\documentclass[11pt, letterpaper]{amsart}
\usepackage{amsthm,amssymb}
\usepackage[dvips, hcentering,
includehead,width=14.2cm,
top=0.5cm,
height=21.6cm,
paperheight=23.8cm,paperwidth=17cm
]{geometry}

\usepackage{tikz-cd}

\usepackage{hyperref}
\usepackage{epsfig}
\usepackage[arrow]{xy}

\usepackage{amssymb,latexsym,cite,epsf,graphics}
\usepackage{pict2e}
\usepackage{graphicx,color}
\usepackage{enumerate}
\usepackage{enumitem}

\usepackage{ccaption}
\changecaptionwidth
\captionwidth{5.6in}

\newtheorem{theorem}{Theorem}[section]

\newtheorem{lemma}[theorem]{Lemma}

\newtheorem{proposition}[theorem]{Proposition}

\newtheorem{corollary}[theorem]{Corollary}

\theoremstyle{remark}

\theoremstyle{definition}

\newtheorem{remark}[theorem]{Remark}
\newtheorem{example}[theorem]{Example}
\newtheorem{definition}[theorem]{Definition}
\newtheorem{problem}[theorem]{Problem}

\def\CC{\mathbb{C}} 
\def\GG{\mathbf{G}}
\def\PP{\mathbb{P}} 
\def\TT{\mathbf{T}}
 
\def\RR{\mathbb{R}} 
\def\VV{\mathbb{V}} 

\def\bb{\mathbf{b}}
\def\cc{\mathbf{c}}
\def\hh{\mathbf{h}}

\def\vv{\mathbf{v}}

\newcommand{\edge}[1]{\arrow[#1, no head]}

\newcommand{\shortedge}[1]{\stackrel{#1}{\rule[.5ex]{1.2em}{0.5pt}}}

\newcommand{\mutation}[1]{\stackrel{#1}{\rule[.5ex]{2em}{0.5pt}}}
\newcommand{\vmutation}[1]{  {\scriptstyle #1}{\Bigl|\Bigr.}  }

\definecolor{darkred}{rgb}{1,0,0}        
\definecolor{lightred}{rgb}{1,0.4,0}     
\definecolor{darkblue}{cmyk}{1,0.4,0,0.4}  
\definecolor{lightblue}{cmyk}{1,0.4,0,0}  
\definecolor{darkgreen}{cmyk}{1,0.5,1,0}  
\definecolor{lightgreen}{cmyk}{1,0,1,0}  
\definecolor{lightgray}{rgb}{0.5,0.5,0.5}  

\newcommand{\red}[1]{{\color{red} #1\color{black}}}
\newcommand{\blue}[1]{{\color{blue} #1\color{black}}}
\newcommand{\green}[1]{{\color{lightgreen} #1}}

\newcommand{\lightgreen}[1]{\color{lightgreen}{#1}\color{black}}

\title{Incidences and tilings}

\numberwithin{equation}{section}

\begin{document}

\author{Sergey Fomin}
\address{\hspace{-.3in} Department of Mathematics, University of Michigan,
Ann Arbor, MI 48109, USA}
\email{fomin@umich.edu}

\author{Pavlo Pylyavskyy}
\address{\hspace{-.3in} Department of Mathematics, University of Minnesota,
Minneapolis, MN 55414, USA}
\email{ppylyavs@umn.edu}

\date{\today}

\thanks{Partially supported by NSF grants DMS-2054231 (S.~F.)
and DMS-1949896 (P.~P.).
}

\subjclass{
Primary
51A20, 
Secondary
05E14, 
14N20, 
51M15. 
}

\keywords{Linear incidence geometry, tiled surface, incidence theorem.}

\begin{abstract}
We show that various classical theorems of real/complex linear incidence geometry, 
such as the theorems of Pappus, Desargues, M\"obius, and so on, 
can be interpreted as special cases of a single ``master theorem''
that involves an arbitrary tiling of a closed oriented surface by quadrilateral tiles. 
This yields a general mechanism for producing new incidence theorems and generalizing the known~ones. 
\end{abstract}

\maketitle


\rightline{\emph{Science is what we understand well enough to explain to a computer.}}

\rightline{\emph{Art is everything else we do.}}
\medskip

\rightline{\emph{-- D.~E.~Knuth \rm\cite{knuth-A+B}}}

\vspace{10pt}
\begin{center}
\includegraphics[scale=0.33, trim=0.1cm 0cm 0cm 0.7cm, clip]{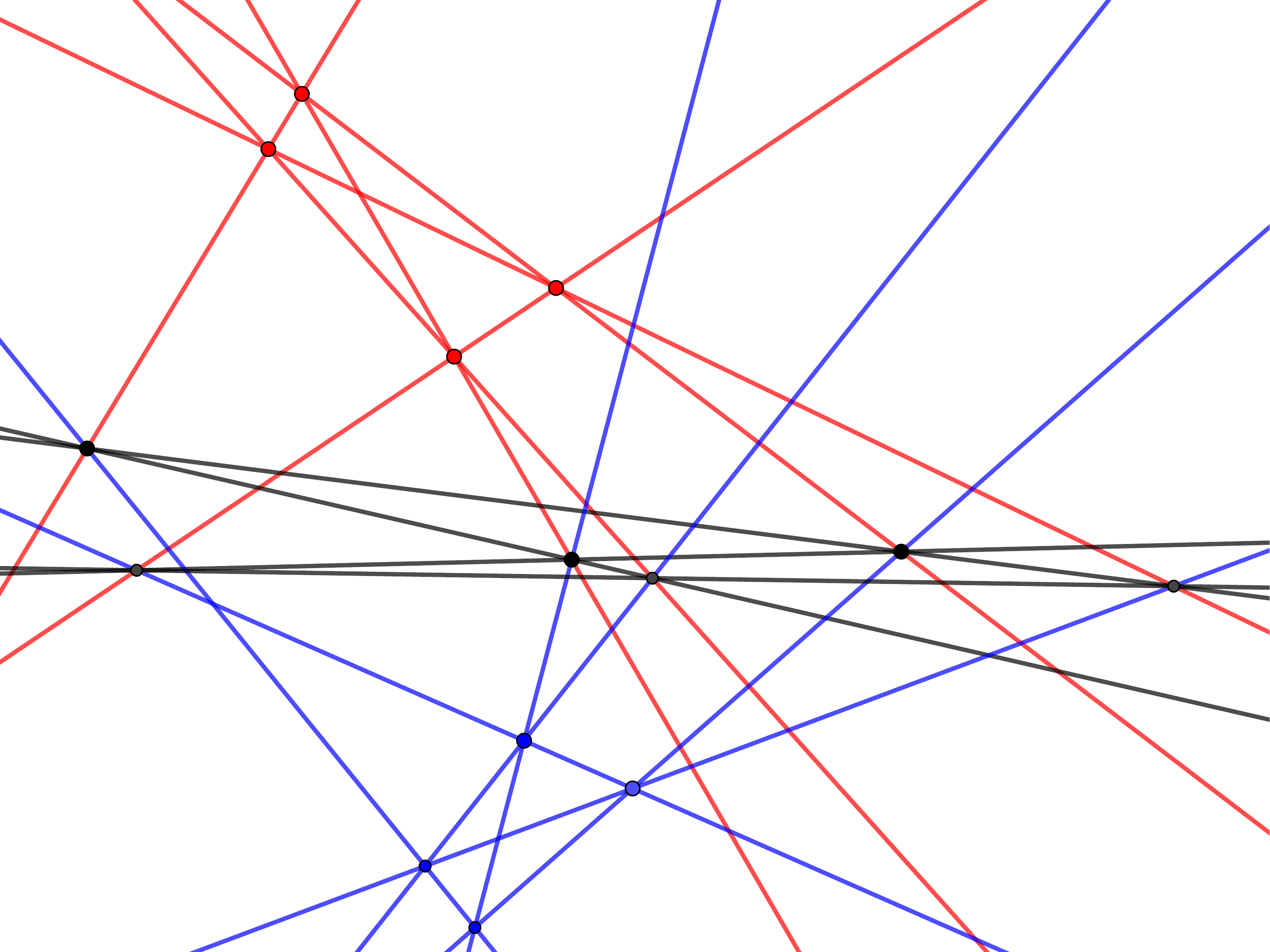}
\end{center}
\vspace{-5pt}


\section{Introduction}

Linear incidence geometry studies configurations of points, lines, and (hyper)planes
from the perspective of their relative position (e.g., whether a point lies on a line, whether two lines in 3-space intersect, etc.).
This venerable subject goes back at least to the times of Euclid, 
with important contributions by Pappus, Desargues, M\"obius, Hilbert, and many~others;
see, e.g., 
\cite{baltus, faghihi, marchisotto, pambuccian-schacht, van-der-waerden-awakening} for historical accounts. 

A typical theorem of (real or complex) linear incidence geometry asserts that, 
given a finite configuration, say of points and lines in the projective plane, 
a particular collection of incidence constraints between these geometric objects 
implies another incidence constraint, under appropriate genericity assumptions. \linebreak[3]
For~compendiums of such incidence theorems, see for example 
\cite{chou-gao-zhang, li-wu, pickert, richter-gebert-mechanical}. 

In recent decades, the advent of computational commutative algebra
led to design and  implementation of highly effective algorithmic approaches to automated proofs 
of incidence theorems, 
see, e.g., \cite{chou-gao-zhang, hongbo-li, li-wu, richter-gebert-mechanical, cinderella, sturmfels-automated}. 
Nowadays any such theorem can be  proved by a computer, with minimal human input. 


These developments, however, leave unanswered a key question:
where do all these incidence theorems come from to begin with?
That is, 
is there a systematic way to \emph{generate} them? 
In this paper, we 
tackle this question by establishing 
a kind of \hbox{``master theorem''} of real/complex linear incidence geometry,
from which various---perhaps all---incidence theorems can be obtained as special cases. 
As a result, we obtain a unifying perspective on
\emph{why} all of these incidence theorems hold. 




To give the reader a quick taste of our master theorem (cf.\ Theorem~\ref{th:master}), 
we formulate right away its simplified version for the case of the projective plane:

\begin{theorem}
\label{th:main-plane}
Consider a tiling of a closed oriented surface by quadrilateral~tiles. 
Assume that the vertices of the tiling are colored black and white, 
so that every edge connects vertices of different color. 
Associate to each black (resp., white) vertex a~point (resp., a line) in the real/complex projective plane, 
so that all these points and lines are distinct and no point lies on any of the lines associated with adjacent vertices. 
For each tile
\begin{equation*}
\begin{tikzpicture}[baseline= (a).base]
\node[scale=1] (a) at (0,0){
\begin{tikzcd}[arrows={-stealth}, sep=small, cramped]
A  \edge{r}  \edge{d}& \ell \edge{d}& \\[3pt]
m \edge{r} & B
\end{tikzcd}
};
\end{tikzpicture}
\end{equation*}
(here $A$ and $B$ are points and $\ell$ and $m$ are lines), consider the incidence condition
\begin{equation*}
\tag{$*$}
\text{the points $A$, $B$, and $\ell\cap m$ are collinear.}
\end{equation*}
If condition {\rm($*$)} holds for all tiles but one, then it also holds for the remaining~tile. 
\end{theorem}

Once this statement has been discovered, the proof is short and straightforward. 
Even so, the theorem turns out to be surprisingly powerful. 
As a proof-of-concept test, 
we consider various widely known theorems of classical real/complex two- and three-dimensional 
linear incidence geometry. 
In each case, we show how to interpret a given incidence theorem 
as a special case of  Theorem~\ref{th:main-plane} (resp., its 3D version). 
In~particular, we present direct tiling-based proofs for the following classical theorems:
\begin{itemize}[leftmargin=.2in]
\item
in dimension two: 
\begin{itemize}[leftmargin=.2in]
\item 
the Desargues theorem (Theorem~\ref{th:desargues}); 
\item
the Pappus theorem (Theorem~\ref{th:pappus}); 
\item
the complete quadrangle  theorem (Theorems~\ref{th:complete-quad} and~\ref{th:harmonic-points-theorem}); 
\item
the permutation theorem (Theorem~\ref{th:perm-thm}); 
\item
Saam's theorems (Theorems~\ref{th:saam-perspectivities} and~\ref{th:saam-5}); 
\item
the Goodman-Pollack  theorem (Theorem~\ref{th:goodman-pollack}); 
\end{itemize} \pagebreak[3]
\item
in dimension three: 
\begin{itemize}[leftmargin=.2in]
\item
the bundle theorem (Theorem~\ref{th:bundle}); 
\item
the sixteen points theorem (Theorem~\ref{th:16-points});
\item
the 
M\"obius theorem and 
the octahedron theorem (Theorems~\ref{th:cube} and~\ref{th:octahedron}). 
\end{itemize}
\end{itemize}
To illustrate, Figure~\ref{fig:bipartite-tilings} shows 
two tilings (of the sphere and the torus, respectively)
that yield the theorems of Desargues and Pappus. 
See Section~\ref{sec:first-applications} 
for detailed explanations. 

It is possible to obtain a given incidence theorem 
from substantially different tilings.
For example, the theorem of Pappus can be obtained from two non-isomorphic tilings of the torus, 
cf.\ Figures~\ref{fig:pappus-torus} and~\ref{fig:pappus2-tiling}. 

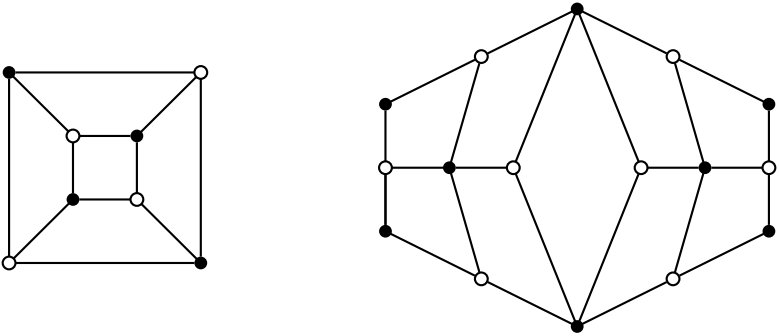
\begin{figure}[ht]
\begin{center}
\vspace{5pt}
\setlength{\unitlength}{1.2pt}
\begin{picture}(60,65)(0,-20)
\thicklines

\put(2,0){{\line(1,0){56}}}
\put(2,60){{\line(1,0){56}}}
\put(22,20){{\line(1,0){16}}}
\put(22,40){{\line(1,0){16}}}

\put(0,2){{\line(0,1){56}}}
\put(60,2){{\line(0,1){56}}}
\put(20,22){{\line(0,1){16}}}
\put(40,22){{\line(0,1){16}}}

\put(1.4,1.4){{\line(1,1){18}}}
\put(58.6,58.6){{\line(-1,-1){18}}}
\put(18.6,41.4){{\line(-1,1){18}}}
\put(41.4,18.6){{\line(1,-1){18}}}

\put(0,0){\circle{4}}
\put(0,60){\circle*{4}}
\put(60,60){\circle{4}}
\put(60,0){\circle*{4}}
\put(20,20){\circle*{4}}
\put(20,40){\circle{4}}
\put(40,20){\circle{4}}
\put(40,40){\circle*{4}}
\end{picture}
\qquad\qquad\qquad
\begin{picture}(120,100)(0,0)
\thicklines

\put(2,50){{\line(1,0){16}}}
\put(22,50){{\line(1,0){16}}}
\put(82,50){{\line(1,0){16}}}
\put(102,50){{\line(1,0){16}}}

\put(0,52){{\line(0,1){16}}}
\put(0,48){{\line(0,-1){16}}}
\put(120,52){{\line(0,1){16}}}
\put(120,48){{\line(0,-1){16}}}
\put(0,48){{\line(0,-1){16}}}
\put(29.4,82.9){{\line(-2,-7){9}}}
\put(29.4,17.1){{\line(-2,7){9}}}
\put(90.6,17.1){{\line(2,7){9}}}
\put(90.6,82.9){{\line(2,-7){9}}}

\put(28,16){{\line(-2,1){28}}}
\put(32,14){{\line(2,-1){28}}}
\put(28,84){{\line(-2,-1){28}}}
\put(32,86){{\line(2,1){28}}}
\put(88,86){{\line(-2,1){28}}}
\put(92,84){{\line(2,-1){28}}}
\put(88,14){{\line(-2,-1){28}}}
\put(92,16){{\line(2,1){28}}}

\put(40.8,52){{\line(2,5){19}}}
\put(40.8,48){{\line(2,-5){19}}}
\put(79.2,52){{\line(-2,5){19}}}
\put(79.2,48){{\line(-2,-5){19}}}

\put(60,0){\circle*{4}}
\put(30,15){\circle{4}}
\put(90,15){\circle{4}}
\put(0,30){\circle*{4}}
\put(120,30){\circle*{4}}
\put(0,50){\circle{4}}
\put(40,50){\circle{4}}
\put(80,50){\circle{4}}
\put(120,50){\circle{4}}
\put(20,50){\circle*{4}}
\put(100,50){\circle*{4}}
\put(0,70){\circle*{4}}
\put(120,70){\circle*{4}}
\put(30,85){\circle{4}}
\put(90,85){\circle{4}}
\put(60,100){\circle*{4}}

\end{picture}
\end{center}
\caption{Tiled surfaces yielding the theorems of Desargues (left) and Pappus (right).
Left: the tiling of the sphere that topologically corresponds to the surface of the cube. 
Right: glue the opposite sides of the shown hexagonal fundamental domain to obtain a tiling of the torus. 
}
\vspace{-15pt}
\label{fig:bipartite-tilings}
\end{figure}

It is tempting, if perhaps too bold, to conjecture that \emph{any} theorem
of linear incidence geometry can be obtained as a \hbox{special case of our master theorem}.
A~result of this kind appears to be out of reach, due to the absence of any reasonable classification
of incidence theorems, even in the case of the projective plane. 
Indeed, the universe of incidence theorems is too ``wild'' to be amenable to 
an explicit description, see 
\cite[Sections 8.3--8.4]{oriented-matroids}, 
\cite[Corollary 1.4]{is-the-missing-axiom-lost}, 
\cite{mnev-configurations, sturmfels-decidability, vamos}, \cite[Theorem~6.5.17]{oxley}. 
On the other hand, such complexity-theoretic results do not necessarily preclude the existence of a universal master theorem. 
That's because the task of identifying a (possibly non-unique) instance of 
a master theorem that yields a given incidence theorem
may not be easily accomplished. 
This is indeed the case for our master theorem:
we do not know an algorithm that inputs an incidence theorem 
and outputs a tiled Riemann surface from which this theorem or its generalization (cf.\ Remark~\ref{rem:remove-redundancies}) 
can be obtained. 
Thus, finding a tiling-based proof for a given incidence theorem remains a form of art rather than science, 
by Knuth's definition in the epigraph.



\smallskip

By design, our master theorem leads to the discovery of new (and generalizations of old)
incidence theorems. 
See, in particular, Theorems~\ref{th:complete-quad-generalization}, \ref{th:planar-bundle-generalized}, 
\ref{th:saam-perspectivities}, \ref{th:dual-stars-of-david}, \ref{th:5-points-theorem}, \ref{th:hexagon-piercing}, 
\ref{th:3d-pappus}, \ref{th:generalized-perm-theorem}, and~\ref{th:pappus-split1}. 

\smallskip

We refer the reader to \cite{artin, baker, casas-alvero, coxeter-non-euclidean, coxeter-projective, hilbert-foundations, hilbert-cohn-vossen, moufang, pickert, richter-gebert-book, seidenberg-book} for general background in classical projective geometry. 

\smallskip

\pagebreak[3]

To simplify exposition and  curb the size of the paper, we restrict our treatment in several ways. 
First, while we realize the importance of more general frameworks of axiomatic projective geometry, 
we decided to focus in this paper on the classical setting of complex or real projective spaces. 
Many of our results can be generalized to geometries over other fields and even noncommutative skew fields. 
Second, we confine our treatment to configurations that consist exclusively of projective subspaces 
and do not involve conics/quadrics or
varieties of higher degree.
In~future work, we plan to extend our framework to more general incidence theorems. 

Third, we do not attempt to formulate the weakest genericity assump\-tions under which each of our  theorems holds. 
This somewhat narrows, although not in an essential way, the class of incidence theorems under consideration. 
We habitually use the term ``generic'' to describe choices (of points, lines, planes, etc.) that belong to a Zariski dense subset 
of the appropriate configuration space.  

The fact that we confine our treatment to ``generic forms'' of incidence theorems 
might provide another explanation for why Mn\"ev universality 
does not preclude the existence of a comprehensive master theorem of linear incidence geometry. 


While we don't make it explicit, our approach is rooted in \hbox{\emph{classical invariant theory}}. 
From the perspective of this theory, 
incidence geometry can be viewed as~a study of particular kinds of identities in certain  rings of invariants. 
For example, collinearity of triples of points on the projective plane can be encoded by the vanishing of~the corresponding Pl\"ucker coordinates ($3\times 3$ minors of a $3\times n$ matrix). 
Various identities that Pl\"ucker coordinates satisfy lead to representability restrictions for the corresponding matroids. 
These identities and restrictions can be very complicated, as the aforementioned 
complexity/universality results attest. 
To bypass these difficulties,~our approach implicitly relates to a different ring of invariants, namely
the ring generated by pairings between vectors and covectors. 
In the case of $\operatorname{SL}_3$ (i.e., the case of the projective plane), 
the ideal of relations among these pairings is generated by $4\times 4$ 
Gramian determinants, each of which has 24 terms of degree~4. 
(By comparison, the \linebreak[3]
Grassmann-Pl\"ucker relations for Pl\"ucker coordinates have degree~2 and involve just 3 terms.) 
It seems reasonable to ignore these Gramian relations and treat the pairings \linebreak[3]
as if they were algebraically independent. 
Identity verification becomes easy, 
while the difficulty shifts towards reformulating incidence theorems in terms of such pairings. 

\smallskip

\enlargethispage{.2cm}

The paper is organized as follows. 
Sections \ref{sec:master-theorem}--\ref{sec:3D-incidence-geometry} constitute the core of the paper. 
In Section~\ref{sec:master-theorem}, \hbox{we establish the general} form of our master theorem 
(Theorem~\ref{th:master}). 
In Sections~\ref{sec:first-applications}--\ref{sec:3D-incidence-geometry},
we demonstrate the power of this theorem
by providing numerous applications to linear incidence geometry over the real or complex field,
first in the plane (Sections~\ref{sec:first-applications}--\ref{sec:coherent-polygons})
and then in 3-space (Section~\ref{sec:3D-incidence-geometry}). 
We also provide (see Section~\ref{sec:first-applications})
an interpretation of our approach 
in the more conventional language of Levi graphs. 


Section~\ref{sec:combinatorial-reformulations} is devoted to combinatorial reformulations of our master theorem. 
The first reformulation uses nodal (multi-)curves on an oriented surface~$\Sigma$; 
such curves are in one-to-one correspondence with quadrilateral tilings of~$\Sigma$. 
The second reformulation involves graphs properly embedded into~$\Sigma$,
so that each face is simply-connected. 
This generalizes 
a beautiful construction proposed by D.~G.~Glynn~\cite{glynn}.

In Section~\ref{sec:deducing-new-incidence-theorems}, 
we discuss ways to obtain new incidence theorems from existing~ones using the tiling technique. 
In particular, we show that \emph{any} incidence theorem associated to a tiling~$\TT$ can be generalized 
in multiple ways by inserting new tiles into~$\TT$. 

\pagebreak[3]

Several geometric corollaries of the master theorem are presented in Section~\ref{sec:geometric-ramifications}.
We~explain how an arbitrary triangulation of an oriented surface gives rise to an incidence theorem
in the plane. 
Both Desargues' and Pappus' theorems can be obtained in this way. 
A three-dimensional analogue of this result yields the M\"obius theorem.  
We also show that when the dual graph of a triangulation is Hamiltonian,
the corresponding incidence theorem can be presented as a closure porism (Schlie\ss ungs\-satz).  

In Section~\ref{sec:consistency}, we prove that our tilings exhibit 3D/4D consistency in the sense of 
Bobenko-Suris~\cite{bobenko-suris-book} 
and provide solutions of Zamolodchikov's tetrahedron equation. 

Variations of our main construction are discussed in Sections~\ref{sec:anticoherent}--\ref{sec:generalizations}, where we in particular outline  
connections with Ceva's theorem, harmonic quadruples, 
basic Schubert Calculus, and circles on the M\"obius plane.

\smallskip

\section*{Acknowledgments}

We took inspiration from the work of J\"urgen Richter-Gebert~\cite{richter-gebert-mechanical, richter-gebert-ceva} and 
Bernd Sturm\-fels~\cite{sturmfels-automated, sturmfels-algorithms-in-invariant-theory}
on invariant-theoretic approaches to incidence geometry
and the work of Alexander Bobenko and Yuri Suris~\cite{bobenko-suris, bobenko-suris-book}
on integrable systems on \hbox{quad-graphs}.

We thank Alexander Barvinok and Michael Shapiro for enlightening discussions. 
We are grateful to Joe Buhler and Jeff Lagarias for their comments on the earlier version of the paper. 

The main results of this paper were presented at the conference ``Cluster algebras and Poisson geometry''
(Levico Terme, Italy) in June 2023.
We thank the organizers and participants of this conference for their substantive feedback. 

We used \texttt{GeoGebra} for drawing point-and-line configurations. 


\section{The master theorem}
\label{sec:master-theorem}


Let $\PP$ be a real or complex finite-dimensional projective space of dimension~${\ge2}$. 
(In this paper, we focus on applications where $\dim\PP=2$ or $\dim\PP=3$, i.e., $\PP$~is a plane or a 3-space.)  

We denote by $\PP^*$ the set of hyperplanes in~$\PP$. 
In particular, when $\PP$ is a plane, the elements of $\PP^*$ are lines. 
A~point $A\in \PP$ and a hyperplane $\ell\in\PP^*$ are called \emph{incident} to each other if $A\in\ell$. 

We denote by $(AB)$ (resp., $(ABC)$) the line passing through two distinct points $A$ and~$B$ 
(resp., the plane passing through distinct points $A, B, C$).  

\begin{definition}
\label{def:coherent-tile}
Throughout this paper, a \emph{tile} is a topological quadrilateral 
(that is, a closed oriented disk with four marked points on its boundary)
whose vertices are clockwise labeled $A, \ell, B, m$, 
where $A,B\in\PP$ are points and $\ell,m\in\PP^*$ are hyperplanes:
\begin{equation}
\label{eq:AlBm}
\begin{tikzpicture}[baseline= (a).base]
\node[scale=1] (a) at (0,0){
\begin{tikzcd}[arrows={-stealth}, cramped, sep=12]
A  \edge{r}  \edge{d}& \ell \edge{d} \\[2pt]
m \edge{r} & B
\end{tikzcd}
};
\end{tikzpicture}
\end{equation}
Such a tile is called \emph{coherent} if 
\begin{itemize}[leftmargin=.2in]
\item 
neither $A$ nor $B$ is incident to either $\ell$ or~$m$; and 
\item
either $A=B$ or $\ell=m$ or
else the line $(AB)$ and the codimension~$2$ subspace $\ell\cap m$ have a nonempty intersection. 
\end{itemize}
\end{definition}

\begin{remark}
\label{rem:coherence-plane}
In the case of the projective plane ($\dim\PP=2$),
a coherent tile involves two points $A,B$ and two lines $\ell, m$ not incident to them  
such that either $A=B$ or $\ell=m$ or
else the line $(AB)$ passes through the point $\ell\cap m$.
See Figure~\ref{fig:tile}. 
\end{remark}

\begin{figure}[ht]
\begin{center}
\includegraphics[scale=0.55, trim=0cm 7.5cm 2cm 0.5cm, clip]{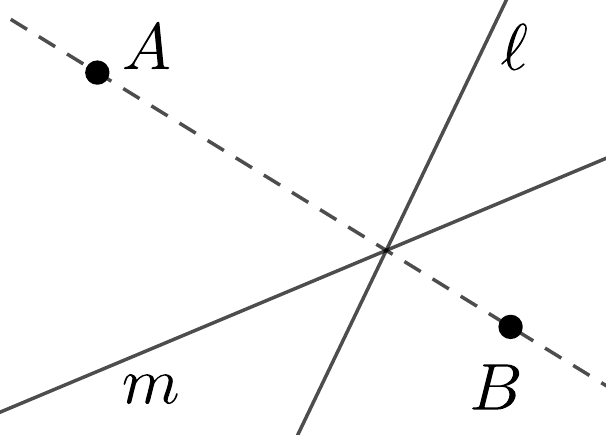}
\end{center}
\vspace{-5pt}
\caption{Definition of a coherent tile. 
}
\label{fig:tile}
\end{figure}


\begin{remark}
The notion of coherence introduced in Definition~\ref{def:coherent-tile} 
does not depend on the (interlacing) linear order in which the points $A,B$ and the lines $\ell,m$ are listed.
That is, rotating or reflecting a tile does not affect its coherence.  
\end{remark}

The notion of coherence can be reformulated algebraically in terms of \emph{cross-ratios}.
This will require a bit of preparation. 
Assume that~$\PP$ is the projectivization of a real or complex vector space~$\VV$,
so that points in $\PP$ are identified with one-dimensional subspaces in~$\VV$.
Hyperplanes in~$\PP$ correspond to codimension~1 subspaces of~$\VV$, so they 
can be identified with one-dimensional subspaces in the dual space~$\VV^*$.


For a vector $\mathbf{V}\!\in\!\VV$ and a covector $\hh\!\in\!\VV^*$, let $\langle \mathbf{V},\hh\rangle$ denote their pairing. 

\begin{definition}
\label{def:mixed-cross-ratio}
Let $\mathbf{A}, \mathbf{B}\!\in\!\VV$ be vectors and let $A,B\!\in\!\PP$ be the corresponding~points.
Let $\mathbf{\boldsymbol \ell}, \mathbf{m}\in\VV^*$ be covectors and let $\ell,m\in\PP^*$ be the corresponding hyperplanes. 
Assume that all four pairings 
$\langle \mathbf{A}, \mathbf{\boldsymbol \ell} \rangle$,
$\langle \mathbf{A}, \mathbf{m} \rangle$, 
$\langle \mathbf{B}, \mathbf{\boldsymbol \ell} \rangle$, 
$\langle \mathbf{B}, \mathbf{m} \rangle$
are nonzero;
equivalently, neither $A$ nor $B$ is incident to $\ell$~or~$m$.
The \emph{mixed cross-ratio} $(A,B;\ell,m)$ is defined~by 
\begin{equation}
\label{eq:mixed-cross-ratio}
(A,B;\ell,m)=
\dfrac{\langle \mathbf{A}, \mathbf{\boldsymbol \ell} \rangle\langle \mathbf{B}, \mathbf{m} \rangle}{\langle \mathbf{A}, \mathbf{m} \rangle\langle \mathbf{B}, \mathbf{\boldsymbol \ell} \rangle}. 
\end{equation}
We note that $(A,B;\ell,m)$ does not depend on the choice of vectors 
$\mathbf{A}, \mathbf{B}$ and covectors $\mathbf{\boldsymbol \ell}, \mathbf{m}$
representing the points 
 $A,B$ and the hyperplanes $\ell,m$, respectively. 
In fact, 
\begin{equation}
\label{eq:ABlm=ABLM}
(A,B;\ell,m)=(A, B; L, M), 
\end{equation}
the ordinary cross-ratio of four collinear points $A$, $B$, $L\!=\!(AB)\cap\ell$ and $M\!=\!(AB)\cap m$. 
\end{definition}

The instances $\dim\PP=2$ and $\dim\PP=3$ of the notion of mixed cross-ratio 
has been studied since at least early 20th century, 
see \cite[(4.49), (4.71)]{coxeter-non-euclidean} and references therein. 

\begin{proposition}
\label{pr:coherence-algebraic}
Let $A,B\in\PP$ be points and $\ell,m\in\PP^*$ be hyperplanes such that neither $A$ nor $B$ is incident to $\ell$~or~$m$. 
The tile
\begin{equation}
\label{eq:AlBm-again}
\begin{tikzpicture}[baseline= (a).base]
\node[scale=1] (a) at (0,0){
\begin{tikzcd}[arrows={-stealth}, sep=small, cramped]
A  \edge{r}  \edge{d}& \ell \edge{d} \\[3pt]
m \edge{r} & B
\end{tikzcd}
};
\end{tikzpicture}
\end{equation}
(cf.\ \eqref{eq:AlBm}) is coherent if and only if 
\begin{equation}
\label{eq:cross-ratio=1}
(A,B;\ell,m)=1.
\end{equation}
\end{proposition}

\begin{proof}
If the tile \eqref{eq:AlBm-again} is coherent, 
then the line $(AB)$ passes through~a point lying on $\ell\cap m$ (cf.\ Figure~\ref{fig:tile}).
This implies that
$\frac{\langle \mathbf{A}, \mathbf{\boldsymbol \ell} \rangle}{\langle \mathbf{A}, \mathbf{m} \rangle}
\!=\!
\frac{\langle \mathbf{B}, \mathbf{\boldsymbol \ell} \rangle}{\langle \mathbf{B}, \mathbf{m} \rangle}$,
and~\eqref{eq:cross-ratio=1} follows, cf.\ \eqref{eq:mixed-cross-ratio}. 

On the other hand, if the tile \eqref{eq:AlBm-again} is not coherent, then 
the points $L=(AB)\cap\ell$ and $M=(AB)\cap m$ on the line~$(AB)$ are different from each other. 
Consequently $\frac{\langle \mathbf{A}, \mathbf{\boldsymbol \ell} \rangle}{\langle \mathbf{A}, \mathbf{m} \rangle}
\neq
\frac{\langle \mathbf{B}, \mathbf{\boldsymbol \ell} \rangle}{\langle \mathbf{B}, \mathbf{m} \rangle}$, 
contradicting~\eqref{eq:cross-ratio=1}. 
\end{proof}

We are now prepared to state and prove our master theorem.

\begin{theorem}
\label{th:master}
Consider a tiling of a closed oriented surface by quadrilateral~tiles. 
Assume that the vertices of the tiling are colored black and white, 
so that every edge connects vertices of different color. 
Associate to each black (resp., white) vertex a point (resp., a hyper\-plane) in the real/complex 
finite-dimensional projective space~$\PP$. 
Assume that for each edge $A-h$ of the tiling, the point~$A$ does not lie on the hyperplane~$h$. 
If all tiles but one are coherent, then the remaining tile is coherent as~well. 
\end{theorem}

\begin{proof}
The theorem follows by combining Proposition~\ref{pr:coherence-algebraic}
with the observation that the product of mixed cross-ratios $(A,B;\ell,m)$ over all
tiles \eqref{eq:AlBm}/\eqref{eq:AlBm-again} in the tiling~is equal to~1. 
To see why the latter statement is true, replace each point~$P\in\PP$ (resp., hyperplane~$h\in\PP^*$) appearing in the tiling
by an appropriate vector~$\mathbf{p}\in\VV$
(resp., covector~$\mathbf{h}\in\VV^*$).
Then for each edge $P\!-\!h$ in the tiling, 
the pairing $\langle \mathbf{P}, \mathbf{h}\rangle$ appears in the numerator (resp., denominator) of the mixed cross-ratio 
for the unique tile in which $P$ immediately precedes~$h$
in the clockwise (resp., counterclockwise) traversal of the boundary of the tile.
\end{proof}

\begin{remark}
In the case $\dim\PP=2$ of Theorem~\ref{th:master}, we recover Theorem~\ref{th:main-plane}. 
\end{remark}

\begin{remark}
There exist tilings of oriented surfaces by quadrilateral tiles 
whose vertices cannot be properly colored in two colors, 
as our construction demands. 
See Figure~\ref{fig:non-bipartite-tiling} for a (counter)example. 

On the other hand, any tiling of a \emph{sphere} by quadrilateral tiles
necessarily has a requisite 2-coloring. 
Indeed, in the spherical case, every simple cycle in the 1-skeleton of the tiling 
can be paved by tiles, hence it cannot be odd. 
\end{remark}

When we want to emphasize that a given tiling~$\TT$ of an oriented surface
comes with a proper coloring of its vertices in two colors
(as in Theorem~\ref{th:master}), we say that~$\TT$ is a \emph{bicolored quadrilateral tiling}. 

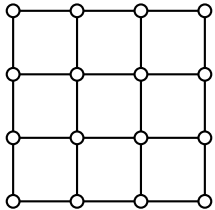
\begin{figure}[ht]
\begin{center}
\vspace{5pt}
\setlength{\unitlength}{1.2pt}
\begin{picture}(60,62)(0,0)
\thicklines

\multiput(2,0)(0,20){4}{{\line(1,0){16}}}
\multiput(22,0)(0,20){4}{{\line(1,0){16}}}
\multiput(42,0)(0,20){4}{{\line(1,0){16}}}

\multiput(0,2)(20,0){4}{{\line(0,1){16}}}
\multiput(0,22)(20,0){4}{{\line(0,1){16}}}
\multiput(0,42)(20,0){4}{{\line(0,1){16}}}

\multiput(0,0)(20,0){4}{\circle{4}}
\multiput(0,20)(20,0){4}{\circle{4}}
\multiput(0,40)(20,0){4}{\circle{4}}
\multiput(0,60)(20,0){4}{\circle{4}}

\end{picture}
\end{center}
\caption{A tiling of the torus that does not yield an incidence theorem. 
(The opposite sides of the shown fundamental domain should be glued to each other.)  
The 1-skeleton of this tiling is not bipartite since it contains 3-cycles. 
}
\vspace{-15pt}
\label{fig:non-bipartite-tiling}
\end{figure}

\newpage

\section*{Trivial instances of the master theorem}

A bicolored quadrilateral tiling of an oriented surface must be sufficiently ``ample'' 
in order for the associated incidence theorem to be nontrivial. 
We illustrate this by several examples of tilings whose associated theorems amount to tautologies. 

\begin{example}
\label{eg:tiling-tautology-1}
Tile a sphere~$S^2$ by two quadrilateral tiles by drawing a circle on~$S^2$ 
and splitting this circle into a 4-cycle. (That is, place four vertices onto the circle.)  
The resulting two tiles will be mirror images of each other:
\begin{equation}
\label{eq:mirror-images}
\begin{tikzpicture}[baseline= (a).base]
\node[scale=1] (a) at (0,0){
\begin{tikzcd}[arrows={-stealth}, sep=small, cramped]
A  \edge{r}  \edge{d}& \ell \edge{d} \\[3pt]
m \edge{r} & B
\end{tikzcd}
};
\end{tikzpicture}
\qquad \qquad
\begin{tikzpicture}[baseline= (a).base]
\node[scale=1] (a) at (0,0){
\begin{tikzcd}[arrows={-stealth}, sep=small, cramped]
\ell  \edge{r}  \edge{d}& A \edge{d} \\[3pt]
B \edge{r} & m
\end{tikzcd}
};
\end{tikzpicture}
\end{equation}
If one of these tiles is coherent, then so is the other---which is obvious because the two coherence statements coincide. 
\end{example}

\begin{example}
\label{eg:tiling-tautology-2}
Split a $2\times 2k$ square into $4k$ unit squares and glue its opposite sides to each other 
to get a tiling of the torus by $4k$ tiles, shown below for $k=1, 2$:
\begin{equation}
\label{eq:torus-2x2k}
\begin{tikzpicture}[baseline= (a).base]
\node[scale=1] (a) at (0,0){
\begin{tikzcd}[arrows={-stealth}, sep=small, cramped]
A \edge{r}  \edge{d}& \ell \edge{r} \edge{d} & A \edge{d} \\
m \edge{r} \edge{d} & B \edge{d} \edge{r} & m \edge{d} \\
A \edge{r} & \ell \edge{r} & A
\end{tikzcd}
};
\end{tikzpicture}
\qquad\qquad
\begin{tikzpicture}[baseline= (a).base]
\node[scale=1] (a) at (0,0){
\begin{tikzcd}[arrows={-stealth}, sep=small, cramped]
A \edge{r}  \edge{d}& \ell \edge{r} \edge{d} & C \edge{d} \edge{r} & p \edge{r} \edge{d} & A \edge{d} \\
m \edge{r} \edge{d} & B \edge{d} \edge{r} & n \edge{d} \edge{r} & D \edge{r} \edge{d} & m \edge{d} \\
A \edge{r} & \ell \edge{r} & C \edge{r} & p \edge{r} & A
\end{tikzcd}
};
\end{tikzpicture}
\end{equation}
The associated incidence theorem is vacuous since the coherence statements encoded by the tiles appearing in the same column are identical to each other. 
Thus, if all tiles but one are coherent, then the remaining tile is coherent just because the conclusion of this implication is contained among its conditions. 
\end{example}

\begin{example}
\label{eg:tiling-tautology-3}
The incidence theorem associated with the tiling of the torus 
\begin{equation}
\label{eq:torus-XX}
\begin{tikzpicture}[baseline= (a).base]
\node[scale=1] (a) at (0,0){
\begin{tikzcd}[arrows={-stealth}, sep=10, cramped]
a \edge{rr, dotted} \edge{rd} \edge{dd, dotted} && b \edge{rr, dotted} \edge{rd} \edge{ld} && a \edge{ld} \edge{dd, dotted} \\
& C \edge{rd} \edge{ld} && D \edge{rd} \edge{ld} \\
a \edge{rr, dotted} && b \edge{rr, dotted} && a
\end{tikzcd}
};
\end{tikzpicture}
\end{equation}
(as above, the opposite sides of the rectangular fundamental domain should be glued to each other)
is trivial for a different reason: 
each of the four tiles has two vertices with the same label, so all four  
coherence statements are vacuous. 
\end{example}

\begin{example}
\label{eg:tiling-3-tiles}
The incidence theorem associated with the tiling of the sphere 
\begin{equation}
\label{eq:sphere-3-tiles}
\begin{tikzpicture}[baseline= (a).base]
\node[scale=1] (a) at (0,0){
\begin{tikzcd}[arrows={-stealth}, sep=5, cramped]
A \edge{rr} \edge{dd} && \ell \edge{ld} \edge{dd} \\
& B \edge{ld} \\
m \edge{rr} && C
\end{tikzcd}
};
\end{tikzpicture}
\end{equation}
is rather shallow. 
If $\ell=m$ or the points $A,B,C$ are not distinct, then the statement is vacuous. 
Otherwise, the theorem says that if $A,B,C\notin\ell$ and
the point $D=\ell\cap m$ lies on the lines $(AB)$ and $(BC)$, then $D\in (AC)$. 
This is clearly true because $B\neq D$ and therefore $(AB)=(BC)=(AC)$. 
\end{example}

\newpage

\section{First applications
}
\label{sec:first-applications}

In the words of S.~A.~Amitsur \cite[Section~A.2]{amitsur-PI}, 
\emph{the three main theorems in (linear) projective geometry are the complete quadrilateral theorem,  
the Desargues theorem, and the Pappus theorem.}
We begin by explaining how each of these three theorems can be obtained 
by direct application of our master theorem (Theorem~\ref{th:main-plane}/\ref{th:master}). 

\section*{The Desargues theorem}

\begin{theorem}[G.~Desargues, \emph{ca.}\ 1639] 
\label{th:desargues}
Let $a,b,c$ be distinct concurrent lines~on the complex/real projective plane. 
Pick generic points \hbox{$A_1, A_2\!\in\! a$, $B_1, B_2\!\in\! b$, $C_1, C_2\!\in\! c$}.
Then the points $A=(B_1C_1)\cap (B_2C_2)$, 
$B=(A_1C_1)\cap (A_2C_2)$,
$C=(A_1B_1)\cap (A_2C_2)$
are collinear. See Figure~\ref{fig:desargues-geogebra}. 
\end{theorem}

\begin{figure}[ht]
\begin{center}
\vspace{-.1in}
\includegraphics[angle=0, scale=0.5, trim=0cm 1cm 2cm 0cm, clip]{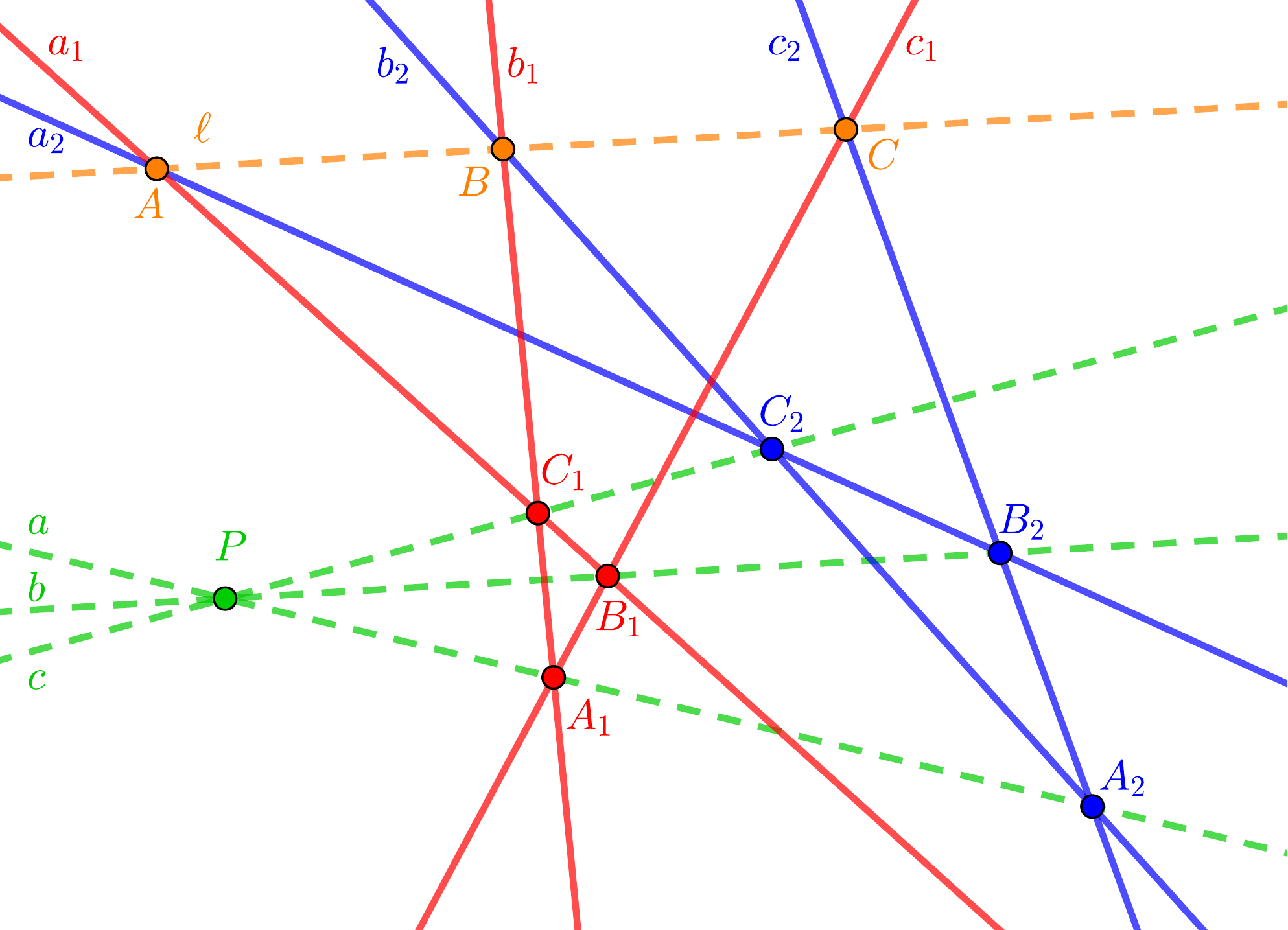}
\end{center}
\vspace{-.1in}
\caption{The Desargues theorem. 
}
\label{fig:desargues-geogebra}
\end{figure}

\vspace{-.2in}

\begin{proof}
Denote
\begin{align}
\nonumber
a&=(A_1A_2), \quad \ b=(B_1B_2),\quad \ c=(C_1C_2), \\[5pt]
\label{eq:a1b1c1}
a_1&=(B_1C_1), \quad b_1=(A_1C_1), \quad c_1=(A_1B_1), \\[5pt]
\label{eq:a2b2c2}
a_2&=(B_2C_2), \quad b_2=(A_2C_2), \quad c_2=(A_2B_2),
\end{align}
so that
$A=a_1\cap a_2$, $B=b_1\cap b_2$, $C=c_1\cap c_2$. 

Let us formulate the pertinent conditions in terms of $a_1, a_2, A_1, A_2, b, c, B, C$: 
\begin{itemize}[leftmargin=.2in]
\item 
$A_1$, $A_2$, and $b\cap c$ are collinear (given); 
\item 
$A_1$, $B$, and $a_1\cap c$ are collinear (given); 
\item
$A_1$, $C$, and $a_1\cap b$ are collinear (given); 
\item 
$A_2$, $B$, and $a_2\cap c$ are collinear (given); 
\item
$A_2$, $C$, and $a_2\cap b$ are collinear (given); 
\item
$B$, $C$, and $a_1\cap a_2$ are collinear (to be proved). 
\end{itemize}
These six conditions correspond to the six tiles in the tiling of the sphere shown in~Figure~\ref{fig:desargues}. 
The claim now follows by the master theorem.
\end{proof}

\begin{figure}[ht]
\vspace{-10pt}
\begin{equation*}
\begin{tikzpicture}[baseline= (a).base]
\node[scale=1] (a) at (0,0){
\begin{tikzcd}[arrows={-stealth}, cramped, sep=15]
A_2 \edge{rrr} \edge{rd} \edge{ddd}&&& a_2  \edge{ddd}\edge{ld}\\[0pt]
& b  \edge{r}  \edge{d}& C \edge{d}& \\[0pt]
& A_1 \edge{r} \edge{ld} & a_1 \edge{rd}& \\[0pt]
c \edge{rrr} & & & B
\end{tikzcd}
};
\end{tikzpicture}
\end{equation*}
\vspace{-18pt}
\caption{The tiling of the sphere used in the proof of the Desargues theorem. 
}
\label{fig:desargues}
\end{figure}
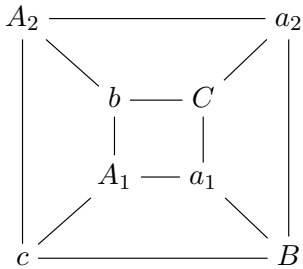

\vspace{-.15in}

For another version of essentially the same argument, see the proof of Corollary~\ref{cor:tileable-hexagons}. 

A more combinatorial viewpoint towards incidence geometry of point-and-line configurations in the projective plane, 
going back at least to the work of H.~S.~M.~Coxeter, 
involves the notion of a Levi graph (or incidence graph); 
see, e.g.,~\cite{pisanski-servatius} for an excellent introduction to this approach.
Let $\mathcal{C}$ be a configuration of points and lines in the real/complex projective plane. 
(In this paper, we use the word \emph{configuration} to describe a collection of geometric objects,
with no \emph{a~priori} combinatorial constraints.) 
The \emph{Levi graph} of~$\mathcal{C}$ 
is a bipartite graph whose vertices correspond to the points and lines of~$\mathcal{C}$
and whose edges record the incidences between them:
whenever a point $P$ lies on a line~$\ell$, the Levi graph includes the edge $P-\ell$. 
(We will use this notion a bit loosely, ignoring ``accidental'' incidences that do not affect the validity of the corresponding
incidence theorems; these accidents can be eliminated by imposing appropriate genericity assumptions.)

Let us illustrate this viewpoint in the case of the Desargues theorem.
In light of Definition~\ref{def:coherent-tile}/Remark~\ref{rem:coherence-plane},
we replace each tile 
\begin{equation*}
\begin{tikzpicture}[baseline= (a).base]
\node[scale=1] (a) at (0,0){
\begin{tikzcd}[arrows={-stealth}, sep=15, cramped]
A  \edge{r}  \edge{d}& \ell \edge{d} \\[0pt]
m \edge{r} & B
\end{tikzcd}
};
\end{tikzpicture}
\end{equation*}
by the ``quadripod''
\begin{equation}
\label{eq:quadripod}
\begin{tikzpicture}[baseline= (a).base]
\node[scale=1] (a) at (0,0){
\begin{tikzcd}[arrows={-stealth}, cramped, sep=10]
A  \edge{drr}&&& \ell \edge{dll} \\[2pt]
& N \edge{r} \edge{ld} & c \edge{rd} \\
m  &&& B
\end{tikzcd}
};
\end{tikzpicture}
\end{equation}
which encodes the coherence of the tile in terms of the Levi graph:
we want the point $N=\ell\cap m$ to lie on the line $c=(AB)$. 
(Note that the edges $A-\ell-B-m-A$ from the boundary of the tile do not appear in the Levi graph,
as they do not correspond to incidences in the configuration.)

Applying this procedure to the tiling in Figure~\ref{fig:desargues},
we obtain the Levi graph of the Desargues configuration, shown in Figure~\ref{fig:desargues-graph-from-tiling}. 
An alternative rendering of this graph, matching the presentation given in \cite[Fig.~5.25]{pisanski-servatius}, 
is given in Figure~\ref{fig:desargues-graph}. 

\pagebreak

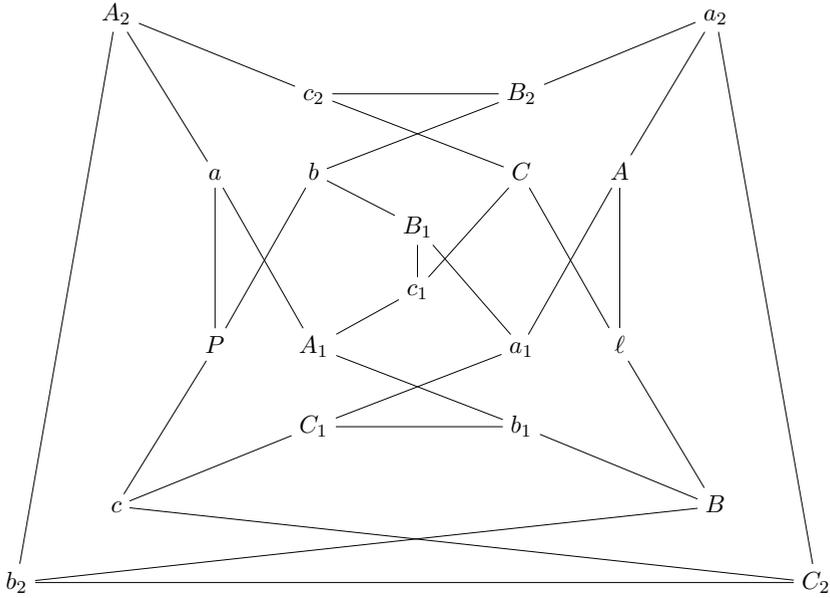
\begin{figure}[ht]
\vspace{-5pt}
\begin{equation*}
\begin{tikzpicture}[baseline= (a).base]
\node[scale=0.85] (a) at (0,0){
\begin{tikzcd}[arrows={-stealth}, cramped]
& A_2 \edge{ldddddddd} \edge{rdd} \edge{rrd} &&&&&& a_2 \edge{lld} \edge{ldd} \edge{rdddddddd} \\
&&& c_2 \edge{rr}  \edge{rrd} && B_2 \edge{lld} \\
&& a \edge{ddd} \edge{rddd} & b \edge{lddd} \edge{rd} && C \edge{ldd} \edge{rddd} & A \edge{lddd} \edge{ddd} \\[-10pt]
&&&& B_1 \edge{d} \edge{rdd} \\[-5pt]
&&&& c_1 \edge{ld} \\[-10pt]
&& P \edge{ldd} & A_1 \edge{rrd} && a_1 \edge{lld} & \ell \edge{rdd} \\
&&& C_1 \edge{rr} \edge{lld} && b_1 \edge{rrd} \\
& c \edge{drrrrrrr} &&&&&& B \edge{dlllllll} \\
b_2 \edge{rrrrrrrr} &&&&&&&& C_2
\end{tikzcd}
};
\end{tikzpicture}
\end{equation*}
\vspace{-5pt}
\caption{Obtaining the Levi graph of the Desargues configuration from the tiling in Figure~\ref{fig:desargues}. 
For example, at the top of the picture, the 5~edges incident to the vertices $c_2$ and/or~$B_2$ come from the coherent tile with vertices $A_2, a_2, C, b$ 
at the top of Figure~\ref{fig:desargues}. 
}
\label{fig:desargues-graph-from-tiling}
\end{figure}

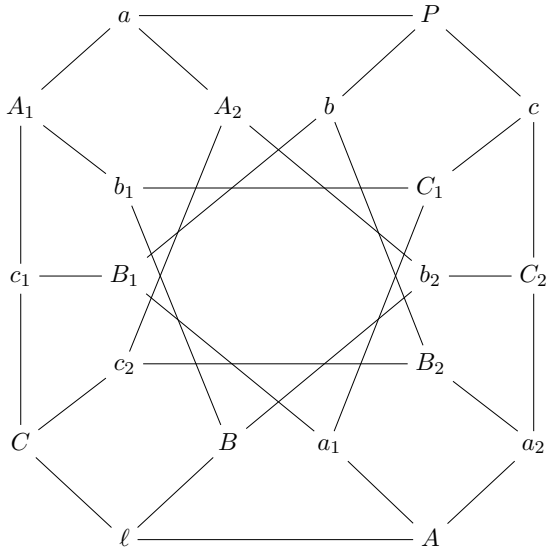
\begin{figure}[ht]
\vspace{-.25in}
\begin{equation*}
\begin{tikzpicture}[baseline= (a).base]
\node[scale=0.85] (a) at (0,0){
\begin{tikzcd}[arrows={-stealth}, cramped]
& a \edge{rrr} \edge{rd} \edge{ld}&&& P  \edge{dr}\edge{dl}\\[.1in]
A_1 \edge{rd} \edge{dd}&& A_2 \edge{lddd} \edge{rrdd}& b \edge{lldd} \edge{rddd}&& c\edge{ld}\edge{dd}\\[.0in]
& b_1 \edge{rrr} \edge{rddd}&&& C_1 \edge{lddd} \\[.05in]
c_1 \edge{r} \edge{dd}& B_1 \edge{rrdd}&&& b_2 \edge{r} \edge{lldd} & C_2 \edge{dd}\\[.05in]
& c_2 \edge{rrr} \edge{ld}&&& B_2 \edge{rd}\\[.0in]
C \edge{rd} && B \edge{ld} & a_1 \edge{rd} && a_2 \edge{ld} \\[.1in]
& \ell \edge{rrr} &&& A
\end{tikzcd}
};
\end{tikzpicture}
\end{equation*}
\vspace{-.15in}
\caption{Alternative rendering of the Levi graph of the Desargues configuration. 
}
\label{fig:desargues-graph}
\end{figure}

\vspace{-.1in}

We will soon see that, in general, additional steps may be required to get the Levi graph 
of an incidence theorem from the corresponding tiling.
This can be seen from the simple fact that each tile gives rise to 5~edges 
whereas the number of edges in a Levi graph of an incidence theorem does not have to be divisible by~5.

\clearpage

\section*{The Pappus theorem}

We consider the following version of the classical Pappus theorem. 

\begin{theorem}[Pappus of Alexandria, around 340 CE]
\label{th:pappus}
Let $P_1,P_2,P_3,P_4,P_5,P_6$ be six generic points on the real/complex plane. 
Then any two of the following three conditions (cf.\ Figure~\ref{fig:pappus1}) imply the third:
\begin{itemize}[leftmargin=.2in]
\item 
the lines $(P_1P_2)$, $(P_3P_4)$, $(P_5P_6)$ intersect at a common point~$A$; 
\item
the lines $(P_2P_3)$, $(P_4P_5)$, $(P_6P_1)$ intersect at a common point~$B$; 
\item
the lines $(P_1P_4)$, $(P_2P_5)$, $(P_3P_6)$ intersect at a common point~$C$. 
\end{itemize}
\end{theorem}

\begin{figure}[ht]
\includegraphics[angle=0, scale=0.4, trim=1.5cm 1cm 2cm 1cm, clip]{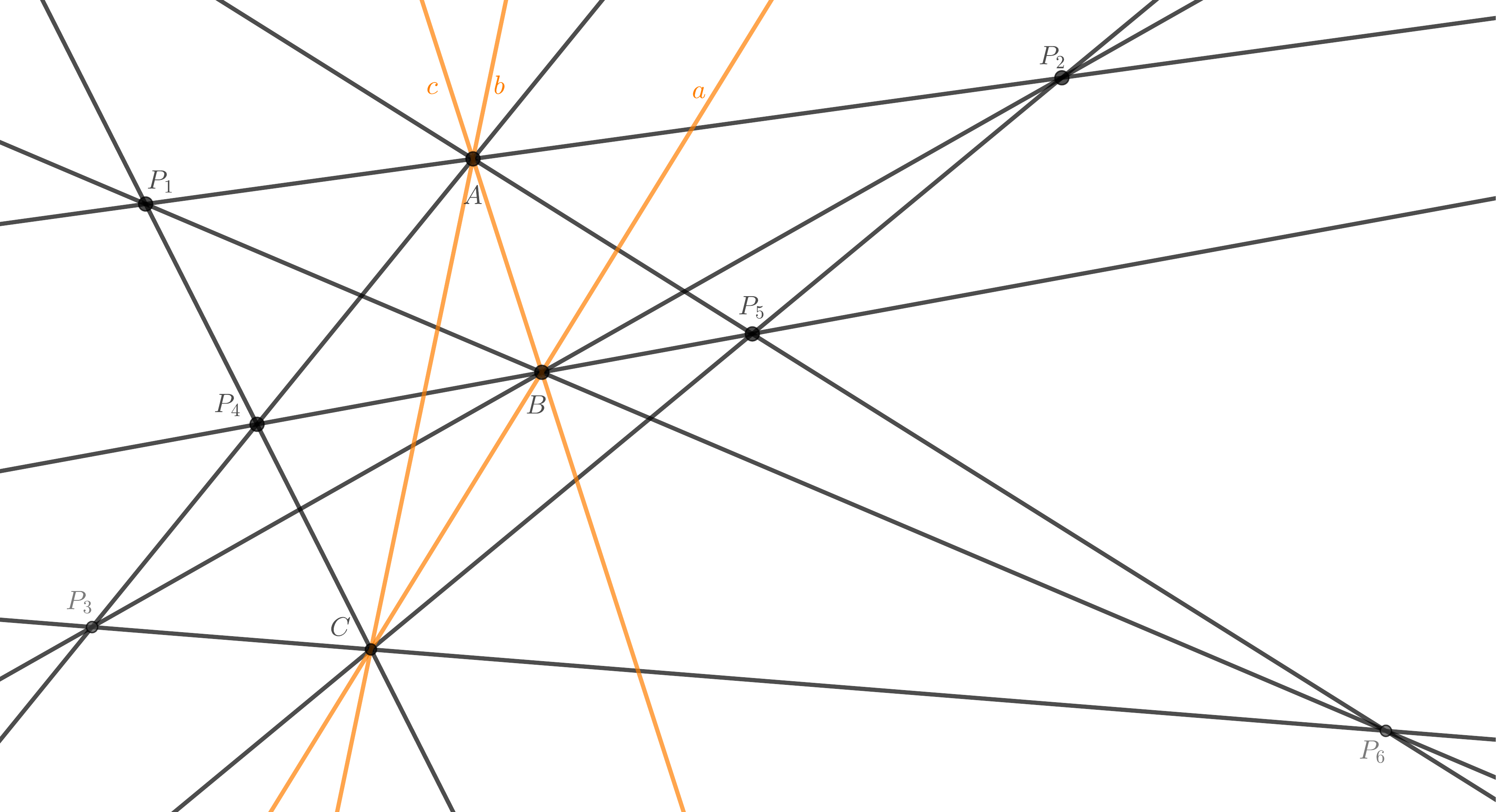}
%
\vspace{-.1in}
\caption{The Pappus theorem. 
}
\label{fig:pappus1}
\end{figure}

\vspace{-.2in}

\begin{proof}
Define the points $A\!=\!(P_3P_4)\cap(P_5P_6), B\!=\!(P_2P_3)\cap(P_1P_6), {C\!=\!(P_1P_4)\cap(P_2P_5)}$
and
the lines $a\!=\!(BC)$, $b\!=\!(AC)$, $c\!=\!(AB)$, so that $A\!=\!b\cap c$, $B\!=\!a\cap c$, $C\!=\!a\cap b$. 
By~construction, we have:
\begin{itemize}[leftmargin=.2in]
\item
$P_3$, $P_4$, and $b\cap c$ are collinear; 
\item
$P_5$, $P_6$, and $b\cap c$ are collinear; 
\item
$P_2$, $P_3$, and $a\cap c$ are collinear; 
\item
$P_6$, $P_1$, and $a\cap c$ are collinear; 
\item
$P_1$, $P_4$, and $a\cap b$ are collinear; 
\item
$P_2$, $P_5$, and $a\cap b$ are collinear. 
\end{itemize}
The conditions appearing in the theorem become:
\begin{itemize}[leftmargin=.2in]
\item 
$P_1$, $P_2$, and $b\cap c$ are collinear; 
\item
$P_4$, $P_5$, and $a\cap c$ are collinear; 
\item
$P_3$, $P_6$, and $a\cap b$ are collinear. 
\end{itemize}
These 9 conditions correspond to the 9 tiles in the tiling of the torus shown in Figure~\ref{fig:pappus-torus}. 
The claim follows by the master theorem. 
\end{proof}

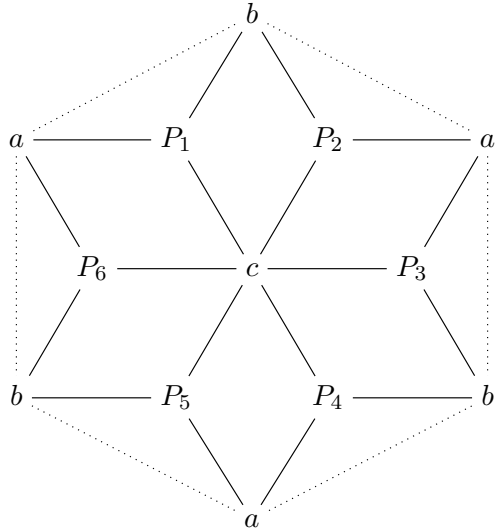
\begin{figure}[ht]
\vspace{-.15in}
\begin{equation*}
\begin{tikzpicture}[baseline= (a).base]
\node[scale=1] (a) at (0,0){
\begin{tikzcd}[arrows={-stealth}, sep=small, cramped]
&&   && b \edge{ld} \edge{rd} \edge{rrrd, dotted} &&  \\[.32in]
& a \edge{rr} \edge{rd} \edge{rrru, dotted} \edge{dd, dotted}&& P_1 \edge{rd} && P_2 \edge{rr} \edge{ld} && a \edge{ld} \edge{dd, dotted} \\[.32in]
 && P_6 \edge{rr} \edge{ld} && c \edge{rr} \edge{rd} \edge{ld}&& P_3 \edge{rd} & &  \\[.32in]
& b \edge{rr} \edge{rrrd, dotted} && P_5 \edge{rd} && P_4 \edge{rr} \edge{ld} && b  \\[.32in]
&&  && a \edge{rrru, dotted} && 
\end{tikzcd}
};
\end{tikzpicture}
\end{equation*}
\vspace{-.15in}
\caption{The tiling of the torus used in the proof of the Pappus theorem.
Opposite sides of the shown hexagonal fundamental domain should be glued to each other.
There are no edges between the vertices $a$ and~$b$. }
\vspace{-.2in}
\label{fig:pappus-torus}
\end{figure}

\vspace{-.15in}

To obtain the Levi graph of the Pappus configuration from the tiling in Figure~\ref{fig:pappus-torus}, 
we need to perform the following steps:
\begin{itemize}[leftmargin=.3in]
\item[(1)]
insert the quadripod~\eqref{eq:quadripod} into each tile, as before; 
\item[(2)]
identify multiple occurrences of each of the points $A, B, C$ 
to get Figure~\ref{fig:pappus-graph-intermediate}; 
\item[(3)]
remove bivalent vertices $a,b,c$ and the edges incident to them 
(bivalent vertices do not impose any restrictions on the rest of the configuration); 
\item[(4)]
rearrange the drawing to get Figure~\ref{fig:pappus-graph}. (Compare with \cite[Fig.~5.23]{pisanski-servatius}.)
\end{itemize}

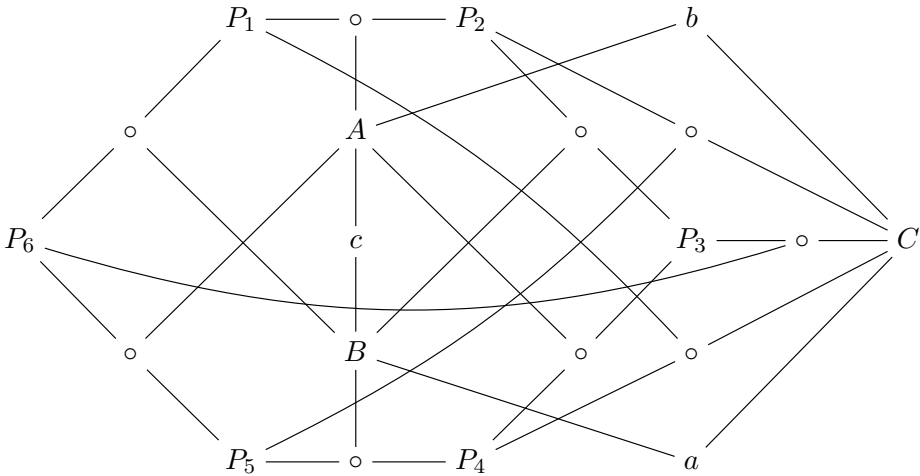
\begin{figure}[ht]
\vspace{-.1in}
\begin{equation*}
\begin{tikzpicture}[baseline= (a).base]
\node[scale=1] (a) at (0,0){
\begin{tikzcd}[arrows={-stealth}, cramped]
&& P_1 \edge{r} \edge{rrrrddd, bend left =10} \edge{dl} & \circ \edge{r} \edge{d} & P_2 \edge{rrd} \edge{dr} && b \edge{dlll}\edge{ddrr}  \\[.1in]
& \circ\edge{ddrr} \edge{dl} && A \edge{ddrr}\edge{d} && \circ \edge{dr} & \circ \edge{drr} \\[.1in]
P_6 \edge{dr} \edge{rrrrrrr, bend right =17}&&& c \edge{d} &&& P_3 \edge{dl} \edge{r}& \circ \edge{r}& C \edge{dll} \\[.1in]
& \circ \edge{uurr}\edge{dr} && B\edge{uurr} \edge{d} \edge{drrr}&& \circ \edge{dl} & \circ \\[.1in]
&& P_5\edge{r} \edge{rrrruuu, bend right =10} & \circ \edge{r}& P_4\edge{rru} && a \edge{uurr}
\end{tikzcd}
};
\end{tikzpicture}
\end{equation*}
\vspace{-.1in}
\caption{The intermediate Levi graph obtained from the tiling shown in Figure~\ref{fig:pappus-torus}. 
Unlabeled vertices marked~``$\circ$'' correspond to lines. 
}
\label{fig:pappus-graph-intermediate}
\end{figure}

\clearpage
\newpage

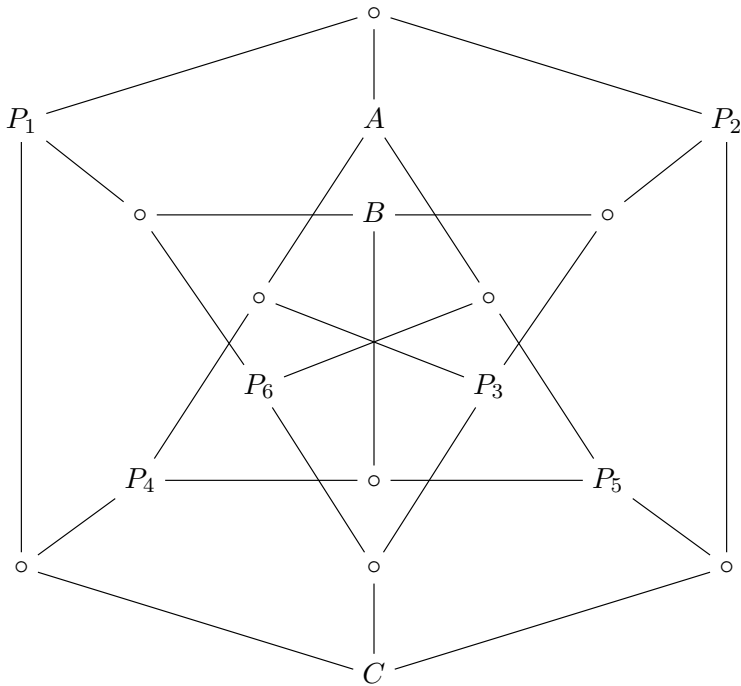
\begin{figure}[ht]
\vspace{-.1in}
\begin{equation*}
\begin{tikzpicture}[baseline= (a).base]
\node[scale=1] (a) at (0,0){
\begin{tikzcd}[arrows={-stealth}, cramped]
&&& \circ \edge{d} \edge{llld} \edge{rrrd}\\[.1in]
P_1 \edge{ddddd} \edge{dr} &&& A \edge{ddl} \edge{ddr} &&& P_2\edge{ddddd} \edge{dl}  \\[.0in]
& \circ \edge{rr} \edge{rdd} && B \edge{ddd} \edge{rr} && \circ \edge{ddl}  \\[.0in]
&& \circ \edge{ddl} \edge{drr} && \circ \edge{dll} \edge{ddr} \\[.0in]
&& P_6 \edge{ddr} && P_3 \edge{ddl} \\[.0in]
& P_4 \edge{dl} \edge{rr} && \circ \edge{rr} && P_5 \edge{dr} \\[.0in]
\circ \edge{rrrd} &&& \circ \edge{d} &&& \circ \edge{llld} \\[.1in]
&&& C
\end{tikzcd}
};
\end{tikzpicture}
\end{equation*}
\vspace{-.15in}
\caption{The Levi graph of the Pappus configuration. 
Each unlabeled vertex~$v$~corresponds to the line passing through the 3 points that label the vertices adjacent~to~$v$. 
}
\label{fig:pappus-graph}
\end{figure}

\vspace{-.2in}

Surprisingly, there exists an alternative proof of the Pappus theorem, 
also based on our master theorem
but utilizing a different tiling of the torus. 
This~tiling, shown in Figure~\ref{fig:pappus2-tiling}, involves 4~points and 5~lines 
(\emph{vs}.~6~points and 3~lines in Figure~\ref{fig:pappus-torus}).

\begin{proof}[Second proof of Theorem~\ref{th:pappus}]
$\!$As before, set $A\!=\!(P_3P_4)\cap(P_5P_6)$, ${C\!=\!(P_1P_4)\cap(P_2P_5)}$, and 
$b=AC$. Also define (cf.\ Figure~\ref{fig:pappus2}):
\begin{equation*}
\begin{array}{llll}
q=P_1P_6, \quad r=P_1P_2, \quad s=P_3P_6, \quad t=P_2P_3, 
\\[4pt]
D= P_1P_4\cap P_5P_6, \quad E= P_2P_5\cap P_3P_4. 
\end{array}
\end{equation*}
Reformulating the 9 collinearity conditions of the Pappus configuration  
in terms of the points $P_4, P_5, D, E$ and the lines
$b, q, r, s, t$, we obtain the following 9 conditions: 
\begin{itemize}[leftmargin=.2in]
\item 
$P_4$, $D$, and $b\cap s$ are collinear, 
\item 
$P_4$, $D$, and $q\cap r$ are collinear, 
\item 
$P_4$, $E$, and $b\cap r$ are collinear; 
\item 
$P_4$, $E$, and $s\cap t$ are collinear; 
\item
$P_5$, $D$, and $b\cap r$ are collinear; 
\item
$P_5$, $D$, and $s\cap q$ are collinear; 
\item
$P_5$, $E$, and $b\cap s$ are collinear,
\item 
$P_5$, $E$, and $r\cap t$ are collinear;
\item 
$P_4$, $P_5$, and $q\cap t$ are collinear. 
\end{itemize}
These 9 conditions correspond to the 9 tiles in the tiling of the torus exhibited in Figure~\ref{fig:pappus2-tiling}. 
The claim follows by the master theorem. 
\end{proof}

\begin{figure}[ht]
\vspace{-.1in}
\includegraphics[angle=0, scale=0.75, trim=1.5cm 12cm 2cm 1cm, clip]{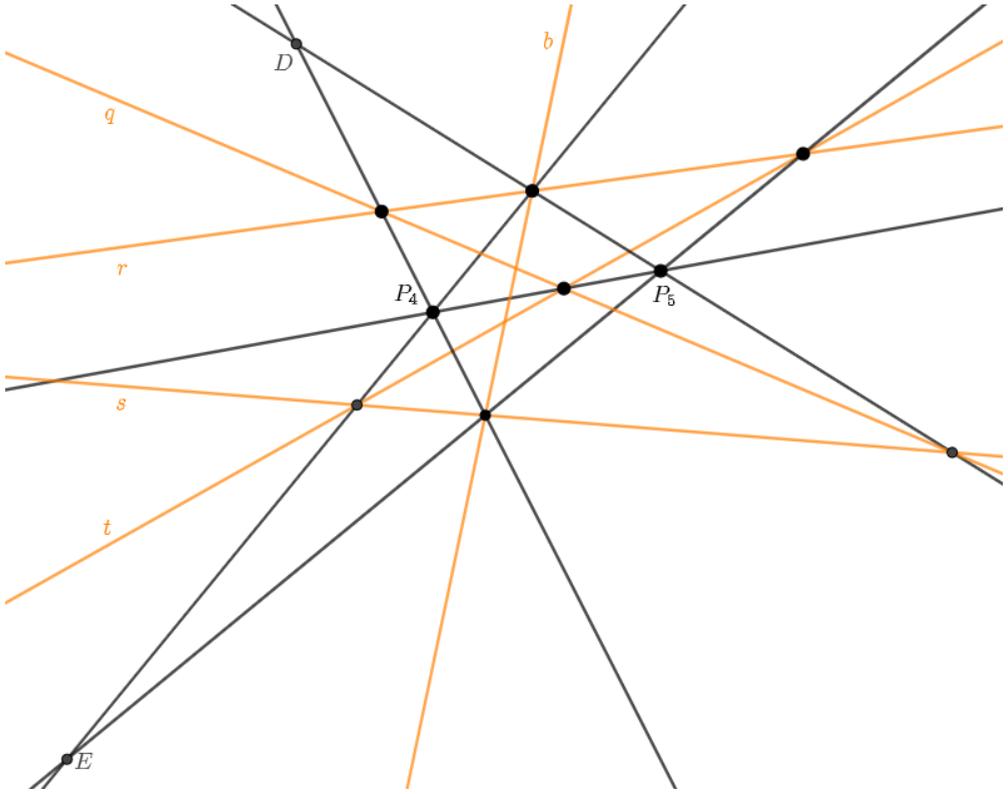}
\vspace{-.1in}
\caption{The Pappus theorem, second version. 
}
\label{fig:pappus2}
\end{figure}

\vspace{-.2in}

\begin{figure}[ht]
\vspace{-.2in}
\begin{center}
\begin{equation*}
\begin{tikzpicture}[baseline= (a).base]
\node[scale=.9] (a) at (0,0){
\begin{tikzcd}[arrows={-stealth}, cramped]
& & & P_5 \edge{rrd} \edge{lld} \edge{dddr} \edge{dddl} \\
& r \edge{ld} \edge{dd} &&&& s \edge{rd} \edge{dd} \\[-15pt]
P_4 \edge{d} &&&&&&  P_4 \edge{d} \\
b \edge{d} \edge{r} & E \edge{r} \edge{dd} & t \edge{rddd} && q \edge{r} \edge{lddd} & D \edge{dd} \edge{r} & b \edge{d} \\
P_5 \edge{rd} &&&&&&  P_5 \edge{ld} \\[-15pt]
& s \edge{rrd} &&&& r \edge{lld} \\
& & & P_4 
\end{tikzcd}
};
\end{tikzpicture}
\end{equation*}
\vspace{-.2in}
\end{center}
\caption{The tiling of the torus for the second proof of the Pappus theorem.\\
Opposite sides of this hexagonal fundamental domain should be glued to each other.}
\vspace{-.2in}
\label{fig:pappus2-tiling}
\end{figure}
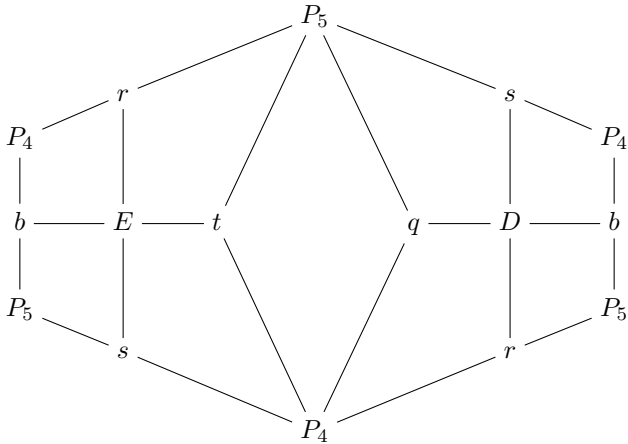


To get the Levi graph of the Pappus configuration from the tiling in Figure~\ref{fig:pappus2-tiling}, 
identify duplicate occurences of $(P_4E)$, $(P_5E)$, $(P_4D)$, $(P_5D)$, $b\cap s$, $b\cap r$, then remove bivalent vertices $b, D, E$
and the edges incident to them.

\clearpage

\newpage

\section*{Complete quadrangles/quadrilaterals}

Theorems~\ref{th:complete-quad} and~\ref{th:harmonic-points-theorem} below 
are apparently due to G.~Desargues, with precursors in the works of Euclid and Pappus of Alexandria,
see \cite[Chapter~VIII, p.~288]{van-der-waerden-awakening}. 
For other presentations and discussions, see, e.g.,  \cite[Example~5]{richter-gebert-mechanical}, 
\cite[Section~8.3]{richter-gebert-book}, and/or \cite[\S2.4]{heyting}. 
These results play an important role in R.~Moufang's foundational work on 
axiomatization of projective geometry, cf.\ \cite[p.~68]{moufang}. 

\begin{theorem}
\label{th:complete-quad}
Let $A_1, A_2, A_3,A_4$ be generic points on the real/complex projective plane. 
Draw six lines $\ell_{ij}\!=\!(A_iA_j)$, ${1\le i<j\le 4}$, through various pairs of these~points.
Pick an arbitrary line and denote by $P_{ij}$ the intersection of $\ell_{ij}$ with the chosen line. 
Now let us try to find another quadruple of points that would yield the same six points~$P_{ij}$. 
To this end, take a generic point~$B_1$ and draw the lines $m_{12}=(B_1P_{12})$, $m_{13}=(B_1P_{13})$, $m_{14}=(B_1P_{14})$.
Pick a generic~point $B_2\in m_{12}$ and draw the lines $m_{23}=(B_2 P_{23})$ and $m_{24}=(B_2 P_{24})$. 
Set $B_3=m_{13}\cap m_{23}$ and $B_4=m_{14}\cap m_{24}$. Then the line $m_{34}=(B_3B_4)$ passes through~$P_{34}$. 
See Figure~\ref{fig:complete-quadrilateral}. 
\end{theorem}


\begin{figure}[ht]
\vspace{-.1in}
\begin{center}
\includegraphics[scale=0.4, trim=0.1cm 0cm 0cm 0.5cm, clip]{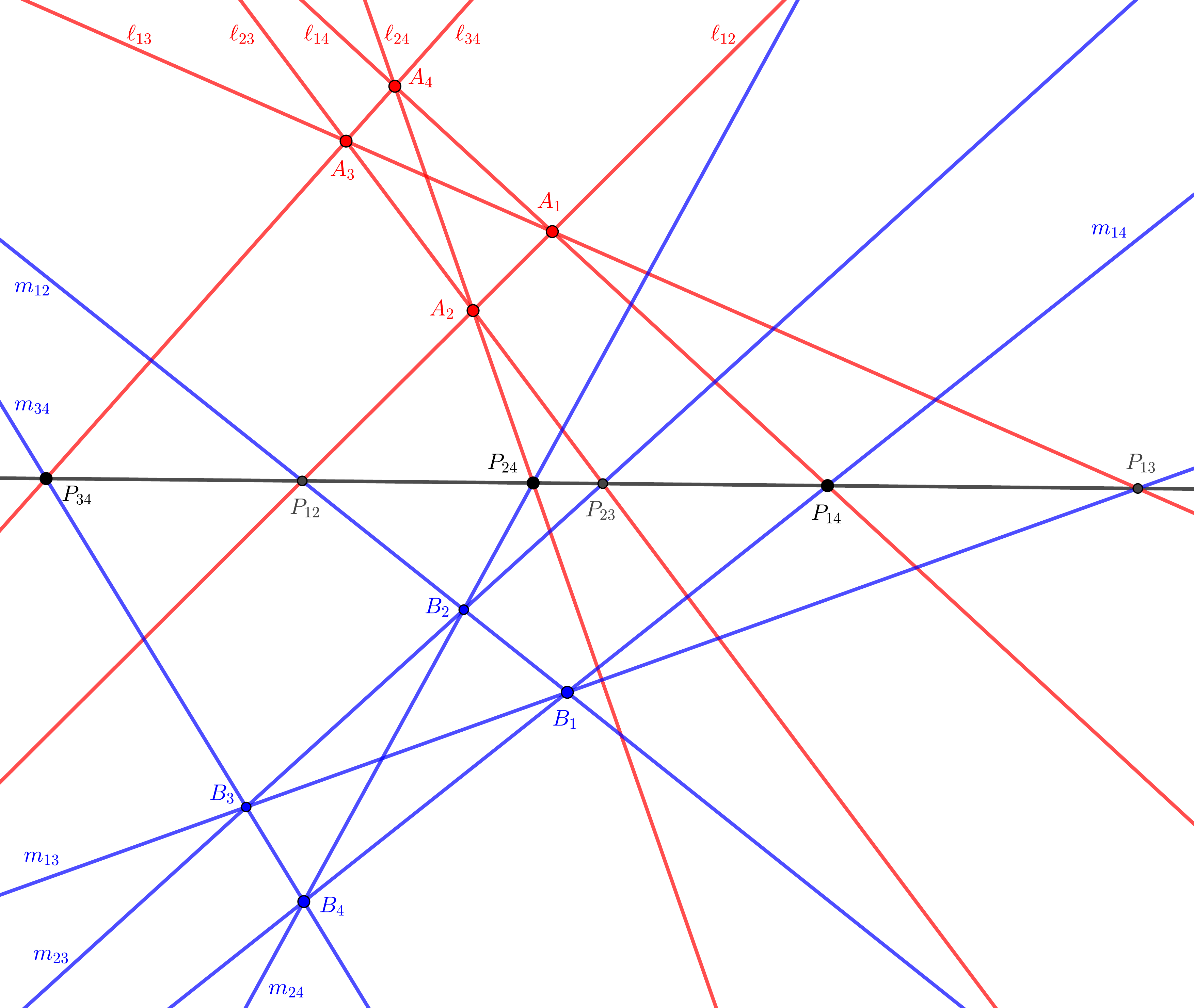}
\end{center}
\vspace{-5pt}
\caption{Theorem~\ref{th:complete-quad}. 
}
\label{fig:complete-quadrilateral}
\end{figure}

\vspace{-.2in}

\begin{proof}
Apply the master theorem to the tiling in Figure~\ref{fig:complete-quadrilateral-tiling}. 
\end{proof}

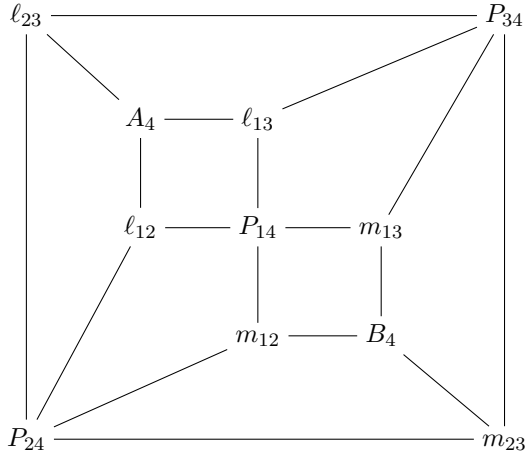
\begin{figure}[ht]
\vspace{-5pt}
\begin{equation*}
\begin{tikzpicture}[baseline= (a).base]
\node[scale=.9] (a) at (0,0){
\begin{tikzcd}[arrows={-stealth}, cramped]
\ell_{23} \edge{rrrr} \edge{rd} \edge{dddd} &&&& P_{34} \edge{lld}\edge{ldd} \edge{dddd} \\[8pt]
& A_4 \edge{r} \edge{d} & \ell_{13} \edge{d}& & \\[10pt]
& \ell_{12} \edge{r} \edge{ldd} & P_{14} \edge{d} \edge{r} & m_{13} \edge{d}& \\[10pt]
& & m_{12} \edge{r} \edge{lld}& B_4 \edge{rd} & \\[8pt]
P_{24}  \edge{rrrr} &&&& m_{23}
\end{tikzcd}
};
\end{tikzpicture}
\end{equation*}
\vspace{-.1in}
\caption{The tiling of the sphere used in the proofs of Theorems~\ref{th:complete-quad}--\ref{th:complete-quad-generalization}. 
}
\vspace{-.2in}
\label{fig:complete-quadrilateral-tiling}
\end{figure}

To obtain the Levi graph associated with the configuration in Figure~\ref{fig:complete-quadrilateral},
begin by inserting a quadripod~\eqref{eq:quadripod} into each tile to obtain the graph 
shown in Figure~\ref{fig:complete-quadrilateral-pre-Levi}. 
There, each of the three unlabeled vertices represents the line 
that passes through the six points~$P_{ij}$.
To get the Levi graph for Theorem~\ref{th:complete-quad}, these three vertices must be identified (``glued'') into a single vertex. 

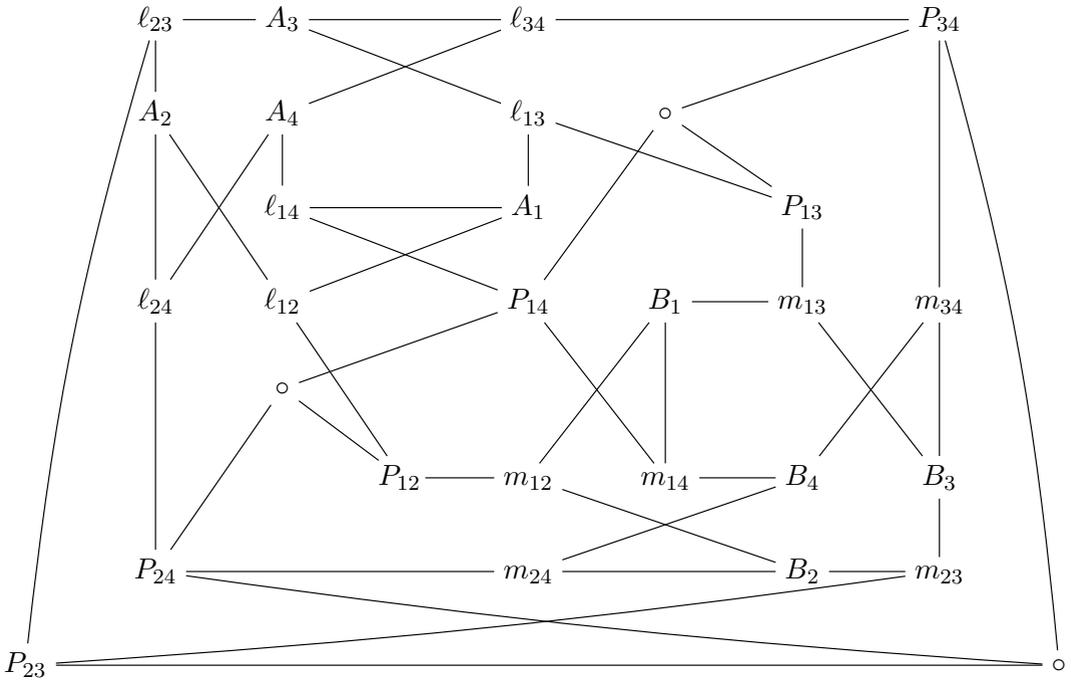
\begin{figure}[ht]
\begin{equation*}
\begin{tikzpicture}[baseline= (a).base]
\node[scale=1] (a) at (0,0){
\!\!\begin{tikzcd}[arrows={-stealth}, cramped]
& \ell_{23} \edge{r} \edge{d} \edge{lddddddd, bend right=5} & A_3\edge{rr} \edge{rrd}   && \ell_{34} \edge{dll} \edge{rrr} &&& P_{34}\edge{dll} \edge{rddddddd, bend left=5} \edge{ddd}   \\
& A_2 \edge{dd} \edge{rdd} & A_4 \edge{d} \edge{ldd} && \ell_{13} \edge{d} \edge{rrd} & \circ \edge{ldd} \edge{rd}  \\
&& \ell_{14} \edge{rr} \edge{rrd} && A_1 \edge{lld} && P_{13} \edge{d} \\
& \ell_{24} \edge{ddd} & \ell_{12}  \edge{rdd} && P_{14} \edge{lld} \edge{rdd} & B_1 \edge{r} \edge{ldd} \edge{dd} & m_{13} \edge{rdd} & m_{34} \edge{ldd} \edge{dd} \\
&& \circ \edge{ldd} \edge{rd} \\
&&& \!\!P_{12}\! \edge{r} & m_{12} \edge{rrd} & m_{14}\edge{r} & B_4 \edge{lld} & B_3 \edge{d} \\
& P_{24} \edge{rrr} \edge{drrrrrrr, bend right=2} &&& m_{24} \edge{rr} && B_2 \edge{r} & m_{23} \edge{dlllllll, bend left=2} \\
P_{23} \edge{rrrrrrrr} &&&&&&&& \circ
\end{tikzcd}
};
\end{tikzpicture}
\end{equation*}
\vspace{-.15in}
\caption{The Levi graph obtained from the tiling in Figure~\ref{fig:complete-quadrilateral-tiling}. 
}
\label{fig:complete-quadrilateral-pre-Levi}
\end{figure}

\vspace{-.1in}

As the above proof suggests, Theorem~\ref{th:complete-quad} can be generalized. 
There is no need to require that the six points $P_{ij}$ lie on the same line; 
it is enough to ask that the three triples 
$\{P_{13},P_{14},P_{34}\}$, $\{P_{12},P_{14},P_{24}\}$, and $\{P_{23}, P_{24},P_{34}\}$ are collinear.
The same proof still works,
with the same tiling as before (see Figure~\ref{fig:complete-quadrilateral-tiling}).
The key difference is that now, the graph in Figure~\ref{fig:complete-quadrilateral-pre-Levi} 
is the true Levi graph of the configuration, since the three unlabeled vertices no longer need to be glued together. 

We thus obtain the following generalization of Theorem~\ref{th:complete-quad}: 

\begin{theorem}
\label{th:complete-quad-generalization}
Let $A_1, A_2, A_3,A_4$ be four generic points on the real/complex projective plane. 
Draw six lines $\ell_{ij}\!=\!(A_iA_j)$, ${1\le i<j\le 4}$, through  pairs of these points.
Pick three generic points $P_{14}\in\ell_{14}$, $P_{24}\in\ell_{24}$, $P_{34}\in\ell_{34}$, 
then set
\begin{equation}
\label{eq:P12-P13-P23}
P_{12}=\ell_{12}\cap (P_{14}P_{24}), \quad 
P_{13}=\ell_{13}\cap (P_{14}P_{34}), \quad
P_{23}=\ell_{23}\cap (P_{24}P_{34}). 
\end{equation}
Pick a generic point~$B_1$ and 
define $m_{12}=(B_1P_{12})$, $m_{13}=(B_1P_{13})$, ${m_{14}=(B_1P_{14})}$.
Pick a generic~point $B_2\in m_{12}$ and draw the lines $m_{23}=(B_2 P_{23})$ and $m_{24}=(B_2 P_{24})$. 
Set $B_3=m_{13}\cap m_{23}$ and $B_4=m_{14}\cap m_{24}$. Then the line $m_{34}=(B_3B_4)$ contains~$P_{34}$. 
See Figure~\ref{fig:complete-quadrilateral-generalized}. 
\end{theorem}

\begin{figure}[ht]
\vspace{-5pt}
\begin{center}
\includegraphics[scale=0.45, trim=0.1cm 0.3cm 0cm 0.5cm, clip]{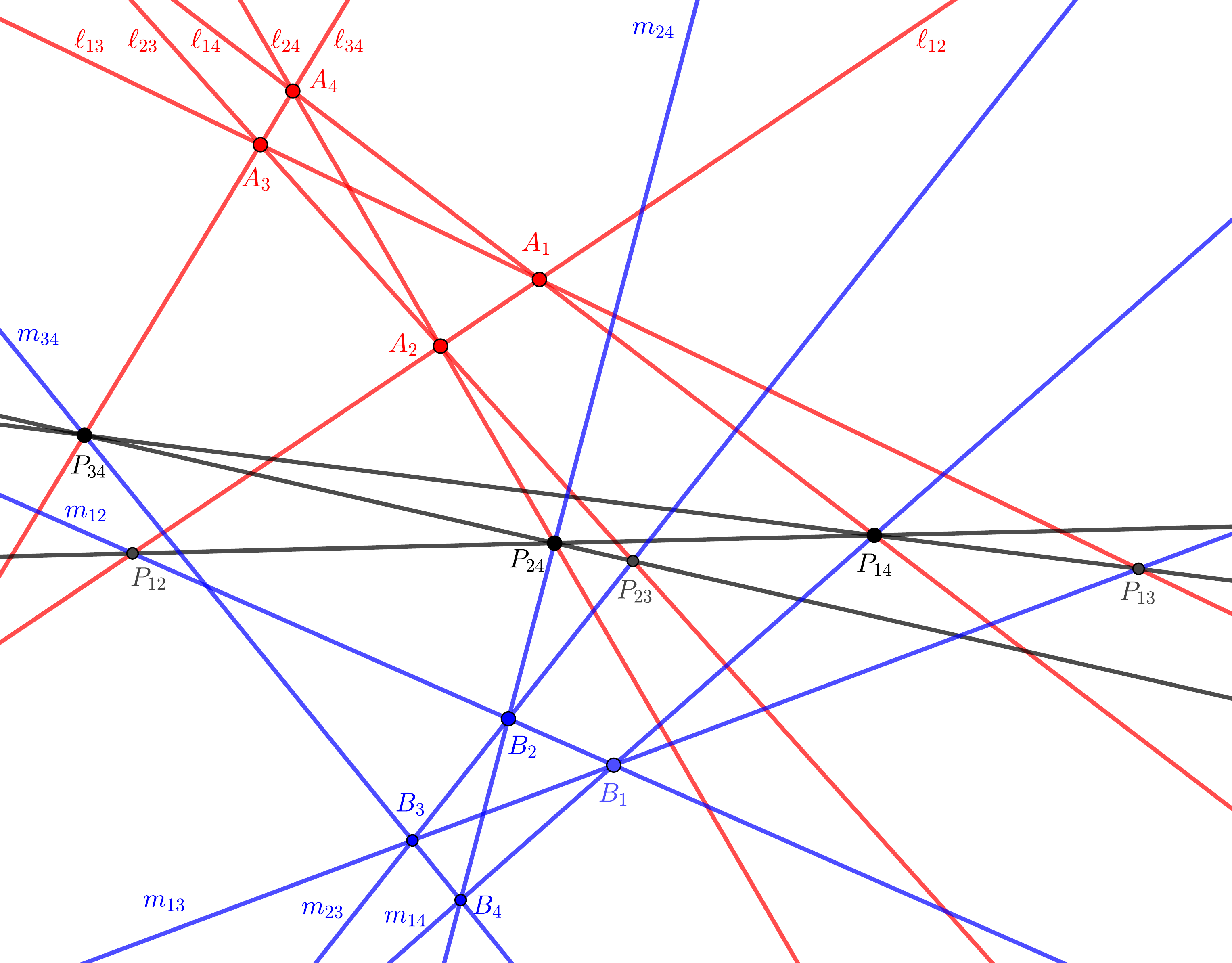}
\end{center}
\vspace{-5pt}
\caption{The configuration in Theorem~\ref{th:complete-quad-generalization}. 
}
\label{fig:complete-quadrilateral-generalized}
\end{figure}

\vspace{-15pt}

\begin{remark}
\label{rem:quads-desargues}
The assumptions in Theorem~\ref{th:complete-quad-generalization} imply (right after~\eqref{eq:P12-P13-P23})
that the points $P_{12},  P_{13}, P_{23}$ are collinear, cf.\ the front page. Our proof does not rely on this observation. 
\end{remark}

Going in the opposite direction, we can further specialize Theorem~\ref{th:complete-quad}
to obtain the following simplification (the ``harmonic points theorem'' \cite[Example~4]{richter-gebert-mechanical}):

\begin{theorem}
\label{th:harmonic-points-theorem}
Let $A_1, A_2, A_3,A_4$ be four generic points on the real/complex projective plane. 
Draw six lines $\ell_{ij}\!=\!(A_iA_j)$, ${1\le i<j\le 4}$. 
Let $P_{12,34}=\ell_{12}\cap\ell_{34}$~and $P_{14, 23}\!=\!\ell_{14}\cap\ell_{23}$. 
Draw the line ${\ell_\circ\!=\!(P_{12,34}P_{14, 23})}$. 
Set ${P_{13}\!=\!\ell_{13}\cap\ell_\circ}$ and ${P_{24}\!=\!\ell_{24}\cap\ell_\circ}$.
Take a generic point~$B_1$ and set ${m_{12}=(B_1P_{12,34})}$, ${m_{13}=(B_1P_{13})}$, ${m_{14}=(B_1P_{14,23})}$.
Then pick a generic~point $B_2\in m_{12}$ and set $m_{23}=(B_2 P_{14,23})$, 
$m_{24}=(B_2 P_{24})$, 
$B_3=m_{13}\cap m_{23}$, $B_4=m_{14}\cap m_{24}$. Then the line $m_{34}=(B_3B_4)$ passes through~$P_{12,34}$. 
See Figure~\ref{fig:harmonic-points-theorem}. 
\end{theorem}

\begin{figure}[ht]
\vspace{-5pt}
\begin{center}
\includegraphics[scale=0.47, trim=0.1cm 0.3cm 0cm 0.5cm, clip]{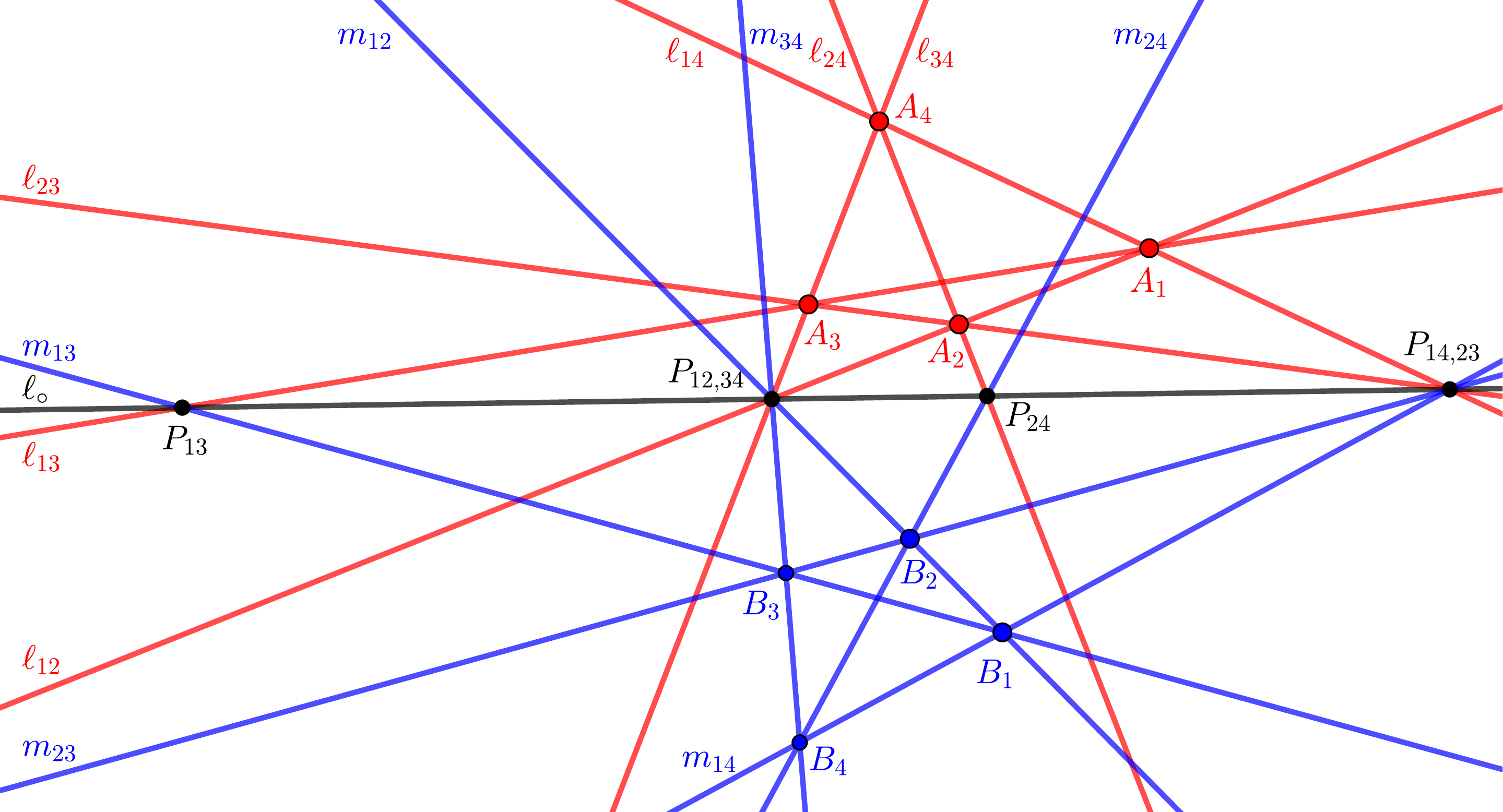}
\end{center}
\vspace{-8pt}
\caption{The configuration in Theorem~\ref{th:harmonic-points-theorem}. 
}
\label{fig:harmonic-points-theorem}
\end{figure}

\vspace{-15pt}

Theorem~\ref{th:harmonic-points-theorem} 
is immediate from Theorem~\ref{th:complete-quad}: just make the additional (redundant) 
assumptions $P_{12}=P_{34}$ and $P_{14}=P_{23}$.
The proof remains the same, except for the extra identifications in the Levi graph.

\begin{remark}
\label{rem:remove-redundancies}
As explained above, the tiling in Figure~\ref{fig:complete-quadrilateral-tiling} 
yields Theorem~\ref{th:complete-quad-generalization}, 
from which Theorems~\ref{th:complete-quad} and~\ref{th:harmonic-points-theorem}
can be obtained
by introducing additional---in point of fact, unnecessary---(co)incidence assumptions.
These examples suggest that if we want our master theorem to potentially encompass all incidence theorems
of plane projective geometry,
then we need to allow, in addition to the choices embedded into the master theorem
(i.e., the choices of a closed oriented surface and an arbitrary quadrilateral tiling of it), 
the ability to impose additional redundant constraints that may potentially simplify the statement
of an incidence theorem. 

From an algorithmic standpoint, this means that in order to exhibit an instance of the master theorem
from which a given incidence theorem can be obtained, we may need to first identify and remove 
some redundant assumptions that do not affect the validity of the theorem. 
This could be a challenging algorithmic task in itself---but one that may have to be carried out 
in order to obtain (a generalization~of) a given incidence theorem 
from a suitable tiled surface. 
\end{remark}


\newpage

Adding the assumption $\ell_{12}=m_{12}$ to Theorem~\ref{th:complete-quad-generalization}
and simplifying the ensuing statement,
we obtain the following result. 

\begin{theorem}
\label{th:planar-bundle-generalized}
Let $A_1, A_2, A_3,A_4$ be four generic points on the plane. 
Draw six~lines $\ell_{ij}\!=\!(A_iA_j)$, ${1\le i<j\le 4}$. 
Pick points $Q\!\in\!\ell_{34}$, $P_{13}\!\in\!\ell_{13}$, $P_{23}\!\in\!\ell_{23}$ and ${B_1,B_2\!\in\!\ell_{12}}$.
Construct 
$P_{14}\!=\!(QP_{13})\cap\ell_{14}$,
$P_{24}\!=\!(QP_{23})\cap\ell_{24}$
and 
$B_3\!=\!(B_1P_{13})\cap(B_2P_{23})$,
${B_4\!=\!(B_1P_{14})\cap(B_2P_{24})}$.
Then the points $B_3, B_4, Q$ are collinear. 
\end{theorem}

\begin{figure}[ht]
\vspace{-5pt}
\begin{center}
\includegraphics[scale=0.48, trim=0.4cm 1.3cm 0.1cm 0.5cm, clip]{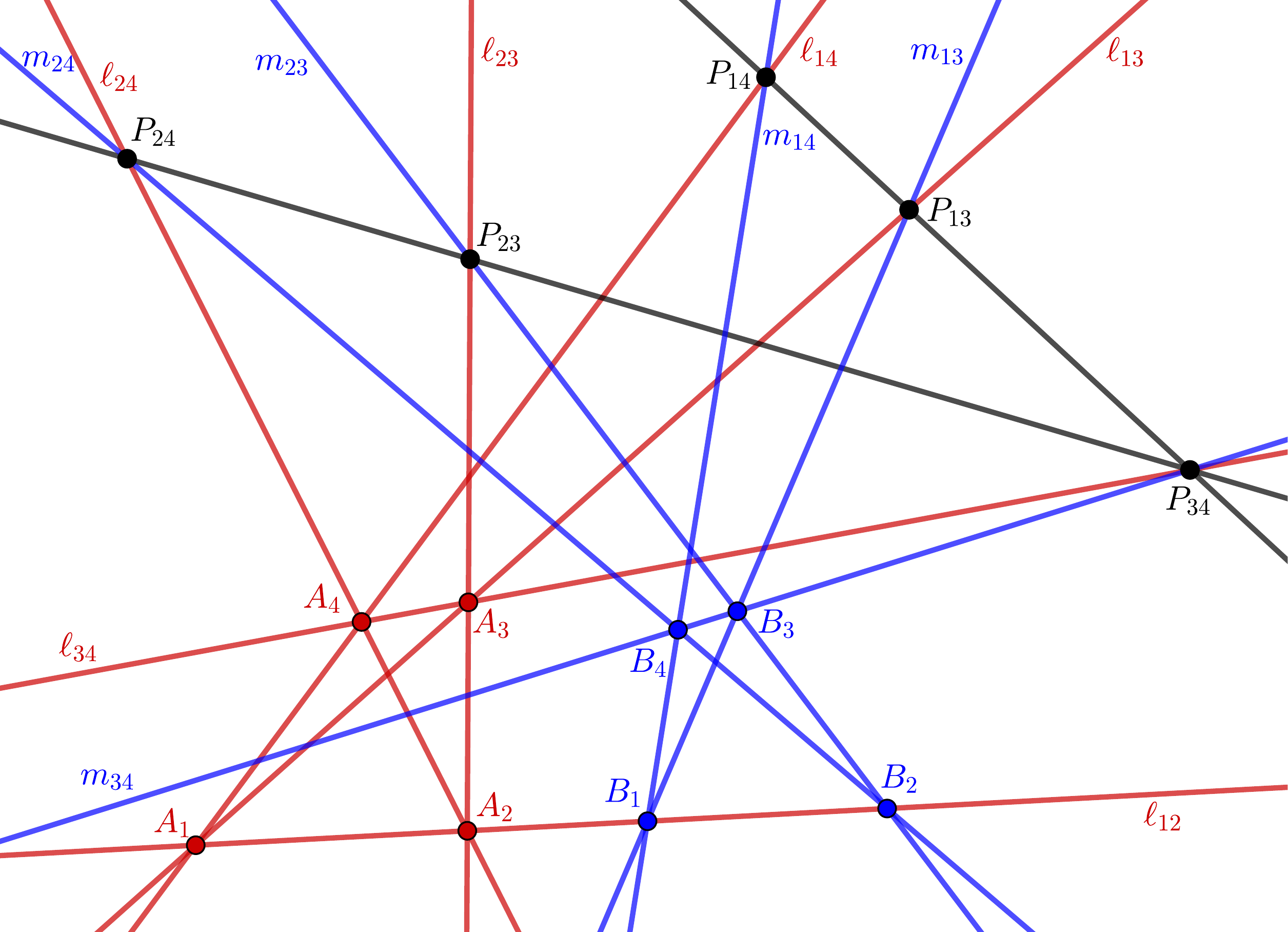}
\end{center}
\vspace{-10pt}
\caption{The configuration in Theorem~\ref{th:planar-bundle-generalized}.
}
\label{fig:planar-bundle-generalized}
\end{figure}

\vspace{-15pt}

\begin{proof}
Apply the master theorem to the tiling in Figure~\ref{fig:planar-bundle-generalized-tiling}. 
(This tiling was obtained from the tiling in in Figure~\ref{fig:complete-quadrilateral-tiling}
by identifying the vertices $\ell_{12}$ and~$m_{12}$.) 
\end{proof}

\begin{figure}[ht]
\vspace{-5pt}
\begin{equation*}
\begin{tikzpicture}[baseline= (a).base]
\node[scale=1] (a) at (0,0){
\begin{tikzcd}[arrows={-stealth}, cramped, sep=8]
P_{34} \edge{rd} &&&& P_{34} \edge{ld} \\
& \ell_{13} \edge{r} \edge{d} & P_{14}\edge{r}  \edge{d} & m_{13} \edge{d} \\[5pt]
& A_4 \edge{r} \edge{d} & \ell_{12} \edge{r} \edge{d} & B_4 \edge{d} \\[5pt]
& \ell_{23} \edge{r} & P_{24} \edge{r} & m_{23} \\
P_{34} \edge{ru} &&&& P_{34} \edge{lu} 
\end{tikzcd}
};
\end{tikzpicture}
\end{equation*}
\vspace{-.2in}
\caption{The tiling of the sphere used in the proof of Theorem~\ref{th:planar-bundle-generalized}. 
}
\label{fig:planar-bundle-generalized-tiling}
\end{figure}
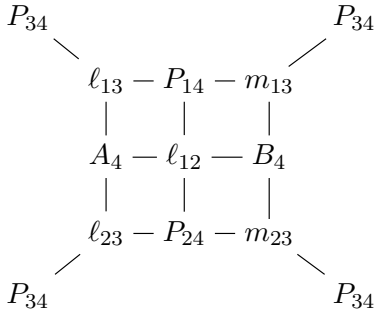

\vspace{-20pt}

\begin{remark}
Theorem~\ref{th:planar-bundle-generalized} can be restated as follows.  
Consider two quadrilaterals that are in perspective. 
(In the above notation, take $A_3P_{13}B_3P_{23}$ and $A_4P_{14}B_4P_{24}$.)  
Construct the points where their corresponding  sides meet. 
(These would be $A_1$, $B_1$, $B_2$, $A_2$.)  
If three of these four points are collinear, then all four points are collinear. 
\end{remark}

\begin{remark}
Degenerating the above construction to the setting where both quad\-ruples 
$\{A_3, A_4,B_3,B_4\}$ and $\{P_{13}, P_{14}, P_{23}, P_{24}\}$ are collinear,
we obtain the \emph{planar bundle theorem} \cite[Theorem~2.5]{BBB}. 
\end{remark}




We conclude this section by a couple of additional applications of our master theorem in plane projective geometry. 

\section*{The permutation theorem}

\begin{theorem}
\label{th:perm-thm}
Take a generic quadruple $p, q, r, s$ of concurrent lines on the real or complex projective plane. 
Let $P_1,Q_1,R_1,S_1$ and $P_2,Q_2,R_2,S_2$ denote the points of intersection
of this quadruple with two generic lines $\ell_1$ and~$\ell_2$, see Figure~\ref{fig:permutation-theorem}. 
If three of the four lines $(P_1Q_2)$, $(R_1S_2)$, $(R_2S_1)$, $(P_2Q_1)$ are concurrent,
then all four of them are concurrent. 
\end{theorem}

\begin{figure}[ht]
\vspace{-.1in}
\begin{center}
\includegraphics[scale=0.48, trim=0cm 0.5cm 1cm 0cm, clip]{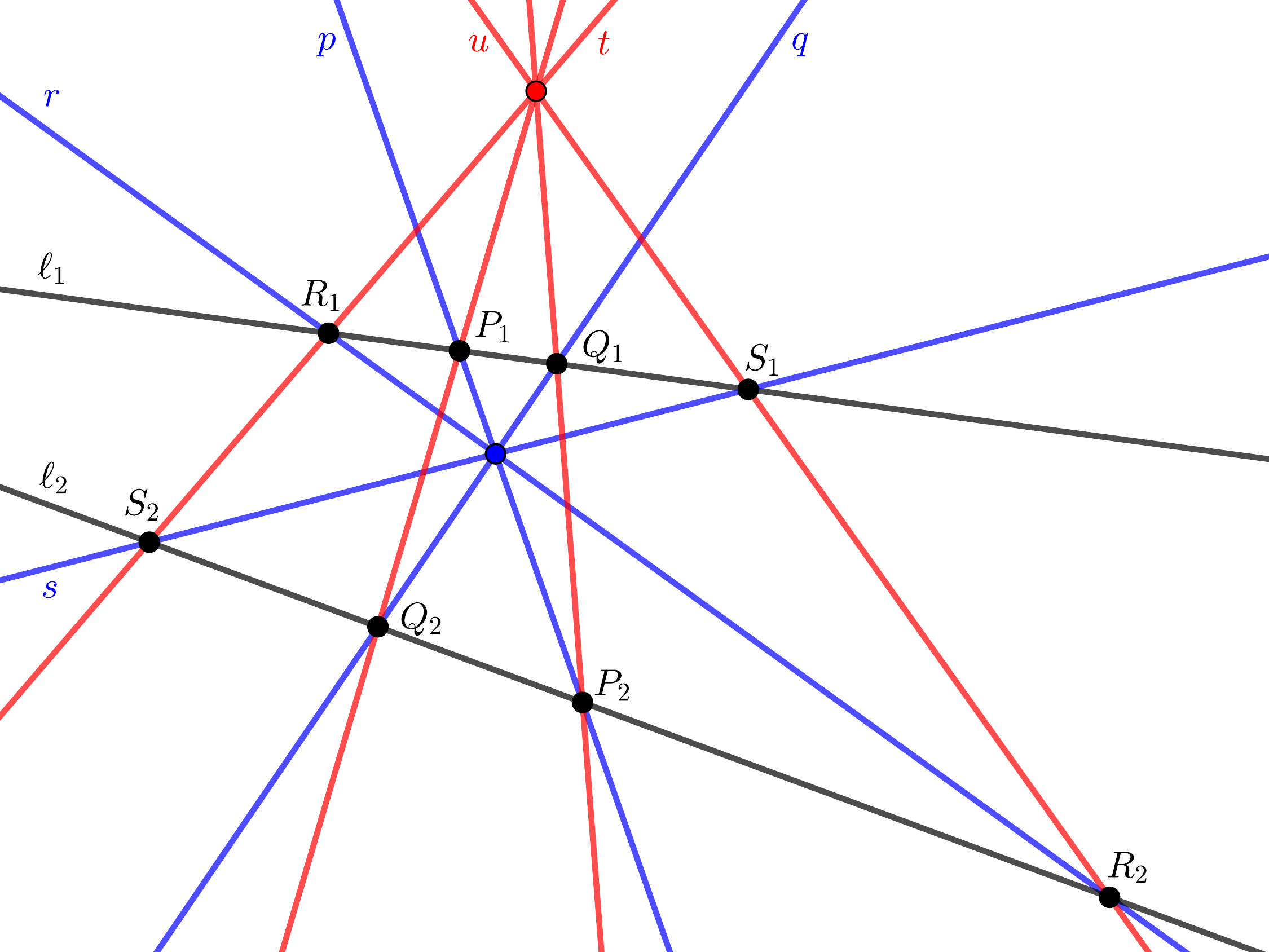}
\end{center}
\vspace{-.1in}
\caption{The permutation theorem. 
}
\vspace{-.25in}
\label{fig:permutation-theorem}
\end{figure}


\begin{proof}
Denote
$t=(R_1S_2)$ and $u=(R_2S_1)$. 
Let us express the conditions appearing in the theorem in terms of the points $P_1, P_2, Q_1, Q_2$
and the lines $r, s, t, u$:
\begin{itemize}[leftmargin=.2in]
\item 
$P_1$, $Q_1$, and $r\cap t$ are collinear; 
\item
$P_1$, $Q_1$, and $s\cap u$ are collinear; 
\item 
$P_2$, $Q_2$, and $r\cap u$ are collinear; 
\item
$P_2$, $Q_2$, and $s\cap t$ are collinear; 
\item 
$(P_1P_2), r, s$ are concurrent; 
\item
$(Q_1Q_2), r, s$ are concurrent; 
\item 
$(P_1Q_2), t, u$ are concurrent; 
\item
$(P_2Q_1), t, u$ are concurrent. 
\end{itemize}
These 8 conditions correspond to 8 tiles in the tiling of the torus shown in Figure~\ref{fig:torus-perm}. 
The claim follows by the master theorem. 
\end{proof}

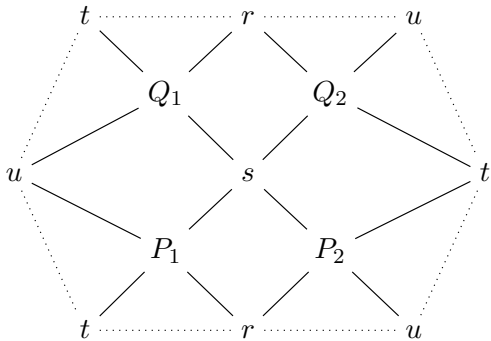
\begin{figure}[ht]
\begin{equation*}
\begin{tikzpicture}[baseline= (a).base]
\node[scale=1] (a) at (0,0){
\begin{tikzcd}[arrows={-stealth}, cramped, sep=15]
& t \edge{rr, dotted} \edge{ldd, dotted} \edge{rd} && r \edge{rr, dotted} \edge{rd} \edge{ld} && u \edge{rdd, dotted} \edge{ld} \\
&& Q_1 \edge{rd} \edge{lld} && Q_2 \edge{rrd} \edge{ld} \\
u \edge{rdd, dotted} \edge{rrd} &&& s \edge{rd} \edge{ld} &&& t \edge{ldd, dotted} \edge{lld} \\
&& P_1 \edge{rd} \edge{ld} && P_2 \edge{rd} \edge{ld} \\
& t \edge{rr, dotted} && r \edge{rr, dotted} && u
\end{tikzcd}
};
\end{tikzpicture}
\end{equation*}
\vspace{-.1in}
\caption{The tiling of the torus used in the proof of the permutation theorem. 
Opposite sides of the hexagonal fundamental domain should be glued to each other.
There are no edges along the sides of the fundamental domain. 
}
\label{fig:torus-perm}
\end{figure}

\section*{Saam's sequence of perspectivities}

The following theorem generalizes a result of A.~Saam \cite{saam-1988},
reproduced in \cite[Examples~7,~20]{richter-gebert-mechanical}. 

\begin{theorem}
\label{th:saam-perspectivities}
Let $A, B, P_1,\dots,P_n, Q_1,\dots,Q_n, R_1,\dots,R_n, S_1,\dots,S_n$ be points on the real/complex projective plane.
Consider the following conditions, where the indices are viewed modulo~$n$:
\begin{itemize}[leftmargin=.2in]
\item 
the lines $(Q_iS_i)$, $(Q_{i-1}S_{i-1})$, $(AP_i)$ are concurrent; 
\item
the lines $(Q_iS_{i-1})$, $(Q_{i+1}S_i)$, $(BR_i)$ are concurrent; 
\item 
the points $P_i$, $Q_i$, $R_i$ are collinear; 
\item
the points $R_i$, $S_i$, $P_{i+1}$ are collinear. 
\end{itemize}
If $4n-1$ of these $4n$ conditions hold, then so does the remaining one. 
\end{theorem}

Figure~\ref{fig:saam-perspectivities-general} illustrates Theorem~\ref{th:saam-perspectivities} for $n=4$. 


\begin{figure}[ht]
\begin{center}
\includegraphics[scale=0.45, trim=2.5cm 0.8cm 2cm 0.3cm, clip]{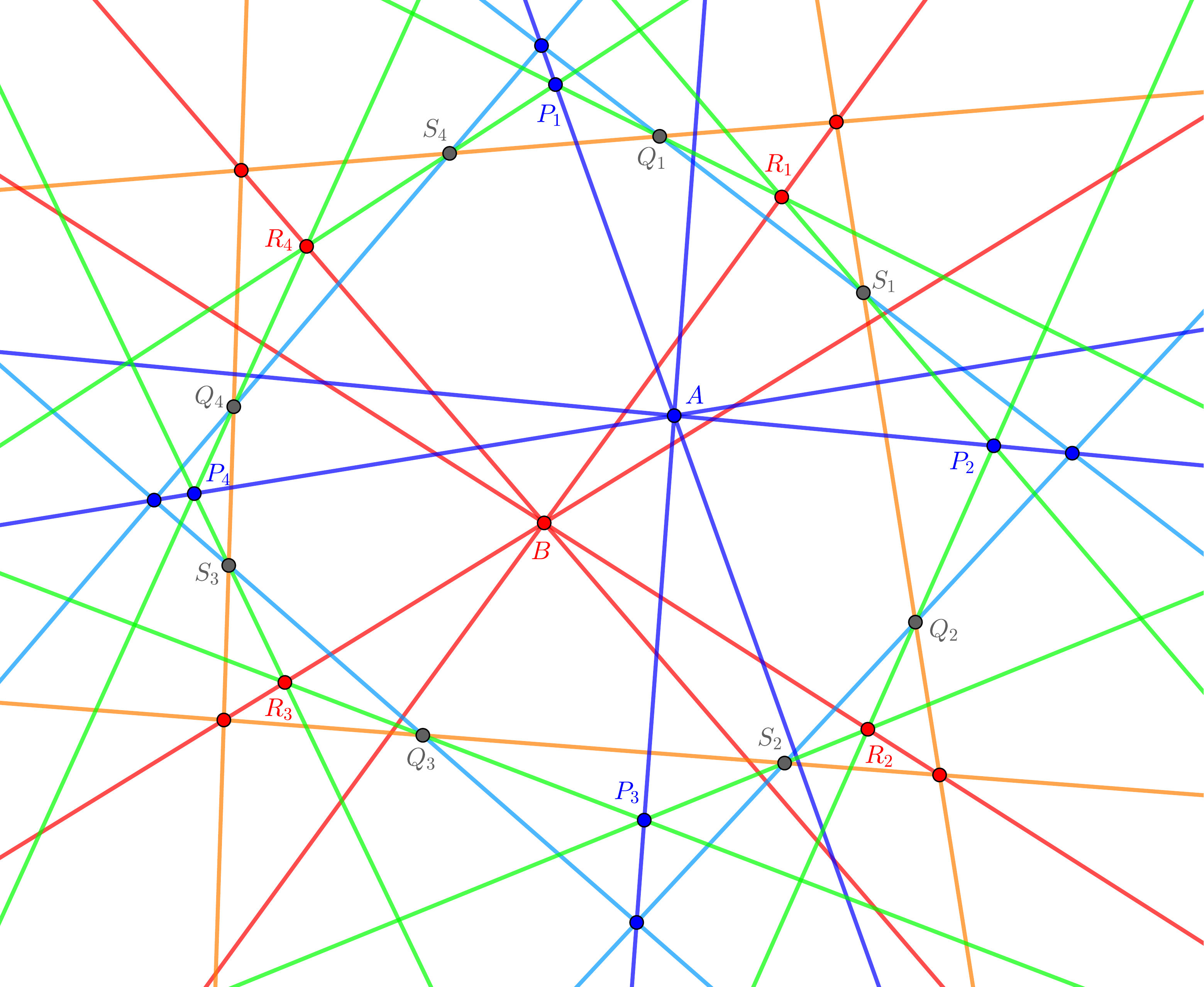}
\vspace{-5pt}
\end{center}
\caption{Generalization of Saam's sequence of perspectivities.}
\label{fig:saam-perspectivities-general}
\end{figure}


\begin{proof}[Proof of Theorem~\ref{th:saam-perspectivities}]
(For simplicity, we restrict the treatment to the case $n=4$.)
Apply the master theorem to the tiling of the sphere shown in Figure~\ref{fig:saam-perspectivities}. 
\end{proof}

\begin{figure}[ht]
\begin{center}
\vspace{-.15in}
\begin{equation*}
\begin{tikzpicture}[baseline= (a).base]
\node[scale=0.9] (a) at (0,0){
\begin{tikzcd}[arrows={-stealth}, cramped, sep=small]
A \edge{rd} &&&&&& A \edge{ld} \\[.1in]
& (Q_1S_1) \edge{rr}\edge{dd} \edge{rd}&& P_1 \edge{d}\edge{rr} && (Q_4S_4)\edge{dd}  \\[.1in]
&& R_1 \edge{r}\edge{d} & (Q_1S_4) \edge{r}\edge{d}& R_4 \edge{d}\edge{ru} \\[.1in]
& P_2 \edge{r}\edge{dd} & (Q_2S_1) \edge{r}\edge{d}& B \edge{r}\edge{d}& (Q_4S_3)\edge{r}\edge{d} & P_4\edge{dd}  \\[.1in]
&& R_2 \edge{r} & (Q_3S_2) \edge{d}\edge{r}& R_3 \edge{rd} \\[.1in]
& (Q_2S_2) \edge{rr}\edge{ru} && P_3 \edge{rr} && (Q_3S_3) \\[.1in]
A \edge{ru}&&&&&& A \edge{lu}
\end{tikzcd}
};
\end{tikzpicture}
\end{equation*}
\vspace{-.2in}
\end{center}

\enlargethispage{1cm}

\caption{The proof of Theorem~\ref{th:saam-perspectivities} for $n=4$. 
%
}
\label{fig:saam-perspectivities}
\end{figure}
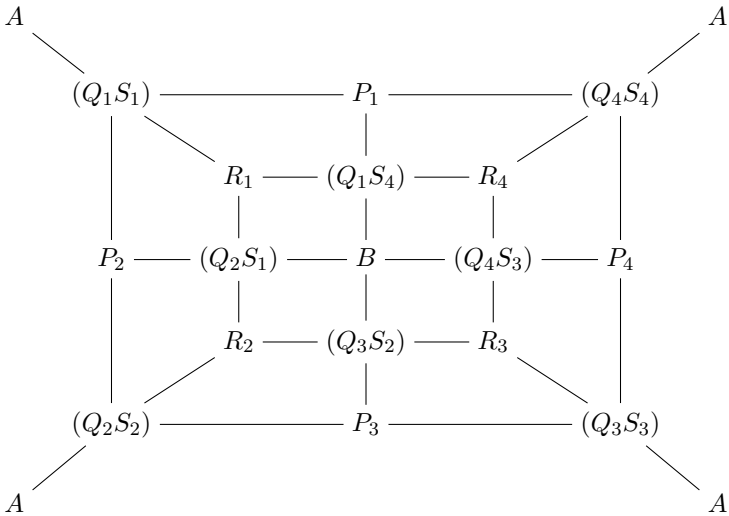

\begin{remark}
In the original formulation of Saam's theorem, the points $A$ and~$B$ are assumed to coincide with each other. 
To use the notation from \cite[Figure~7]{richter-gebert-mechanical}, 
there is no need to require that all eight radial lines pass through the same point~$H$.
It suffices to ask that the lines through 1, 5, 9, $D$ pass through common point~$H$,
whereas the lines through 3, 7, $B$, and~$F$ pass through common point~$H'$. 
\end{remark}

\begin{remark}
\label{rem:nehring}
The special case $n=3$ of Saam's sequence-of-perspectivities theorem (Theorem~\ref{th:saam-perspectivities} with $A=B$),
known as Nehring's theorem \cite{nehring}, can be seen to be a special case of the Pappus theorem.
Using the notation from
\cite[Example~6.27]{chou-gao-zhang}, this can be explained as follows: 
remove the point~$O$ and the lines $OA, OB, OC$, then unmark the points $A, B, C$
to get the Pappus configuration. 
\end{remark}

\clearpage

\newpage

\section*{Twin stars of David}

\begin{theorem}
\label{th:dual-stars-of-david}
Let $A, B, C, A', B', C'$ be six points on the real or complex plane. \linebreak[3]
Let $p, q, r, p', q', r'$ be six lines none of which passes through any of those six points. \linebreak[3]
If~eleven of the twelve triples of lines 
\begin{align*}
&\{p,q,(A'C)\}, \{q,r,(BC')\}, \{p,r,(AB')\}, \{p',q',(AC')\}, \{q',r',(B'C)\}, \{p',r',(A'B)\} \\
&\{p,q',(AC)\}, \{q,r',(BC)\}, \{p',r,(AB)\}, \{p',q,(A'C')\}, \{q',r,(B'C')\}, \{p,r',(A'B')\} 
\end{align*}
are concurrent, then the twelfth triple is concurrent as well. See Figure~\ref{fig:dual-stars-of-david}. 
\end{theorem}

\begin{figure}[ht]
\begin{center}
\includegraphics[scale=0.5, trim=0cm 0cm 0cm 0cm, clip]{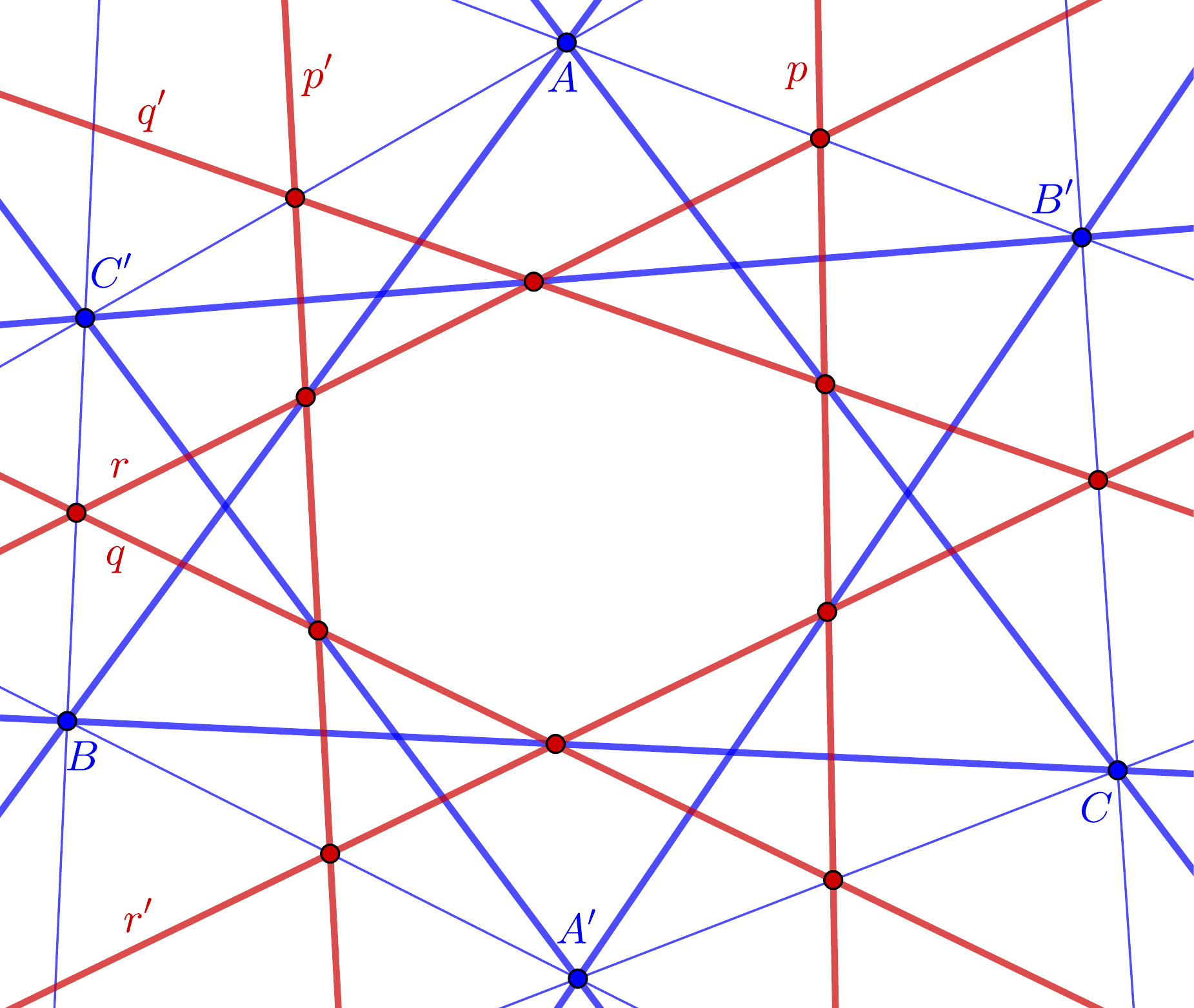}
\vspace{-5pt}
\end{center}
\caption{Twin stars of David.}
\vspace{-10pt}
\label{fig:dual-stars-of-david}
\end{figure}

\begin{proof}
Apply the master theorem to the tiling of the torus shown in Figure~\ref{fig:dual-stars-of-david-tiling}. 
\end{proof}

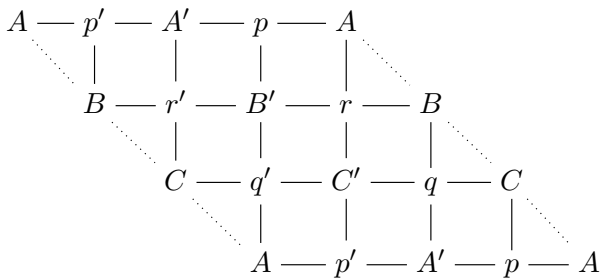
\begin{figure}[ht]
\begin{center}
\vspace{-.15in}
\begin{equation*}
\begin{tikzpicture}[baseline= (a).base]
\node[scale=1] (a) at (0,0){
\begin{tikzcd}[arrows={-stealth}, cramped, sep=14]
A \edge{r} \edge{rd, dotted} & p' \edge{r} \edge{d} & A' \edge{r} \edge{d} & p \edge{r} \edge{d} & A \edge{d} \edge{rd, dotted} \\
& B \edge{r} \edge{rd, dotted} & r' \edge{r} \edge{d} & B' \edge{r} \edge{d} & r \edge{r} \edge{d} & B \edge{d} \edge{rd, dotted} \\
&& C \edge{r} \edge{rd, dotted} & q' \edge{r} \edge{d} & C' \edge{r} \edge{d} & q \edge{r} \edge{d} & C \edge{d} \edge{rd, dotted} \\
&&& A \edge{r} & p' \edge{r} & A' \edge{r} & p \edge{r} & A
\end{tikzcd}
};
\end{tikzpicture}
\end{equation*}
\vspace{-.1in}
\end{center}
\caption{The proof of Theorem~\ref{th:dual-stars-of-david}.
The opposite sides of the fundamental domain should be glued to each other. 
}
\label{fig:dual-stars-of-david-tiling}
\end{figure}

\newpage

\section{Coherent polygons}
\label{sec:coherent-polygons}

\begin{definition}
\label{def:coherent-polygon}
Let $n\ge2$.
Let $A_1,\dots,A_n$ (resp., $\ell_1,\dots,\ell_n$) be an $n$-tuple of~points (resp., lines) on 
the real/complex projective plane. 
Assume that each point~$A_i$ does not lie on either of the lines~$\ell_i$ and~$\ell_{i-1}$
(with the indexing understood modulo~$n$). 
Consider a topological disk $\mathbf{P}$ with $2n$ marked points on the boundary, 
labeled by $A_1, \ell_1, \dots, A_n,\ell_n$, in this order. 
(See Figure~\ref{fig:coherent-polygon}.)
We call the $2n$-gon $\mathbf{P}$ \emph{coherent}~if 
the associated ``generalized mixed cross ratio'' $(A_1,\dots,A_n;\ell_1,\dots,\ell_n)$
is equal to~1: 
\begin{equation}
\label{eq:coherent-polygon}
(A_1,\dots,A_n;\ell_1,\dots,\ell_n)
\stackrel{\rm def}{=}
\dfrac{
\langle \mathbf{A_1}, \mathbf{\boldsymbol \ell_1} \rangle
\langle \mathbf{A_2}, \mathbf{\boldsymbol \ell_2} \rangle
\cdots
\langle \mathbf{A_{\boldsymbol n}}, \mathbf{\boldsymbol \ell_{\boldsymbol n}} \rangle
}{
\langle \mathbf{A_2}, \mathbf{\boldsymbol \ell_1} \rangle
\langle \mathbf{A_3}, \mathbf{\boldsymbol \ell_2} \rangle
\cdots
\langle \mathbf{A_1}, \mathbf{\boldsymbol \ell_{\boldsymbol n}} \rangle
}=1,
\end{equation}
where $\mathbf{A_{\boldsymbol i}}$ (resp.,~$\mathbf{\boldsymbol \ell_{\boldsymbol i}}$) is a vector (resp., covector)
defining~$A_i$ (resp.,~$\ell_i$). 

We note that $(A_1,\dots,A_n;\ell_1,\dots,\ell_n)$ does not depend on the choice of vectors 
$\mathbf{A_{\boldsymbol i}}$ and covectors $\mathbf{\boldsymbol \ell_{\boldsymbol i}}$
representing the points  $A_i$ and the hyperplanes $\ell_i$, respectively. 

In the case of a quadrilateral ($n=2$), we recover the notion of coherence introduced in 
Definition~\ref{def:coherent-tile},
in light of Proposition~\ref{pr:coherence-algebraic}. 
\end{definition}

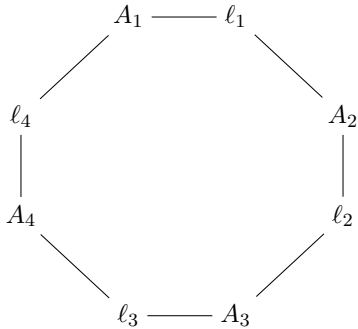
\begin{figure}[ht]
\begin{center}
\vspace{-.1in}
\begin{equation*}
\begin{tikzpicture}[baseline= (a).base]
\node[scale=.88] (a) at (0,0){
\begin{tikzcd}[arrows={-stealth}, cramped]
& A_1 \edge{r} & \ell_1 \edge{rd} & \\[.1in]
\ell_4 \edge{d} \edge{ru} & & & A_2 \\[.1in]
A_4 \edge{rd} & & & \ell_2 \edge{u}\\[.1in]
& \ell_3 \edge{r} & A_3 \edge{ru}
\end{tikzcd}
};
\end{tikzpicture}
\end{equation*}
\vspace{-.15in}
\end{center}
\caption{A coherent $2n$-gon, for $n=4$.}
\label{fig:coherent-polygon}
\end{figure}

\vspace{-.15in}

\begin{proposition}
\label{pr:coherence=tilability}
Let $\mathbf{P}\!=\!A_1 \ell_1 \cdots A_n\ell_n$ be a $2n$-gon with vertices labeled by 
$n$~points $A_1,\dots,A_n$ and $n$ lines $\ell_1,\dots,\ell_n$ on the real/complex projective plane. 
Assume that each point~$A_i$ does not lie on either of the lines~$\ell_i$ and~$\ell_{i-1}$
(with the indexing~$\bmod n$). 
Suppose that no three of the points~$A_i$ are collinear.
Then the following are equivalent:
\begin{itemize}[leftmargin=.2in]
\item 
the polygon $\mathbf{P}$ is coherent; 
\item
the polygon $\mathbf{P}$ can be tiled by coherent quadrilateral tiles. 
\end{itemize}
In fact, a coherent polygon $\mathbf{P}$ as above can always be tiled by $2n-3$ coherent tiles
all of whose vertices lying in the interior of~$\mathbf{P}$ are labeled by lines (and not by points). 
\end{proposition}

\begin{proof}
Suppose that $\mathbf{P}$ is tiled by coherent quadrilateral tiles.
Multiplying equations~\eqref{eq:cross-ratio=1}
for all these tiles, we obtain equation~\eqref{eq:coherent-polygon}. 
That is, $\mathbf{P}$ is coherent. 

We prove the reverse implication by induction on~$n$. The base case $n=2$ is trivial. 
Suppose that our $2n$-gon~$\mathbf{P}$ is coherent. 
Define
\begin{equation*}
B= (A_nA_1)\cap \ell_n, \quad C=(A_{n-1}A_n)\cap \ell_{n-1}, \quad \ell_{n-1}'=(BC). 
\end{equation*}
Recall that $A_{n-1}, A_n, A_1$ are not collinear. 
Also, neither $A_1$ nor $A_n$ lie~on~$\ell_n$, hence the collinear points $B, A_n, A_1$ are distinct. 
Similarly, the collinear points $C, A_{n-1}, A_n$ are distinct.
Hence the line $\ell_{n-1}'=(BC)$ contains none of the points $A_{n-1}, A_n, A_1$. 

It is now straightforward to check that the two tiles shown below are coherent: 
\begin{equation}
\label{eq:two-tiles-2n-gon}
\begin{tikzpicture}[baseline= (a).base]
\node[scale=1] (a) at (0,0){
\begin{tikzcd}[arrows={-stealth}, sep=small, cramped]
A_1  \edge{r}  \edge{d}& \ell _{n-1}'\edge{d} \edge{r} & A_{n-1} \edge{d} \\[3pt]
\ell_n \edge{r} & A_n\edge{r} & \ell_{n-1} 
\end{tikzcd}
};
\end{tikzpicture} 
\end{equation}
Indeed, $(A_1A_n)\cap \ell_{n-1}'\cap \ell_n=B$
and $(A_{n-1}A_n)\cap \ell_{n-1}'\cap \ell_{n-1}=C$. 

We now embed the tiles from~\eqref{eq:two-tiles-2n-gon} into the given $2n$-gon~$\mathbf{P}$,
as shown in \hbox{Figure~\ref{fig:coherent-polygon-induction-step}}. 
Since $\mathbf{P}$ is coherent, as are these two quadrilateral tiles,
it follows that the remaining $(2n-2)$-gon $\mathbf{P'}\!=\!A_1 \ell_1 \cdots A_{n-1}\ell_{n-1}'$
is coherent as well.
Moreover the noncollinearity condition for the points $A_1,\dots,A_{n-1}$ still holds,
so we can apply the induction assumption and obtain the desired tiling by $2+2(n-1)-3=2n-3$ tiles.
\end{proof}

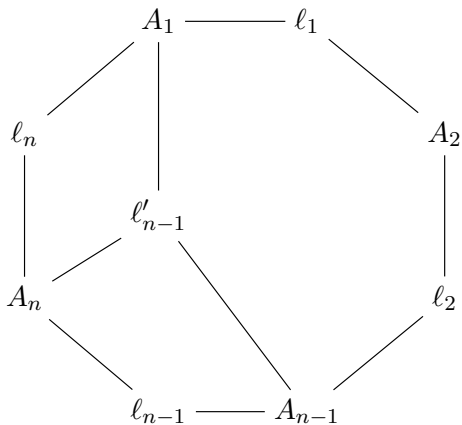
\begin{figure}[ht]
\begin{center}
\vspace{-.2in}
\begin{equation*}
\begin{tikzpicture}[baseline= (a).base]
\node[scale=1] (a) at (0,0){
\begin{tikzcd}[arrows={-stealth}, cramped]
& A_1 \edge{r} \edge{dd}& \ell_1 \edge{rd} & \\[.1in]
\ell_n \edge{dd} \edge{ru} & & & A_2 \\[-6pt]
& \ell_{n-1}' \\[-6pt]
A_n \edge{rd} \edge{ru} & & & \ell_2 \edge{uu}\\[.1in]
& \ell_{n-1} \edge{r} & A_{n-1} \edge{ru}\edge{luu}
\end{tikzcd}
};
\end{tikzpicture}
\end{equation*}
\vspace{-.2in}
\end{center}
\caption{Proof of Proposition~\ref{pr:coherence=tilability}.}
\label{fig:coherent-polygon-induction-step}
\end{figure}

\vspace{-.2in}

\begin{corollary}
\label{cor:tilability-genus}
Let $\Sigma$ be an oriented bordered surface whose boundary~$\partial\Sigma$ is homeo\-morphic to a circle.
Select $2n$ points on~$\partial\Sigma$ and label them with $n$ points and $n$ lines 
$A_1, \ell_1, \dots A_n,\ell_n$, in this order. 
Suppose that this labeling can be completed to a tiling of $\Sigma$ by coherent quadrilateral tiles. 
Then the polygon $\mathbf{P}\!=\!A_1 \ell_1 \cdots A_n\ell_n$ can be tiled by coherent tiles. 
\end{corollary}

To rephrase, adding handles to the surface does not expand the class of tileable polygons. 

\begin{proof}
Multiplying the coherence conditions for all tiles in the tiling, we conclude that 
the polygon $\mathbf{P}$ is coherent. 
By Proposition~\ref{pr:coherence=tilability}, this implies that $\mathbf{P}$ can be tiled  by coherent quadrilateral tiles. 
\end{proof}

While our arguments appearing below in this section utilize the notion of a coherent polygon,
they do not really rely on Proposition~\ref{pr:coherence=tilability}.
A~combinatorially-minded reader can replace ``coherent'' by ``tileable'' throughout without any loss of content. 

\newpage

\section*{Coherent hexagons}

The case $n\!=\!3$ of Proposition~\ref{pr:coherence=tilability} yields the following version of Desargues' theorem: 

\begin{corollary}
\label{cor:tileable-hexagons}
Let $A_1B_1C_1$ and $A_2B_2C_2$ be two generic triangles on the plane,~with sides 
${a_1\!=\!(B_1C_1)}, {b_1\!=\!(A_1C_1)}, {c_1\!=\!(A_1B_1)}$ and ${a_2\!=\!(B_2C_2)}, {b_2\!=\!(A_2C_2)}, {c_2\!=\!(A_2B_2)}$, respectively. 
The following are equivalent:
\begin{itemize}[leftmargin=.2in]
\item
the points $a_1\cap a_2$, $b_1\cap b_2$, $c_1\cap c_2$ are collinear; 
\item 
the lines $(A_1A_2)$, $(B_1B_2)$, $(C_1C_2)$ are concurrent; 
\item
the hexagon
shown in Figure~\ref{fig:6-cycle-tiling}(i) is coherent; 
\item
this hexagon can be tiled as shown in Figure~\ref{fig:6-cycle-tiling}(ii); 
\item
this hexagon can be tiled as shown in Figure~\ref{fig:6-cycle-tiling}(iii).
\end{itemize}
\end{corollary}

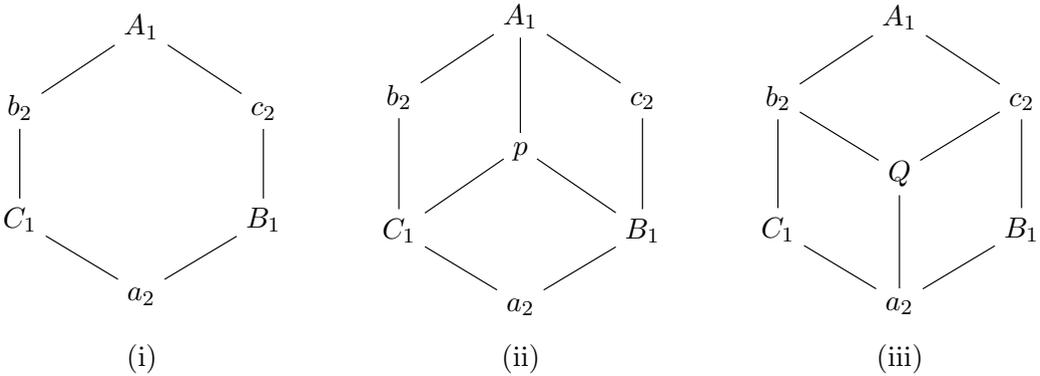
\begin{figure}[ht]
\begin{equation*}
\begin{array}{ccccc}
\begin{tikzcd}[arrows={-stealth}, cramped]
&  A_1 \edge{ld} \edge{rd}   & \\[-.07in]
b_2 \edge{d}  &     & c_2 \edge{d}  \\[7pt]
C_1 \edge{rd} & & B_1\edge{ld} \\[-.07in]
& a_2
\end{tikzcd}
&\hspace{.25in}
&
\begin{tikzcd}[arrows={-stealth}, cramped]
&  A_1 \edge{ld} \edge{rd} \edge{dd}  & \\[-.07in]
b_2 \edge{dd}  & & c_2 \edge{dd}  \\[-.19in]
& p \edge{rd} \edge{ld}     \\[-.06in]
C_1 \edge{rd}&& B_1\edge{ld} \\[-.07in]
& a_2
\end{tikzcd}
&\hspace{.25in}&
\begin{tikzcd}[arrows={-stealth}, cramped]
&  A_1 \edge{ld} \edge{rd}  & \\[-.07in]
b_2 \edge{dd} \edge{rd} && c_2 \edge{dd} \edge{ld} \\[-.1in]
& Q \edge{dd} \\[-.2in]
C_1 \edge{rd} &  & B_1\edge{ld} \\[-.07in]
& a_2
\end{tikzcd}
\\[.9in]
\text{(i)} && \text{(ii)}&& \text{(iii)}
\end{array}
\end{equation*}
\vspace{-.1in}
\caption{A hexagon and its two tilings. 
}
\label{fig:6-cycle-tiling}
\end{figure}

\vspace{-.1in}

\begin{proof}
The last three statements are equivalent to each other by Proposition~\ref{pr:coherence=tilability}.
The existence of the tiling in Figure~\ref{fig:6-cycle-tiling}(ii)
means that we can draw a line~$p$ through the points $a_1\cap a_2$, $b_1\cap b_2$, $c_1\cap c_2$;
likewise, the existence of the tiling in Figure~\ref{fig:6-cycle-tiling}(iii) 
means that we can find a point~$Q$ lying on the lines $(A_1A_2)$, $(B_1B_2)$, $(C_1C_2)$. 
\end{proof}

\begin{remark}
\label{rem:tilability-generic}	
In Corollary~\ref{cor:tileable-hexagons}, when we say that $A_1B_1C_1$ and $A_2B_2C_2$ are \emph{triangles},
we implicitly require that 
\begin{itemize}[leftmargin=.2in]
\item 
the points $A_1,B_1,C_1$ are not collinear and
\item
the lines $a_2,b_2,c_2$ is not concurrent.
\end{itemize}
To illustrate the role of this requirement, let's assume, on the contrary, that $A_1,B_1,C_1$ are distinct \emph{collinear} points.
Let $Q$ be a fourth generic point. 
Pick generic points $A_2\in (A_1Q)$, ${B_2\in (B_1Q)}$, $C_2\!\in\! (C_1Q)$.
Set $a_2\!=\!(B_2C_2)$, $b_2\!=\!(A_2C_2)$, $c_2\!=\!(A_2B_2)$. 
Then the tiling in \hbox{Figure~\ref{fig:6-cycle-tiling}(iii)} is coherent by construction. 
However, there is no coherent tiling as in Figure~\ref{fig:6-cycle-tiling}(ii),
since the line~$p$ would have to pass through the points $(A_1B_1)\cap c_2$, $(A_1C_1)\cap b_2$, $(B_1C_1)\cap a_2$,
implying that $p$ is the line through the three collinear points $A_1,B_1,C_1$. 
But this would violate the condition that in a coherent tile, no two adjacent labels (a point and a line)
are incident to each other. 
\end{remark}

\newpage

\section*{Coherent octagons}

In this section, we discuss two kinds of coherent octagons
(see Propositions~]\ref{pr:octagon-special} and~\ref{pr:octagon}), which come from special choices
of vertex labels. 

\begin{proposition}
\label{pr:octagon-special}
Let $P_1, P_2, P_3, P_4$ be four generic points on the plane.
Let $a$ and~$b$ be lines passing through \hbox{$(P_1P_2)\cap (P_3P_4)$} and \hbox{$(P_1P_4)\cap (P_2P_3)$}, respectively. 
%
Then the octagon whose vertices are labeled by
$P_1, a, P_2, b, P_3, a, P_4, b$ (see Figure~\ref{fig:8-gon-tiling-special}) is coherent. 
\end{proposition}

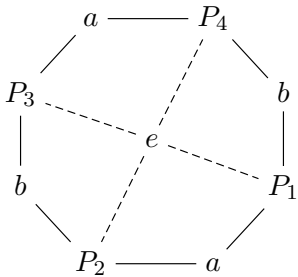
\begin{figure}[ht]
\begin{center}
\vspace{-.2in}
\begin{equation*}
\begin{tikzpicture}[baseline= (a).base]
\node[scale=1] (a) at (0,0){
\begin{tikzcd}[arrows={-stealth}, cramped, sep=8]
& a \edge{rr} && P_4 \edge{rd} & \\[5pt]
P_3 \edge{dd} \edge{ru} & && & b \\[-4pt]
&& e \edge{llu, dashed}\edge{ruu, dashed} \edge{rrd, dashed} \edge{ldd, dashed}\\[-4pt]
b \edge{rd} & && & P_1 \edge{uu}\\[5pt]
& P_2 \edge{rr} && a \edge{ru}
\end{tikzcd}
};
\end{tikzpicture}
\end{equation*}
\vspace{-.15in}
\end{center}
\caption{The octagon from Proposition~\ref{pr:octagon-special} and its tiling. 
Note that each of the lines $a$ and $b$ appears twice along the boundary. }
\label{fig:8-gon-tiling-special}
\end{figure}

\vspace{-.1in}

\begin{proof}
The proof is based on the tiling shown in Figure~\ref{fig:8-gon-tiling-special}, where $e=(MN)$ is the line
through the points $M=(P_1P_2)\cap (P_3P_4)$
and $N=(P_1P_4)\cap (P_2P_3)$. 
It is strightforward to check that the conditions for all four tiles are satisfied. 
\end{proof}

\begin{remark}
Our second proof of the Pappus theorem can be viewed as an application of Proposition~\ref{pr:octagon-special}: 
the octagons $P_4 \,r P_5\,t P_4\,s P_5\, b$ and $P_5 \,s P_4\,b P_5\,r P_4\, q$
in Figure~\ref{fig:pappus2-tiling} satisfy, up to projective duality, the conditions in Proposition~\ref{pr:octagon-special}.
\end{remark}


\begin{proposition}
\label{pr:octagon}
Let $P_1, P_2, P_3, P_4$ be four generic points on the plane.
Then the 
octa\-gon 
$(P_1P_2) P_3 (P_4P_1) P_2 (P_3P_4) P_1 (P_2P_3) P_4$
(see~Figure~\ref{fig:8-gon}) is coherent. 
\end{proposition}


\begin{proof} 
The claim follows from the tiling shown in Figure~\ref{fig:8-gon}, with 
$H=(P_1P_3)\cap(P_2P_4)$. 
It is easy to check that all four tiles are coherent. 
For example, the lines $(P_1P_2)$, $(P_2P_3)$, and $(P_4H)=(P_2P_4)$ intersect at~$P_2$. 
\end{proof}

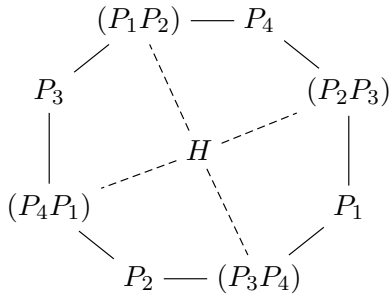
\begin{figure}[ht]
\begin{center}
\vspace{-.2in}
\begin{equation*}
\begin{tikzpicture}[baseline= (a).base]
\node[scale=1] (a) at (0,0){
\begin{tikzcd}[arrows={-stealth}, cramped, sep=-5]
& (P_1P_2) \edge{rr} && P_4 \edge{rd} & \\[.2in]
P_3 \edge{dd} \edge{ru} & && & (P_2P_3) \\[11pt]
&& H \edge{luu, dashed}\edge{rru, dashed} \edge{rdd, dashed} \edge{lld, dashed}\\[11pt]
(P_4P_1) \edge{rd} & && & P_1 \edge{uu}\\[.2in]
& P_2 \edge{rr} && (P_3P_4) \edge{ru}
\end{tikzcd}
};
\end{tikzpicture}
\end{equation*}
\vspace{-.25in}
\end{center}
\caption{A coherent octagon obtained from a plane quadrilateral $P_1P_2P_3P_4$. 
}
\label{fig:8-gon}
\end{figure}

\newpage

\newpage

\section*{Coherent decagons}

Proposition~\ref{pr:octagon} straightforwardly generalizes to larger polygons with the number of sides not divisible by~3. 
The next interesting example is the decagon:

\begin{proposition}
\label{prop:10-gon}
The decagon shown in Figure~\ref{fig:10-gon} is coherent.
\end{proposition}

\begin{figure}[ht]
\begin{center}
\vspace{-.15in}
\begin{equation*}
\begin{tikzcd}[arrows={-stealth}, cramped, sep=small]
P_3 \edge{d} \edge{rr} && (P_4P_5) \edge{rr} && P_1 \edge{d} \\[5pt]
(P_1P_2) \edge{d} &&&& (P_2P_3) \edge{d}\\[5pt]
P_5 \edge{r} & (P_3P_4) \edge{r} & P_2 \edge{r} & (P_1P_5) \edge{r} & P_4
\end{tikzcd}
\end{equation*}
\vspace{-.1in}
\end{center}
\caption{A coherent decagon.}
\label{fig:10-gon}
\end{figure}
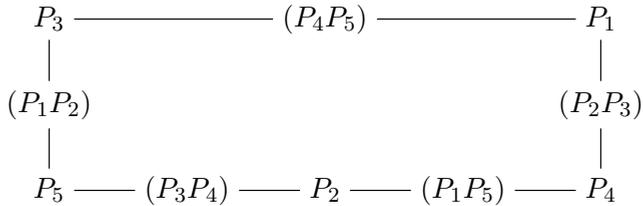

\vspace{-.3in}

\begin{proof}[Proof~1]
The generalized mixed cross-ratio (cf.\ \eqref{eq:coherent-polygon}) is equal to~$1$,
as each of the five terms in the numerator cancels out a term in the denominator. 
\end{proof}

The above proof, while very simple, relies on algebraic calculation.
Since our key goal is to show the universal applicability of tiling-based arguments, we provide two alternative proofs below. 
A~fourth proof (which is simpler but relies on a modification of the tiling method)
will be given in Section~\ref{sec:anticoherent}. 

\begin{proof} [Proof~2]
Denote 
\begin{equation*}
\begin{array}{llll}
Q = (P_1P_5)\cap (P_2P_3), \quad & S =(P_1P_2) \cap (P_3P_4),  
\\[5pt]
\!M =(P_1P_2) \cap (P_4Q),              & \!N = (P_2P_3) \cap (P_5S), \quad & T = (P_4Q)\cap (P_5S), \\[5pt]
U = (P_1N)\cap(P_3M), 
\end{array}
\end{equation*}
see Figure~\ref{fig:10-gon-tiling-picture}. 
The claim can now be obtained from the tiling 
shown in Figure~\ref{fig:10-gon-tiling-no-holes}. 
\end{proof}


\enlargethispage{1cm}

\begin{figure}[ht]
\begin{center}
\includegraphics[scale=0.5, trim=8.5cm 12cm 10cm 12.5cm, clip]{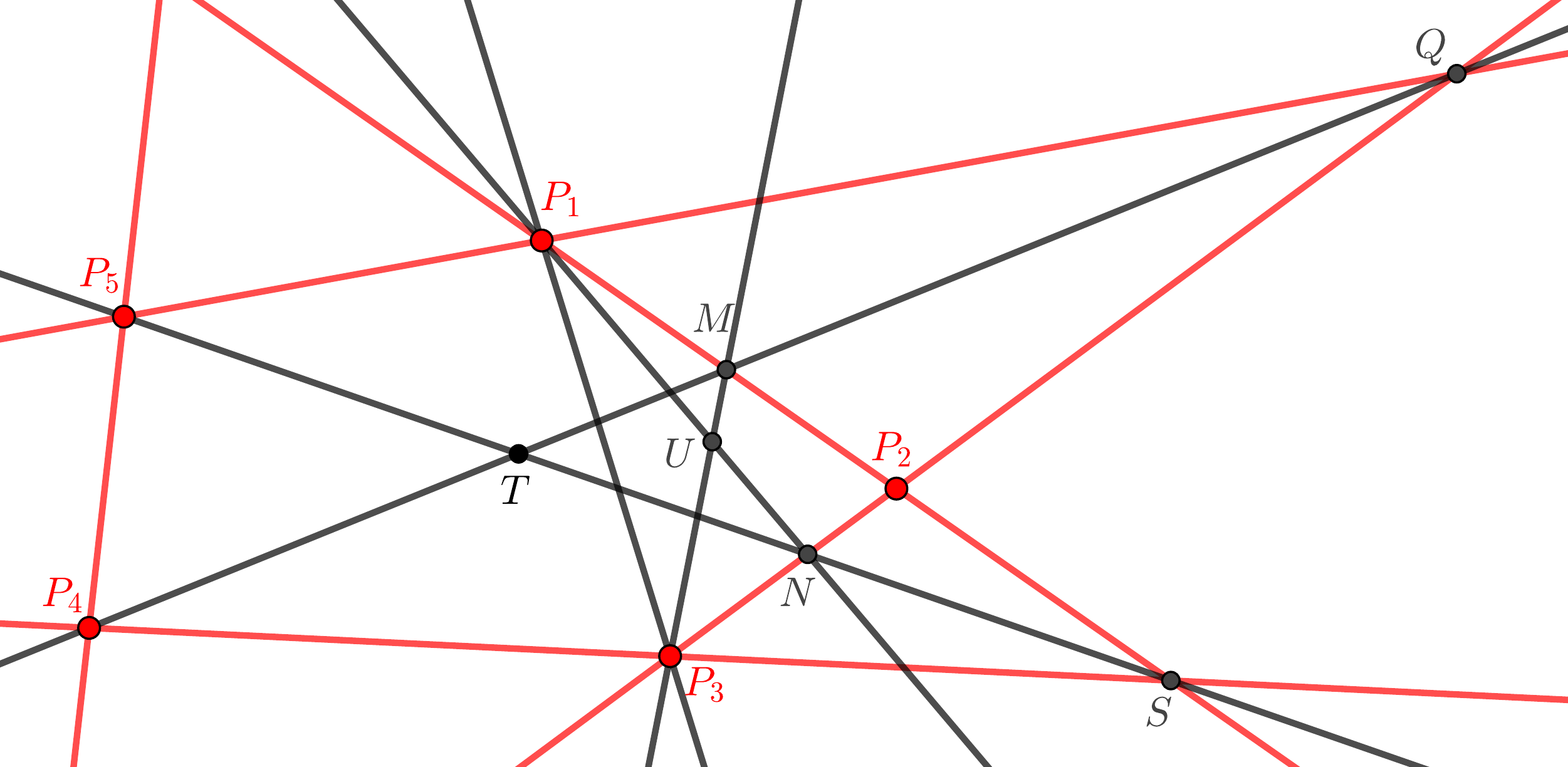}
\end{center}
\vspace{-.1in}
\caption{Notation used in the proof of Proposition~\ref{prop:10-gon}. 
}
\label{fig:10-gon-tiling-picture}
\end{figure}

\clearpage

\begin{figure}[ht]
\begin{center}
\vspace{-.2in}
\begin{equation*}
\!
\begin{tikzpicture}[baseline= (a).base]
\node[scale=0.7] (a) at (0,0){
\begin{tikzcd}[arrows={-stealth}, cramped]
&&&&& P_4P_5 \edge{dll}\edge{drr}\edge{dddlll, bend right =15, line width=1.7pt}\edge{dddrrr, bend left =15, line width=1.7pt}
\\[0pt]
&&& S \edge{dlll} \edge{ddll} &&&& Q \edge{drrr} \edge{ddrr}
\\[15pt]
MN \edge{d} &&& P_5 \edge{d, line width=1.7pt}\edge{r, line width=1.7pt} & P_3P_4  \edge{d}\edge{r, line width=1.7pt} & P_2  \edge{d}\edge{r, line width=1.7pt} & P_1P_5  \edge{d}\edge{r, line width=1.7pt} & P_4  \edge{d, line width=1.7pt} &&& MN \edge{d} 
\\[0pt]
T \edge{r}\edge{d} & NP_4 \edge{r}\edge{rdd} & P_3 \edge{r, line width=1.7pt} \edge{rrrddd}& P_1P_2 \edge{r} \edge{drr, dashed}& N \edge{r} & P_1P_3 \edge{r} \edge{d, dashed}& M \edge{r} & P_2P_3 \edge{r, line width=1.7pt} \edge{dll, dashed}& P_1 \edge{r}\edge{lllddd} & MP_5 \edge{r} \edge{ldd}& T  \edge{d}
\\[0pt]
P_4P_5 \edge{rrd} &&&&& U \edge{dd, dashed}&&&&& P_4P_5 \edge{lld}
\\[0pt]
&& Q \edge{rrrd}&&&&&& S \edge{llld}
\\[15pt]
&&&&& MN
\end{tikzcd}
};
\end{tikzpicture}
\end{equation*}
\vspace{-.15in}
\end{center}
\caption{Tileability of the decagon from Figure~\ref{fig:10-gon} (shown in~\textbf{bold}). 
Opposite~sides of the shown hexagonal fundamental domain should be identified, yielding a torus. 
The tiling of the octagon 
$P_1P_2$\,--\,$P_3$\,--\,$MN$\,--\,$P_1$\,--\,$P_2P_3$\,--\,$M$\,--\,$P_1P_3$\,--\,$N$\,--\,$P_1P_2$
around the vertex~$U$
follows the pattern of Figure~\ref{fig:8-gon}. 
Here we use that $(P_1P_2)\!=\!(P_1M)$ and ${(P_2P_3)\!=\!(P_3N)}$. 
Furthermore, replacing the quadrilateral $P_1$\,--\,$MN$\,--\,$P_3$\,--\,$P_4P_5$\,--\,$P_1$ by a single tile yields
the tiling in Figure~\ref{fig:pappus2-tiling}, so the tileability of this quadrilateral can be interpreted as an application of 
the Pappus theorem. 
}
\vspace{-.25in}
\label{fig:10-gon-tiling-no-holes}
\end{figure}
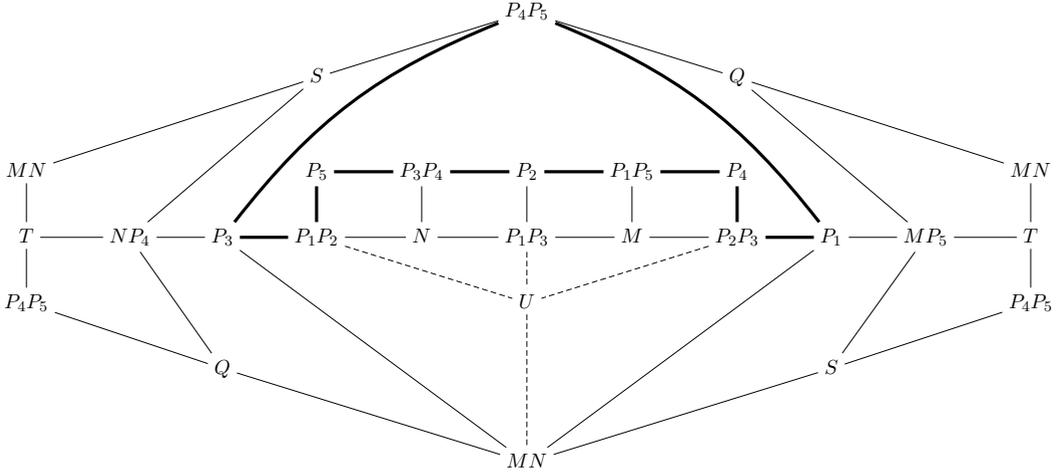

\begin{proof}[Proof~3]
This proof utilizes the tiling of a genus~2 surface
shown in Figure~\ref{fig:10-gon-tiling-holes}. 
\end{proof}


\begin{figure}[ht]
\begin{center}
\vspace{-.1in}
\begin{equation*}
\begin{tikzpicture}[baseline= (a).base]
\node[scale=0.9] (a) at (0,0){
\begin{tikzcd}[arrows={-stealth}, cramped]
P_3 \edge{rd} \edge{rrr, line width=1.5pt} \edge{dddd, line width=1.5pt} &&& P_4 P_5 \edge{ld} \edge{rd} \edge{rrr, line width=1.5pt} \edge{ddddd}  &&& P_1  \edge{dddd, line width=1.5pt} \edge{ld}  \\[0pt]
& P_2P_4 \edge{dd}  \edge{r}   & S \edge{ddd}  && Q \edge{ddd} \edge{r} & P_2P_5 \edge{dd} \\[-15pt]
&& \hspace{-.75in} \text{hole\#1} &&& \hspace{-.75in} \text{hole\#1} \\[-15pt]
& Q \edge{dl} \edge{dr} &&&& S \edge{dr}\\[-15pt]
P_1P_2  \edge{dr} \edge{dddd, line width=1.5pt} & & P_2P_5 \edge{dr}& & P_2P_4 \edge{dr}\edge{dl} \edge{ur} && P_2P_3 \edge{dddd, line width=1.5pt}  \\[-15pt]
& N  \edge{dr} \edge{dddd} \edge{ur} && T \edge{dr} \edge{dddd} && M \edge{ur} \edge{dl} \edge{dddd}  \\[-15pt]
&& P_1P_5 \edge{dd} \edge{ur} \edge{d} && P_3P_4 \edge{dd}  \\[-15pt]
&& \hspace{-.75in} \text{hole\#2} &&& \hspace{-.75in} \text{hole\#2} \\[-15pt]
P_5 \edge{dr, line width=1.5pt}&& M \edge{dr}&& N \edge{dr}&& P_4 \\[-15pt]
& P_3P_4 \edge{ur} \edge{rrd, line width=1.5pt} && QS \edge{d} \edge{ur} && P_1P_5 \edge{ur, line width=1.5pt} \\
&&& P_2 \edge{urr, line width=1.5pt} 
\end{tikzcd}
};
\end{tikzpicture}
\end{equation*}
\vspace{-.1in}
\end{center}
\caption{Another tiling of the decagon from Figure~\ref{fig:10-gon}. 
The two holes in each pair (\#1 and \#2) should be glued to each other.
}
\vspace{-.1in}
\label{fig:10-gon-tiling-holes}
\end{figure}
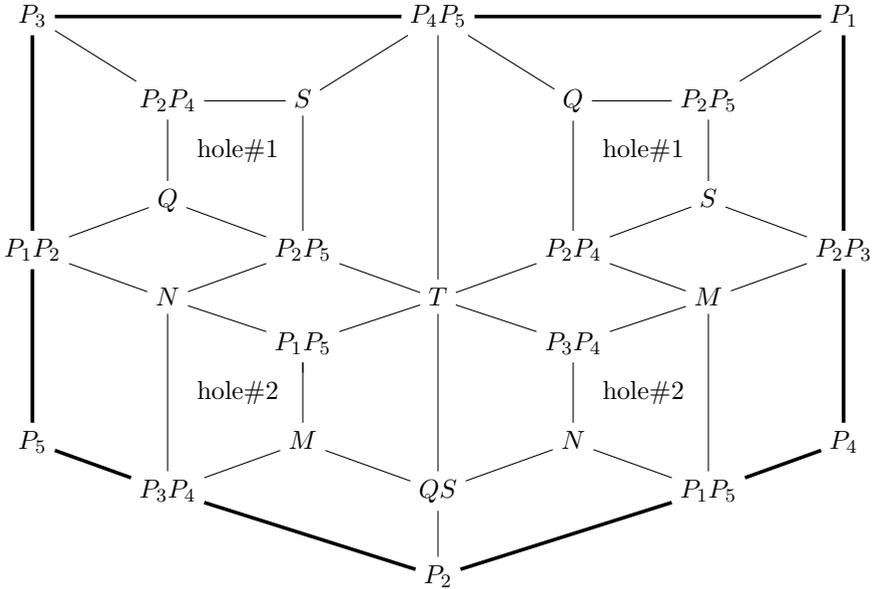

%
\clearpage

\newpage 

\section*{The Goodman-Pollack theorem}

Coherent polygons of the kind shown in Figure~\ref{fig:10-gon} 
(cf.\ Proposition~\ref{pr:octagon}) appear in the proof of the ``$n$-star theorem'' of J.~Goodman and R.~Pollack, 
see \cite[Figure~3]{goodman-pollack} and \cite[Example~14]{richter-gebert-mechanical}. 
Here is the statement for $n=5$. 

\begin{theorem}
\label{th:goodman-pollack}
Let $P_1,\dots,P_5$ be five generic points on the plane. 
Denote
\begin{align}
\notag
Q_1&=(P_3P_4)\cap(P_2P_5), \\
\notag
Q_2&=(P_4P_5)\cap(P_3P_1), \\
\label{eq:Q1...Q5}
Q_3&=(P_5P_1)\cap(P_4P_2), \\
\notag
Q_4&=(P_1P_2)\cap(P_5P_3), \\
\notag
Q_5&=(P_2P_3)\cap(P_1P_4). 
\end{align}
If four of the points $Q_1,\dots,Q_5$ lie on a line, then the fifth point lies on the same line. 
\end{theorem}

\begin{proof}
The conditions in Theorem~\ref{th:goodman-pollack} are encoded by the 5~tiles shown in Figure~\ref{fig:10-gon-tiling}.  
The claim then follows by Proposition~\ref{prop:10-gon}. 
\end{proof}

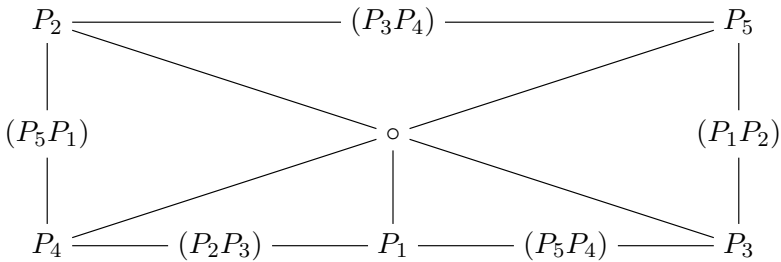
\begin{figure}[ht]
\begin{center}
\begin{equation*}
\begin{tikzpicture}[baseline= (a).base]
\node[scale=1] (a) at (0,0){
\begin{tikzcd}[arrows={-stealth}, cramped]
P_2 \edge{d} \edge{rr} && (P_3P_4) \edge{rr} && P_5 \edge{d} \\[5pt]
(P_5P_1) \edge{d} & &\circ \edge{llu} \edge{rru}  \edge{d} \edge{lld} \edge{rrd} &  & (P_1P_2) \edge{d}\\[5pt]
P_4 \edge{r} & (P_2P_3) \edge{r} & P_1 \edge{r} & (P_5P_4) \edge{r} & P_3
\end{tikzcd}
};
\end{tikzpicture}
\end{equation*}
\end{center}
\caption{The proof of the Goodman-Pollack theorem.
The line at the center passes through the points $Q_1,\dots,Q_5$.}
\label{fig:10-gon-tiling}
\end{figure}

\begin{remark}
To get a self-contained tiling-based proof of Theorem~\ref{th:goodman-pollack}, paste the tiling shown in Figure~\ref{fig:10-gon-tiling} 
into the decagonal region appearing in Figure~\ref{fig:10-gon-tiling-no-holes}. 
Alternatively, glue together the boundaries of the tilings shown in Figures~\ref{fig:10-gon-tiling-holes} and~\ref{fig:10-gon-tiling}. 
\end{remark}

Proposition~\ref{pr:coherence=tilability} suggests the following generalization of Theorem~\ref{th:goodman-pollack}.

\begin{theorem}
\label{th:5-points-theorem}
Let $P_1, \dots, P_5$ be five generic points on the plane. 
Define $Q_1,\dots,Q_5$ as in~\eqref{eq:Q1...Q5}. 
Then the points $(P_1P_2)\cap(Q_3Q_5)$, 
$(P_5P_1)\cap(Q_4Q_2)$, and~$Q_1$ are collinear. 
See Figure~\ref{fig:5-points-theorem}. 
\end{theorem}

Before proving Theorem~\ref{th:5-points-theorem}, let us explain why it can be viewed as a generalizion
of Theorem~\ref{th:goodman-pollack}.
If in Theorem~\ref{th:5-points-theorem}, we make the additional (redundant) assumption
that the points $Q_2,Q_3,Q_4,Q_5$ lie on a line~$\ell$,
then $(Q_3Q_5)=(Q_4Q_2)=\ell$ and therefore
Theorem~\ref{th:5-points-theorem} implies that $Q_1\in\ell$. 

\begin{proof}[Proof of Theorem~\ref{th:5-points-theorem}]
Combine Proposition~\ref{prop:10-gon} with the tiling shown in Figure~\ref{fig:5-points-theorem-tiling}. 
\end{proof}


\begin{figure}[ht]
\begin{center}
\includegraphics[scale=0.55, trim=0.6cm 0cm 0.5cm 0cm, clip]{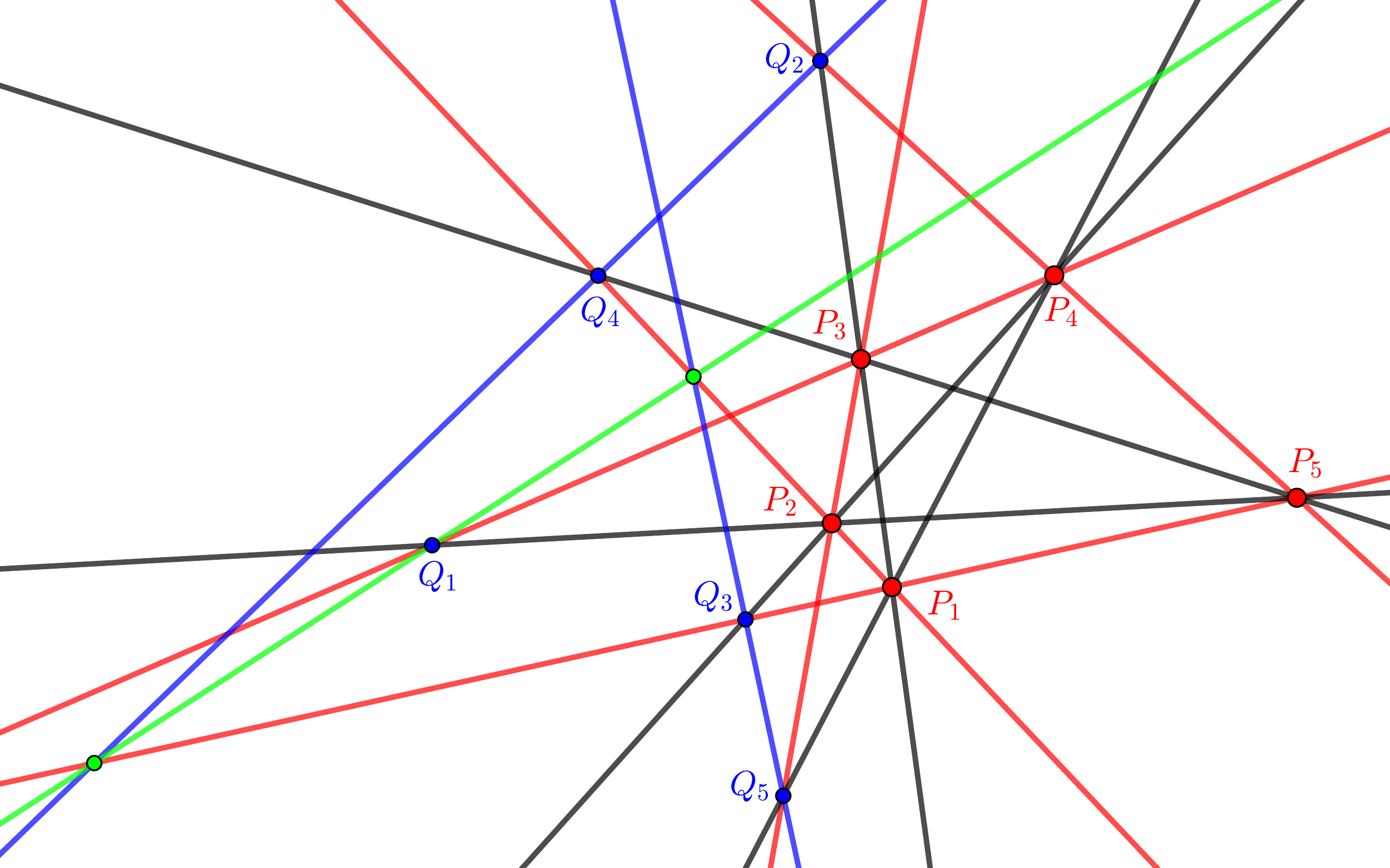}
\end{center}
\caption{The configuration in Theorem~\ref{th:5-points-theorem}. 
}
\label{fig:5-points-theorem}
\end{figure}

\begin{figure}[ht]
\begin{center}
\begin{equation*}
\begin{tikzpicture}[baseline= (a).base]
\node[scale=1] (a) at (0,0){
\begin{tikzcd}[arrows={-stealth}, cramped]
P_2 \edge{d} \edge{rr} && (P_3P_4) \edge{rr} && P_5 \edge{d} \\[5pt]
(P_5P_1) \edge{d} & (Q_3Q_5) \edge{lu} \edge{ld} \edge{rd} &\circ \edge{llu} \edge{rru}  \edge{d} & (Q_2Q_4) \edge{ld} \edge{rd} \edge{ru} & (P_1P_2) \edge{d}\\[5pt]
P_4 \edge{r} & (P_2P_3) \edge{r} & P_1 \edge{r} & (P_5P_4) \edge{r} & P_3
\end{tikzcd}
};
\end{tikzpicture}
\end{equation*}

\end{center}
\caption{The proof of Theorem~\ref{th:5-points-theorem}.}
\label{fig:5-points-theorem-tiling}
\end{figure}
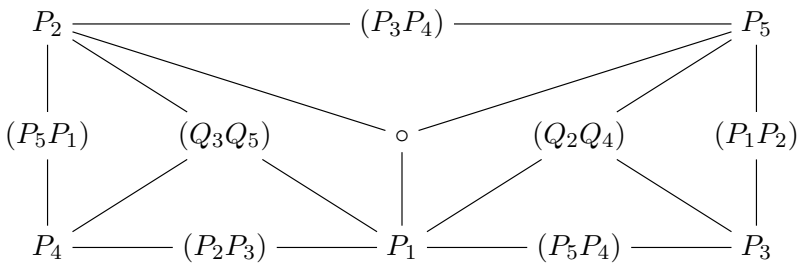

\begin{remark}
\label{rem:12-3-45-6-}
Propositions \ref{pr:octagon}--\ref{prop:10-gon} generalize
to larger $2n$-gons (here ${n\not\equiv 0\bmod 3}$) 
with boundary labels $P_1$\,-\,$(P_2P_3)$\,-\,$P_4$\,-\,$(P_6P_7)$\,-\,$\cdots$.
We provide hints for a tiling-based proof of this statement at the end of Section~\ref{sec:anticoherent}. 
\end{remark}

\begin{remark}
The proof of the general Goodman-Pollack theorem, for $n\ge 5$ points on the plane, breaks into two cases. 
If $n$ is not divisible by~3, the proof is analogous to the one given above for $n=5$, cf.\ Remark~\ref{rem:12-3-45-6-}. 
If $n$ is divisible by~3, then the statement becomes vacuous, since each concurrence condition is repeated more than once. 
\end{remark}

\newpage

\section*{Saam's theorem}

The following result is the $n=5$ instance of Saam's theorem \cite{saam-1987, saam-1988}, 
reproduced in \cite[Example~6]{richter-gebert-mechanical}. 

\begin{theorem}
\label{th:saam-5}
Let $B, P_1, P_2,P_3,P_4,P_5$ be six generic points on the plane.
Pick a generic point $Q_1\in(BP_1)$. Define (reading by columns): 
\begin{align*}
                                                        & Q_1'=(Q_5P_3)\cap(BP_1), \\
Q_2'=(Q_1P_4)\cap (BP_2), \quad & Q_2=(Q_1'P_4)\cap (BP_2), \\
Q_3=(Q_2'P_5)\cap (BP_3), \quad & Q_3'=(Q_2'P_5)\cap (BP_3), \\
Q_4'=(Q_3P_1)\cap (BP_4), \quad & Q_4'=(Q_3'P_1)\cap (BP_4), \\
Q_5=(Q_4'P_2)\cap (BP_5), \quad & Q_5'=(Q_4'P_2)\cap (BP_5). 
\end{align*}
Then the points $Q_5', P_3, Q_1$ are collinear. 
See Figure~\ref{fig:saam-5}. 
\end{theorem}

\begin{figure}[ht]
\begin{center}
\vspace{-.1in}
\includegraphics[scale=0.5, trim=0cm 0cm 0cm 0cm, clip]{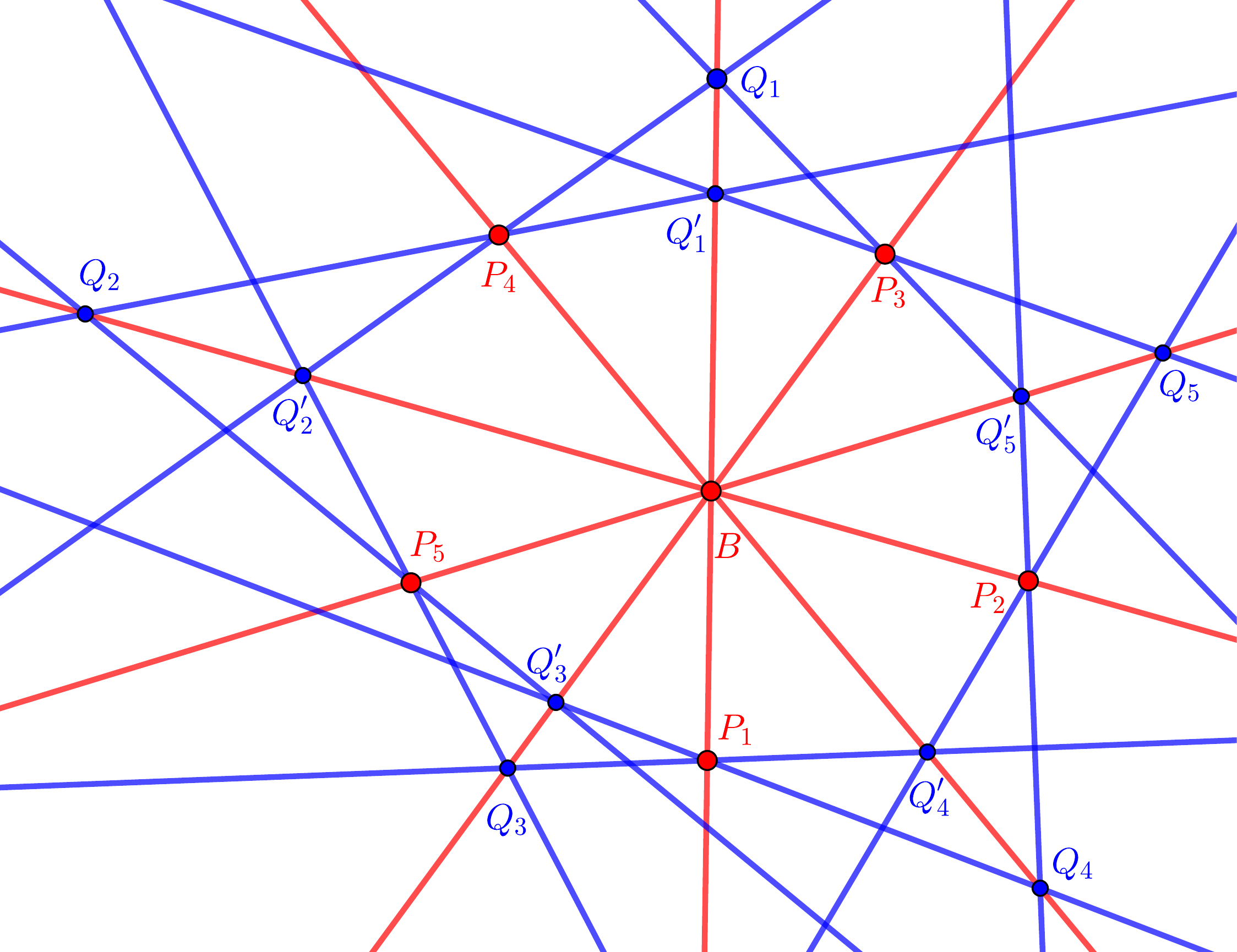}
\end{center}
\vspace{-.1in}
\caption{The configuration in Theorem~\ref{th:saam-5}. 
}
\label{fig:saam-5}
\end{figure}

\vspace{-.25in}

\begin{proof}
The claim follows from the tiling shown in Figure~\ref{fig:saam-tiling},
once we observe that for the 20-gon on the perimeter of this tiling, 
the generalized mixed ratio is equal to~1 for trivial algebraic reasons. 
(All factors cancel each other out in pairs.)  

To get an entirely tiling-based proof, we need to tile the aforementioned 20-gon.
One possible tiling is shown in Figure~\ref{fig:saam-tiling-20-cycle}. 
There, each octagonal ``petal'' is an 8-cycle from Proposition~\ref{pr:octagon},
so it can be tiled as in Figure~\ref{fig:8-gon}. 
In the interior 20-gon, the 10 sides involving the vertices labeled~$B$ can be glued in 5~pairs,
respecting both the labeling and the orientation. 
The remaining 10~sides form the (properly oriented) 10-gon $Q_1-(Q_2Q_4)-Q_5-(Q_1Q_3)-Q_4-(Q_2Q_5)-Q_3-(Q_1Q_4)-Q_2-(Q_3Q_5)-Q_1$. 
This is a 10-gon from Figure~\ref{fig:10-gon}, with respect to the cyclic ordering $Q_1Q_4Q_2Q_5Q_3$;
so it can be tiled as shown in Figure~\ref{fig:10-gon-tiling-no-holes} or
Figure~\ref{fig:10-gon-tiling-holes}. 
\end{proof}


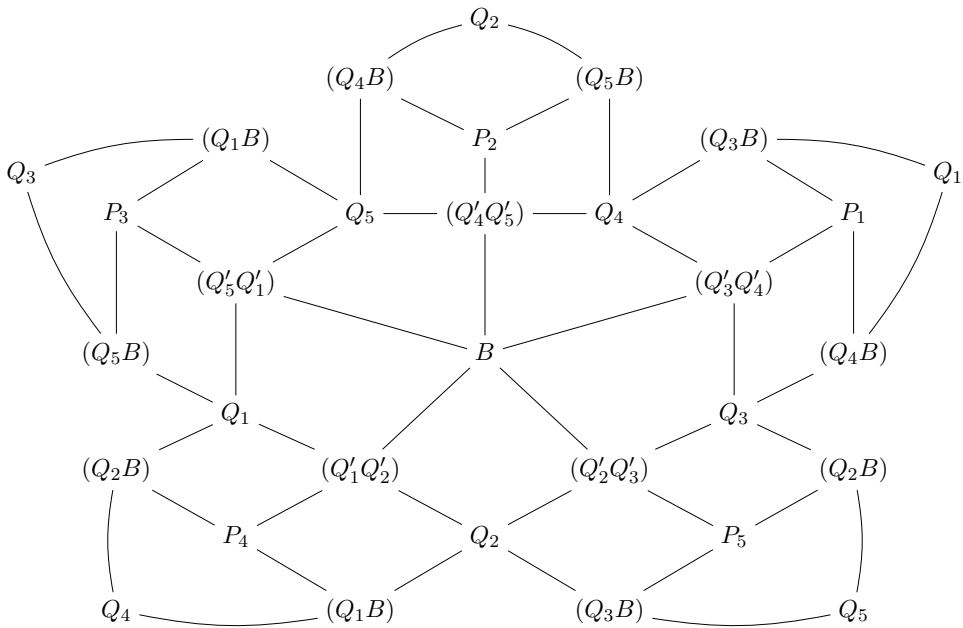
\begin{figure}[ht]
\vspace{-.1in}
\begin{equation*}
\begin{tikzpicture}[baseline= (a).base]
\node[scale=.85] (a) at (0,0){
\begin{tikzcd}[arrows={-stealth}, cramped, sep=small]
&&&& Q_2 \edge{dl, bend right=10} \edge{dr, bend left=10} \\
& & & (Q_4B) \edge{ddd} \edge{dr} && (Q_5B)\edge{ddd}   \\
& & (Q_1B) \edge{rdd} \edge{ldd} &  & P_2 \edge{dd} \edge{ru} &   & (Q_3B)\edge{ldd} \edge{rdd} \edge{drr, bend left=10} \\[-12pt]
Q_3 \edge{dddr, bend right=10} \edge{rru, bend left=10} &&&&&&&& Q_1\edge{dddl, bend left=10} \\[-9pt]
& P_3 \edge{dd} \edge{rd} & & Q_5\edge{r} \edge{ld} & (Q_4'Q_5')\edge{dd} \edge{r} & Q_4 \edge{rd} & & P_1 \edge{dd} \edge{ld} &  \\[3pt]
& & (Q_5'Q_1') \edge{dd} \edge{rrd} & & & & (Q_3'Q_4') \edge{dd} \edge{lld} \\[4pt]
& (Q_5B) \edge{rd} &&& B \edge{rdd} \edge{ldd} &&& (Q_4B)\edge{dl}  \\
& & Q_1 \edge{ld} \edge{rd} &&&& Q_3\edge{ld} \edge{rd}  \\[-2pt]
& (Q_2B) \edge{dd, bend right=10} \edge{rd} && (Q_1'Q_2')\edge{ld} \edge{rd}  && (Q_2'Q_3')\edge{ld} \edge{rd}  && (Q_2B)\edge{dd, bend left=10} \edge{ld}  \\[3pt]
&& P_4\edge{rd}  && Q_2\edge{ld} \edge{rd}  && P_5 \edge{ld} \\[6pt]
& Q_4 \edge{rr, bend right=10}&& (Q_1B) && (Q_3B) \edge{rr, bend right=10}&& Q_5
\end{tikzcd}
};
\end{tikzpicture}
\end{equation*}
\vspace{-.1in}
\caption{The tiling used in the proof of Theorem~\ref{th:saam-5}. 
}
\label{fig:saam-tiling}
\end{figure}
%
%
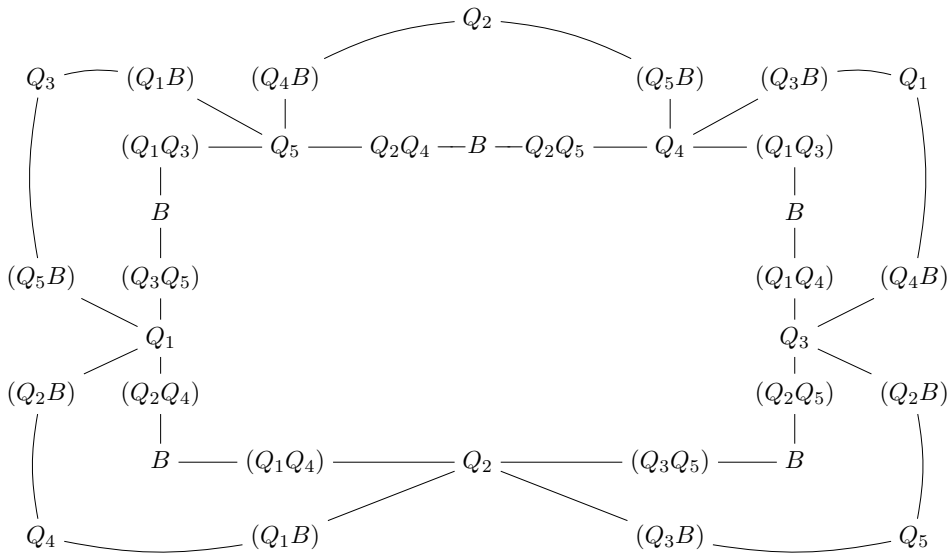
\begin{figure}[ht]
\vspace{-.2in}
\begin{equation*}
\hspace{-10pt}
\begin{tikzpicture}[baseline= (a).base]
\node[scale=.85] (a) at (0,0){
\begin{tikzcd}[arrows={-stealth}, cramped, sep=small]
&&& Q_2 \edge{dl, bend right=10} \edge{dr, bend left=10} \\
Q_3 \edge{ddd, bend right=10} \edge{r, bend left=10}& (Q_1B) \edge{rd} &(Q_4B) \edge{d}  && (Q_5B)\edge{d} & (Q_3B) \edge{ld} \edge{r, bend left=10} & Q_1 \edge{ddd, bend left=10}\\[3pt]
&  (Q_1Q_3) \edge{d}& Q_5\edge{r} \edge{l} & Q_2Q_4-\!\!\!-B-\!\!\!-Q_2Q_5 \edge{r} & Q_4 \edge{r} & (Q_1Q_3) \edge{d} \\[3pt]
& B \edge{d}  & & & & B \edge{d} \\[4pt]
(Q_5B) \edge{rd} &(Q_3Q_5) \edge{d}&&  && (Q_1Q_4) \edge{d}& (Q_4B)\edge{dl}  \\
& Q_1\edge{d} \edge{ld}  &&&& Q_3 \edge{d}\edge{rd}  \\[-2pt]
(Q_2B) \edge{dd, bend right=10}  & (Q_2Q_4) \edge{d}& & & & (Q_2Q_5) \edge{d}& (Q_2B)\edge{dd, bend left=10}   \\[3pt]
& B\edge{r}  & (Q_1Q_4) \edge{r} & Q_2\edge{ld} \edge{rd}\edge{r}   &(Q_3Q_5) \edge{r} & B  \\[6pt]
Q_4 \edge{rr, bend right=10}&& (Q_1B) && (Q_3B) \edge{rr, bend right=10}&& Q_5
\end{tikzcd}
};
\end{tikzpicture}
\end{equation*}
\vspace{-.15in}
\caption{Tiling the 20-gon appearing along the perimeter of Figure~\ref{fig:saam-tiling}. 
}
\label{fig:saam-tiling-20-cycle}
\end{figure}

\clearpage

\newpage

\section{Applications in three-dimensional geometry
}
\label{sec:3D-incidence-geometry}

\section*{The bundle theorem}

The \emph{bundle theorem} (see, e.g., \cite{BBB, kahn} and references therein)
is the following result. 

\begin{theorem}
\label{th:bundle}
Let $\ell_1,\ell_2,\ell_3,\ell_4$ be four lines in the real/complex projective 3-space, no three of them coplanar. 
If~five of the pairs $(\ell_i,\ell_j)$ ($i\neq j$) are coplanar, then the sixth pair is also coplanar. 
\end{theorem}

\begin{proof}
Choose two generic points $P_i$ and $Q_i$ on each of the four lines~$\ell_i$. 
In the tiling of the sphere shown in Figure~\ref{fig:bundle}, the six tiles 
correspond to the six coplanarities for the pairs $(\ell_i,\ell_j)$.
The claim follows.
\end{proof}

\begin{figure}[ht]
\vspace{-.1in}
\begin{equation*}
\begin{tikzcd}[arrows={-stealth}, cramped, sep=6]
Q_3 \edge{rrr} \edge{rd} \edge{ddd}&&& (P_1P_3P_4)  \edge{ddd}\edge{ld}\\[.4in]
& (P_1P_2P_3)  \edge{r}  \edge{d}& Q_1 \edge{d}& \\[.4in]
& Q_2 \edge{r} \edge{ld} & (P_1P_2P_4) \edge{rd}& \\[.4in]
(P_2P_3P_4) \edge{rrr} & & & Q_4
\end{tikzcd}
\end{equation*}
\vspace{-5pt}
\caption{The tiling of the sphere used in the proof of the bundle theorem. 
Each tile~(a~face of the cube) corresponds to a coplanarity condition.
For example, the coherence of the outer face means that the points $P_3,P_4,Q_3,Q_4$ (equivalently, the lines $\ell_3$ and~$\ell_4$) are coplanar. 
}
\label{fig:bundle}
\end{figure}
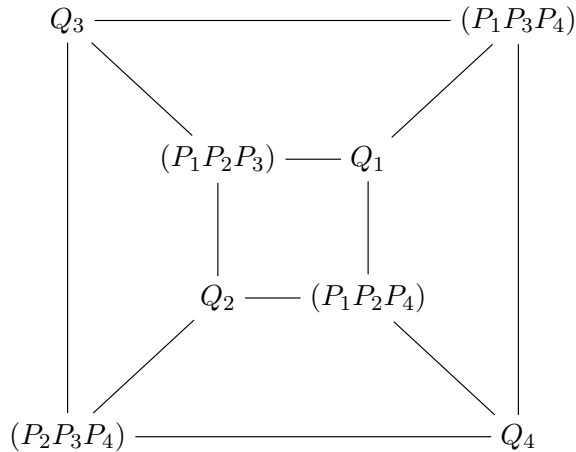

\vspace{-.2in}

\begin{remark}
If three lines do not lie in a plane but are pairwise coplanar, then these three lines intersect at a common point. 
So the four lines $\ell_1,\ell_2,\ell_3,\ell_4$
in Theorem~\ref{th:bundle} do, in fact, intersect at a single point, thus forming a ``bundle'' (hence the name). 
\end{remark}

\begin{remark}
The above proof of Theorem~\ref{th:bundle} implicitly relies on the non-coplanarity condition
on the triples of lines~$\ell_i$. 
Indeed, the coherence of the tiles shown in Figure~\ref{fig:bundle} requires each point~$Q_i$ 
to lie outside the planes of the form $(P_iP_jP_k)$,
which is ensured by the fact that the lines $\ell_i, \ell_j,\ell_k$ are non-coplanar. 
\end{remark}

\begin{remark}
Ignoring the labels, the tiling in Figure~\ref{fig:bundle} coincides with the tiling in Figure~\ref{fig:desargues}.
Thus, the bundle theorem can be viewed as the 3-dimensional counterpart of Desargues' theorem. 

A 3-dimensional counterpart of the Pappus theorem will be presented later in this section, see Theorem~\ref{th:3d-pappus}. 
\end{remark}

\newpage

\section*{The cube theorem}

The \emph{cube theorem} below is due to A.~F.~M\"obius~\cite{moebius}, cf.\ \cite[vol.~1, Exercise I.II.10]{baker} or 
{\cite[Example~11]{richter-gebert-mechanical}}. 

\begin{theorem}
\label{th:cube}
Let $P_1, \dots, P_8$ be points in the 3-space such that none of the quadruples $\{P_i,P_j,P_k,P_\ell\}$
($1\le i<j<k\le4<\ell$) are coplanar. 
If seven of the eight quadruples 
\begin{equation}
\label{eq:8-quadruples}
\begin{array}{l}
\{P_1, P_2, P_7, P_8\}, \{P_1, P_2, P_5, P_6\}, \{P_1, P_4, P_5, P_8\}, \{P_1, P_4, P_6, P_7\}, \\[.05in]
\{P_2, P_3, P_5, P_8\}, \{P_2, P_3, P_6, P_7\}, \{P_3, P_4, P_7, P_8\}, \{P_3, P_4, P_5, P_6\}
\end{array}
\end{equation}
are coplanar, then the remaining quadruple is also coplanar. 
See Figure~\ref{fig:16-points-cube}. 
\end{theorem}

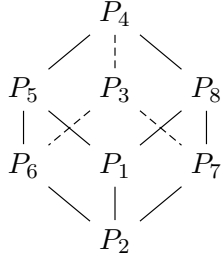
\begin{figure}[ht]
\vspace{-10pt}
\begin{equation*}
\begin{tikzcd}[arrows={-stealth}, cramped, sep=8]
& &  P_4 \arrow[lld, no head] \arrow[rrd, no head] \arrow[d, no head, dashed]  & & \\[5pt]
P_5 \arrow[d, no head] \arrow[rrd, no head] && P_3 \arrow[rrd, no head, dashed] \arrow[lld, no head, dashed]   
     && P_8 \arrow[d, no head] \arrow[lld, no head] \\[5pt]
P_6 \arrow[rrd, no head] && P_1 \arrow[d, no head] && P_7\arrow[lld, no head] \\[5pt]
&& P_2
\end{tikzcd}
\end{equation*}
\vspace{-7pt}
\caption{The cube theorem. Given eight points $P_i$ associated with the vertices of the cube as shown in the picture, 
the theorem asserts that if the four points associated with each face of the cube are coplanar, 
as is the quadruple $\{P_1, P_4, P_6, P_7\}$, then the quadruple $\{P_2, P_3, P_5, P_8\}$ is coplanar as well. 
}
\label{fig:16-points-cube}
\end{figure}

\vspace{-15pt}


\begin{proof}
Apply the master theorem to the tiling of the torus shown in Figure~\ref{fig:cube-thm-tiling}. 
\end{proof}

\enlargethispage{.5cm}

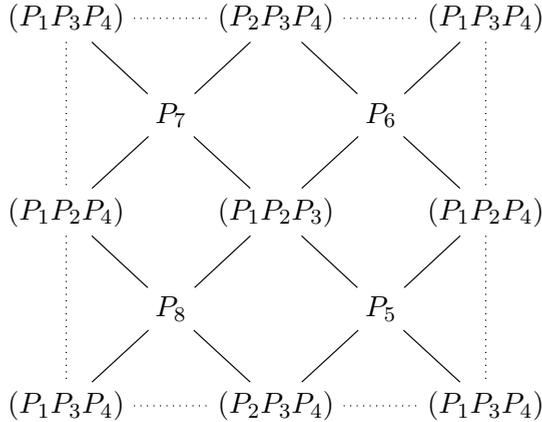
\begin{figure}[ht]
\vspace{-7pt}
\begin{equation*}
\begin{tikzcd}[arrows={-stealth}, cramped, sep=5]
(P_1P_3P_4) \arrow[rd, no head] \edge{rr, dotted} \edge{dd, dotted} && (P_2P_3P_4) \arrow[ld, no head] \arrow[rd, no head]  \edge{rr, dotted} && (P_1P_3P_4) \arrow[ld, no head]  \edge{dd, dotted} \\[15pt]
& P_7 \arrow[ld, no head] \arrow[rd, no head]  && P_6 \arrow[ld, no head] \arrow[rd, no head] \\[15pt]
(P_1P_2P_4) \edge{dd, dotted} \arrow[rd, no head] & & (P_1P_2P_3) \arrow[ld, no head] \arrow[rd, no head] & & (P_1P_2P_4) \edge{dd, dotted} \arrow[ld, no head] \\[15pt]
& P_8 \arrow[ld, no head] \arrow[rd, no head]  && P_5 \arrow[ld, no head] \arrow[rd, no head] \\[15pt]
(P_1P_3P_4)  \edge{rr, dotted} & & (P_2P_3P_4)  \edge{rr, dotted} & & (P_1P_3P_4)
\end{tikzcd}
\end{equation*}
\caption{The tiling of the torus that yields the cube theorem. 
Opposite sides of the rectangular fundamental domain should be glued to each other.
There are no edges between the vertices connected by dotted lines.
The eight coplanarity conditions~\eqref{eq:8-quadruples} are encoded by the eight tiles. 
For example, the tile at the center-top with the vertices $P_7$, $(P_2P_3P_4)$, $P_6$ and $(P_1P_2P_3)$
encodes the coplanarity of the lines $(P_6P_7)$ and $(P_2P_3P_4)\cap (P_1P_2P_3) = (P_2P_3)$,
i.e., the coplanarity of the quadruple $\{P_2,P_3,P_6,P_7\}$. 
}
\label{fig:cube-thm-tiling}
\end{figure}

\clearpage

\newpage

\section*{The octahedron theorem}

In the forthcoming treatment of 4D consistency (see Section~\ref{sec:consistency}), 
we will utilize a version of Theorem~\ref{th:cube}
known as the \emph{octahedron theorem}
(see, e.g., \cite[Example~11]{nixon-schulze-whiteley} and references therein).

\begin{theorem}
\label{th:octahedron}
Let $f_{12}, f_{13}, f_{14}, f_{23}, f_{24}, f_{34}$ be six points in 3-space. 
Assume that the planes $(f_{12}f_{13}f_{14})$, 
$(f_{12}f_{23}f_{24})$, 
$(f_{13}f_{23}f_{34})$,  
$(f_{14}f_{24}f_{34})$ have a common point and the six given points are otherwise generic. 
Then the planes 
$(f_{12}f_{13}f_{23})$ 
$(f_{12}f_{14}f_{24})$, 
$(f_{13}f_{14}f_{34})$,
$(f_{23}f_{24}f_{34})$ 
have a common point. 
\end{theorem}

Theorem~\ref{th:octahedron} can be visually presented as follows, cf.\ Figure~\ref{fig:octahedron}. 
Given an (irregular) octahedron in the real \hbox{3-space}, color its eight faces in two colors, red and blue,
so that adjacent faces are colored in different colors.
Then the four red faces intersect at a point if and only if the four blue faces intersect at a point.

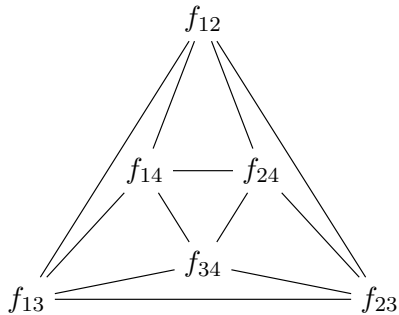
\begin{figure}[ht]
\begin{equation*}
\begin{tikzcd}[arrows={-stealth}, cramped, sep=1]
&\ &\  &\  &  f_{12} \edge{llllddd} \edge{ld} \edge{rd} \edge{rrrrddd} \\[40pt]
&& & f_{14} \edge{rr} \edge{rd} \edge{llldd}&&  f_{24} \edge{ld} \edge{rrrdd}\\[18pt]
&& && f_{34} \edge{lllld} \edge{rrrrd}\\[-4pt]
f_{13} \edge{rrrrrrrr}  &&&&&&\ &\ & f_{23}  
\end{tikzcd}
\end{equation*}
\caption{The octahedron theorem (Theorem~\ref{th:octahedron}). 
}
\vspace{-.1in}
\label{fig:octahedron}
\end{figure}

\begin{proof}
The octahedron theorem is obtained from the cube theorem (Theorem~\ref{th:cube})
under the following change of notation:
\begin{align*}
P_1 &= (f_{12}f_{13}f_{14})\cap (f_{12}f_{23}f_{24})\cap (f_{13}f_{23}f_{34})\cap (f_{14}f_{24}f_{34}), \\
P_2 &= f_{12}, \\
P_3 &= (f_{12}f_{13}f_{23})\cap (f_{12}f_{14}f_{24})\cap (f_{13}f_{14}f_{34})\cap (f_{23}f_{24}f_{34}), \\
P_4 &= f_{34}, \\
P_5 &= f_{13}, \\
P_6 &= f_{14}, \\
P_7 &= f_{24}, \\
P_8 &= f_{23}. 
\end{align*}
To be a bit more precise, set $P_3=(f_{12}f_{13}f_{23})\cap (f_{12}f_{14}f_{24})\cap (f_{13}f_{14}f_{34})$.
Then all quadruples in~\eqref{eq:8-quadruples} except for $\{P_3, P_4, P_7, P_8\}$
are coplanar, hence the latter quadruple is coplanar as well. 
\end{proof}

\begin{remark}
Cox's first theorem \cite[Exercise~2.28]{bobenko-suris-book}
is essentially a reformulation of the octahedron theorem.
\end{remark}

\newpage

\section*{The sixteen points theorem}

The sixteen points theorem is the following classical result in 3-dimensional projective geometry, see, e.g., 
\cite[p.~46]{baker}, 
\cite[$\S$4.2]{heyting} or~\cite[Example~11]{richter-gebert-mechanical}.

\begin{theorem}[The sixteen points theorem]
\label{th:16-points}
Let $a_1, a_2, a_3, a_4, b_1, b_2, b_3, b_4$ be a generic collection of eight lines in 3-space
such that fifteen of the pairs $(a_i, b_j)$ are coplanar. Then the remaining pair is also~coplanar.
\end{theorem}

\begin{proof}[Proof~1]
Let the eight given lines be 
\begin{equation*}
\begin{array}{llll}
a_1=(P_1Q_1), \qquad & a_2=(P_2Q_2), \qquad & a_3=(P_3Q_3), \qquad & a_4=(P_4Q_4), \\[.05in]
b_1=(P_1R_1), & b_2=(P_2R_2), & b_3=(P_3R_3), & b_4=(P_4R_4). 
\end{array}
\end{equation*}
Then the four conditions $a_i\cap b_i\neq\varnothing$ are automatic, while the remaining 12 conditions
$a_i\cap b_j\neq\varnothing$ ($i\neq j$) are encoded by the 12 tiles of the tiling of the torus shown in Figure~\ref{fig:16-points-torus}.
The claim follows. 
\end{proof}

\begin{figure}[ht]
\vspace{-.1in}
\begin{equation*}
\label{eq:16-points-torus}
\begin{tikzpicture}[baseline= (a).base]
\node[scale=1] (a) at (0,0){
\begin{tikzcd}[arrows={-stealth}, cramped, sep=small]
&& R_1 \edge{rr} \edge{ld} && (P_1P_2P_3) \edge{rr} \edge{ld} \edge{rd}&& Q_1 \edge{rd} \\[40pt]
& (P_1P_3P_4) \edge{rr} \edge{ld} \edge{rd} && Q_3 \edge{rd} && R_2 \edge{rr} \edge{ld} && (P_1P_2P_4) \edge{ld} \edge{rd}\\[40pt]
Q_1 \edge{rd} && R_4 \edge{rr} \edge{ld} && (P_2P_3P_4) \edge{rr} \edge{rd} \edge{ld}&& Q_4\edge{rd} & & R_1\edge{ld} \\[40pt]
& (P_1P_2P_4) \edge{rr} \edge{rd} && Q_2 \edge{rd} && R_3 \edge{rr} \edge{ld} && (P_1P_3P_4) \edge{ld} \\[40pt]
&& R_1 \edge{rr} && (P_1P_2P_3) \edge{rr} && Q_1 
\end{tikzcd}
};
\end{tikzpicture}
\end{equation*}
\vspace{-.1in}
\caption{The tiling of the torus used in the proof of the sixteen points theorem.
Opposite sides of the hexagonal fundamental domain should be glued to each~other.
The~coherence of each tile can be interpreted as coplanarity of two lines $a_i$ and~$b_j$ with $i\neq j$.
For~example, the tile at the center-top with the vertices $Q_3$, $(P_1P_2P_3)$, $R_2$, and $(P_2P_3P_4)$
encodes the coplanarity of the points $P_2, P_3, Q_3, R_2$ or, equivalently, 
the coplanarity of the lines $a_3=(P_3Q_3)$ and $b_2=(P_2R_2)$.
}
\label{fig:16-points-torus}
\end{figure}
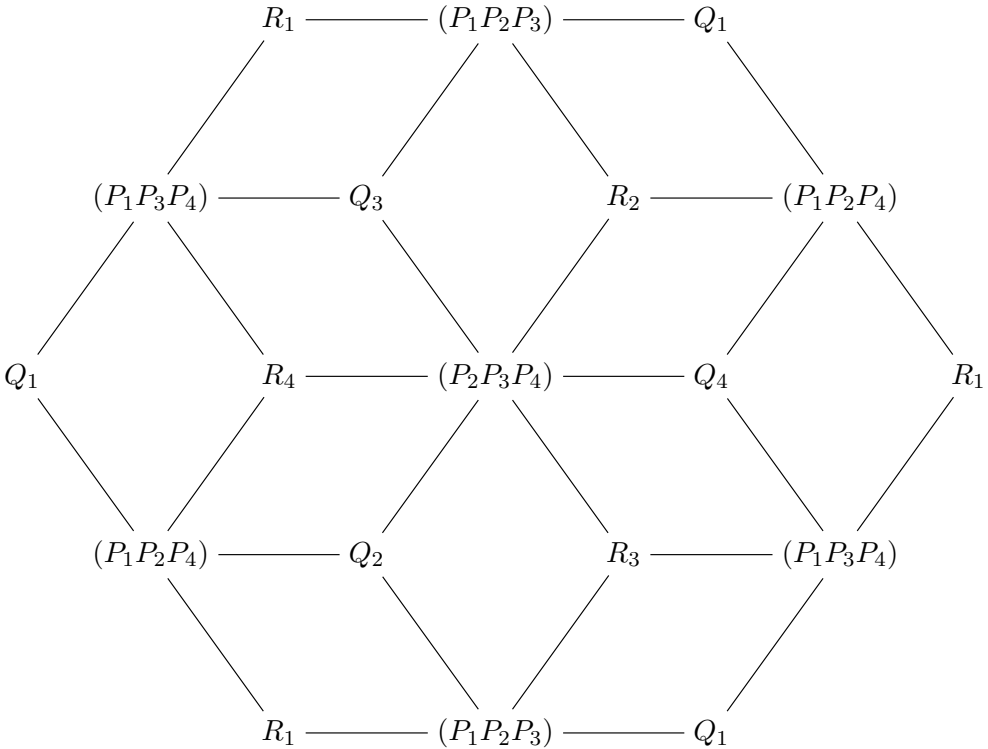

\vspace{-.1in}

\begin{proof}[Proof~2]
The sixteen points theorem can be viewed as the cube theorem in disguise, as we shall now explain. 
Assume that the pairs $(a_i, b_j)$, for $i,j\in\{1,2\}$ are coplanar,
as are the pairs $(a_i, b_j)$, for $i,j\in\{3,4\}$.
Denote
\begin{equation*}
\begin{array}{l}
P_1 = a_1\cap b_2, \\[.05in]
P_2 = a_1\cap b_1, 
\end{array}
\quad
\begin{array}{l}
P_3 = a_2\cap b_1, \\[.05in]
P_4 = a_2\cap b_2, 
\end{array}
\quad
\begin{array}{l}
P_5 = a_3\cap b_4, \\[.05in]
P_6 = a_4\cap b_4, 
\end{array}
\quad
\begin{array}{l}
P_7 = a_4\cap b_3, \\[.05in]
P_8 = a_3\cap b_3,
\end{array}
\end{equation*}
so that 
\begin{equation}
\label{eq:ab-via-P}
\begin{array}{llll}
a_1 = (P_1 P_2), \ \ &
a_2 = (P_3 P_4), \ \ &
a_3 = (P_5 P_8), \ \ &
a_4 = (P_6 P_7), 
\\[.05in]
b_1 = (P_2 P_3), &
b_2 = (P_1 P_4), &
b_3 = (P_7 P_8), &
b_4 = (P_5 P_6). 
\end{array}
\end{equation}
Let us now express the sixteen coplanarity conditions in Theorem~\ref{th:16-points} in terms of
the eight points $P_1, P_2, P_3, P_4, P_5, P_6, P_7, P_8$.
Eight of these sixteen coplanarity conditions have already been assumed
(and are automatic from~\eqref{eq:ab-via-P}). 
The remaining eight conditions concern the coplanarity of the pairs of lines
\begin{align*}
&(a_1,b_3), (a_1, b_4), (a_2,b_3), (a_2, b_4), \\
&(a_3, b_1), (a_4, b_1), (a_3,b_2), (a_4,b_2), 
\end{align*}
or equivalently the coplanarity of the quadruples of points
\begin{equation*}
\begin{array}{l}
\{P_1, P_2, P_7, P_8\}, \{P_1, P_2, P_5, P_6\}, \{P_3, P_4, P_7, P_8\}, \{P_3, P_4, P_5, P_6\}, \\[.05in]
\{P_2, P_3, P_5, P_8\}, \{P_2, P_3, P_6, P_7\}, \{P_1, P_4, P_5, P_8\}, \{P_1, P_4, P_6, P_7\}. 
\end{array}
\end{equation*}
Note that these are precisely the quadruples appearing in~\eqref{eq:8-quadruples} (in different order). 
Thus, the statement that the coplanarity of seven of these eight quadruples implies the coplanarity of the remaining one
is precisely Theorem~\ref{th:cube}. 
\end{proof}

\begin{remark}
\label{rem:16-lines-Schubert}
The sixteen points theorem has the following interpretation rooted in classical Schubert Calculus.
Consider three generic lines $a_1$, $a_2$, $a_3$ in $\CC\PP^3$.
The lines~$b$ piercing each of $a_1$, $a_2$, $a_3$ form a one-parameter family, 
a ruling~$\mathcal{B}$  of a hyperboloid~$\mathbf{H}$. 
The original lines $a_1$, $a_2$, $a_3$ belong to the other ruling of~$\mathbf{H}$, which we denote by~$\mathcal{A}$;
see, e.g., \cite[Section~1]{leykin-sottile}. 
Now consider a fourth line~$a_4$. 
If $a_4$ is also generic, then the one-parameter family~$\mathcal{B}$ contains precisely two lines $b_1$ and $b_2$ 
that pierce all four lines $a_1$, $a_2$, $a_3$, $a_4$.
Namely, locate the two intersection points of $a_4$ with the hyperboloid~$\mathbf{H}$ and take the lines in~$\mathcal{B}$ 
that pass through either of these two points.
If, on the other hand, $a_4$ is not generic, so that there exist at least three lines $b_1$, $b_2$, $b_3\in\mathcal{B}$ that pierce $a_1$, $a_2$, $a_3$, $a_4$
(as in the sixteen points theorem),
then $a_4$ intersects $\mathbf{H}$ in three points and therefore $a_4$ must lie on~$\mathbf{H}$. 
More concretely, $a_4$ must belong to the same ruling~$\mathcal{A}$ of~$\mathbf{H}$ that contains $a_1$, $a_2$, $a_3$. 
In that case, \emph{any} line~$b_4$ that pierces $a_1$, $a_2$, $a_3$ (that is, any line $b_4\in\mathcal{B}$) 
also pierces~$a_4$, and we recover the assertion of the  sixteen points theorem. 

Put differently, we know from Schubert Calculus that given four generic lines $a_1$, $a_2$, $a_3$, $a_4$,
there exist exactly two lines $b_1$ and $b_2$ that pierce each of $a_1$, $a_2$, $a_3$, $a_4$.
(This classical fact is a special case of Lemma~\ref{lem:flats} below.) 
The~sixteen points theorem says that if four lines $a_1$, $a_2$, $a_3$, $a_4$ 
can be pierced by three lines $b_1$, $b_2$, $b_3$ (thus $a_1$, $a_2$, $a_3$, $a_4$ are not generic), 
then there must exist infinitely many lines piercing $a_1$, $a_2$, $a_3$, $a_4$, 
so that any line $b_4$ that pierces $a_1$, $a_2$, $a_3$ 
will pierce $a_4$ as well. 
\end{remark}

\newpage

\section*{The Glynn configuration}

The following ``sibling'' of the M\"obius configuration from Theorem~\ref{th:cube}
was discovered by D.~G.~Glynn~\cite[Section~4.1]{glynn-2010}: 

\begin{theorem}
\label{th:glynn-config}
Let $P_1, \dots, P_8$ be points in the 3-space such that none of the quadruples $\{P_i,P_j,P_k,P_\ell\}$
($1\le i<j<k\le4<\ell$) is coplanar. 
If seven of the eight quadruples 
\begin{equation}
\label{eq:8-quadruples-glynn}
\begin{array}{l}
\{P_1, P_2, P_5, P_7\}, \{P_1, P_3, P_5, P_8\}, \{P_1, P_3, P_6, P_7\}, \{P_1, P_4, P_6, P_8\}, \\[.05in]
\{P_2, P_3, P_6, P_8\}, \{P_2, P_4, P_5, P_6\}, \{P_2, P_4, P_7, P_8\}, \{P_3, P_4, P_5, P_7\}
\end{array}
\end{equation}
are coplanar, then the remaining quadruple is also coplanar. 
See Figure~\ref{fig:glynn}. 
\end{theorem}

\begin{figure}[ht]
\begin{center}
\vspace{-.1in}
\includegraphics[scale=0.5, trim=0cm 0.4cm 0cm 0.2cm, clip]{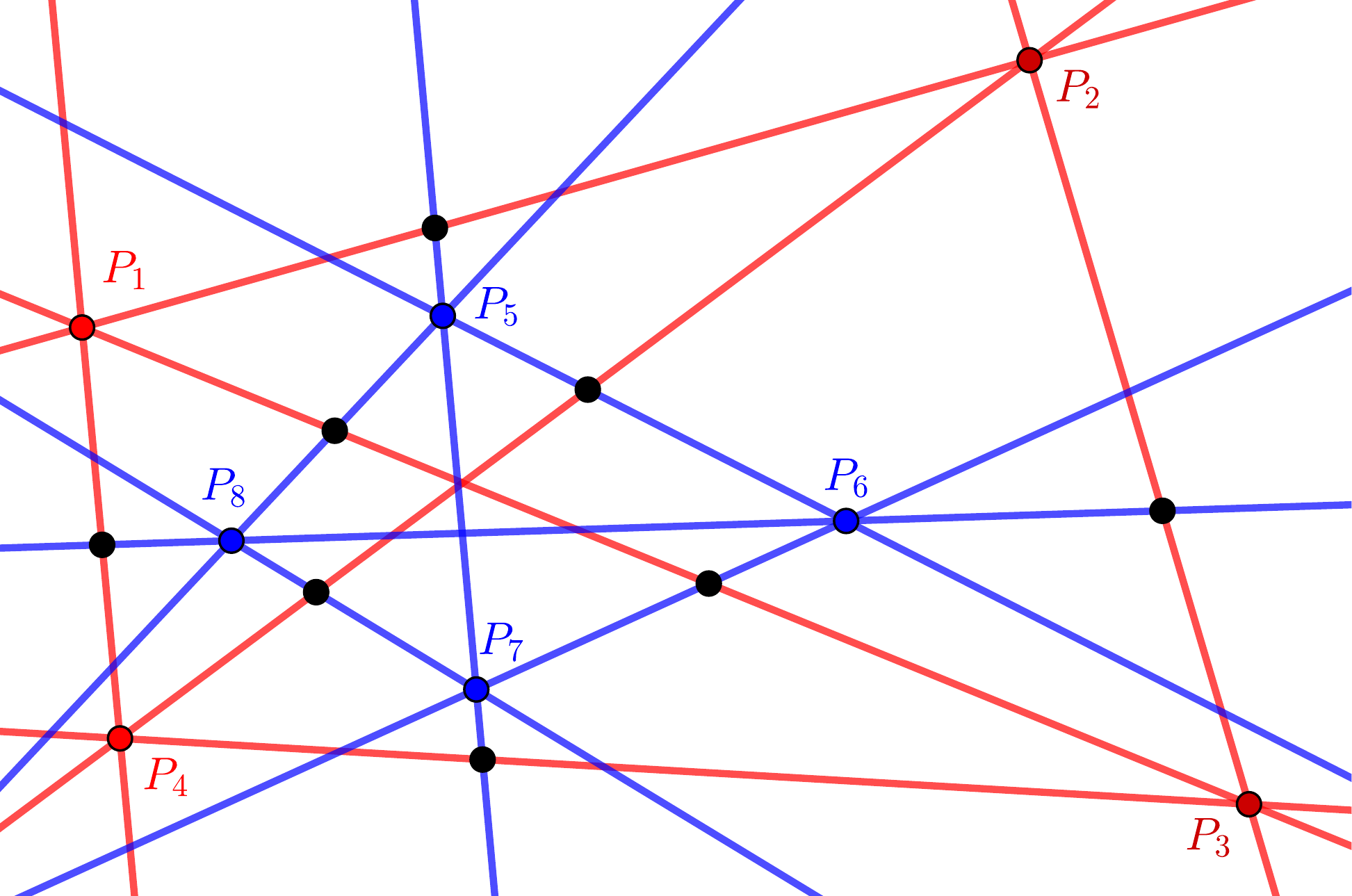}
\end{center}
\vspace{-.1in}
\caption{The Glynn configuration. 
}
\label{fig:glynn}
\end{figure}

\vspace{-.25in}

\begin{proof}
Apply the master theorem to the tiling of the torus shown in Figure~\ref{fig:glynn-tiling}. 
\end{proof}

\enlargethispage{.5cm}

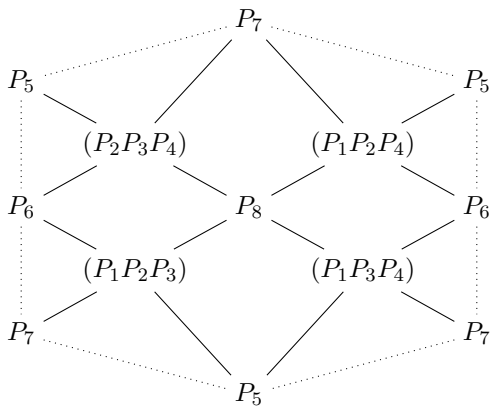
\begin{figure}[ht]
\vspace{-3pt}
\begin{equation*}
\begin{tikzpicture}[baseline= (a).base]
\node[scale=0.9] (a) at (0,0){
\begin{tikzcd}[arrows={-stealth}, cramped, sep=small]
&& P_7 \edge{rrd, dotted} \edge{lld, dotted} \edge{ddr} \edge{ddl} \\
P_5 \edge{dd, dotted} \edge{dr} && && P_5 \edge{dd, dotted} \edge{dl} \\ 
& (P_2P_3P_4) \edge{dl} \edge{dr} & & (P_1P_2P_4) \edge{dl} \edge{dr} \\
P_6 \edge{dd, dotted} \edge{dr} && P_8 \edge{dr} \edge{dl} && P_6 \edge{dd, dotted} \edge{dl} \\ 
& (P_1P_2P_3)  \edge{dl} \edge{ddr} & & (P_1P_3P_4) \edge{ddl}\edge{dr}  \\
P_7 \edge{rrd, dotted} && && P_7 \edge{lld, dotted} \\ 
&& P_5
\end{tikzcd}
};
\end{tikzpicture}
\end{equation*}
\vspace{-10pt}
\caption{The tiling yielding Theorem~\ref{th:glynn-config}. 
Opposite sides of the hexagonal fundamental domain should be glued to each other.
There are no edges between $P_5, P_6, P_7$.
}
\label{fig:glynn-tiling}
\end{figure}

\newpage

\section*{Two lines piercing five lines}

\begin{theorem}
\label{th:hexagon-piercing}
Let $\ell_1, \ell_2$ be a generic pair of lines and let $m_1,m_2,m_3,m_4,m_5$ be five lines 
each of which pierces both $\ell_1$ and~$\ell_2$. 
Pick a generic point $P_1$ on the~line~$m_3$. \linebreak[3]
Find $P_2\!\in\! m_4$ such that $(P_1P_2)$ pierces~$m_1$. 
Find $P_3\!\in\! m_5$ such that $(P_2P_3)$ pierces~$m_2$. \linebreak[3]
Find $P_4\!\in\! m_3$ such that $(P_3P_4)$ pierces~$m_1$. 
Find $P_5\!\in\! m_4$ such that $(P_4P_5)$ pierces~$m_2$. 
Find $P_6\!\in\! m_5$ such that $(P_5P_6)$ pierces~$m_1$. 
Then $(P_6P_1)$ pierces~$m_2$. 
See~Figure~\ref{fig:piercing2x5}. 
\end{theorem}

\begin{figure}[ht]
\begin{center}
\includegraphics[scale=0.5, trim=0cm 0cm 0cm 0cm, clip]{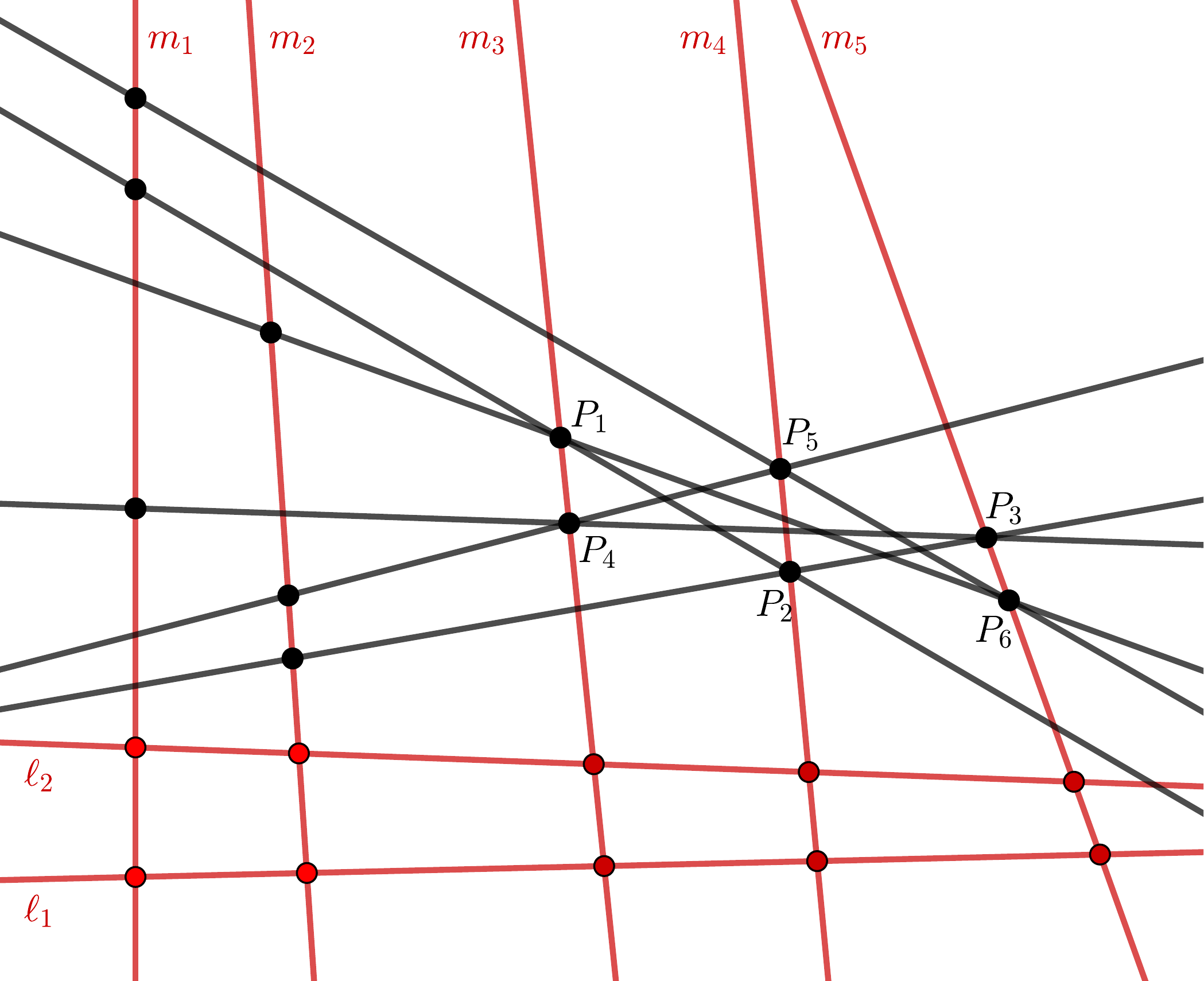}
\end{center}
\vspace{-.1in}
\caption{The 3-dimensional configuration in Theorem~\ref{th:hexagon-piercing}. 
}
\label{fig:piercing2x5}
\end{figure}

\begin{proof}
Let $h_{ij}$ be the plane containing the lines $\ell_i$ and $m_j$, for $i,j\in\{1,2\}$.
Then 
\begin{equation*}
\begin{array}{ll}
h_{11}\cap h_{12}=\ell_1, \quad
&h_{21}\cap h_{22}=\ell_2, \quad \\[5pt]
h_{11}\cap h_{21}=m_1, \quad
&h_{12}\cap h_{22}=m_2. 
\end{array}
\end{equation*}
Also, $(P_1P_4)=m_3$, $(P_2P_5)=m_4$, $(P_3P_6)=m_5$. 
It is then easy to check that the piercing conditions in the theorem are encoded by the tiles of the tiling
shown in Figure~\ref{fig:6points-piercing}.
The claim follows by the master theorem. 
\end{proof}

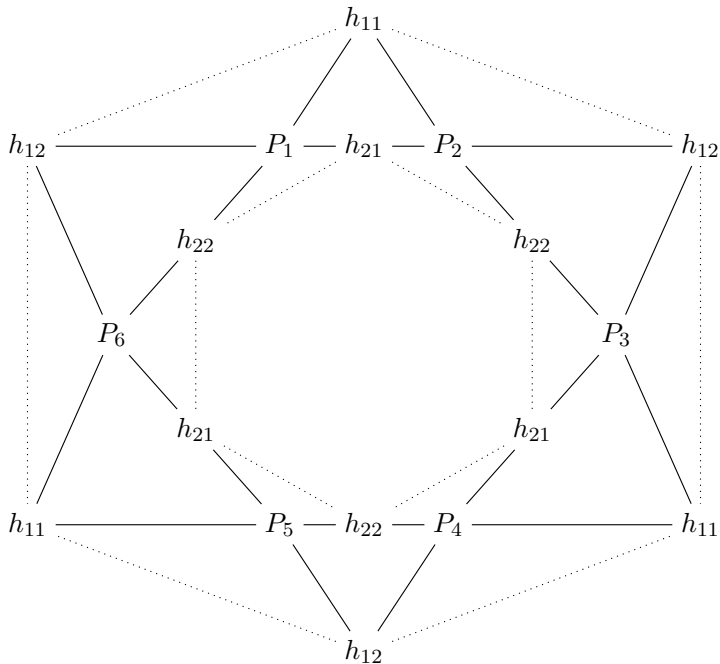
\begin{figure}[ht]
\vspace{-.1in}
\begin{equation*}
\begin{tikzpicture}[baseline= (a).base]
\node[scale=.95] (a) at (0,0){
\begin{tikzcd}[arrows={-stealth}, cramped, sep=small]
&&&& h_{11} \edge{dllll, dotted} \edge{drrrr, dotted} \edge{dl}\edge{dr}\\[25pt]
h_{12} \edge{rrr} \edge{ddr}\edge{dddd, dotted} &&& P_1 \edge{r} \edge{dl} & h_{21} \edge{r} \edge{dll, dotted}\edge{drr, dotted}& P_2 \edge{rrr} \edge{dr}&&& h_{12} \edge{dddd, dotted}\edge{ddl}\\[12pt]
&& h_{22} \edge{dl}\edge{dd, dotted}&&&& h_{22} \edge{dr}\edge{dd, dotted}\\[12pt]
& P_6 \edge{ddl} \edge{dr}&&&&&& P_3 \edge{ddr}\edge{dl}\\[12pt]
&& h_{21} \edge{dr}\edge{drr, dotted}&&&& h_{21} \edge{dll, dotted}\edge{dl}\\[12pt]
h_{11} \edge{rrr} &&& P_5 \edge{r}  \edge{dr}& h_{22}\edge{r}  & P_4 \edge{rrr} \edge{dl}&&& h_{11} \\[25pt]
&&&& h_{12}\edge{ullll, dotted} \edge{urrrr, dotted}
\end{tikzcd}
};
\end{tikzpicture}
\end{equation*}
\vspace{-.1in}
\caption{The tiling of the genus 2 surface used in the proof of Theorem~\ref{th:hexagon-piercing}. 
The opposite sides of each of the two dotted hexagons should be glued to each other.
}
\label{fig:6points-piercing}
\end{figure}

\begin{remark}
As mentioned in Remark~\ref{rem:16-lines-Schubert}, 
four generic lines in 3-space can be pierced, in two different ways, by a fifth line.
When \emph{five} lines $m_1,\dots,m_5$ can be pierced in this way, they are in special position:
as Theorem~\ref{th:hexagon-piercing} shows, the existence of two piercing lines $\ell_1$ and~$\ell_2$
forces the hexagon $P_1P_2P_3P_4P_5P_6$ inscribed in the triple of lines $(m_3,m_4,m_5)$ to close up. 
Note that instead of $(m_3,m_4,m_5)$, we could have chosen any sub-triple of the 5-tuple $(m_1,\dots,m_5)$. 
\end{remark}

\section*{Three-dimensional counterpart of the Pappus theorem}

Consider the tiling in Figure~\ref{fig:pappus-torus}, which was used in the proof of the Pappus theorem.
As before, let $P_1, \dots, P_6$ be points---but this time, let $a,b,c$ be planes in 3-space. 
As a result, we obtain the following  ``three-dimensional Pappus theorem.''

\begin{theorem}
\label{th:3d-pappus}
Let $P_1,\dots,P_6$ be six generic points in 3-space.
Take a generic line~$m_1$ that pierces the lines $(P_1P_2)$, $(P_3P_4)$, $(P_5P_6)$.
Find a line $m_2$ that pierces $m_1$ together with $(P_2P_3)$, $(P_4P_5)$, $(P_6P_1)$. 
Let $O=m_1\cap m_2$. 
Then the line $\ell_1$ that passes through~$O$ and pierces $(P_1P_4)$ and~$(P_2P_5)$ will also pierce~$(P_3P_6)$. 
\end{theorem}

\begin{proof}
Let $a$, $b$, and $c$ be the planes spanned by the pairs of lines $(m_2,\ell_1)$, $(m_1,\ell_1)$, and $(m_1,m_2)$, 
respectively. 
The claim can now be obtained by applying the master theorem to the tiling in Figure \ref{fig:pappus-torus}.
\end{proof}

\begin{remark}
Theorem~\ref{th:3d-pappus} can be illustrated using Figure~\ref{fig:piercing2x5}. 
Just remove line $\ell_2$ and make the three lines $m_1,m_2,\ell_1$ intersect at a point~$O$. 
\end{remark}

\begin{remark}
As noted in Remark~\ref{rem:16-lines-Schubert}, the lines that pierce a given triple of~generic lines $\{a_1,a_2,a_3\}$
form a ruling of a hyperboloid $\mathbf{H}=\mathbf{H}(a_1,a_2,a_3)$. 
With this notation, line~$m_1$ in Theorem~\ref{th:3d-pappus} lies on the hyperboloid $\mathbf{H}_1=\mathbf{H}((P_1P_2),(P_3P_4),(P_5P_6))$.
Similarly, line~$m_2$ lies on $\mathbf{H}_2=\mathbf{H}((P_2P_3),(P_4P_5),(P_6P_1))$.
Theorem~\ref{th:3d-pappus} says that any point 
${O\in \mathbf{H}_1\cap \mathbf{H}_2}$
lies on $\mathbf{H}_3=\mathbf{H}((P_1P_4),(P_2P_5),(P_3P_6))$.
Thus, the intersection
$\mathbf{H}_1
\cap \mathbf{H}_2
\cap \mathbf{H}_3$ 
is 1-dimensional rather than 0-dimensional, as one would expect. 
\end{remark}

\section*{Incidence theorems in changing dimensions}

\begin{remark}
For a given a bicolored quadrilateral tiling~$\TT$,
our master theorem (Theorem~\ref{th:master}) yields an incidence theorem associated to~$\TT$, 
for an arbitrary choice of dimension $d=\dim\PP\ge 2$ of the ambient real/complex projective space~$\PP$. 
It~is~then natural to wonder: what happens to the resulting incidence theorem as $d$ varies? 
Thus far, we have seen two examples of this kind:
\begin{itemize}[leftmargin=.2in]
\item 
the tiling in Figures~\ref{fig:desargues} and~\ref{fig:bundle} (the surface of a cube) yields
the Desargues theorem (Theorem~\ref{th:desargues}) for $d=2$ 
and the bundle theorem (Theorem~\ref{th:bundle}) for $d=3$; 
\item
the tiling in Figure~\ref{fig:pappus-torus} yields the Pappus theorem (Theorem~\ref{th:pappus}) for $d=2$
and Theorem~\ref{th:3d-pappus} for $d=3$. 
\end{itemize}
Speaking in general, we have not observed a transparent direct relationship 
between the 2-dimensional and 3-dimensional statements coming from the same tiling, 
short of their common generalization in the form of Theorem~\ref{th:master}. 

In most cases, an incidence theorem becomes simpler (and sometimes trivial)
when the ambient dimension~$d$ decreases. 
One notable exception to this rule is Desargues' theorem, 
which is ostensibly more substantial than its 3-dimensional counterpart. 
\end{remark}

There are many examples where a tiling yields a trivial incidence theorem for $d=2$ but not for $d=3$.  
We next discuss one common mechanism behind such occurrences. 

\begin{lemma}
\label{lem:four-tiles-triv}
Let $Q_1, Q_2, Q_3, Q_4$ be four distinct points on the real/complex projective plane 
and let $\ell_1,\ell_2,\ell_3,\ell_4$ be four distinct lines. 
Suppose that the tiles
\begin{equation}
\label{eq:four-tiles-triv}
\begin{tikzpicture}[baseline= (a).base]
\node[scale=1] (a) at (0,0){
\begin{tikzcd}[arrows={-stealth}, sep=small, cramped]
Q_1  \edge{r}  \edge{d}& \ell_1 \edge{d} \\[3pt]
\ell_2 \edge{r} & Q_2
\end{tikzcd}
};
\end{tikzpicture}
\qquad
\begin{tikzpicture}[baseline= (a).base]
\node[scale=1] (a) at (0,0){
\begin{tikzcd}[arrows={-stealth}, sep=small, cramped]
Q_1  \edge{r}  \edge{d}& \ell_3 \edge{d} \\[3pt]
\ell_4 \edge{r} & Q_2
\end{tikzcd}
};
\end{tikzpicture}
\qquad
\begin{tikzpicture}[baseline= (a).base]
\node[scale=1] (a) at (0,0){
\begin{tikzcd}[arrows={-stealth}, sep=small, cramped]
Q_3  \edge{r}  \edge{d}& \ell_1 \edge{d} \\[3pt]
\ell_2 \edge{r} & Q_4
\end{tikzcd}
};
\end{tikzpicture}
\qquad
\begin{tikzpicture}[baseline= (a).base]
\node[scale=1] (a) at (0,0){
\begin{tikzcd}[arrows={-stealth}, sep=small, cramped]
Q_3  \edge{r}  \edge{d}& \ell_3 \edge{d} \\[3pt]
\ell_4 \edge{r} & Q_4
\end{tikzcd}
};
\end{tikzpicture}
\end{equation}
are coherent. 
Then one of the following two statements holds:
\begin{itemize}[leftmargin=.3in]
\item[{\rm (i)}]
the six lines $\ell_1,\ell_2,\ell_3,\ell_4, (Q_1Q_2), (Q_3Q_4)$ are concurrent; 
\item[{\rm (ii)}]
the six points $Q_1,Q_2,Q_3,Q_4, \ell_1\cap\ell_2, \ell_3\cap\ell_4$ are collinear. 
\end{itemize}
\end{lemma}

\begin{proof}
The coherence of the tiles \eqref{eq:four-tiles-triv} means that each of the points $\ell_1\cap\ell_2$ and $\ell_3\cap\ell_4$
lies on each of the lines $(Q_1Q_2)$ and~$(Q_3Q_4)$. 
If the latter two lines are distinct, then their intersection must coincide with both $\ell_1\cap\ell_2$ and $\ell_3\cap\ell_4$,
yielding~(i). 
Otherwise, the points $Q_1,Q_2,Q_3,Q_4$ lie on the same line that moreover contains $\ell_1\cap\ell_2$ and $\ell_3\cap\ell_4$,
yielding~(ii). 
\end{proof}

Lemma~\ref{lem:four-tiles-triv} shows that whenever a tiling contains four tiles as in~\eqref{eq:four-tiles-triv},
the $d\!=\!2$ case of the corresponding incidence theorem involves a rather degenerate
point-and-line configuration that includes either six concurrent lines or six collinear points. 
As a result, the final statement may end up collapsing to a trivial one. 

\begin{example}
\label{eg:cube-in-dim-2}
Figure~\ref{fig:cube-thm-tiling} contains two instances of Lemma~\ref{lem:four-tiles-triv}.
Because of that, the 2-dimensional counterpart of the M\"obius theorem (Theorem~\ref{th:cube}) turns out to be trivial. 
For~another example of a similar kind, see Theorem~\ref{th:hexagon-piercing}/Figure~\ref{fig:6points-piercing}.
\end{example}

\begin{remark}
In Section~\ref{sec:master-theorem} (see Examples~\ref{eg:tiling-tautology-1}--\ref{eg:tiling-3-tiles}),
we discussed several situations wherein an incidence theorem associated with a given tiling is trivial regardless of dimension. 
The phenomena discussed in Example~\ref{eg:cube-in-dim-2} are more subtle than those. 
\end{remark}

\newpage

\section{Combinatorial reformulations of the master theorem
}
\label{sec:combinatorial-reformulations}

Fix a finite-dimensional real/complex projective space~$\PP$. 
For any bicolored quadrilateral tiling of an oriented closed surface~$\Sigma$ (of arbitrary genus),
Theorem~\ref{th:master} yields an incidence theorem in~$\PP$.
In this section, we discuss a couple of alternative ways to describe such tilings. 
In each case, we obtain reformulations of the master theorem (or its individual instances) in a new combinatorial language. 

\section*{Graphs embedded into surfaces}

\begin{definition}
\label{def:graphs-on-surfaces}
Consider a tiling $\TT$ of an oriented closed surface~$\Sigma$ by quadrilateral tiles. 
(Each pair of tiles are either disjoint, or share a side, or share a vertex.)
Assume that the vertices of the tiling are colored black and white, 
so that each edge connects vertices of different color. 
Construct a graph $\GG=\GG(\TT)$ embedded in~$\Sigma$ as follows.
The vertices of~$\GG$ are the white vertices of the tiling~$\TT$.
Each tile in~$\TT$ gives rise to an edge in~~$\GG$
that connects the two white vertices of the tile. 

We note that the \emph{faces} of~$\GG$ (i.e., connected components of the complement of~$\GG$ inside~$\Sigma$) 
correspond to the black vertices of the tiling. 
\end{definition}

A couple of examples of the above construction are given in Figure~\ref{fig:bipartite-tilings-G}.

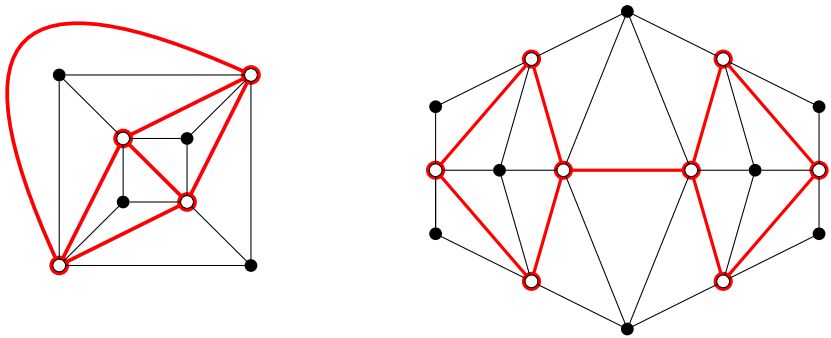
\begin{figure}[ht]
\begin{center}
\vspace{5pt}
\setlength{\unitlength}{1.2pt}
\begin{picture}(60,65)(0,-20)
\linethickness{1.4pt}

\put(1,2){\red{{\line(1,2){18}}}}
\put(2,1){\red{{\line(2,1){36}}}}
\put(41,22){\red{{\line(1,2){18}}}}
\put(22,41){\red{{\line(2,1){36}}}}
\put(21.5,38.5){\red{{\line(1,-1){17}}}}

\put(0,0){\red{\circle{5}}}
\put(20,40){\red{\circle{5}}}
\put(40,20){\red{\circle{5}}}
\put(60,60){\red{\circle{5}}}

\red{\qbezier(-1,2)(-50,110)(58,61)}

\thinlines

\put(2,0){{\line(1,0){56}}}
\put(2,60){{\line(1,0){56}}}
\put(22,20){{\line(1,0){16}}}
\put(22,40){{\line(1,0){16}}}

\put(0,2){{\line(0,1){56}}}
\put(60,2){{\line(0,1){56}}}
\put(20,22){{\line(0,1){16}}}
\put(40,22){{\line(0,1){16}}}

\put(1.4,1.4){{\line(1,1){18}}}
\put(58.6,58.6){{\line(-1,-1){18}}}
\put(18.6,41.4){{\line(-1,1){18}}}
\put(41.4,18.6){{\line(1,-1){18}}}

\put(0,0){\circle{4}}
\put(0,60){\circle*{4}}
\put(60,60){\circle{4}}
\put(60,0){\circle*{4}}
\put(20,20){\circle*{4}}
\put(20,40){\circle{4}}
\put(40,20){\circle{4}}
\put(40,40){\circle*{4}}
\end{picture}
\qquad\qquad\qquad
\begin{picture}(120,100)(0,0)
\linethickness{1.2pt}

\put(0,50){\red{\circle{5}}}
\put(40,50){\red{\circle{5}}}
\put(80,50){\red{\circle{5}}}
\put(120,50){\red{\circle{5}}}
\put(30,15){\red{\circle{5}}}
\put(30,85){\red{\circle{5}}}
\put(90,15){\red{\circle{5}}}
\put(90,85){\red{\circle{5}}}

\put(1.2,51.4){\red{{\line(6,7){27.6}}}}
\put(1.2,48.6){\red{{\line(6,-7){27.6}}}}
\put(118.8,51.4){\red{{\line(-6,7){27.6}}}}
\put(118.8,48.6){\red{{\line(-6,-7){27.6}}}}
\put(42,50){\red{{\line(1,0){36}}}}
\put(39.4,52.1){\red{{\line(-2,7){8.8}}}}
\put(39.4,47.9){\red{{\line(-2,-7){8.8}}}}
\put(80.6,52.1){\red{{\line(2,7){8.8}}}}
\put(80.6,47.9){\red{{\line(2,-7){8.8}}}}

\thinlines

\put(2,50){{\line(1,0){16}}}
\put(22,50){{\line(1,0){16}}}
\put(82,50){{\line(1,0){16}}}
\put(102,50){{\line(1,0){16}}}

\put(0,52){{\line(0,1){16}}}
\put(0,48){{\line(0,-1){16}}}
\put(120,52){{\line(0,1){16}}}
\put(120,48){{\line(0,-1){16}}}
\put(0,48){{\line(0,-1){16}}}
\put(29.4,82.9){{\line(-2,-7){9}}}
\put(29.4,17.1){{\line(-2,7){9}}}
\put(90.6,17.1){{\line(2,7){9}}}
\put(90.6,82.9){{\line(2,-7){9}}}

\put(28,16){{\line(-2,1){28}}}
\put(32,14){{\line(2,-1){28}}}
\put(28,84){{\line(-2,-1){28}}}
\put(32,86){{\line(2,1){28}}}
\put(88,86){{\line(-2,1){28}}}
\put(92,84){{\line(2,-1){28}}}
\put(88,14){{\line(-2,-1){28}}}
\put(92,16){{\line(2,1){28}}}

\put(40.8,52){{\line(2,5){19}}}
\put(40.8,48){{\line(2,-5){19}}}
\put(79.2,52){{\line(-2,5){19}}}
\put(79.2,48){{\line(-2,-5){19}}}

\put(60,0){\circle*{4}}
\put(30,15){\circle{4}}
\put(90,15){\circle{4}}
\put(0,30){\circle*{4}}
\put(120,30){\circle*{4}}
\put(0,50){\circle{4}}
\put(40,50){\circle{4}}
\put(80,50){\circle{4}}
\put(120,50){\circle{4}}
\put(20,50){\circle*{4}}
\put(100,50){\circle*{4}}
\put(0,70){\circle*{4}}
\put(120,70){\circle*{4}}
\put(30,85){\circle{4}}
\put(90,85){\circle{4}}
\put(60,100){\circle*{4}}

\end{picture}
\vspace{-5pt}
\end{center}
\caption{The correspondence $\TT\mapsto\GG=\GG(\TT)$ for the tilings from Figure~\ref{fig:bipartite-tilings},
cf.~also Figures~\ref{fig:desargues} and~\ref{fig:pappus2-tiling}. 
On the right, the opposite sides of the hexagonal fundamental domain should be glued to each other, 
producing a graph $\GG$ embedded in the torus. 
This graph has 5~vertices (corresponding to the hyperplanes $b, r, s, q, t$ in Figure~\ref{fig:pappus2-tiling}),
9~edges (shown in the figure above) and 4~faces
(2~pentagons and 2~quadrilaterals).
}
\label{fig:bipartite-tilings-G}
\end{figure}

\vspace{-.2in}

\begin{proposition}
\label{pr:graphs-on-surfaces}
The map $\TT\mapsto\GG$ 
described in Definition~\ref{def:graphs-on-surfaces} 
is a bijection~between 
\begin{itemize}[leftmargin=.2in]
\item 
bicolored quadrilateral tilings~$\TT$
of a Riemann surface~$\Sigma$ (viewed up to isotopy of the surface~$\Sigma$) and
\item
graphs~$\GG$ embedded in~$\Sigma$ (and viewed up to isotopy of~$\Sigma$) 
in which every vertex has degree $\ge 2$, every face has $\ge 2$ sides, 
and each face 
is homeomorphic to a disk. 
\end{itemize}
\end{proposition}

\begin{proof}
The properties of $\GG$ listed above are readily checked for any graph $\GG(\TT)$ obtained from a triangulation~$\TT$.
Conversely, given such an embedded graph~$\GG$, we can recover the corresponding triangulation~$\TT$ 
by placing a black vertex inside each face of~$\GG$ and connecting each of these black vertices,
inside the corresponding face, to the (white) vertices of~$\GG$ lying on the boundary of that face. 
\end{proof}

\begin{remark}
Interchanging the black and white colors in a triangulation~$\TT$ results in replacing the embedded graph $\GG(\TT)$
by its (Poincar\'e) dual graph. 
The corresponding incidence theorems are projectively dual to each other. 
\end{remark}

\begin{remark}
Graphs embedded into surfaces so that every face is a topological disk are sometimes called 
(topological) \emph{maps} (on surfaces); see, e.g., \cite[Section~1.3.2]{lando-zvonkin}. 
\end{remark}

\begin{remark}
The above correspondence becomes especially transparent when the surface~$\Sigma$ 
is a sphere while the graph~$\GG$ comes from the 1-skeleton of a convex 3-dimensional polytope~$\mathbf{P}$
(whose boundary can be identified with the sphere~$\Sigma$). 
Thus, every 3-dimensional  polytope~$\mathbf{P}$ gives rise to an incidence theorem.
Dual polytopes produce projectively dual theorems (which we view as equivalent). 
Several specific instances of this construction are listed in Figure~\ref{fig:polytopes->configs}. 
\end{remark}

\begin{figure}[ht]
\begin{center}
\vspace{10pt}
\begin{tabular}{|c|c|}
\hline
Desargues theorem, the bundle theorem & tetrahedron \\
\hline 
complete quadrilateral theorem & 
triangular prism 
\\
\hline
Theorem~\ref{th:planar-bundle-generalized} & square pyramid
\\
\hline
Saam's sequence of perspectivities & square antiprism
\\
\hline
\end{tabular}
\end{center}
\caption{Incidence theorems coming from 3-dimensional convex polytopes.
Each polytope can be replaced by its dual.}
\label{fig:polytopes->configs}
\end{figure}

\vspace{-.2in}

\begin{remark}
\label{rem:glynn}
Denote $d=\dim\PP$. 
In the special case when the number of vertices in the graph~$\GG$
(equivalently, the number of white vertices in the corresponding tiling~$\TT$)
is equal to $d+1$,
we recover the ingenious construction of D.~G.~Glynn~\cite{glynn}. \linebreak[3]
Let us explain. 

Let $\GG$ be a graph embedded into a closed oriented surface as in Proposition~\ref{pr:graphs-on-surfaces}, 
with vertex set~$V$ of cardinality $d+1$ and a set of faces~$F$ of arbitrary cardinality. 
Associate with each vertex $v\in V$ (resp., a face $f\in F$) a point $P_v$ (resp.,~$P_f$). 
For each $v\in V$, let $h_v\subset\PP$ be the hyperplane 
spanned by the set $\{P_{u}\mid u\in V, u\neq v\}$.
We thus obtain a labeling of the vertices~$v$ (resp., faces~$f$) of~$\GG$
by the hyperplanes~$h_v$ (resp., points~$P_f$). 
Equivalently, we get a labeling of the white (resp., black) vertices of the corresponding tiling~$\TT$
by hyperplanes (resp., points) in~$\PP$. 
The corresponding instance of our master theorem is a restatement of the main result of~\cite{glynn}.

A particularly nice feature of the special case $|V|=d+1$ treated in~\cite{glynn} is that the resulting  
incidence theorems are immediately stated in ``matroidal'' terms,
i.e., in terms of a finite collection of points in~$\PP$ and dependency relations among them. \linebreak[3]
It~is also shown in~\cite{glynn} that in the genus~0 case (i.e., when $\Sigma$ is a sphere),
the ensuing incidence theorems hold over arbitrary skew fields (such as the quaternions). 

On the other hand, the restriction $|V|=d+1$ substantially limits the range of potential applications.
For example, in the case of the projective plane ($d=2$), this restriction would limit us to
tilings with three white (or three black) vertices. 
It~appears that most incidence theorems discussed in this paper cannot be directly obtained
from such tilings. 
\end{remark}

\pagebreak[3]

\section*{Nodal curves on a surface}

We next discuss, somewhat informally, another ``cryptomorphic'' version 
of the combinatorial construction underlying our master theorem.

\begin{definition}
\label{def:tiling-to-curve}
Given a tiling~$\TT$ of a surface~$\Sigma$, construct a nodal curve $\mathbf{C}=\mathbf{C}(\TT)$ on~$\Sigma$ as follows. 
Choose a ``midpoint'' inside each edge of the tiling.
For each tile, connect both pairs of opposite midpoints to each other.
If you like, smoothen the resulting curve. (Everything is viewed up to isotopy.) 
Examples related to Desargues, Pappus, and permutation theorems are shown in Figures~\ref{fig:nodal-curves}--\ref{fig:nodal-curves-2}. 
\end{definition}

\begin{figure}[ht]
\begin{center}
\vspace{-5pt}
\qquad
\setlength{\unitlength}{1.1pt}
\begin{picture}(60,65)(0,-20)
\linethickness{1.4pt}

\put(0,30){\lightgreen{{\line(1,0){60}}}}
\put(0,-20){\lightgreen{{\line(1,0){60}}}}
\put(-10,-10){\lightgreen{{\line(0,1){30}}}}
\put(70,-10){\lightgreen{{\line(0,1){30}}}}

\lightgreen{
\qbezier(0,-20)(-10,-20)(-10,-10)
\qbezier(60,-20)(70,-20)(70,-10)
\qbezier(60,30)(70,30)(70,20)
\qbezier(0,30)(-10,30)(-10,20)
}

\put(30,0){\blue{{\line(0,1){60}}}}
\put(-20,0){\blue{{\line(0,1){60}}}}
\put(-10,-10){\blue{{\line(1,0){30}}}}
\put(-10,70){\blue{{\line(1,0){30}}}}

\blue{
\qbezier(-20,0)(-20,-10)(-10,-10)
\qbezier(-20,60)(-20,70)(-10,70)
\qbezier(30,60)(30,70)(20,70)
\qbezier(30,0)(30,-10)(20,-10)
}

\put(10,20){\red{{\line(0,1){20}}}}
\put(50,20){\red{{\line(0,1){20}}}}
\put(20,10){\red{{\line(1,0){20}}}}
\put(20,50){\red{{\line(1,0){20}}}}

\red{
\qbezier(10,20)(10,10)(20,10)
\qbezier(50,20)(50,10)(40,10)
\qbezier(10,40)(10,50)(20,50)
\qbezier(50,40)(50,50)(40,50)
}

\thinlines

\put(2,0){{\line(1,0){56}}}
\put(2,60){{\line(1,0){56}}}
\put(22,20){{\line(1,0){16}}}
\put(22,40){{\line(1,0){16}}}

\put(0,2){{\line(0,1){56}}}
\put(60,2){{\line(0,1){56}}}
\put(20,22){{\line(0,1){16}}}
\put(40,22){{\line(0,1){16}}}

\put(1.4,1.4){{\line(1,1){18}}}
\put(58.6,58.6){{\line(-1,-1){18}}}
\put(18.6,41.4){{\line(-1,1){18}}}
\put(41.4,18.6){{\line(1,-1){18}}}

\put(0,0){\circle{4}}
\put(0,60){\circle*{4}}
\put(60,60){\circle{4}}
\put(60,0){\circle*{4}}
\put(20,20){\circle*{4}}
\put(20,40){\circle{4}}
\put(40,20){\circle{4}}
\put(40,40){\circle*{4}}
\end{picture}
\qquad\qquad\qquad
\begin{picture}(120,100)(0,0)
\linethickness{1.4pt}

\put(0,60){\red{{\line(3.5,1){50}}}}
\put(0,40){\red{{\line(3.5,-1){50}}}}
\put(120,40){\red{{\line(-3.5,-1){50}}}}
\put(120,60){\red{{\line(-3.5,1){50}}}}
\put(50,25){\red{{\line(1,2.5){20}}}}
\put(50,75){\red{{\line(1,-2.5){20}}}}

\put(10,50){\green{{\line(1,5.5){5}}}}
\put(10,50){\green{{\line(1,-5.5){5}}}}

\put(90,50){\green{{\line(-3,8.5){15}}}}
\put(90,50){\green{{\line(-3,-8.5){15}}}}

\put(110,50){\blue{{\line(-1,5.5){5}}}}
\put(110,50){\blue{{\line(-1,-5.5){5}}}}

\put(30,50){\blue{{\line(3,8.5){15}}}}
\put(30,50){\blue{{\line(3,-8.5){15}}}}

\thinlines

\put(2,50){{\line(1,0){16}}}
\put(22,50){{\line(1,0){16}}}
\put(82,50){{\line(1,0){16}}}
\put(102,50){{\line(1,0){16}}}

\put(0,52){{\line(0,1){16}}}
\put(0,48){{\line(0,-1){16}}}
\put(120,52){{\line(0,1){16}}}
\put(120,48){{\line(0,-1){16}}}
\put(0,48){{\line(0,-1){16}}}
\put(29.4,82.9){{\line(-2,-7){9}}}
\put(29.4,17.1){{\line(-2,7){9}}}
\put(90.6,17.1){{\line(2,7){9}}}
\put(90.6,82.9){{\line(2,-7){9}}}

\put(28,16){{\line(-2,1){28}}}
\put(32,14){{\line(2,-1){28}}}
\put(28,84){{\line(-2,-1){28}}}
\put(32,86){{\line(2,1){28}}}
\put(88,86){{\line(-2,1){28}}}
\put(92,84){{\line(2,-1){28}}}
\put(88,14){{\line(-2,-1){28}}}
\put(92,16){{\line(2,1){28}}}

\put(40.8,52){{\line(2,5){19}}}
\put(40.8,48){{\line(2,-5){19}}}
\put(79.2,52){{\line(-2,5){19}}}
\put(79.2,48){{\line(-2,-5){19}}}

\put(60,0){\circle*{4}}
\put(30,15){\circle{4}}
\put(90,15){\circle{4}}
\put(0,30){\circle*{4}}
\put(120,30){\circle*{4}}
\put(0,50){\circle{4}}
\put(40,50){\circle{4}}
\put(80,50){\circle{4}}
\put(120,50){\circle{4}}
\put(20,50){\circle*{4}}
\put(100,50){\circle*{4}}
\put(0,70){\circle*{4}}
\put(120,70){\circle*{4}}
\put(30,85){\circle{4}}
\put(90,85){\circle{4}}
\put(60,100){\circle*{4}}

\end{picture}
\end{center}
\caption{Nodal curves corresponding to the tilings in Figure~\ref{fig:bipartite-tilings}
(or Figures~\ref{fig:desargues} and~\ref{fig:pappus2-tiling}). 
Each component of a nodal curve is shown in different color. 
}
\label{fig:nodal-curves}
\end{figure}
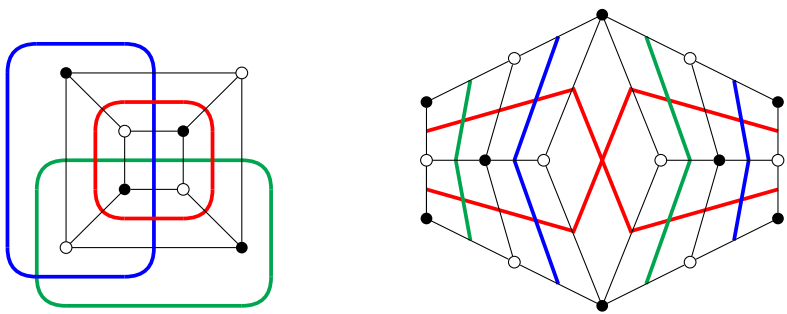

\vspace{-10pt}

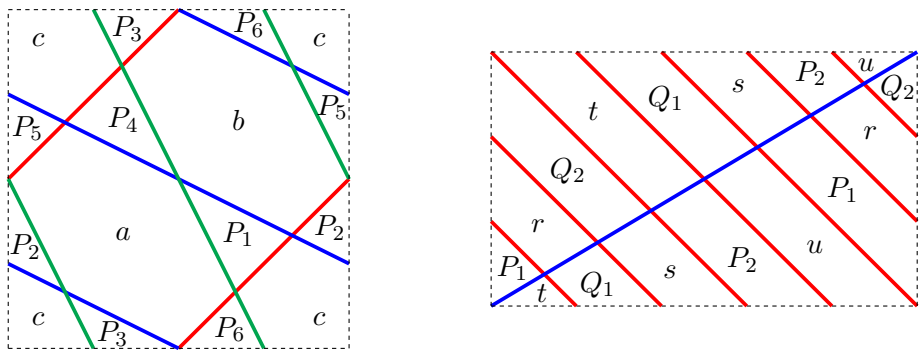
\begin{figure}[ht]
\begin{center}
\vspace{-5pt}
\qquad
\setlength{\unitlength}{1.6pt}
\begin{picture}(80,80)(0,0)
\linethickness{1.4pt}

\put(0,40){\red{{\line(1,1){40}}}}
\put(40,0){\red{{\line(1,1){40}}}}

\put(0,20){\blue{{\line(2,-1){40}}}}
\put(0,60){\blue{{\line(2,-1){80}}}}
\put(40,80){\blue{{\line(2,-1){40}}}}

\put(0,40){\lightgreen{{\line(1,-2){20}}}}
\put(20,80){\lightgreen{{\line(1,-2){40}}}}
\put(60,80){\lightgreen{{\line(1,-2){20}}}}

\thinlines

\multiput(0,0)(0,2){40}{\line(0,1){1}}
\multiput(80,0)(0,2){40}{\line(0,1){1}}
\multiput(0,0)(2,0){40}{\line(1,0){1}}
\multiput(0,80)(2,0){40}{\line(1,0){1}}

\put(27,27){\makebox(0,0){$a$}}
\put(54,54){\makebox(0,0){$b$}}
\put(27,54){\makebox(0,0){$P_4$}}
\put(54,27){\makebox(0,0){$P_1$}}

\put(24.5,3.5){\makebox(0,0){$P_3$}}
\put(28,75.5){\makebox(0,0){$P_3$}}
\put(52,4.5){\makebox(0,0){$P_6$}}
\put(56.5,76.5){\makebox(0,0){$P_6$}}

\put(4,24){\makebox(0,0){$P_2$}}
\put(4.5,52){\makebox(0,0){$P_5$}}

\put(76,57){\makebox(0,0){$P_5$}}
\put(75.5,28.5){\makebox(0,0){$P_2$}}

\put(7,7){\makebox(0,0){$c$}}
\put(7,73){\makebox(0,0){$c$}}
\put(73,7){\makebox(0,0){$c$}}
\put(73,73){\makebox(0,0){$c$}}

\end{picture}
\qquad\qquad\qquad
\begin{picture}(100,60)(10,-10)
\linethickness{1.4pt}

\put(0,20){\red{{\line(1,-1){20}}}}
\put(0,40){\red{{\line(1,-1){40}}}}
\put(0,60){\red{{\line(1,-1){60}}}}
\put(20,60){\red{{\line(1,-1){60}}}}
\put(40,60){\red{{\line(1,-1){60}}}}
\put(60,60){\red{{\line(1,-1){40}}}}
\put(80,60){\red{{\line(1,-1){20}}}}

\put(0,0){\blue{{\line(5,3){100}}}}

\thinlines

\multiput(0,0)(0,2){30}{\line(0,1){1}}
\multiput(100,0)(0,2){30}{\line(0,1){1}}
\multiput(0,0)(2,0){50}{\line(1,0){1}}
\multiput(0,60)(2,0){50}{\line(1,0){1}}

\put(5,9){\makebox(0,0){$P_1$}}
\put(11,19){\makebox(0,0){$r$}}
\put(18,32){\makebox(0,0){$Q_2$}}
\put(24,46){\makebox(0,0){$t$}}
\put(41,49){\makebox(0,0){$Q_1$}}
\put(58,52){\makebox(0,0){$s$}}
\put(75,55){\makebox(0,0){$P_2$}}
\put(88,57){\makebox(0,0){$u$}}

\put(12,3){\makebox(0,0){$t$}}
\put(25,5){\makebox(0,0){$Q_1$}}
\put(42,8){\makebox(0,0){$s$}}
\put(59,11){\makebox(0,0){$P_2$}}
\put(76,14){\makebox(0,0){$u$}}
\put(82.5,27.5){\makebox(0,0){$P_1$}}
\put(89,41){\makebox(0,0){$r$}}
\put(95.5,51.5){\makebox(0,0){$Q_2$}}

\end{picture}

\end{center}
\caption{Nodal curves on the torus corresponding to the tilings in Figures~\ref{fig:pappus-torus} and~\ref{fig:torus-perm}. 
Region labels match the vertex labels of the tiling. 
}
\vspace{-15pt}
\label{fig:nodal-curves-2}
\end{figure}

\begin{remark}
\label{rem:curve-to-tiling}
Let $\TT$ be a bicolored quadrilateral tiling of a surface~$\Sigma$.
Let ${\mathbf{C}=\mathbf{C}(\TT)}$ be the associated curve on~$\Sigma$. 
A~\emph{region} of~$\mathbf{C}$ is a connected component of the complement $\Sigma-\mathbf{C}$. 
We then have: 
\begin{itemize}[leftmargin=.2in]
\item 
every region of $\mathbf{C}$ is homeomorphic to a disk; 
\item
every region of $\mathbf{C}$ is not a monogon (i.e., has at least two nodes on its boundary);
\item
each region 
is colored black or white, 
so that adjacent regions have different colors. 
\end{itemize}
To~recover the tiling from a nodal curve~$\mathbf{C}$ satisfying these conditions, 
place a black (resp., white) vertex inside each black (resp., white) region 
and connect these vertices across curve segments separating neighboring regions. 
\end{remark}

\clearpage

\newpage

\section{Deducing incidence theorems from each other
} 
\label{sec:deducing-new-incidence-theorems}

The study of logical dependencies among various incidence theorems
has been an important research direction within incidence geometry since its inception in ancient Greece. 
Some of these dependencies can be naturally interpreted in terms of associated tilings. 
In this section, we discuss several settings in which this phenomenon occurs. 

\medskip

We begin by examining several combinatorial constructions that yield 
material equivalences between incidence theorems, i.e., 
instances where two incidence theorems are equivalent to each other, for some transparent reason. 

\section*{Projective duality}

Projective duality between points and hyperplanes
implies that each incidence theorem has a dual counterpart in which the points of the original 
theorem are replaced by hyperplanes, and vice versa.
This is readily seen at the level of tilings: simply swap the black and white colors of the vertices
to get a tiling for the dual theorem. 

To illustrate, dualizing Proposition~\ref{pr:octagon-special} yields the following statement. 

\begin{proposition}
\label{pr:octagon-special-dual}
Let $p_1, p_2, p_3, p_4$ be four generic lines on the plane.
Suppose that
\begin{itemize}[leftmargin=.25in]
\item[{\rm (a)}]
$A$ is a generic point on the line passing through $p_1\cap p_2$ and $p_3\cap p_4$; 
\item[{\rm (b)}]
$B$ is a generic point on the line passing through $p_1\cap p_4$ and $p_2\cap p_3$. 
\end{itemize}
Then 
\begin{itemize}[leftmargin=.25in]
\item[{\rm (c)}]
the octagon with vertex labels
$p_1, A, p_2, B, p_3, A, p_4, B$ (see Figure~\ref{fig:8-gon-tiling-special-dual}) is coherent. 
\end{itemize}
Conversely, {\rm (a)} and~{\rm (c)} imply~{\rm (b)}. 
\end{proposition}

\begin{figure}[ht]
\begin{center}
\vspace{-.2in}
\begin{equation*}
\begin{tikzpicture}[baseline= (a).base]
\node[scale=1] (a) at (0,0){
\begin{tikzcd}[arrows={-stealth}, cramped, sep=8]
& A \edge{rr} && p_4 \edge{rd} & \\[5pt]
p_3 \edge{dd} \edge{ru} & && & B \\[-4pt]
&& E \edge{llu, dashed}\edge{ruu, dashed} \edge{rrd, dashed} \edge{ldd, dashed}\\[-4pt]
B \edge{rd} & && & p_1 \edge{uu}\\[5pt]
& p_2 \edge{rr} && A \edge{ru}
\end{tikzcd}
};
\end{tikzpicture}
\end{equation*}
\vspace{-.15in}
\end{center}
\caption{The octagon from Proposition~\ref{pr:octagon-special-dual} and its tiling.}
\label{fig:8-gon-tiling-special-dual}
\end{figure}
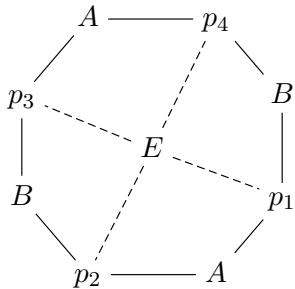

Dualizing the second proof of Theorem~\ref{th:pappus} (cf.\ Figure~\ref{fig:pappus2-tiling}),
we obtain the following version of the Pappus theorem. 

\begin{theorem}[Pappus]
\label{th:pappus-dual}
Let $Q, S, R, T$ be four points on the real/complex plane
and let $d, e, p_4, p_5$ be four lines. 
Consider the six tiles appearing  in Figure~\ref{fig:pappus2-tiling-dual} (including the octagonal tile in the middle). 
If  five of these six tiles are coherent, then the remaining tile is coherent as well. 
Cf.~Figure~\ref{fig:pappus-dual}. 
\end{theorem}

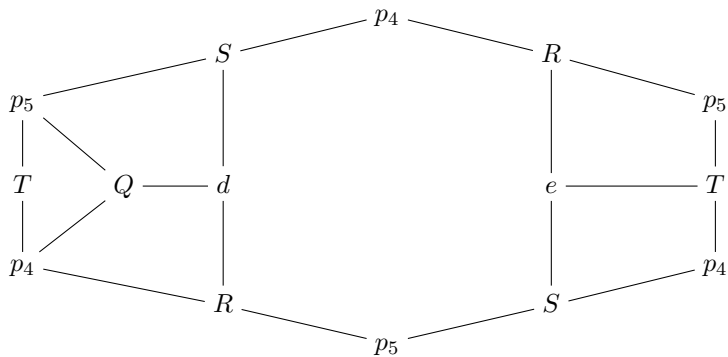
\begin{figure}[ht]
\vspace{-.1in}
\begin{center}
\begin{equation*}
\begin{tikzpicture}[baseline= (a).base]
\node[scale=.9] (a) at (0,0){
\begin{tikzcd}[arrows={-stealth}, cramped]
& & & & p_4 \edge{rrd} \edge{lld}  \\[-19pt]
& & S \edge{lld} \edge{dd} &&& & R \edge{rrd} \edge{dd}  \\[-12pt]
p_5 \edge{d} \edge{dr} &&&&&&&&  p_5 \edge{d} \\
T \edge{d}  & Q \edge{r} \edge{dl} & d \edge{dd} &&&& e \edge{rr} \edge{dd} & & T \edge{d} \\
p_4 \edge{rrd} &&&&&&&&  p_4 \edge{lld} \\[-19pt]
& & R \edge{rrd} &&&& S \edge{lld} \\[-15pt]
& & & & p_5 
\end{tikzcd}
};
\end{tikzpicture}
\end{equation*}
\vspace{-.15in}
\end{center}
\caption{The tiling of the torus for the dual version of the Pappus theorem. 
Opposite sides of the shown hexagonal fundamental domain should be glued to each other.}
\label{fig:pappus2-tiling-dual}
\end{figure}

\begin{figure}[ht]
\vspace{-.1in}
\begin{center}
\includegraphics[scale=0.5, trim=0.1cm 0.2cm 0.6cm 0cm, clip]{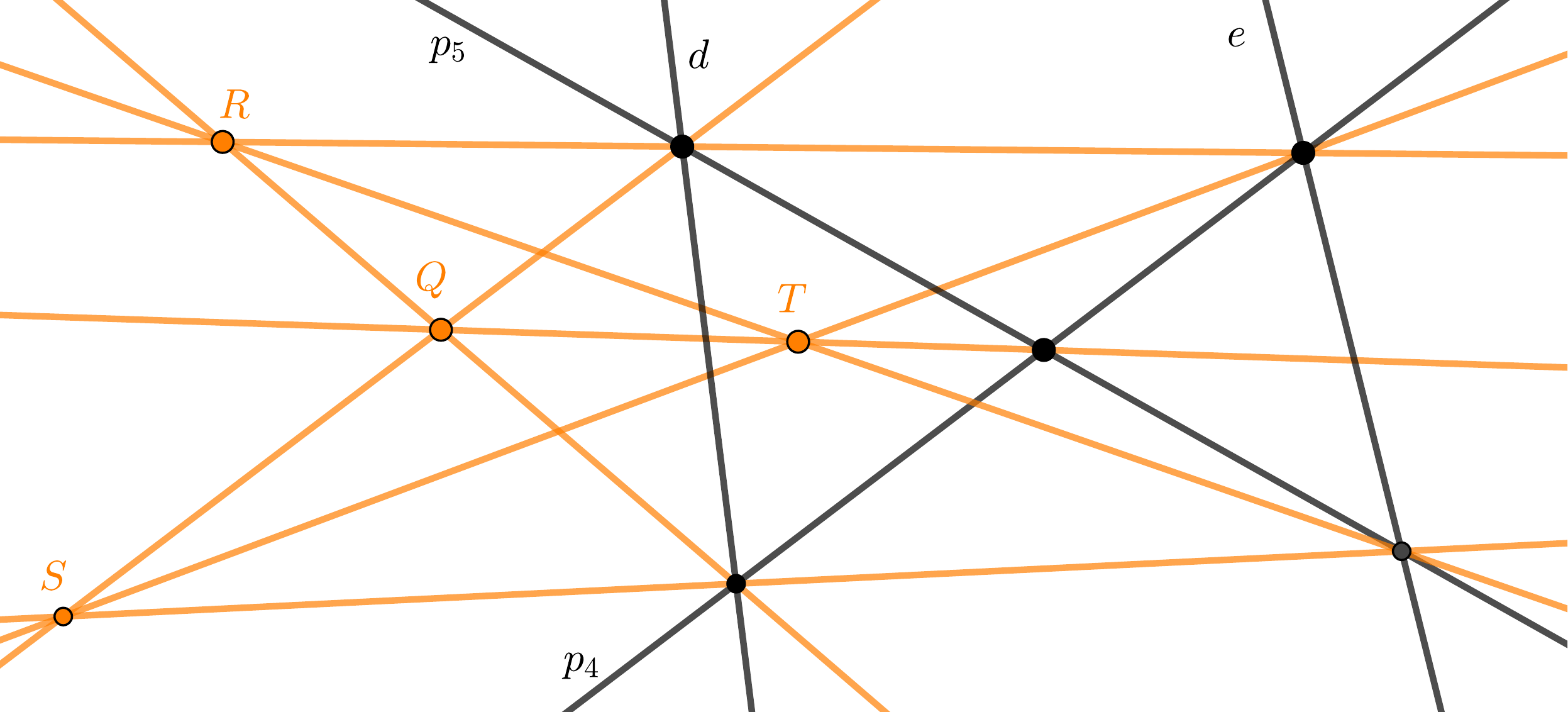}
\end{center}
\vspace{-.1in}
\caption{The dual version of the Pappus  theorem. 
}
\vspace{-.15in}
\label{fig:pappus-dual}
\end{figure}

\section*{Local moves on tilings}

We next discuss several types of \emph{local moves} (that is, local transformations of tilings 
of an oriented surface)  
that translate into equivalences between the corresponding incidence theorems.  

\begin{definition}[\emph{Fusion of adjacent tiles}]
\label{def:fusion-of-tiles}
Suppose that a tiling~$\TT$ of an oriented surface by quadrilateral tiles contains a vertex~$v$ of degree~2. 
Remove~$v$ together with the two edges incident to it to obtain a new tiling~$\TT'$
in which the two tiles surrounding~$v$ have been \emph{fused} into a single quadrilateral tile. 
It~is not hard to see that the incidence theorems associated to $\TT$ and~$\TT'$ 
are closely related; in fact, each of them immediately implies the other. 
See Figure~\ref{fig:fuse-tilings}. 

The removal of~$v$ from~$\TT$ can potentially produce a tiling~$\TT'$ 
that contains a ``self-folded'' quadrilateral tile 
in which two adjacent sides have been glued to each other,
cf.\ the bottom row of Figure~\ref{fig:fuse-tilings}. 
Although thus far in the paper, we did not allow such tiles, 
they may appear as a result of fusion moves. 
This is not going to create problems since self-folded tiles can be easily eliminated, 
see Definition~\ref{def:removal-self-folded} below. 
\end{definition}

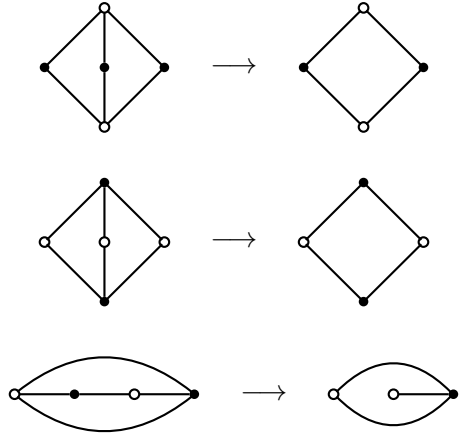
\begin{figure}[ht]
\begin{center}
\setlength{\unitlength}{0.45pt}
\begin{picture}(100,100)(0,0)
\thicklines
\put(1.8,51.8){{\line(1,1){46}}}
\put(1.8,48.2){{\line(1,-1){46}}}
\put(98.2,51.8){{\line(-1,1){46}}}
\put(98.2,48.2){{\line(-1,-1){46}}}

\put(50,4){{\line(0,1){42}}}
\put(50,54){{\line(0,1){42}}}

\put(0,50){\circle*{8}}
\put(50,0){\circle{8}}
\put(50,100){\circle{8}}
\put(100,50){\circle*{8}}

\put(50,50){\circle*{8}}
\end{picture}
\begin{picture}(100,100)(0,0)
\thicklines
\put(50,50){\makebox(0,0){$\longrightarrow$}}
\end{picture}
\begin{picture}(100,100)(0,0)
\thicklines
\put(1.8,51.8){{\line(1,1){46}}}
\put(1.8,48.2){{\line(1,-1){46}}}
\put(98.2,51.8){{\line(-1,1){46}}}
\put(98.2,48.2){{\line(-1,-1){46}}}

\put(0,50){\circle*{8}}
\put(50,0){\circle{8}}
\put(50,100){\circle{8}}
\put(100,50){\circle*{8}}
\end{picture}
\\[20pt]
\begin{picture}(100,100)(0,0)
\thicklines
\put(1.8,51.8){{\line(1,1){46}}}
\put(1.8,48.2){{\line(1,-1){46}}}
\put(98.2,51.8){{\line(-1,1){46}}}
\put(98.2,48.2){{\line(-1,-1){46}}}

\put(50,4){{\line(0,1){42}}}
\put(50,54){{\line(0,1){42}}}

\put(0,50){\circle{8}}
\put(50,0){\circle*{8}}
\put(50,100){\circle*{8}}
\put(100,50){\circle{8}}

\put(50,50){\circle{8}}
\end{picture}
\begin{picture}(100,100)(0,0)
\thicklines
\put(50,50){\makebox(0,0){$\longrightarrow$}}
\end{picture}
\begin{picture}(100,100)(0,0)
\thicklines
\put(1.8,51.8){{\line(1,1){46}}}
\put(1.8,48.2){{\line(1,-1){46}}}
\put(98.2,51.8){{\line(-1,1){46}}}
\put(98.2,48.2){{\line(-1,-1){46}}}

\put(0,50){\circle{8}}
\put(50,0){\circle*{8}}
\put(50,100){\circle*{8}}
\put(100,50){\circle{8}}
\end{picture}
\\[25pt]
\begin{picture}(150,40)(0,0)
\thicklines
\qbezier(2.2,22)(75,80)(147.8,22)
\qbezier(2.2,18)(75,-40)(147.8,18)

\put(4,20){{\line(1,0){42}}}
\put(54,20){{\line(1,0){42}}}
\put(104,20){{\line(1,0){42}}}

\put(0,20){\circle{8}}
\put(50,20){\circle*{8}}
\put(100,20){\circle{8}}
\put(150,20){\circle*{8}}

\end{picture}
\begin{picture}(100,40)(0,0)
\thicklines
\put(50,20){\makebox(0,0){$\longrightarrow$}}
\end{picture}
\begin{picture}(100,40)(0,0)
\thicklines
\qbezier(2.4,22)(50,70)(97.6,22)
\qbezier(2.4,18)(50,-30)(97.6,18)

\put(54,20){{\line(1,0){42}}}

\put(0,20){\circle{8}}
\put(50,20){\circle{8}}
\put(100,20){\circle*{8}}
\end{picture}
\end{center}
\caption{Fusion of tiles. 
}
\vspace{-15pt}
\label{fig:fuse-tilings}
\end{figure}

\begin{definition}[\emph{Getting rid of self-folded tiles}]
\label{def:removal-self-folded}
For a self-folded tile, the coherence condition is always vacuous. 
Consequently, such tiles can be removed (collapsed), 
as shown in Figure~\ref{fig:removal-of-self-folded-tiles}, 
without affecting the content of the associated incidence theorem. 
\end{definition}

\begin{figure}[ht]
\begin{center}
\setlength{\unitlength}{0.45pt}
\begin{picture}(100,40)(0,0)
\thicklines
\qbezier(2.4,22)(50,70)(97.6,22)
\qbezier(2.4,18)(50,-30)(97.6,18)

\put(54,20){{\line(1,0){42}}}

\put(0,20){\circle{8}}
\put(50,20){\circle{8}}
\put(100,20){\circle*{8}}
\end{picture}
\begin{picture}(100,40)(0,0)
\thicklines
\put(50,20){\makebox(0,0){$\longrightarrow$}}
\end{picture}
\begin{picture}(100,40)(0,0)
\thicklines
\put(4,20){{\line(1,0){92}}}

\put(0,20){\circle{8}}
\put(100,20){\circle*{8}}

\end{picture}
\\[20pt]
\begin{picture}(100,40)(0,0)
\thicklines
\qbezier(2.4,22)(50,70)(97.6,22)
\qbezier(2.4,18)(50,-30)(97.6,18)

\put(54,20){{\line(1,0){42}}}

\put(0,20){\circle*{8}}
\put(50,20){\circle*{8}}
\put(100,20){\circle{8}}
\end{picture}
\begin{picture}(100,40)(0,0)
\thicklines
\put(50,20){\makebox(0,0){$\longrightarrow$}}
\end{picture}
\begin{picture}(100,40)(0,0)
\thicklines
\put(4,20){{\line(1,0){92}}}

\put(0,20){\circle*{8}}
\put(100,20){\circle{8}}

\end{picture}
\vspace{-5pt}
\end{center}
\caption{Removal of self-folded tiles. 
}
\vspace{-15pt}
\label{fig:removal-of-self-folded-tiles}
\end{figure}
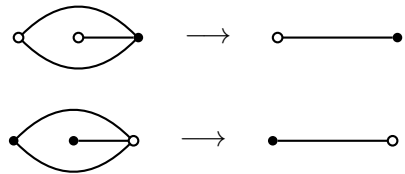

\begin{definition}[\emph{Desargues moves}]
\label{def:desargues-move}
The most interesting type of local move is a \emph{Desargues move} 
shown in Figure~\ref{fig:desargues-moves},
corresponds to applying the Desargues theorem, see Corollary~\ref{cor:tileable-hexagons}. 
If two tilings differ by a Desargues move, then the corresponding incidence theorems
can be obtained from each other via a single application of Desargues' theorem. 
\end{definition}

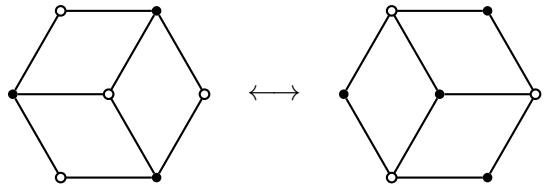
\begin{figure}[ht]
\begin{center}
\setlength{\unitlength}{0.45pt}
\begin{picture}(160,140)(0,0)
\thicklines

\put(2,73.5){{\line(4,7){36}}}
\put(2,66.5){{\line(4,-7){36}}}
\put(44,0){{\line(1,0){72}}}
\put(44,140){{\line(1,0){72}}}
\put(158,73.5){{\line(-4,7){36}}}
\put(158,66.5){{\line(-4,-7){36}}}

\put(4,70){{\line(1,0){72}}}
\put(82,73.5){{\line(4,7){36}}}
\put(82,66.5){{\line(4,-7){36}}}

\put(0,70){\circle*{8}}
\put(120,0){\circle*{8}}
\put(120,140){\circle*{8}}

\put(160,70){\circle{8}}
\put(40,0){\circle{8}}
\put(40,140){\circle{8}}

\put(80,70){\circle{8}}
\end{picture}
\begin{picture}(100,140)(0,0)
\thicklines
\put(50,70){\makebox(0,0){$\longleftrightarrow$}}
\end{picture}
\begin{picture}(160,140)(0,0)
\thicklines

\put(2,73.5){{\line(4,7){36}}}
\put(2,66.5){{\line(4,-7){36}}}
\put(44,0){{\line(1,0){72}}}
\put(44,140){{\line(1,0){72}}}
\put(158,73.5){{\line(-4,7){36}}}
\put(158,66.5){{\line(-4,-7){36}}}

\put(84,70){{\line(1,0){72}}}
\put(78,73.5){{\line(-4,7){36}}}
\put(78,66.5){{\line(-4,-7){36}}}

\put(0,70){\circle*{8}}
\put(120,0){\circle*{8}}
\put(120,140){\circle*{8}}

\put(160,70){\circle{8}}
\put(40,0){\circle{8}}
\put(40,140){\circle{8}}

\put(80,70){\circle*{8}}
\end{picture}
\vspace{-5pt}
\end{center}
\caption{A Desargues move shown above can be performed left-to-right or right-to-left. 
The six vertices of the ambient hexagon do not have to be distinct. 
}
\vspace{-15pt}
\label{fig:desargues-moves}
\end{figure}

\begin{problem}
\label{problem:tiling-moves}

If two tilings can be obtained from each other by a sequence of local moves
described in Definitions~\ref{def:fusion-of-tiles}--\ref{def:desargues-move}
(see Figures~\ref{fig:fuse-tilings}--\ref{fig:desargues-moves}),
then the corresponding incidence theorems are equivalent modulo Desargues' theorem. 
Therefore it would be interesting to effectively describe the equivalence classes of tilings 
with respect to these local moves.
\end{problem}

\clearpage

\newpage

\section*{Local moves on nodal curves}

\begin{remark}
The construction described in Definition~\ref{def:tiling-to-curve} associates a nodal curve to a tiled surface, and vice versa
(cf.\ Remark~\ref{rem:curve-to-tiling}). 
Accordingly, the local moves discussed above can be translated into the language of curves on a surface.  
To be specific, local moves on tilings shown in Figures~\ref{fig:fuse-tilings}--\ref{fig:desargues-moves} correspond, 
under the construction from Definition~\ref{def:tiling-to-curve}, 
to local moves on nodal curves that are shown in Figure~\ref{fig:yb-move}. \linebreak[3]
In~particular, the \emph{braid move} in the bottom row of Figure~\ref{fig:yb-move} 
corresponds to a Desargues move from Definition~\ref{def:desargues-move}/Figure~\ref{fig:desargues-moves}. 
Note that a~braid move may create a monogon even if the original curve did not have such regions, 
see Figure~\ref{fig:move-confluence}. 
\end{remark}

\begin{figure}[ht] 
\begin{center} 
\vspace{-5pt}
\setlength{\unitlength}{1.2pt} 
\begin{picture}(30,30)(0,0)
\thicklines
\qbezier(30,25)(30,8)(7,8)
\qbezier(30,-5)(30,12)(7,12)
\qbezier(7,8)(1,8)(1,10)
\qbezier(7,12)(1,12)(1,10)
\end{picture} 
\setlength{\unitlength}{1.5pt} 
\begin{picture}(33,20)(0,0)
\put(16.5,10){\makebox(0,0){$\longrightarrow$}}
\end{picture}
\setlength{\unitlength}{1.2pt} 
\begin{picture}(30,30)(0,0)
\thicklines
\put(15,-5){{\line(0,1){30}}}
\end{picture} 
\\[20pt]
\setlength{\unitlength}{1.5pt} 
\begin{picture}(30,10)(0,0)
\thicklines
\qbezier(0,0)(15,20)(30,0)
\qbezier(0,10)(15,-10)(30,10)
\end{picture} 
\begin{picture}(33,10)(0,0)
\put(16.5,5){\makebox(0,0){$\longrightarrow$}}
\end{picture}
\begin{picture}(30,10)(0,0)
\thicklines
\put(0,0){{\line(3,1){30}}}
\put(0,10){{\line(3,-1){30}}}
\end{picture} 
\\[10pt]
\setlength{\unitlength}{1.8pt} 
\begin{picture}(24,33)(0,0)
\thicklines
\qbezier(12,2)(5,16.5)(12,31)
\qbezier(0,10)(14,12)(24,23)
\qbezier(24,10)(14,21)(0,23)
\end{picture} 
\setlength{\unitlength}{1.35pt} 
\begin{picture}(33,33)(0,0)
\put(16.5,21){\makebox(0,0){$\longleftrightarrow$}}
\end{picture}
\setlength{\unitlength}{1.8pt} 
\begin{picture}(24,33)(0,0)
\thicklines
\qbezier(12,2)(19,16.5)(12,31)
\qbezier(0,10)(10,21)(24,23)
\qbezier(24,10)(10,12)(0,23)
\end{picture} 
\\
\vspace{-5pt}
\end{center} 
\caption{Local moves on nodal curves. 
The moves in the top and bottom rows are similar to Reidemeister moves of types~I and~III, 
respectively. The move in the middle (which corresponds to fusion moves on tilings, 
see Definition~\ref{def:fusion-of-tiles}/Figure~\ref{fig:fuse-tilings})
is however different from Reidemeister~II. 
} 
\label{fig:yb-move} 
\end{figure}
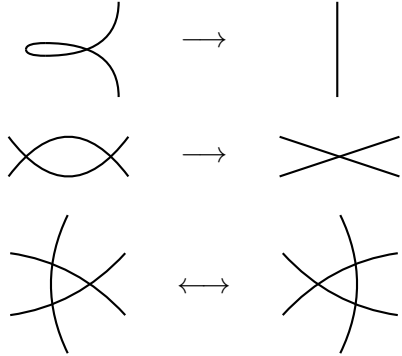 

\begin{figure}[ht] 
\vspace{-35pt}
\begin{equation*}
\begin{tikzpicture}[baseline= (a).base]
\node[scale=1] (a) at (0,0){
\begin{tikzcd}[arrows={-stealth}, sep=small, cramped]
\setlength{\unitlength}{1.2pt} 
\begin{picture}(40,30)(0,7.5)
\thicklines
\qbezier(30,25)(30,8)(7,8)
\qbezier(30,-5)(30,12)(7,12)
\qbezier(7,8)(1,8)(1,10)
\qbezier(7,12)(1,12)(1,10)
\put(10,-5){{\line(0,1){30}}}
\end{picture} 
  \arrow[r]   & 
  \setlength{\unitlength}{1.2pt} 
\begin{picture}(45,30)(-10,7.5)
\thicklines
\qbezier(30,25)(30,8)(7,8)
\qbezier(30,-5)(30,12)(7,12)
\qbezier(7,8)(1,8)(1,10)
\qbezier(7,12)(1,12)(1,10)
\put(25,-5){{\line(0,1){30}}}
\end{picture} 
  \arrow[r] & 
  \setlength{\unitlength}{1.2pt} 
\begin{picture}(35,30)(5,7.5)
\thicklines
\qbezier(30,25)(30,20)(20,15)
\qbezier(20,15)(10,10)(20,5)
\qbezier(30,-5)(30,0)(20,5)
\put(25,-5){{\line(0,1){30}}}
\end{picture} 
\\[3pt]
{\ } \arrow[d] & & {\ } \arrow[d] 
\\[5pt]
\setlength{\unitlength}{1.2pt} 
\begin{picture}(40,27)(0,2)
\thicklines
\qbezier(30,25)(30,5)(20,5)
\qbezier(10,25)(10,5)(20,5)
\qbezier(30,-5)(30,15)(20,15)
\qbezier(10,-5)(10,15)(20,15)
\end{picture} 
\arrow[rr]
&&
 \setlength{\unitlength}{1.2pt} 
\begin{picture}(35,27)(5,2)
\thicklines
\put(15,-5){{\line(1,2){15}}}
\put(30,-5){{\line(-1,2){15}}}
\end{picture} 
\end{tikzcd}
};
\end{tikzpicture}
\end{equation*}
\caption{The phenomenon of confluence of local moves on nodal curves. 
} 
\vspace{-15pt}
\label{fig:move-confluence} 
\end{figure}
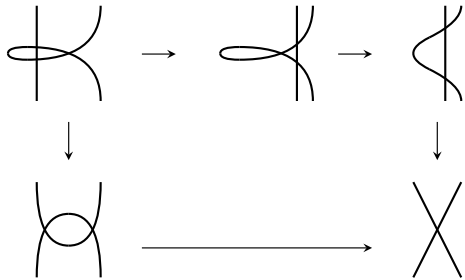 

\begin{problem}
\label{problem:nodal-moves}
Reformulating Problem~\ref{problem:tiling-moves}, it would be interesting to classify 
the equivalence classes of  nodal curves on closed oriented surfaces 
with respect to the local moves in Figure~\ref{fig:yb-move}. 
As in Remark~\ref{rem:curve-to-tiling}, we assume that 
each region (a connected component of the complement of a curve) is homeomorphic to a disk 
and that these regions are properly colored in two colors. 
It is natural to ask whether all such curves within a given move-equivalence class 
that have the smallest number of nodes are braid-equivalent to each other, 
cf.\ Figure~\ref{fig:move-confluence}. 
(Incidence theorems associated with curves that differ by a braid move 
are equivalent modulo Desargues' theorem.) 
\end{problem}

\clearpage

\newpage

\section*{Collapse and insertion of tiles}

We next discuss several combinatorial transformations of tiled surfaces 
that yield logical implications between incidence theorems, 
that is, instances where one such theorem directly follows from some other theorem(s).
The simplest, yet nontrivial, example of such combinatorial transformations
involves a single tile:

\begin{definition}
\label{def:tile-collapse-insertion}
Let $\TT$ be a bicolored quadrilateral tiling of an oriented surface~$\Sigma$.
Pick a tile~$\tau$ in~$\TT$ and construct a new tiling~$\TT'$ 
by removing the interior of~$\tau$ and gluing its four edges in pairs,
as shown below:
\begin{equation*}
\begin{tikzpicture}[baseline= (a).base]
\node[scale=1] (a) at (0,0){
\begin{tikzcd}[arrows={-stealth}, sep=small, cramped]
&& A \edge{rrd}  \edge{lld} \\
\ell \edge{rrd} &&\boxed{\tau}&& m \edge{lld}&  \longrightarrow & \ell \edge{rr} && A && m \edge{ll}& \\
&& A'
\end{tikzcd}
};
\end{tikzpicture}
\end{equation*}
We say that the resulting tiling~$\TT'$ is obtained from~$\TT$ by \emph{collapsing} 
the tile~$\tau$. 
Going in the opposite direction, we say that $\TT$ is obtained from~$\TT'$ by \emph{inserting} 
the tile~$\tau$ into the length-2 path $\ell -\!\!\!- A -\!\!\!- m$ in~$\TT'$. 

The incidence theorem associated with $\TT'$ is obtained from the one associated with~$\TT$
by making the additional assumption $A=A'$. 
\end{definition} 

Tile insertion can be used to produce nontrivial generalizations of existing incidence theorems. 
Several examples illustrating this technique are given below. 

\begin{example}
\label{eg:generalized-perm-theorem}
Recall that the permutation theorem (Theorem~\ref{th:perm-thm})
can be obtained from the tiling in Figure~\ref{fig:torus-perm}. 
Insert two tiles into the latter tiling to obtain the tiling shown in Figure~\ref{fig:torus-perm-generalized}. 
The latter tiling yields a generalization of the permutation theorem stated below in Theorem~\ref{th:generalized-perm-theorem}. 
\end{example}

\begin{figure}[ht]
\begin{equation*}
\begin{tikzpicture}[baseline= (a).base]
\node[scale=1] (a) at (0,0){
\begin{tikzcd}[arrows={-stealth}, cramped, sep=10]
& t \edge{rr, dotted} \edge{ldd, dotted} \edge{rd} && r \edge{rrr, dotted} \edge{rd} \edge{rrd} \edge{ld} &&& u \edge{rdd, dotted} \edge{ld} \\
&& Q_1 \edge{rd} \edge{lld} && Q_2 \edge{ld} \edge{rrrd} & Q_2' \edge{rrd}  \\
u \edge{rdd, dotted} \edge{rrd} &&& s \edge{rd} \edge{rrd} \edge{ld} &&&& t \edge{ldd, dotted} \edge{lld} \\
&& P_1 \edge{rd} \edge{ld} && P_2 \edge{ld} \edge{rrd} & Q_2' \edge{rd} \\
& t \edge{rr, dotted} && r \edge{rrr, dotted} &&& u
\end{tikzcd}
};
\end{tikzpicture}
\end{equation*}
\vspace{-.1in}
\caption{The tiling of the torus obtained by inserting two tiles into the tiling in Figure~\ref{fig:torus-perm}. 
}
\vspace{-.1in}
\label{fig:torus-perm-generalized}
\end{figure}
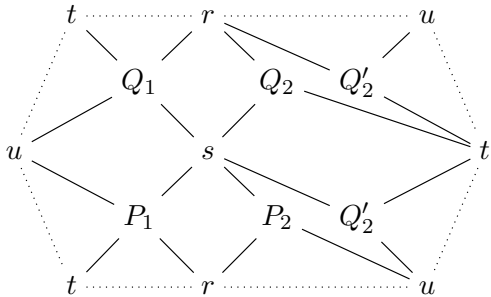

\begin{theorem}
\label{th:generalized-perm-theorem}
Let $P_1, Q_1, R_1, S_1$ be four distinct collinear points in the real/complex projective plane.
Let $R_2, S_2$ be two generic points in the plane. 
Define $r=(R_1R_2)$, $s=(S_1S_2)$, $t=(R_1S_2)$, $u=(R_2S_1)$, 
$A=r\cap s$, $B=t\cap u$. 
Pick a point $P_2\in (AP_1)$. 
Set 
$P_2'=(S_1P_2)\cap(BQ_1)$, 
$Q_2=(S_2P_2')\cap(AQ_1)$,
$Q_2'=(R_1Q_2)\cap(BP_1)$.
Then the points $Q_2', P_2, R_2$ are collinear. 
See Figure~\ref{fig:perm2}. 
\end{theorem}

\begin{figure}[ht]
\begin{center}
\includegraphics[scale=0.5, trim=0.1cm 0.1cm 0.6cm 0.2cm, clip]{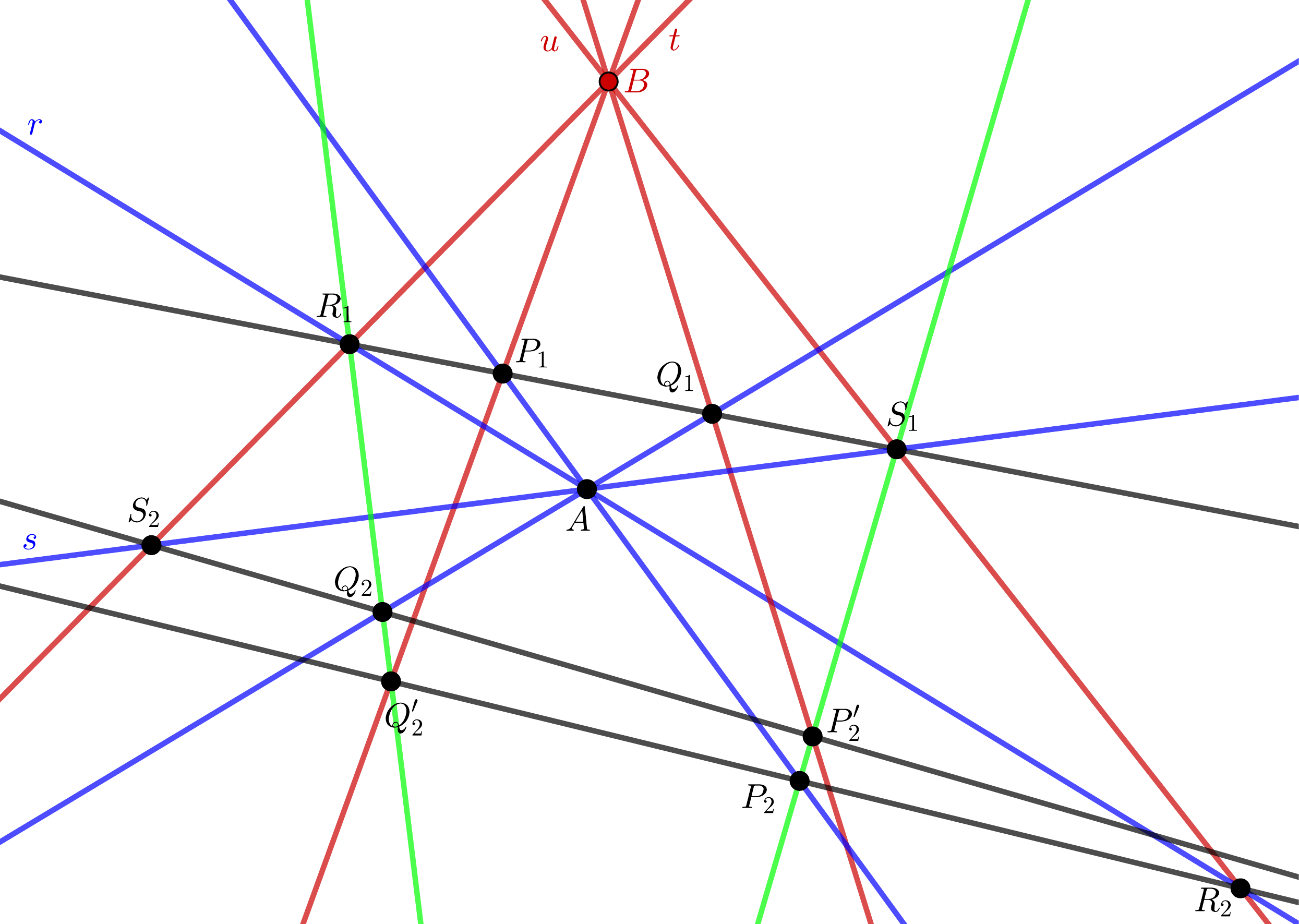}
\end{center}
\vspace{-.1in}
\caption{Generalization of the permutation theorem.}
\vspace{-.15in}
\label{fig:perm2}
\end{figure}

\begin{remark}
\label{rem:sequential-construction}
Theorem~\ref{th:generalized-perm-theorem} has a basic property 
that the original permutation theorem
(Theorem~\ref{th:perm-thm}) lacks:
configurations appearing in Theorem~\ref{th:generalized-perm-theorem}
can~be~constructed, step by step, using a straightedge alone.
By contrast, Theorem~\ref{th:perm-thm}~assumes a certain condition 
(\emph{viz.}, that three of the four lines $(P_1Q_2)$, $(R_1S_2)$, $(R_2S_1)$, $(P_2Q_1)$ are concurrent)
that cannot be guaranteed by a straightedge-only construction. 
\end{remark}

\begin{example}
\label{eg:generalized-perm-theorem-4}
A further generalization of the permutation theorem 
can be obtained by inserting two additional tiles into the tiling in Figure~\ref{fig:torus-perm-generalized}
(thus, inserting four tiles into the tiling in Figure~\ref{fig:torus-perm}),
as shown in Figure~\ref{fig:torus-perm-generalized-further}. 
The resulting generalization of the permutation theorem is illustrated in Figure~\ref{fig:perm3}. 
We omit a formal statement of this generalization, which can be recovered from the figure.
\enlargethispage{1.5cm}
\end{example}

\begin{figure}[ht]
\vspace{-.1in}
\begin{equation*}
\begin{tikzpicture}[baseline= (a).base]
\node[scale=1] (a) at (0,0){
\begin{tikzcd}[arrows={-stealth}, cramped, sep=10]
& t \edge{rrr, dotted} \edge{ldd, dotted} \edge{rd} \edge{rrd} &&& r \edge{rrr, dotted} \edge{rd} \edge{rrd} \edge{ld} &&& u \edge{rdd, dotted} \edge{ld} \\
&& Q_1 \edge{rrd} \edge{lld} & Q_1' \edge{rd} && Q_2 \edge{ld} \edge{rrrd} & Q_2' \edge{rrd}  \\
u \edge{rdd, dotted} \edge{rrd} \edge{rrrd} &&&& s \edge{rd} \edge{rrd} \edge{ld} &&&& t \edge{ldd, dotted} \edge{lld} \\
&& P_1 \edge{rrd} \edge{ld} & P_1' \edge{rd}&& P_2 \edge{ld} \edge{rrd} & Q_2' \edge{rd} \\
& t \edge{rrr, dotted} &&& r \edge{rrr, dotted} &&& u
\end{tikzcd}
};
\end{tikzpicture}
\end{equation*}
\vspace{-.15in}
\caption{The tiling of the torus obtained by inserting four tiles into the tiling in 
Figure~\ref{fig:torus-perm}. }
\label{fig:torus-perm-generalized-further}
\end{figure}
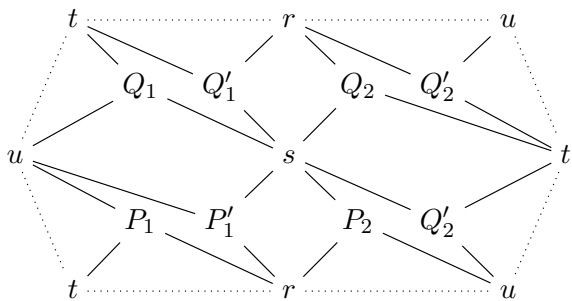

\begin{figure}[ht]
\vspace{-.1in}
\begin{center}
\includegraphics[scale=0.42, trim=0.1cm 0.2cm 0.6cm 0cm, clip]{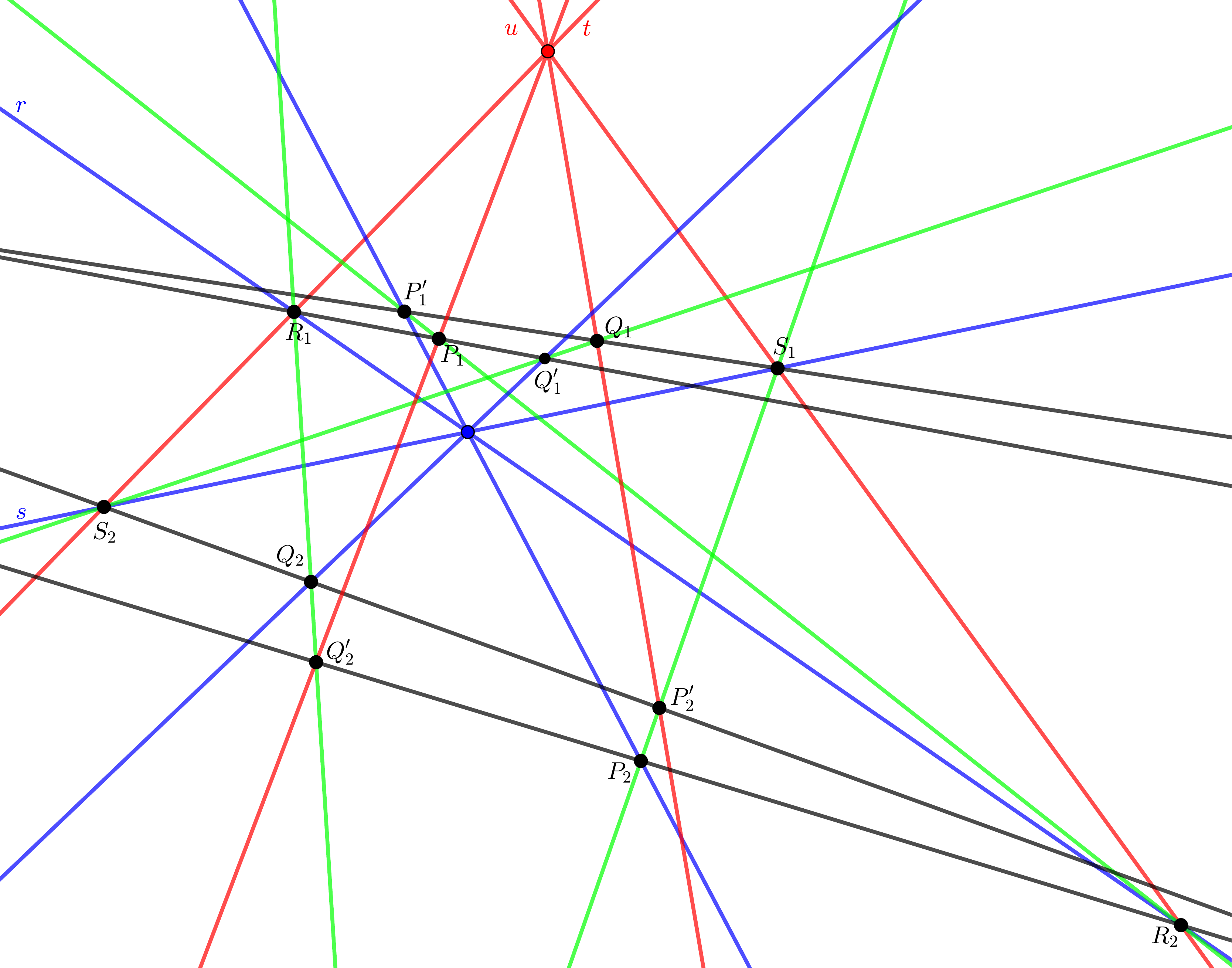}
\end{center}
\vspace{-.1in}
\caption{Further generalization of the permutation theorem. 
}
\vspace{-.15in}
\label{fig:perm3}
\end{figure}

\begin{example}
\label{eg:pappus-split1}
Recall that the Pappus theorem can be obtained from the tiling in Figure~\ref{fig:pappus-torus}. 
Insert a tile in the middle of the latter tiling to obtain the tiling 
shown in Figure~\ref{fig:pappus-torus-split1} 
The latter tiling yields a generalization of the Pappus theorem stated below in Theorem~\ref{th:pappus-split1}. 
\end{example}

\begin{figure}[ht]
\vspace{-.15in}
\begin{equation*}
\begin{tikzpicture}[baseline= (a).base]
\node[scale=1] (a) at (0,0){
\begin{tikzcd}[arrows={-stealth}, sep=5, cramped]
&& &  && b \edge{lld} \edge{rrd} \edge{rrrrd, dotted} &&  \\[15pt]
& a \edge{rr} \edge{rd} \edge{rrrru, dotted} \edge{dd, dotted}&& P_1 \edge{rd} &&&& P_2 \edge{rr} \edge{ld} && a \edge{ld} \edge{dd, dotted} \\[15pt]
 && P_6 \edge{rr} \edge{ld} && c \edge{ld} \edge{rrru} && c' \edge{rr} \edge{rd} && P_3 \edge{rd} & &  \\[15pt]
& b \edge{rr} \edge{rrrrd, dotted} && P_5 \edge{rrd} \edge{rrru} &&&& P_4 \edge{rr} \edge{lld} && b  \\[15pt]
&&  &&& a \edge{rrrru, dotted} && 
\end{tikzcd}
};
\end{tikzpicture}
\end{equation*}
\vspace{-.15in}
\caption{The tiling of the torus obtained by inserting a tile into the tiling in Figure~\ref{fig:pappus-torus}. }
\vspace{-.2in}
\label{fig:pappus-torus-split1}
\end{figure}
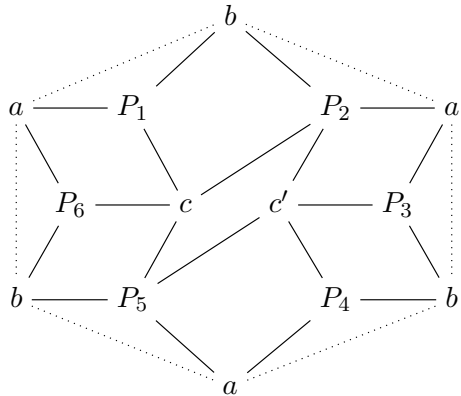

\begin{theorem}
\label{th:pappus-split1}
Consider three distinct lines $(P_1P_4)$, $(P_2P_5)$, $(P_3P_6)$ on the real or complex plane
that intersect at a common point~$C$.
(The points $P_i$ are generically chosen on the corresponding lines.) 
Set $A=(P_1P_2)\cap(P_5P_6)$, 
$B=(P_2P_3)\cap (P_4P_5)$,
$A'=(AC)\cap(P_3P_4)$, 
$B'=(BC)\cap(P_1P_6)$. 
Then the lines $(AB'), (A'B), (P_2P_5)$ are concurrent. 
See Figure~\ref{fig:pappus-1split}. 
\end{theorem}

\begin{figure}[ht]
\vspace{0pt}
\begin{center}
\includegraphics[scale=0.5, trim=0.5cm 1cm 0.6cm 0.4cm, clip]{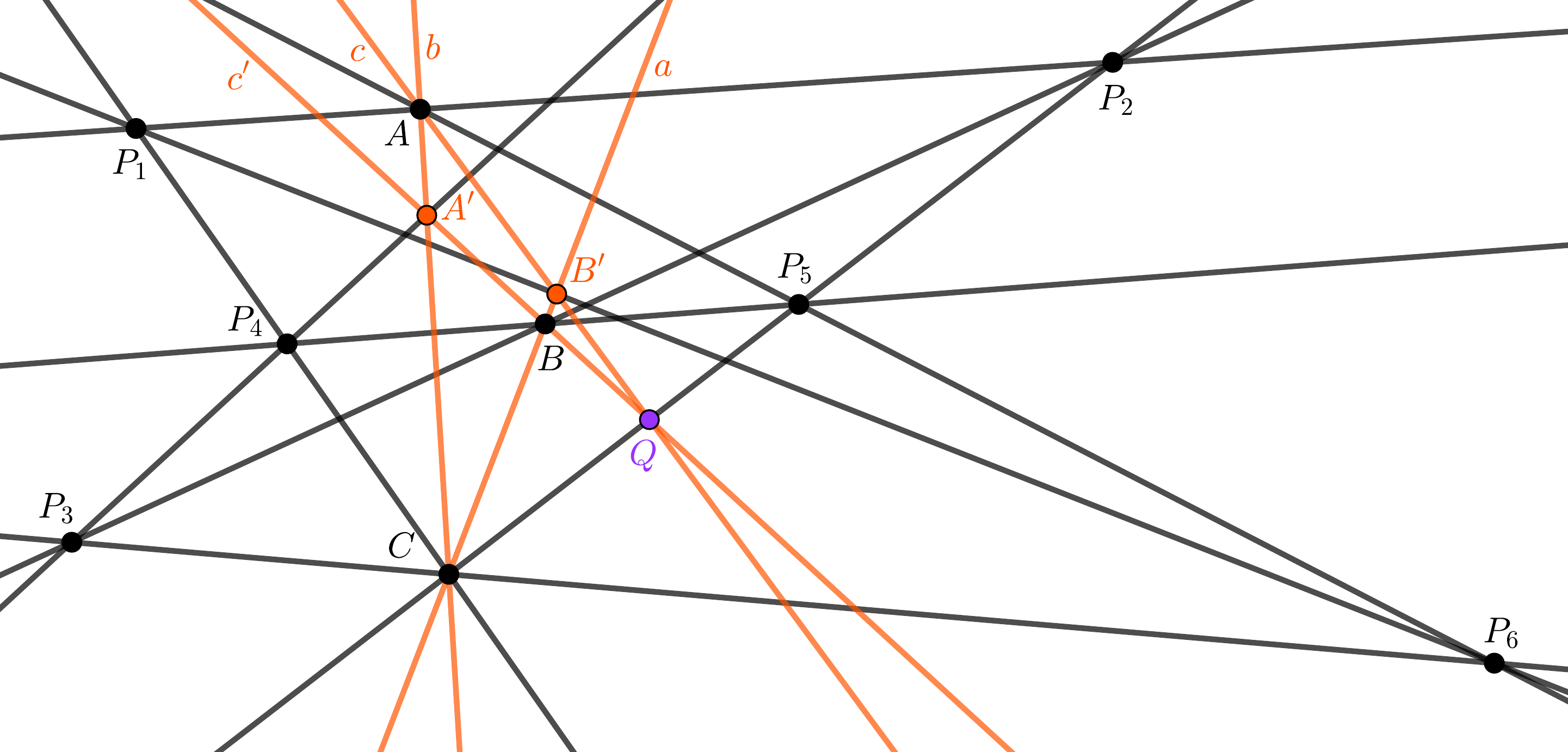}
\end{center}
\vspace{-.1in}
\caption{Generalization of the Pappus theorem. 
}
\vspace{-.15in}
\label{fig:pappus-1split}
\end{figure}

\begin{example}
\label{eg:pappus-split3}
A further generalization of the Pappus theorem 
can be obtained by inserting two additional tiles into the tiling in Figure~\ref{fig:pappus-torus-split1}
(thus, inserting three tiles into the tiling in Figure~\ref{fig:pappus-torus}),
as shown in Figure~\ref{fig:pappus-split3-tiling}. 
The resulting generalization of the Pappus theorem is illustrated in Figure~\ref{fig:pappus-3splits}. 
We omit a formal statement of this generalization, which can be recovered from the figure.
\end{example}

\begin{figure}[ht]
\begin{equation*}
\begin{tikzcd}[arrows={-stealth}, cramped, sep=small]
 && b \edge{ld} \edge{rd}  \edge{rr, dotted} && b' \edge{ld} \edge{rd} \edge{rr, dotted}  && a \edge{ld} \edge{rd} \\[5pt]
& P_4 \edge{ld} \edge{rd}  && P_1 \edge{ld} \edge{rd} && P_2 \edge{ld} \edge{rd} && P_3\edge{ld} \edge{rd} \\[5pt]
a \edge{rd} & & a' \edge{ld} \edge{rd} & & c \edge{rd}  \edge{ld} && c' \edge{ld} \edge{rd} && b \edge{ld} \\[5pt]
& P_3  \edge{rd}  && P_6 \edge{ld} \edge{rd} && P_5 \edge{ld} \edge{rd} && P_4 \edge{ld} \\[5pt]
 & & b  \edge{rr, dotted} & & b' \edge{rr, dotted} && a 
\end{tikzcd}
\end{equation*}
\caption{The tiling of the torus obtained by inserting three tiles into the tiling in Figure~\ref{fig:pappus-torus}. 
}
\label{fig:pappus-split3-tiling}
\end{figure}
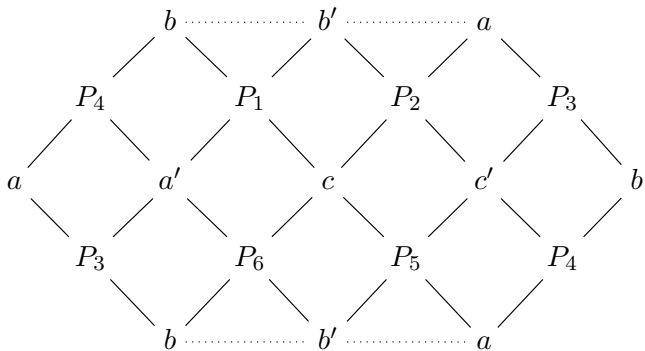

\begin{figure}[ht]
\vspace{-.05in}
\begin{center}
\includegraphics[scale=0.5, trim=0.9cm 0.3cm 0.9cm 0.5cm, clip]{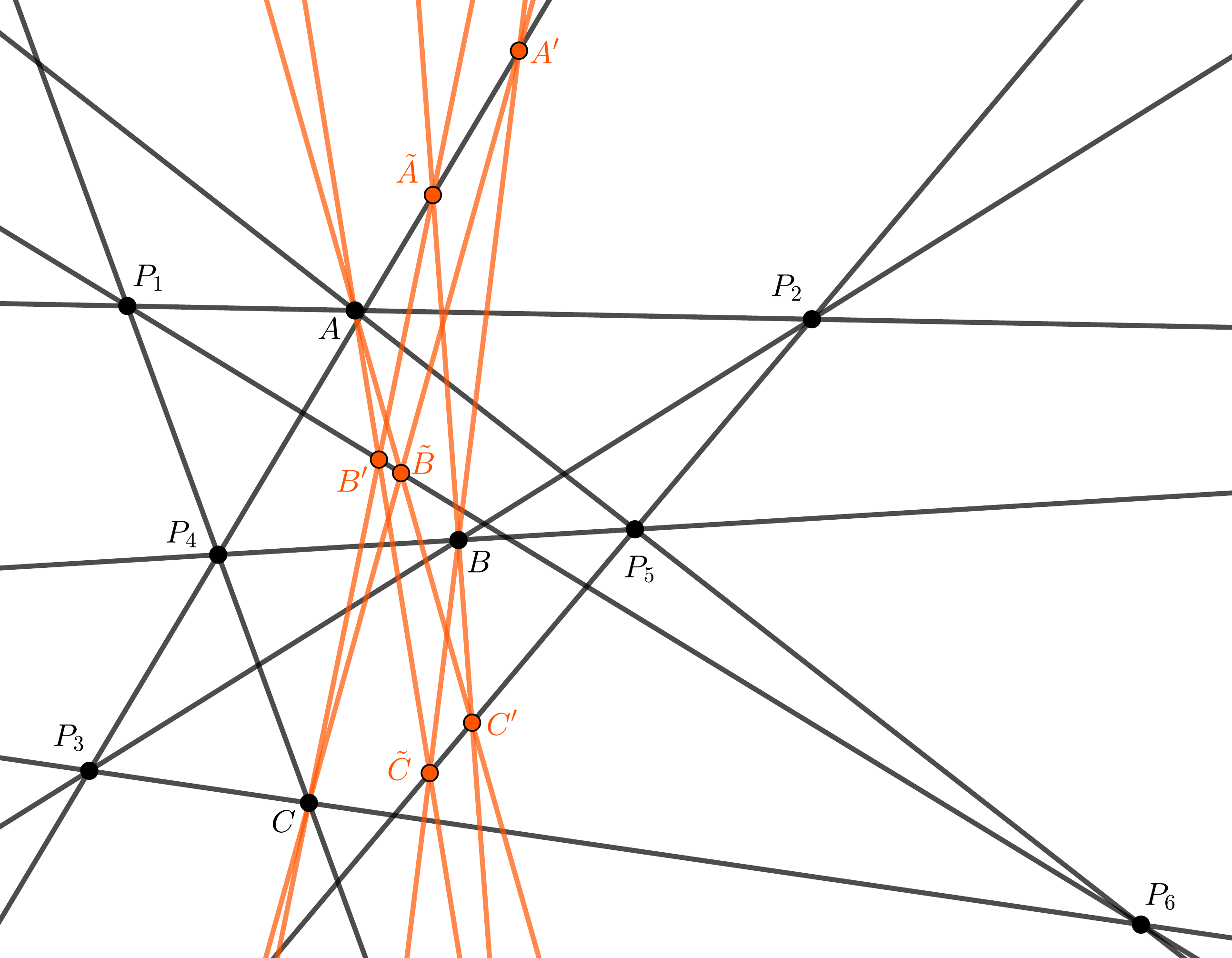}
\end{center}
\vspace{-.1in}
\caption{Further generalization of the Pappus theorem. 
}
\vspace{-.15in}
\label{fig:pappus-3splits}
\end{figure}


\begin{remark}
\label{rem:making-vertices-trivalent}
The tile insertion construction of Definition~\ref{def:tile-collapse-insertion}
can be used to ``lift'' an arbitrary bicolored quadrilateral tiling~$\TT$ of an oriented surface~$\Sigma$ 
to a tiling~$\TT'$ of~$\Sigma$ in which every black vertex is trivalent, see Figure~\ref{fig:split-vertex}. 
(If~$\TT$ contains bivalent vertices, remove them first, as explained in Definition~\ref{def:fusion-of-tiles}.)
Translating this into the language of graphs properly embedded into the surface
(see Definition~\ref{def:graphs-on-surfaces}), 
we conclude that each face of the graph~$\GG(\TT')$ is a triangle,
i.e., $\GG(\TT')$ describes a \emph{triangulation} of the surface~$\Sigma$. 
Thus, any incidence theorem associated to a tiled surface 
can be generalized to an incidence theorem associated to a triangulated surface. 
We will make the latter statement explicit in Corollary~\ref{cor:two-triangulations}, cf.\ also 
Remark~\ref{rem:polygonal-subdivisions}. 
\end{remark}

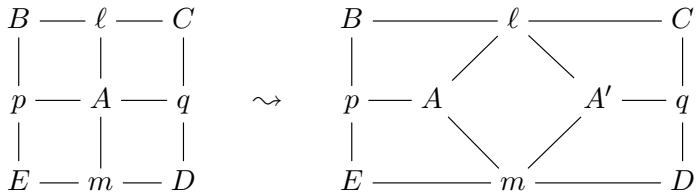
\begin{figure}[ht]
\begin{equation*}
\begin{tikzpicture}[baseline= (a).base]
\node[scale=1] (a) at (0,0){
\begin{tikzcd}[arrows={-stealth}, cramped, sep=15]
B \edge{r} \edge{d} & \ell \edge{r} \edge{d} & C \edge{d} && B \edge{rr} \edge{d} && \ell \edge{rr} \edge{dl} \edge{dr} && C \edge{d}\\
p \edge{r} \edge{d} & A \edge{d} \edge{r} & q \edge{d}  & \leadsto & p \edge{r} \edge{d} & A \edge{rd} && A' \edge{r} \edge{dl} & q \edge{d}\\
E \edge{r} & m \edge{r} & D && E \edge{rr} && m \edge{rr} && D
\end{tikzcd}
};
\end{tikzpicture}
\end{equation*}
\vspace{-.1in}
\caption{Splitting a four-valent vertex~$A$ into two, by inserting a tile. 
Iterating this operation, we can get rid of all black vertices of degree~$\ge4$.
Indeed, a black vertex of degree~$d\ge4$ can be split into two black vertices of degrees~3 and~${d-1}$, respectively. 
}
\vspace{-.1in}
\label{fig:split-vertex}
\end{figure}

\clearpage

\newpage

\section*{Collapsing around a four-valent vertex}

\begin{remark}
\label{rem:collapse-4-valent}
Another situation in which assigning identical labels to different vertices of a tiling 
gives rise to a ``collapsing''  transformation is shown in Figure~\ref{fig:another-collapse}. 
Let $\TT$ be a tiling that contains a vertex~$A$ of degree~4, as shown in Figure~\ref{fig:another-collapse} on the left. 
Removing the four tiles surrounding $A$ from~$\TT$ 
and replacing them by the two tiles shown in Figure~\ref{fig:another-collapse} on the right
produces a tiling~$\TT'$. 
We claim that the incidence theorem associated with~$\TT'$ is immediate from 
the incidence theorem associated with~$\TT$, 
and moreover this implication does not rely on the master theorem.  
Suppose that all tiles in~$\TT'$, with the exception of one tile~$\tau$, are known to be coherent. 
There are two possible cases, depending on whether~$\tau$ 
is one of the two tiles shown in Figure~\ref{fig:another-collapse} on the right, or not. 
In each case, the claim follows from the special case $p=p'$ 
of the incidence theorem associated with~$\TT$ by setting $A=(BC)\cap (DE)$. 
\end{remark}

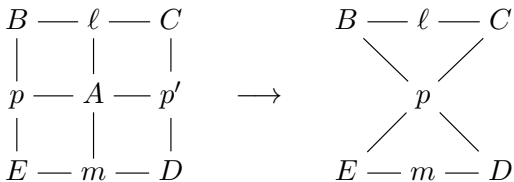
\begin{figure}[ht]
\vspace{-.1in}
\begin{equation*}
\begin{tikzpicture}[baseline= (a).base]
\node[scale=1] (a) at (0,0){
\begin{tikzcd}[arrows={-stealth}, cramped, sep=13]
B \edge{r} \edge{d} & \ell \edge{r} \edge{d} & C \edge{d} && B \edge{r} \edge{dr} & \ell \edge{r} & C \edge{dl}\\
p \edge{r} \edge{d} & A \edge{d} \edge{r} & p' \edge{d}  & \longrightarrow & & p \edge{dr} \edge{dl} \\
E \edge{r} & m \edge{r} & D && E \edge{r} & m \edge{r} & D 
\end{tikzcd}
};
\end{tikzpicture}
\end{equation*}
\vspace{-.15in}
\caption{Tile collapse around a four-valent vertex. 
}
\vspace{-.15in}
\label{fig:another-collapse}
\end{figure}

\enlargethispage{0.5cm}

\begin{example}
Both the complete quad theorem (Theorem~\ref{th:complete-quad})
an its generalization given in Theorem~\ref{th:complete-quad-generalization} 
can be deduced from Desargues' theorem, as follows. 
Consider the tiling in Figure~\ref{fig:P14-twice--tiling},
which can be viewed as invoking Desargues' theorem twice (cf.\ Remark~\ref{rem:gluing-tilings} below). 
Apply the collapsing transformation from Remark~\ref{rem:collapse-4-valent}
around vertex~$\ell$ to obtain the tiling in Figure~\ref{fig:complete-quadrilateral-tiling}, 
which was used to establish Theorem~\ref{th:complete-quad-generalization}. 
\end{example}

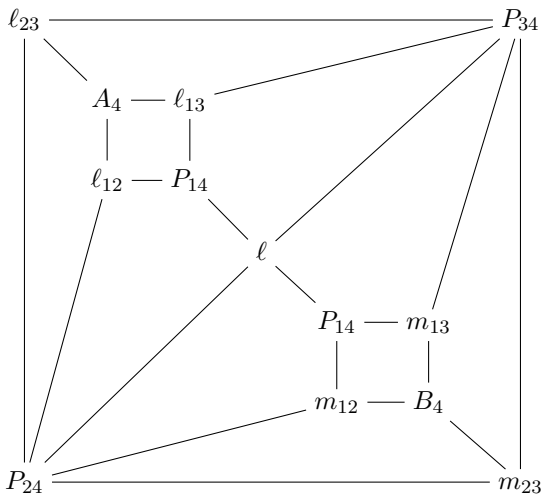
\begin{figure}[ht]
\vspace{-5pt}
\begin{equation*}
\begin{tikzpicture}[baseline= (a).base]
\node[scale=.9] (a) at (0,0){
\begin{tikzcd}[arrows={-stealth}, cramped, sep=small]
\ell_{23} \edge{rrrrrr} \edge{rd} \edge{dddddd} &&&&&& P_{34} \edge{lllld}\edge{lllddd}\edge{ldddd} \edge{dddddd} \\[8pt]
& A_4 \edge{r} \edge{d} & \ell_{13} \edge{d}& & \\[8pt]
& \ell_{12} \edge{r} \edge{ldddd} & P_{14} \edge{dr}  \\[5pt]
&&& \ell \edge{lllddd} \edge{rd}\\[5pt]
& & && P_{14} \edge{r} \edge{d} & m_{13} \edge{d} \\[8pt]
&&&& m_{12} \edge{r} \edge{lllld}& B_4 \edge{rd} & \\[8pt]
P_{24}  \edge{rrrrrr} &&&&&& m_{23}
\end{tikzcd}
};
\end{tikzpicture}
\end{equation*}
\vspace{-10pt}
\caption{
Deducing the complete quad theorem from the Desargues theorem. 
}
\label{fig:P14-twice--tiling}
\end{figure}

\clearpage

\newpage

\section*{Gluing of tilings}

\begin{remark}[\emph{Gluing of tilings}]
\label{rem:gluing-tilings}
Given a pair of tilings of closed surfaces, we can get a new tiling by gluing the given tilings along disks. 
The theorems associated to the two original tilings will accordingly imply the theorem associated to the third. 

The simplest version of this procedure glues two tilings along an arbitrary pair of quadrilateral tiles
(one in each tiling). 
To give an example, start with a tiling of the sphere into three tiles, see~\eqref{eq:sphere-3-tiles},
and glue in two cubes (each corresponding to an instance of Desargues' theorem) 
to obtain the tiling in Figure~\ref{fig:P14-twice--tiling}. 

A bit more complicated version of gluing involves locating, in each of the two input tilings, 
a hexagon tiled by three tiles. 
We can then glue the two tilings along these hexagons, potentially using Corollary~\ref{cor:tileable-hexagons} (Desargues' theorem)
in between. 
A further generalization of this construction involves gluing along $2n$-gons bounding a disk.
\end{remark}

\section*{Pappus implies Desargues}

It is well known that Pappus' theorem implies Desargues'; 
the first proof of this fact was given by G.~Hessenberg~\cite{hessenberg}. 
It was later discovered that his proof had gaps. 
In~particular, it only established a ``generic'' version of the statement (which is sufficient for our purposes). 
For the history of this result and its proofs, see \cite{pambuccian-2004, seidenberg-monthly}.  
As pointed out in~\cite{pambuccian-2004},
all known arguments require (at least) three applications of Pappus to get Desargues. 

We next provide a tiling-based proof of Hessenberg's theorem. 
It is inspired by A.~Seidenberg's presentation~in \cite[pp.~92--94]{seidenberg-book}. 
We will invoke, three times in a row, the version of the Pappus theorem given in 
Theorem~\ref{th:pappus-dual}/Figures~\ref{fig:pappus2-tiling-dual}--\ref{fig:pappus-dual}. 
As a result, we will obtain Desargues' theorem, in the formulation  
given in Theorem~\ref{th:desargues}. 


\begin{figure}[ht]
\vspace{-5pt}
\begin{center}
\includegraphics[scale=0.5, trim=1cm 0.3cm 0.6cm 0.2cm, clip]{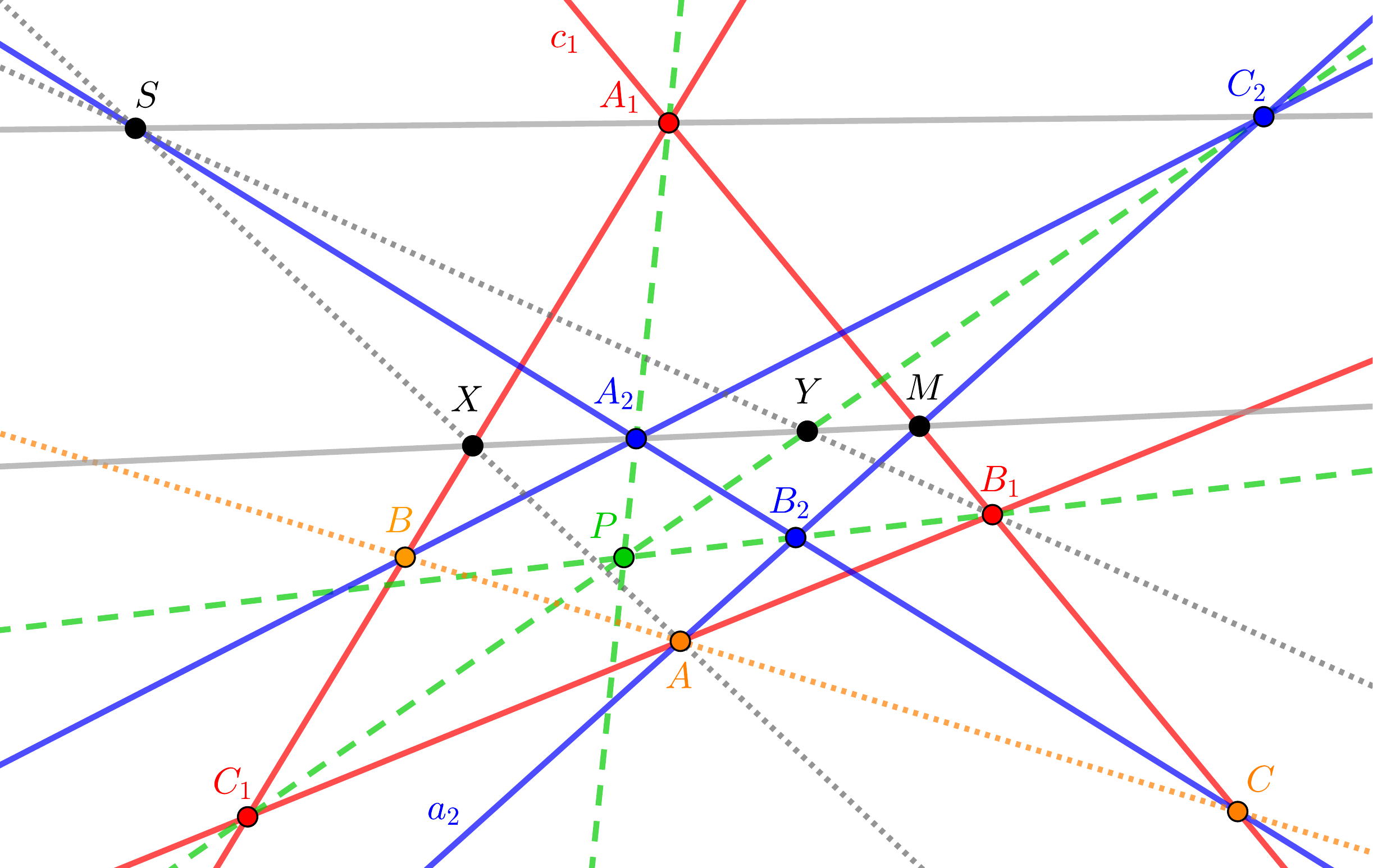}
\end{center}
\vspace{-5pt}
\caption{Pappus implies Desargues: 	the points-and-lines configuration. 
}
\vspace{-5pt}
\label{fig:pappus-to-desargues}
\end{figure}

We refer the reader to Figure~\ref{fig:pappus-to-desargues}. 
Let $A_1B_1C_1$ and $A_2B_2C_2$ be triangles  perspective from a point~$P$. 
As before (cf.\ Theorem~\ref{th:desargues} and \eqref{eq:a1b1c1}--\eqref{eq:a2b2c2}),
denote 
\begin{align*}
a_i=(B_iC_i), \quad b_i=(A_iC_i), \quad c_i=(A_iB_i), \\
A=a_1\cap a_2, \quad B=b_1\cap b_2, \quad C=c_1\cap c_2. 
\end{align*}
Our goal is to prove that the points $A,B,C$ are collinear. 

We set 
\begin{equation*}
S=c_2\cap(A_1C_2), \quad 
M=c_1\cap a_2, \quad 
X=b_1 \cap (A_2M),\quad 
Y=(A_2M)\cap(C_1C_2).
\end{equation*} 
We begin by applying the tiling version of the Pappus theorem (more specifically, Theorem~\ref{th:pappus-dual})
to the following tiling of the torus: 
\begin{equation*}
\begin{tikzpicture}[baseline= (a).base]
\node[scale=.9] (a) at (0,0){
\begin{tikzcd}[arrows={-stealth}, cramped, sep=15]
& & & & a_2 \edge{rrd} \edge{lld}  \\[-19pt]
& & P\edge{lld} \edge{dd} &&& & S \edge{rrd} \edge{dd}  \\[-12pt]
c_1 \edge{d} \edge{dr} &&&&&&&&  c_1 \edge{d} \\
Y \edge{d}  & A_2 \edge{r} \edge{dl} & (A_1B_2) \edge{dd} &&&& (B_1C_2) \edge{rr} \edge{dd} & & Y \edge{d} \\
a_2 \edge{rrd} &&&&&&&&  a_2 \edge{lld} \\[-19pt]
& & S \edge{rrd} &&&& P \edge{lld} \\[-15pt]
& & & & c_1 
\end{tikzcd}
};
\end{tikzpicture}
\end{equation*}
It is straightforward to check that all tiles except for the one in the upper-right~corner are coherent.
(For the octagonal tile, use Proposition~\ref{pr:octagon-special-dual}.)
It follows that the latter tile is coherent as well. 
Again invoking Theorem~\ref{th:pappus-dual} and Proposition~\ref{pr:octagon-special-dual}
for the tiling 
\begin{equation*}
\begin{tikzpicture}[baseline= (a).base]
\node[scale=.9] (a) at (0,0){
\begin{tikzcd}[arrows={-stealth}, cramped, sep=15]
& & & & a_2 \edge{rrd} \edge{lld}  \\[-19pt]
& & C_1 \edge{lld} \edge{dd} &&& & S \edge{rrd} \edge{dd}  \\[-12pt]
c_1 \edge{d} \edge{dr} &&&&&&&&  c_1 \edge{d} \\
Y \edge{d}  & X \edge{r} \edge{dl} & (AA_1) \edge{dd} &&&& (B_1C_2) \edge{rr} \edge{dd} & & Y \edge{d} \\
a_2 \edge{rrd} &&&&&&&&  a_2 \edge{lld} \\[-19pt]
& & S \edge{rrd} &&&& C_1 \edge{lld} \\[-15pt]
& & & & c_1 
\end{tikzcd}
};
\end{tikzpicture}
\end{equation*}
we conclude that the tile with vertices $a_2, X, (AA_1), S$ is coherent. 
We then verify that all five quadrilateral tiles appearing in the tiling 
\begin{equation*}
\begin{tikzpicture}[baseline= (a).base]
\node[scale=.9] (a) at (0,0){
\begin{tikzcd}[arrows={-stealth}, cramped, sep=15]
& & & & a_2 \edge{rrd} \edge{lld}  \\[-19pt]
& & B \edge{lld} \edge{dd} &&& & S \edge{rrd} \edge{dd}  \\[-12pt]
c_1 \edge{d} \edge{dr} &&&&&&&&  c_1 \edge{d} \\
A_2 \edge{d}  & X \edge{r} \edge{dl} & (AA_1) \edge{dd} &&&& (CC_2) \edge{rr} \edge{dd} & & A_2 \edge{d} \\
a_2 \edge{rrd} &&&&&&&&  a_2 \edge{lld} \\[-19pt]
& & S \edge{rrd} &&&& B \edge{lld} \\[-15pt]
& & & & c_1 
\end{tikzcd}
};
\end{tikzpicture}
\end{equation*}
are coherent, so the octagonal tile is coherent as well. 
Applying Proposition~\ref{pr:octagon-special-dual} to the latter tile,
we conclude that the points $A, B, C$ are collinear. 
\qed

\newpage

\section{Geometric ramifications of the master theorem
}
\label{sec:geometric-ramifications}

In this section, we discuss various geometric interpretations and corollaries of our master theorem. 

\section*{The polygon trick}

On a plane, the vertices of a polygon encode as much information as its sides.
This~simple observation can be used to reformulate incidence theorems: 

\begin{remark}
\label{rem:polygon-trick}
Any theorem stated in terms of $m$~points $P_1,\dots,P_m$
and $n\ge 3$ generic lines $\ell_1,\dots,\ell_n$ in the plane can be restated entirely in terms of points 
by replacing the lines $\ell_1,\dots,\ell_n$ by the points $Q_1,\dots, Q_n$ defined by
$Q_i=\ell_i\cap\ell_{i+1}$, where the subscripts are taken modulo~$n$. 
Put differently, we set $\ell_i=(Q_{i-1}Q_i)$---again, with subscripts viewed modulo~$n$. 

To illustrate, Figure~\ref{fig:pappus1} yields an incidence theorem stated in terms of the points
$P_1,\dots,P_6$ and the lines $a,b,c$. 
The above observation allows us to restate this theorem in terms of the nine points 
$P_1,\dots,P_6, A=c\cap b, B=a\cap c, C=a\cap b$, resulting in the traditional version of the Pappus theorem. 

For $n\ge4$, the above reformulation depends on the choice of a cyclic ordering among the lines
$\ell_1,\dots, \ell_n$ (more precisely, cyclic ordering viewed up to global reversal). 

For $n\ge6$, we may alternatively split the lines into collections of size~$\ge3$ and apply the above 
substitution procedure to each of these collections. 
\end{remark}

The construction in Remark~\ref{rem:polygon-trick} straightforwardly extends to higher dimensions.
In particular, in 3-space we get the following: 

\begin{remark}
\label{rem:polygon-trick-3D}
Any theorem stated in terms of $m$~points $P_1,\dots,P_m$
and $n\ge 3$ generic planes $h_1,\dots,h_n$ in 3-space can be restated entirely in terms of points 
by replacing the planes $h_1,\dots,h_n$ by the points $Q_1,\dots, Q_n$ defined by
$Q_i=h_i\cap h_{i+1}\cap h_{i+2}$, where the subscripts are taken modulo~$n$. 
Equivalently, set $h_i=(Q_{i-2}Q_{i-1}Q_i))$---again, with subscripts viewed modulo~$n$.

To illustrate, the tiling in Figure~\ref{fig:cube-thm-tiling} yields an incidence theorem stated in terms of four points
and four planes. 
Replacing the four planes by four points defined as above yields 
the cube theorem of M\"obius (Theorem~\ref{th:cube}). 

As in Remark~\ref{rem:polygon-trick}, the construction depends on the choice of a cyclic ordering (indexing)
of the $n$ given planes~$h_1,\dots,h_n$.
It can also be generalized by utilizing a grouping of these $n$ planes into $n$~overlapping blocks of size~$3$ 
such that each plane appears in exactly $3$~blocks. 
The points $Q_i$ are then defined by taking intersections of the planes that make up each block. 
To illustrate, we can use the Fano block design to replace $7$~generic planes $h_1,\dots,h_7$ by
$7$ generic points $Q_1,\dots,Q_7$ as follows:
\begin{equation*}
\begin{array}{ll}
Q_{i} = h_i \cap h_{i-1} \cap h_{i-3}\quad &\text{($i=1,\dots,7$, all indices viewed $\bmod 7$);}\\[5pt]
h_{i} = (Q_i Q_{i+1} Q_{i+3})          \quad &\text{($i=1,\dots,7$, all indices viewed $\bmod 7$).}
\end{array}
\end{equation*}
\end{remark}

\medskip

We note that in the particular case when $n=\dim\PP+1$, the reformulation described in 
Remarks~\ref{rem:polygon-trick}--\ref{rem:polygon-trick-3D} yields  
D.~G.~Glynn's construction~\cite{glynn} discussed in Remark~\ref{rem:glynn}.

\newpage

\section*{Triangulated surfaces and 2D incidence theorems}

Returning to the construction of Definition~\ref{def:graphs-on-surfaces} 
and assuming that all faces of an embedded graph are triangles (cf.\ Remark~\ref{rem:making-vertices-trivalent}),
we arrive at the following corollary of our master theorem. 

\begin{corollary}
\label{cor:two-triangulations}
Let $T$ be a triangulation of a closed oriented surface.
For each vertex $v$ in~$T$, choose a different point~$P_v$ on the real/complex plane. 
For each edge~$u\shortedge{e} v$ in~$T$, choose a point $P_e$ on the line~$(P_u P_v)$.
Assume that all these points are distinct. 
For each triangle in~$T$ with sides $a, b, c$, consider the condition
\begin{equation}
\label{eq:PaPbPc-collinear}
\text{the points $P_a$, $P_b$, and $P_c$ are collinear.}
\end{equation}
If \eqref{eq:PaPbPc-collinear} holds for all triangles in $T$ except one, then it holds for the remaining triangle.
\end{corollary}

Figure~\ref{fig:2triangles} illustrates condition~\eqref{eq:PaPbPc-collinear} for two adjacent triangles in a triangulation.

\begin{figure}[ht]
\begin{center}
\includegraphics[scale=0.5, trim=0.2cm 0.6cm 0.2cm 0.6cm, clip]{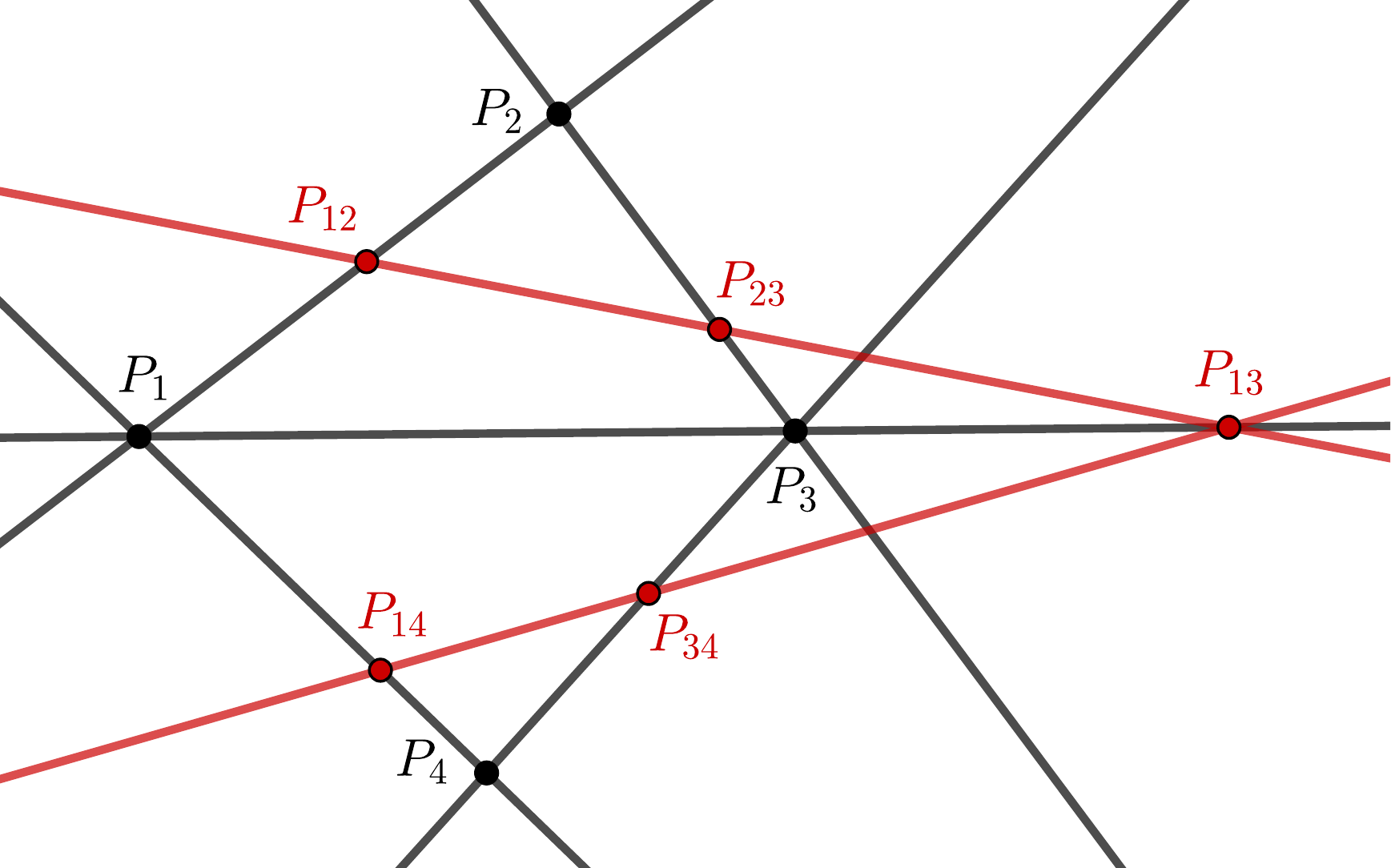}
\end{center}
\vspace{-.1in}
\caption{Condition~\eqref{eq:PaPbPc-collinear} for two adjacent triangles. 
The points $P_{ij}, P_{ik}, P_{jk}$ chosen on the sides of each triangle must be collinear. 
}
\vspace{-.2in}
\label{fig:2triangles}
\end{figure}

\begin{proof}
Construct a tiling of the surface as explained in the proof of Proposition~\ref{pr:graphs-on-surfaces},
but with points and lines interchanged: 
\begin{itemize}[leftmargin=.2in]
\item 
color the vertices of~$T$ black; 
\item
place a white vertex inside each triangle of~$T$;
\item 
connect this vertex to the three vertices of the triangle;
\item
remove the edges of $T$ from the resulting graph. 
\end{itemize}
The tiles $\square_e$ of the resulting tiling are in bijection with the edges~$e$ of the original triangulation~$T$.
We then label the vertices of the tiling:
\begin{itemize}[leftmargin=.2in]
\item 
to each black vertex~$v$ of the tiling, we associate the point~$P_v$; 
\item
to each white vertex of the tiling corresponding to a triangle in~$T$ satisfying~\eqref{eq:PaPbPc-collinear},
we assign the line passing through the points $P_a$, $P_b$, and $P_c$; 
\item
to the remaining triangle, say with sides $a_\circ, b_\circ, c_\circ$, we assign the line~$(P_{a_\circ}P_{b_\circ})$. 
\end{itemize}
For each edge $u\shortedge{e} v$ in~$T$ different from~$c_\circ$, 
the coherence of the corresponding tile~$\square_e$ 
follows from the above definitions: 
the lines associated to the two white vertices of the tile intersect at the point~$P_e$,
which lies on the line $(P_uP_v)$. 
The master theorem implies that the remaining tile is coherent, implying that the line~$(P_{a_\circ}P_{b_\circ})$
passes through~$P_{c_\circ}$, as desired. 
\end{proof}

\begin{example}[\emph{Desargues' Theorem}]
\label{eg:triangulation-desargues}
Triangulate the sphere into four triangles, as in a tetrahedron.
Corresponding to the four vertices of this triangulation, 
pick four points $P_1, P_2, P_3, P_4$ on the plane. 
Corresponding to the six edges of the triangulation, 
pick six points $P_{ij}$, each lying on its respective line $(P_iP_j)$. 
Corresponding to the four faces of the triangulation, 
we have four instances of condition~\eqref{eq:PaPbPc-collinear}, shown as red lines in  Figure~\ref{fig:desargues-triangulation}. 
If three of these instances hold, then the forth one holds as well.
This statement is a reformulation of the Desargues theorem. 
\end{example}

\newcommand{\rowsep}{-9pt}

\begin{figure}[ht]
\begin{equation*}
\hspace{5pt}\begin{array}{cc}
\hspace{-20pt}\begin{array}{c}
\ \\[-130pt]
\begin{tikzpicture}[baseline= (a).base]
\node[scale=0.8] (a) at (0,0){
\begin{tikzcd}[arrows={-stealth}, cramped, sep=15]
&& P_2 \edge{ddl} \edge{d} \edge{ddr} \\[15pt]
& & P_{12}\edge{dd} \\[\rowsep]
& P_{24} \edge{dddl} &  & P_{23} \edge{dddr} \\[\rowsep]
& & P_1 \edge{dl} \edge{dr} &  \\[\rowsep]
& P_{14} \edge{dl} && P_{13} \edge{dr} \\[\rowsep]
P_4 \edge{rr} && P_{34} \edge{rr} && P_3
\end{tikzcd}
};
\end{tikzpicture}
\end{array}
\hspace{10pt}
&\includegraphics[scale=0.45, trim=0.2cm 0.6cm 0.2cm 0.6cm, clip]{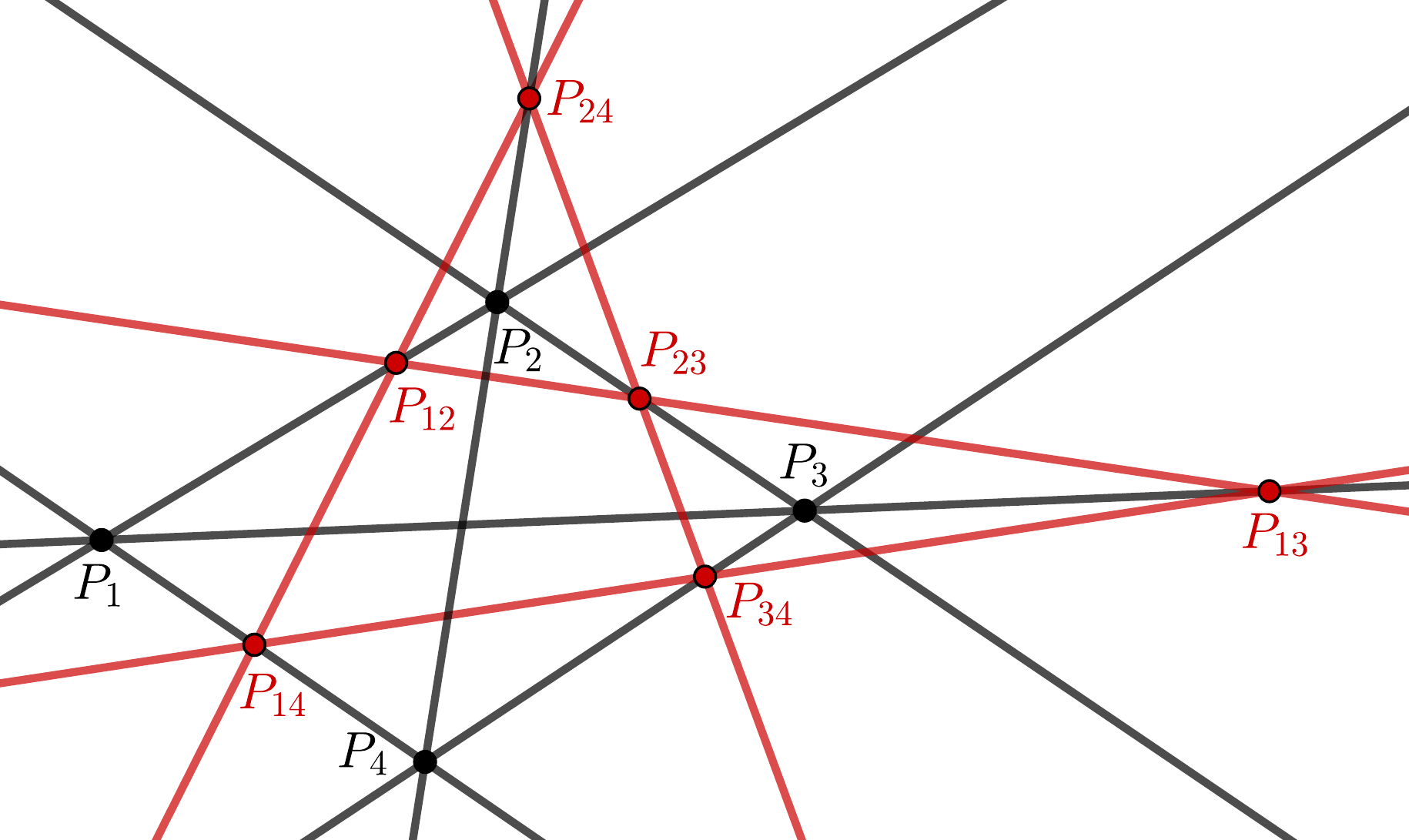}
\end{array}
\end{equation*}
\vspace{-.2in}
\caption{Obtaining the Desargues configuration from a triangulation of the sphere. 
}
\vspace{-.2in}
\label{fig:desargues-triangulation}
\end{figure}
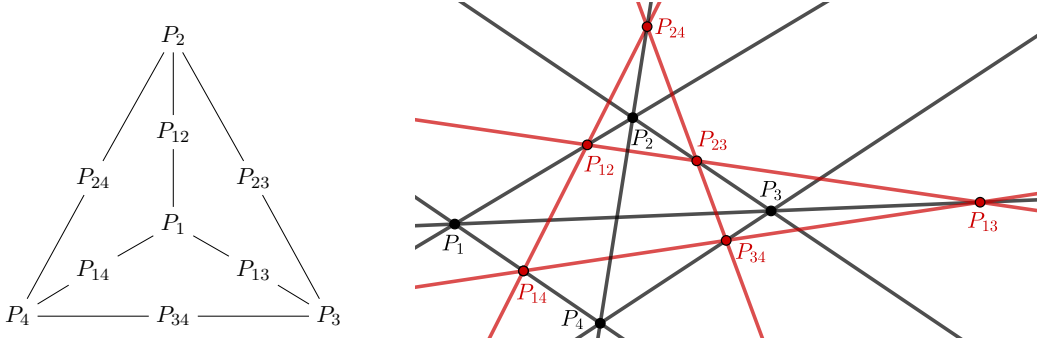

\begin{example}[\emph{Pappus' Theorem}]
\label{eg:triangulation-pappus}
Triangulate the torus into six triangles as shown in Figure~\ref{fig:pappus-triangulation}.
Corresponding to the three vertices of this triangulation, 
pick three~points $A, B, C$ on the plane. 
Draw the lines $(AB), (AC), (BC)$. 
Corresponding to the nine edges of the triangulation, 
place nine points $A_i\in (BC), B_i\in(AC), C_i\in(AB)$. 
Corresponding to the six faces of the triangulation, 
we have six instances of condition~\eqref{eq:PaPbPc-collinear}, shown as red lines in  Figure~\ref{fig:pappus-triangulation}. 
If five of these instances hold, then the sixth one holds as well.
This statement is a reformulation of the Pappus theorem. 
\end{example}

\begin{figure}[ht]
\vspace{-.1in}
\begin{equation*}
\hspace{5pt}\begin{array}{cc}
\hspace{-20pt}\begin{array}{c}
\ \\[-10pt]
\begin{tikzpicture}[baseline= (a).base]
\node[scale=0.8] (a) at (0,0){
\begin{tikzcd}[arrows={-stealth}, cramped, sep=1]
&& A \edge{rr} \edge{ld} \edge{rd}&& C_1  \edge{rr}&& B \edge{rd}\edge{ld} \\[15pt]
& C_2 \edge{ld}  && B_3 \edge{rd} && A_2  \edge{ld} && C_3  \edge{rd}\\[15pt]
B \edge{rd}\edge{rr} && A_1 \edge{rr}  && C \edge{rr} \edge{rd} \edge{ld}&& B_1\edge{rr} & & A\edge{ld} \\[15pt]
& C_3  \edge{rd} && B_2 \edge{ld} && A_3  \edge{rd} && C_2 \edge{ld} \\[15pt]
&& A \edge{rr} && C_1 \edge{rr} && B 
\end{tikzcd}
};
\end{tikzpicture}
\end{array}
\hspace{20pt}
&
\begin{array}{c}
\ \\[-5pt]
\includegraphics[scale=0.5, trim=0.45cm 0.5cm 0.5cm 0.5cm, clip]{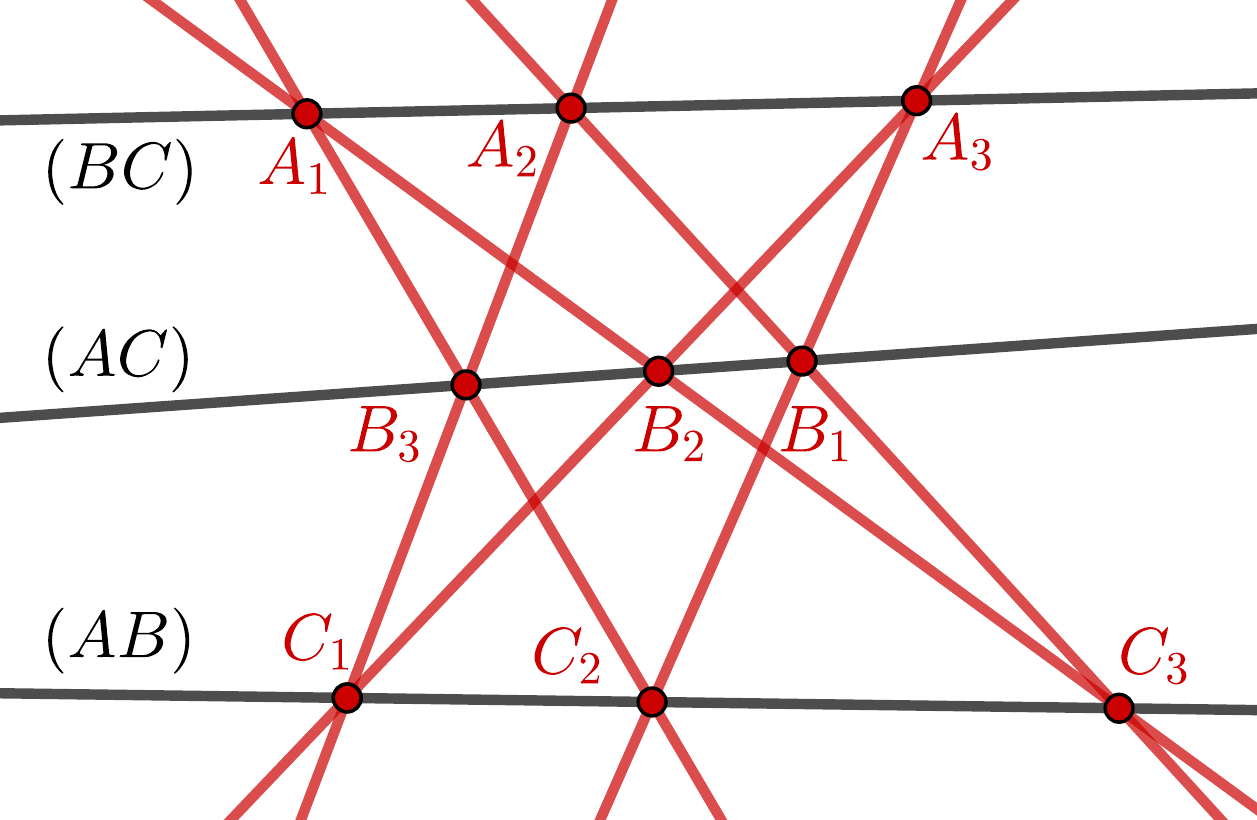}
\end{array}
\end{array}
\end{equation*}
\vspace{-10pt}
\caption{Obtaining the Pappus configuration from a triangulation of the torus.
The opposite sides of the hexagonal fundamental domain should be glued to each other. 
}
\vspace{-.2in}
\label{fig:pappus-triangulation}
\end{figure}
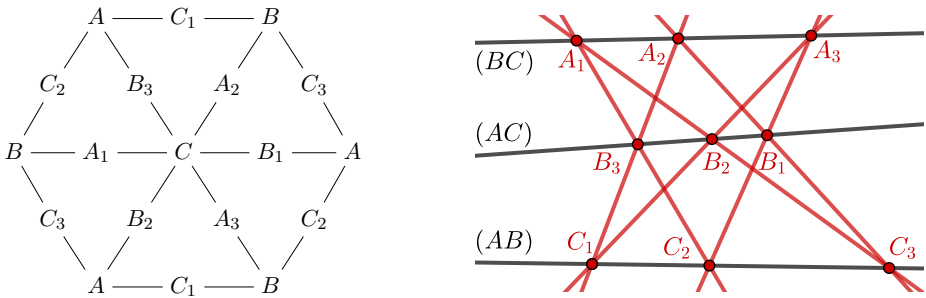

\begin{remark}
\label{rem:polygonal-subdivisions}
Corollary~\ref{cor:two-triangulations} generalizes from triangulations to arbitrary polygonal subdivisions 
of a surface. (The proof remains essentially the same.)  
On the other hand, the resulting theorem 
follows from Corollary~\ref{cor:two-triangulations} by considering a triangulation obtained by subdividing
the polygonal faces into triangles. 
(Here we assume that the conclusion of the theorem refers to a triangular face.)
\end{remark}

\newpage

\section*{Hamiltonian cycles and closure theorems}

Most theorems of classical planar linear incidence geometry 
can be presented in the form of a \emph{closure theorem} (Schlie\ss ungs\-satz),
an assertion that a particular straightedge-only construction closes up, 
producing a periodic sequence of points. (Cf.\ Remark~\ref{rem:sequential-construction}.) 
We next describe a combinatorial technique that can be used~to interpret individual instances of 
Corollary~\ref{cor:two-triangulations} as closure theorems. 
This technique applies whenever the dual graph of a triangulation 
satisfies the relatively mild combinatorial condition of \emph{Hamiltonicity}.
(For simplicial polyhedra, this condition features in a famous conjecture of P.~G.~Tait,
disproved by W.~Tutte, see, e.g., \cite[Section~18.2]{bondy-murty}.) 

\begin{theorem}
\label{th:hamiltonicity}
Let $T$ be a triangulation of a closed oriented surface.
Assume that each pair of triangles in~$T$ have at most one side in common. 
Let $G$~be the dual~tri\-valent graph of~$T$. 
Suppose that $G$ has a Hamiltonian cycle. Let ${\Delta_1,\dots,\Delta_N}$ 
be the corresponding sequence of triangles in~$T$. 
In other words, consecutive triangles $\Delta_i$~and $\Delta_{i+1}$ (all our indexing is modulo~$N$) 
are adjacent for every~$i$, 
and each triangle in~$T$ appears exactly once in the sequence~$(\Delta_i)$. 
Then the following closure theorem holds: 

For each vertex $v$ in~$T$, choose a different point~$P_v$ on the real/complex plane. 
For each edge~$u\shortedge{e} v$ in~$T$ shared by two consecutive triangles $\Delta_i$ and~$\Delta_{i+1}$, 
denote $\ell_i=(P_u P_v)$.
For each of the remaining edges~$u\shortedge{e} v$ in~$T$, 
choose a point $P_e\in (P_u P_v)$.

Pick a point $Q_1\in \ell_1$. 
Assume that all the chosen points are distinct. 
Recursively define the points $Q_i\in \ell_i$, $i\ge 2$, 
by $Q_i=\ell_i\cap (Q_{i-1}P_e)$, where $e$ is the unique side of the triangle~$\Delta_i$ not shared with 
the triangles~$\Delta_{i\pm1}$. 
Then $Q_{N+1}=Q_1$. 
\end{theorem}

\begin{example}
In the case of Desargues' theorem (cf.\ Example~\ref{eg:triangulation-desargues}),
one choice of~a Hamiltonian cycle is shown in Figure~\ref{fig:desargues-pappus-hamiltonian} on the left.
Then Theorem~\ref{th:hamiltonicity} asserts the following.
Pick generic points $P_1, P_2, P_3, P_4$ on the plane. 
Set $\ell_1\!=\!(P_1P_3)$, $\ell_2\!=\!(P_1P_4)$, $\ell_3\!=\!(P_2P_4)$, $\ell_4\!=\!(P_2P_3)$. 
Choose $P_{12}\in (P_1P_2)$ and $P_{34}\in(P_3P_4)$. 
Pick $P_{13}\in \ell_1$. 
Set $P_{14}\!=\!\ell_2\cap (P_{13}P_{34})$, $P_{24}\!=\!\ell_3\cap(P_{12}P_{14})$, $P_{23}\!=\!\ell_4\cap(P_{24}P_{34})$. 
Then $P_{13}\!=\!\ell_1\cap (P_{12}P_{23})$. 
\end{example}

\begin{example}
For the Pappus theorem (cf.\ Example~\ref{eg:triangulation-pappus}),
we can choose a Hamiltonian cycle as shown in Figure~\ref{fig:desargues-pappus-hamiltonian} on the right. 
Set ${Q_1\!=\!A_1}$, 
so that ${\ell_1\!=\!(BC)}$, ${\ell_2\!=\!(AC)}$, $\ell_3\!=\!(AB)$, $\ell_4\!=\!(BC)$, $\ell_5\!=\!(AC)$, $\ell_6\!=\!(AB)$. 
Theorem~\ref{th:hamiltonicity} asserts~that, given $A_2\in(BC)$, $B_2\in(AC)$, $C_2\in(AC)$, 
the recursively defined sequence $A_1$, $B_3$, $C_1$, $A_3$, $B_1$, $C_3$,\dots (cf.\ Figure~\ref{fig:pappus-triangulation}) 
is 6-periodic. The details are left to the reader. 
\end{example}

\enlargethispage{.7cm}

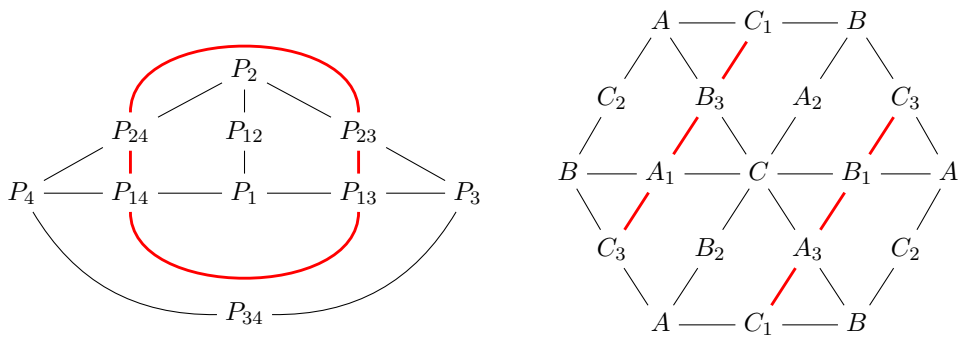
\begin{figure}[ht]
\vspace{-10pt}
\begin{equation*}
\begin{tikzpicture}[baseline= (a).base]
\node[scale=0.9] (a) at (0,0){
\begin{tikzcd}[arrows={-stealth}, cramped, sep=25]
&& P_2 \edge{dl} \edge{d} \edge{dr} \\[-15pt]
& P_{24} \edge{dl} \edge{d, color=red, very thick} \edge{rr, color=red, very thick, bend left=90} 
                & P_{12}\edge{d} & P_{23} \edge{d, color=red, very thick} \edge{dr} \\[-15pt]
P_4 \edge{r} & P_{14} \edge{r} \edge{rr, color=red, very thick, bend right=90} & P_1 \edge{r} & P_{13} \edge{r} &P_3 \\[10pt]
&& P_{34} \edge{rru, bend right=30} \edge{llu, bend left=30} 
\end{tikzcd}
};
\end{tikzpicture}
\qquad
\begin{tikzpicture}[baseline= (a).base]
\node[scale=0.9] (a) at (0,0){
\begin{tikzcd}[arrows={-stealth}, cramped, sep=1]
&& A \edge{rr} \edge{ld} \edge{rd}&& C_1  \edge{rr}  \edge{dl, color=red, very thick} && B \edge{rd}\edge{ld} \\[15pt]
& C_2 \edge{ld}  && B_3 \edge{rd} \edge{dl, color=red, very thick} && A_2  \edge{ld} && C_3  \edge{rd} \edge{dl, color=red, very thick} \\[15pt]
B \edge{rd}\edge{rr} && A_1 \edge{rr} \edge{dl, color=red, very thick} && C \edge{rr} \edge{rd} \edge{ld}&& B_1\edge{rr} \edge{dl, color=red, very thick} & & A\edge{ld} \\[15pt]
& C_3  \edge{rd} && B_2 \edge{ld} && A_3  \edge{rd} \edge{dl, color=red, very thick} && C_2 \edge{ld} \\[15pt]
&& A \edge{rr} && C_1 \edge{rr} && B 
\end{tikzcd}
};
\end{tikzpicture}
\end{equation*}
\vspace{-15pt}
\caption{Left: A triangulation from Figure~\ref{fig:desargues-triangulation}
and a closed curve through it 
that corresponds to a Hamiltonian cycle in the dual trivalent graph. 
Right: a similar closed curve through the triangulation from Figure~\ref{fig:pappus-triangulation}. 
}
\vspace{-.2in}
\label{fig:desargues-pappus-hamiltonian}
\end{figure}

\newpage

\section*{Quadrangulated surfaces and 3D incidence theorems}

The construction of Corollary~\ref{cor:two-triangulations} extends to 3D, with essentially the same proof:

\begin{corollary}
\label{cor:quad-3D}
Let $T$ be a tiling of a closed oriented surface by quadrilateral tiles.
For~each vertex $v$ in~$T$, choose a different point~$P_v$ in the real/complex 3-space. 
For~each edge~$u\shortedge{e} v$ in~$T$, choose a point $P_e$ on the line~$(P_u P_v)$.
Assume that all these points are distinct. 
For each quadrilateral tile in~$T$ with sides $a, b, c, d$ (listed consecutively), consider the condition
\begin{equation}
\label{eq:PaPbPcPd-coplanar}
\text{the points $P_a$, $P_b$, $P_c$, and $P_d$ are coplanar.}
\end{equation}
If \eqref{eq:PaPbPcPd-coplanar} holds for all quadrilateral tiles in $T$ except one, then it holds for the remaining quadrilateral tile.
\end{corollary}

While the tiling in Corollary~\ref{cor:quad-3D}
is an object of the same kind as the tilings appearing in our master theorem,
the geometric constructions associated with these tilings differ between the two contexts.
In the master theorem, black vertices correspond~to points and white vertices to planes,
whereas in Corollary~\ref{cor:quad-3D}, all vertices are labeled by points 
and moreover we associate additional points to the edges of the tiling. 

\begin{example}[\emph{The M\"obius configuration}]
Tile the torus by four quadrilaterals as shown in Figure~\ref{fig:quad-moebius}.
Corresponding to the four vertices of this tiling, 
pick four~points $A, B, C, D$ in 3-space. 
Draw the lines $(AB), (BC), (CD), (DA)$. 
Corresponding to~the eight edges of the tiling, 
pick eight points $P_5, P_8\!\in\!(AB)$, $P_2, P_3\!\in\!(BC)$, $P_6, P_7\!\in\!(CD)$, $P_1, P_4\!\in\! (AD)$. 
Corresponding to the four faces of the tiling, 
we have four instances of condition~\eqref{eq:PaPbPcPd-coplanar}.
Now state everything in terms of $P_1,\dots,P_8$. 
For example, the existence of the point~$C$ means that the lines $(P_2P_3)\!=\!(BC)$ and $(P_6P_7)\!=\!(CD)$
intersect, i.e., the points $P_2, P_3, P_6,  P_7$ are coplanar. 
Adding four coplanarity conditions corresponding to the points $A, B, C, D$
to the four coplanarity conditions corresponding to the four instances of~\eqref{eq:PaPbPcPd-coplanar},
we recover the M\"obius configuration (Theorem~\ref{th:cube}). 
\end{example}

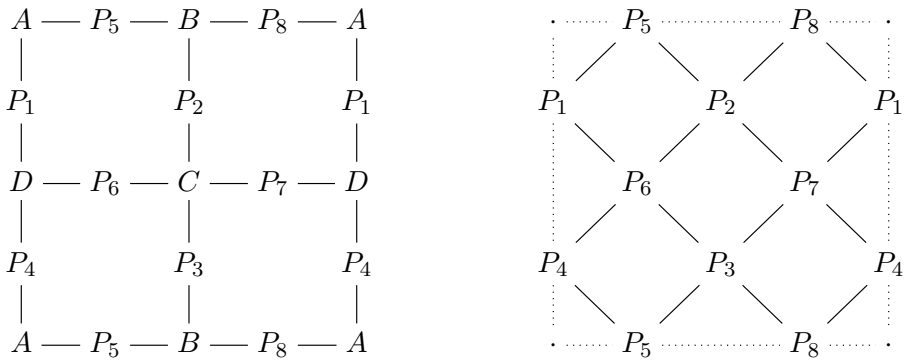
\begin{figure}[ht]
\begin{equation*}
\begin{tikzpicture}[baseline= (a).base]
\node[scale=1] (a) at (0,0){
\begin{tikzcd}[arrows={-stealth}, cramped, sep=small]
A \edge{r} \edge{d} & P_5 \edge{r} & B \edge{r} \edge{d} & P_8 \edge{r} & A \edge{d} \\[5pt]
P_1 \edge{d} && P_2 \edge{d} && P_1 \edge{d} \\[5pt]
D \edge{r} \edge{d} & P_6 \edge{r} & C \edge{r} \edge{d} & P_7 \edge{r} & D \edge{d} \\[5pt]
P_4 \edge{d} && P_3 \edge{d} && P_4 \edge{d} \\[5pt]
A \edge{r} & P_5 \edge{r} & B \edge{r} & P_8 \edge{r} & A
\end{tikzcd}
};
\end{tikzpicture}
\qquad\qquad\quad
\begin{tikzpicture}[baseline= (a).base]
\node[scale=1] (a) at (0,0){
\begin{tikzcd}[arrows={-stealth}, cramped, sep=small]
\cdot\edge{r, dotted} \edge{d, dotted} & P_5 \edge{dl} \edge{dr} \edge{rr, dotted} &  & P_8  \edge{dl} \edge{dr} \edge{r, dotted} & \cdot\edge{d, dotted} \\[5pt]
P_1 \edge{dd, dotted}  \edge{dr} && P_2  \edge{dl} \edge{dr} && P_1 \edge{dl} \edge{dd, dotted} \\[5pt]
 & P_6  \edge{dl} \edge{dr} & & P_7  \edge{dl} \edge{dr} &  \\[5pt]
P_4  \edge{d, dotted} \edge{dr} && P_3  \edge{dl} \edge{dr} && P_4  \edge{dl} \edge{d, dotted} \\[5pt]
\cdot \edge{r, dotted} & P_5 \edge{rr, dotted} &  & P_8\edge{r, dotted}   & \cdot
\end{tikzcd}
};
\end{tikzpicture}
\end{equation*}
\vspace{-10pt}
\caption{Left: Quadrilateral tiling of the torus that illustrates Corollary~\ref{cor:quad-3D}.
The opposite sides of the rectangular fundamental domain should be glued to each other. 
Right: the vertices of each of the eight tiles are coplanar, cf.~\eqref{eq:8-quadruples}. 
}
\vspace{-15pt}
\label{fig:quad-moebius}
\end{figure}

\begin{remark}
Corollary~\ref{cor:quad-3D} has a generalization to polygonal tilings, 
similar to the one discussed in Remark~\ref{rem:polygonal-subdivisions} in the context of 2D incidence geometry. 

\end{remark}
\newpage

\section*{Edges in polytopes}


\begin{corollary}
\label{cor:edge-crosses-edge}
Let $\mathbf{P}$ be a convex polytope in~$\RR^3$ with vertices $V_1,\dots,V_n$. 
Pick $n$ distinct points $V_1',\dots,V_n'\in\RR^3$ not lying on the planes containing the faces of~$\mathbf{P}$. 
For~each edge $\{V_i,V_j\}$ of~$\mathbf{P}$, consider the condition
\begin{equation}
\label{eq:edge-crosses-edge}
\text{the lines $(V_iV_j)$ and $(V_i'V_j')$ intersect.} 
\end{equation}
If this condition is satisfied for all edges of~$\mathbf{P}$ except one,
then it is also satisfied for the remaining edge. 
\end{corollary}

\begin{proof}
Once again, we adapt the construction from (the proof of) Proposition~\ref{pr:graphs-on-surfaces}. 
Viewing $\mathbf{P}$ as a PL-manifold, we tile it by quadrangular tiles as follows.
The vertices of the tiling include the vertices of~$\mathbf{P}$
plus an additional vertex $v_f$ for each two-dimensional face~$f$ of~$\mathbf{P}$. 
The edges of~$\mathbf{P}$ do \emph{not} include the edges of~$\mathbf{P}$.
Instead, we connect each vertex~$v_f$ with all the vertices $V_i$ lying on the boundary of the face~$f$. 
The tiles of the resulting tiling of the sphere correspond to the edges of~$\mathbf{P}$.

We now associate points and planes to the vertices of the tiling, as follows.
To a vertex of the tiling that comes from a vertex~$V_i$ of the polytope~$\mathbf{P}$, 
we associate the point~$V_i'$. 
To a vertex $v_f$ of the tiling that comes from a face~$f$ of~$\mathbf{P}$, 
we associate the plane containing the face~$f$.  
Then the coherence conditions associated to the tiles are precisely the conditions~\eqref{eq:edge-crosses-edge},
and the claim follows. 
\end{proof}

\begin{example}
Let $\mathbf{P}$ be a tetrahedron with vertices $P_1,P_2,P_3,P_4$. 
Corollary~\ref{cor:edge-crosses-edge} asserts that given four generic points $P_1',P_2',P_3',P_4'$,
five of the conditions 
\begin{equation*}
\text{the lines $(P_iP_j)$ and $(P_i'P_j')$ intersect}
\end{equation*}
imply the sixth. 
We thus recover Theorem~\ref{th:bundle}. 
\end{example}

Corollary~\ref{cor:edge-crosses-edge} straightforwardly generalizes to arbitrary oriented two-dimensional
PL-manifolds~$\mathbf{P}$ immersed in~$\RR^3$ 
so that each facet of~$\mathbf{P}$ is homeomorphic to a disk embedded into~$\RR^3$ 
as a flat polygon (not necessarily convex). 
Even more generally:

\begin{corollary}
Consider a tiling of a closed oriented surface by polygons.
Associate a point $P_v$ in 3-space to each vertex~$v$.
Associate a plane~$h_f$ to each polygonal face~$f$.
Associate two lines to each edge~$u\shortedge{e} v$ separating two faces $f$ and~$g$:
\begin{itemize}[leftmargin=.2in]
\item 
the line $(P_uP_v)$ connecting the vertices associated with the endpoints of~$e$; 
\item
the line $h_f\cap h_g$ where the planes associated with the faces separated~by~$e$ meet. 
\end{itemize}
If these two lines intersect for each edge on the surface except one,
then this condition is also satisfied for the remaining edge. 
\end{corollary}

\section*{More geometric reformulations}

A general correspondence principle established by S.~Tabachnikov~\cite{tabachnikov-skewers}
transforms any incidence theorem in the plane 
(in particular, any instance of our master theorem)
into a theorem about \emph{skewers}, i.e., common perpendiculars to lines in 3-space. 

In a similar spirit, \emph{stereographic projection} links incidence theorems
about circles in the plane and their counterparts involving circles on a sphere or planes in 3-space. \linebreak[3]
Thus, the cube theorem of M\"obius (Theorem~\ref{th:cube})
together with Hesse's associated points theorem yields Wallace's theorem, 
cf.\ \cite[Section~7]{glynn-2007} \cite[vol.~4, Section~1, 18--20]{baker}.

\newpage

\section{
Bobenko-Suris consistency of coherent tilings
}
\label{sec:consistency}

In this section, we discuss the phenomena of 3D and 4D consistency of coherent tilings. 
They can be viewed as analogues of the notions of 3D and 4D consistency of 2D discrete dynamical systems that 
play important roles in the context of the 
``integrability as consistency'' paradigm developed by 
A.~Bobenko and Yu.~Suris~\cite{bobenko-suris-book}. 

\section*{
3D consistency for points and lines in the plane}

\begin{proposition}
\label{pr:3D-consistency-Desargues}
On the real/complex projective plane, let $f, f_{12}, f_{13}, f_{23}$ be four points and let  $f_1,f_2,f_3$ be three lines. 
Assume that the three tiles 
\begin{equation}
\label{eq:f-fi-fij}
\begin{tikzpicture}[baseline= (a).base]
\node[scale=1] (a) at (0,0){
\begin{tikzcd}[arrows={-stealth}, sep=small, cramped]
& f_{12} \edge{ld} \edge{rd} \\[3pt]
f_1  \edge{rd}  && f_2 \edge{ld} \\[3pt]
& f 
\end{tikzcd}
};
\end{tikzpicture}
\qquad
\begin{tikzpicture}[baseline= (a).base]
\node[scale=1] (a) at (0,0){
\begin{tikzcd}[arrows={-stealth}, sep=small, cramped]
& f_{13} \edge{ld} \edge{rd} \\[3pt]
f_1  \edge{rd}  && f_3 \edge{ld} \\[3pt]
& f 
\end{tikzcd}
};
\end{tikzpicture}
\qquad
\begin{tikzpicture}[baseline= (a).base]
\node[scale=1] (a) at (0,0){
\begin{tikzcd}[arrows={-stealth}, sep=small, cramped]
& f_{23} \edge{ld} \edge{rd} \\[3pt]
f_2  \edge{rd}  && f_3 \edge{ld} \\[3pt]
& f 
\end{tikzcd}
};
\end{tikzpicture}
\end{equation}
are coherent and that the seven inputs $f, f_i, f_{ij}$ are generic subject to this condition. \pagebreak[3]
Then there exists a unique line $f_{123}$ such that the three tiles
\begin{equation}
\label{eq:fi-fij-fijk}
\begin{tikzpicture}[baseline= (a).base]
\node[scale=1] (a) at (0,0){
\begin{tikzcd}[arrows={-stealth}, sep=10, cramped]
& f_{123} \edge{ld} \edge{rd} \\[5pt]
f_{12}  \edge{rd}  && f_{13} \edge{ld} \\[5pt]
& f_1
\end{tikzcd}
};
\end{tikzpicture}
\qquad
\begin{tikzpicture}[baseline= (a).base]
\node[scale=1] (a) at (0,0){
\begin{tikzcd}[arrows={-stealth}, sep=10, cramped]
& f_{123} \edge{ld} \edge{rd} \\[5pt]
f_{12}  \edge{rd}  && f_{23} \edge{ld} \\[5pt]
& f_2
\end{tikzcd}
};
\end{tikzpicture}
\qquad
\begin{tikzpicture}[baseline= (a).base]
\node[scale=1] (a) at (0,0){
\begin{tikzcd}[arrows={-stealth}, sep=10, cramped]
& f_{123} \edge{ld} \edge{rd} \\[5pt]
f_{13}  \edge{rd}  && f_{23} \edge{ld} \\[5pt]
& f_3
\end{tikzcd}
};
\end{tikzpicture}
\end{equation}
are coherent. 
The same result holds with the words ``point'' and ``line'' interchanged. 
\end{proposition}

To rephrase, if five vertices of a cube are labeled 
by points and lines so that the three tiles that do not involve the unlabeled vertex are coherent,
then the sixth vertex can be labeled, in a unique way, 
so that the remaining three tiles are coherent as well. 
See Figure~\ref{fig:3D-consistency-Desargues}. 
 
\begin{figure}[ht]
\vspace{-3pt}
\begin{equation*}
\begin{tikzcd}[arrows={-stealth}, cramped, sep=small]
& &  f_{123} \arrow[lld, no head] \arrow[rrd, no head] \arrow[d, no head]  & & \\[10pt]
f_{23} \arrow[d, no head] \arrow[rrd, no head] && f_{13} \arrow[rrd, no head] \arrow[lld, no head]   
     && f_{12} \arrow[d, no head] \arrow[lld, no head] \\[10pt]
f_1 \arrow[rrd, no head] && f_2 \arrow[d, no head] && f_3\arrow[lld, no head] \\[10pt]
&& f
\end{tikzcd}
\end{equation*}
\vspace{-5pt}
\caption{3D consistency of coherent tilings for points and lines in the plane. 
}
\vspace{-10pt}
\label{fig:3D-consistency-Desargues}
\end{figure}
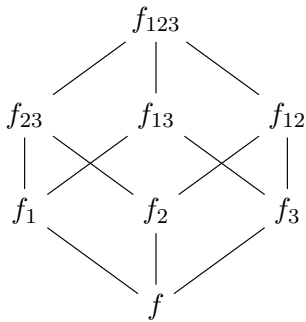

\vspace{-10pt}

\begin{proof}
This statement can be easily seen to be a restatement of Desargues' theorem, cf.\ Figure~\ref{fig:desargues}.
\end{proof}

\newpage

\section*{3D consistency for points and lines in 3-space}

\begin{definition}
\label{def:coherent-tile-3D}
Let $A,B$ (resp., $\ell,m$) be distinct points (resp., distinct lines) in the 3-dimensional real/complex projective space. 
We call the tile with vertices labeled ${A, \ell, B,m}$ (see~\eqref{eq:AlBm}) \emph{coherent} if 
\begin{itemize}[leftmargin=.2in]
\item 
neither $A$ nor $B$ is incident to either $\ell$ or~$m$; 
\item
the lines $\ell$ and $m$ intersect at some point~$P$; 
\item
the line $(AB)$ passes through~$P$. 
\end{itemize}
\end{definition}

\begin{remark}
The original coherence property introduced in Definition~\ref{def:coherent-tile} is a codimension~1 condition:
it imposes a single equation on the labels of a tile. 
By contrast, the coherence property in Definition~\ref{def:coherent-tile-3D} is a codimension~3 condition.
It can be shown that this condition is equivalent to requiring that for any point $O$ in 3-space,
the tile whose vertices are labeled by $A$, $B$, and the spans of $\{O,\ell\}$ and $\{O,m\}$ is coherent
in the original sense. 
It follows that if the vertices of a tiling of a closed oriented surface are labeled by points and lines in 3-space
so that all tiles but one are coherent (as in Definition~\ref{def:coherent-tile-3D}), then the remaining tile is coherent as well. 
\end{remark}

\begin{proposition}
\label{pr:3D-consistency-pts+lines-3D}
In the real/complex projective 3-space, let $f, f_{12}, f_{13}, f_{23}$ be four points in general position.
Let $f_1,f_2,f_3$ be three lines. 
Assume that 
the three tiles \eqref{eq:f-fi-fij} are coherent (cf.\ Definition~\ref{def:coherent-tile-3D})
and the seven inputs $f, f_i, f_{ij}$ are generic subject to the above conditions. 
(In particular, the three lines $f_1,f_2,f_3$ lie in one plane~$h$.) \linebreak[3]
Then~there exists a unique line $f_{123}$ such that the three tiles \eqref{eq:fi-fij-fijk} are coherent.  \linebreak[3]
Specifically, $f_{123}=h\cap (f_{12}f_{13}f_{23})$, the ``Desargues line''
for the perspective triangles \linebreak[3]
$f_{12}f_{13}f_{23}$ and $f_1f_2f_3$. 
\end{proposition}

\begin{proof}
%
For the tiles \eqref{eq:fi-fij-fijk} to be coherent, the line $f_{123}$ must intersect each of the lines
$f_1,f_2,f_3$, hence lie in the plane~$h$. 
These conditions also imply~that $f_{123}$ intersects each of the lines $(f_{12}f_{13}), (f_{12}f_{23}), (f_{13}f_{23})$,
hence lies in the plane~$(f_{12}f_{13}f_{23})$. 
Thus $f_{123}=h\cap (f_{12}f_{13}f_{23})$.

Conversely, let $f_{123}=h\cap (f_{12}f_{13}f_{23})$.
Set $f_{12}'=f_1\cap f_2$, $f_{13}'=f_1\cap f_3$, ${f_{23}'=f_2\cap f_3}$. 
The coherence of the tiles \eqref{eq:f-fi-fij} means that each triple $\{f, f_{ij}, f'_{ij}\}$ is collinear.
Thus the points $f_{12}, f_{12}', f_{13}, f_{13}'$ are coplanar, 
so the lines $f_1=(f_{12}'f_{13}')$ and $(f_{12}f_{13})$ intersect (necessarily at a point lying on~$h\cap (f_{12}f_{13}f_{23})=f_{123}$). 
We conclude that the first tile in \eqref{eq:fi-fij-fijk} is coherent, and so are the other two. 
\end{proof}


Interchanging points and lines in Proposition~\ref{pr:3D-consistency-pts+lines-3D}
requires a coplanarity assumption: 

\begin{proposition}
\label{pr:3D-consistency-pts+lines-3D-l}
In the real/complex projective 3-space, let $f, f_{12}, f_{13}, f_{23}$ be four lines and let $f_1,f_2,f_3$ be three points. 
Assume that the lines $f_{12},f_{13},f_{23}$ lie in one plane, 
the three tiles \eqref{eq:f-fi-fij} are coherent (cf.\ Definition~\ref{def:coherent-tile-3D}), 
and the seven inputs $f, f_i, f_{ij}$ are generic subject to the above conditions. \pagebreak[3]
Then there exists a unique point $f_{123}$ such that the three tiles \eqref{eq:fi-fij-fijk} are coherent. 
\end{proposition}

\begin{proof}
Set $f_1'=f_{12}\cap f_{13}$, $f_2'=f_{12}\cap f_{23}$, $f_3'=f_{13}\cap f_{23}$. 
The tiles \eqref{eq:f-fi-fij} are coherent, 
so each pair $(f_if_j), (f_i'f_j')$ is coplanar. 
Hence each pair $(f_if_i'), (f_jf_j')$ is coplanar. 
If all three lines $(f_if_i')$ are coplanar, then everything lies in a plane and the claim follows~by Desargues' theorem.
Otherwise, the three lines $(f_if_i')$ are concurrent and the point of their intersection is the unique point $f_{123}$ satisfying  
the coherence conditions~\eqref{eq:fi-fij-fijk}.
\end{proof}

\newpage

\section*{4D consistency for points and lines in 3-space}

\begin{theorem}
\label{th:4d-3D}
In the real/complex projective 3-space, 
let $f, f_{12}, f_{13}, f_{14}, f_{23}, f_{24}, f_{34}$ be seven points and let $f_1, f_2, f_3, f_4$ be four lines. 
Assume that the six tiles 
\begin{equation}
\label{eq:f-fi-fij-4}
\begin{tikzpicture}[baseline= (a).base]
\node[scale=1] (a) at (0,0){
\begin{tikzcd}[arrows={-stealth}, sep=6, cramped]
& f_{ij} \edge{ld} \edge{rd} \\[-2pt]
f_i  \edge{rd}  && f_j \edge{ld} \\[-2pt]
& f 
\end{tikzcd}
};
\end{tikzpicture}
\end{equation}
are coherent and the eleven inputs $f, f_i, f_{ij}$ are generic subject to this condition. 
For each triple $1\le i<j<k\le 4$, let $f_{ijk}$ be the unique line that makes all tiles in 
\begin{equation}
\label{eq:fi-fij-fijk-4}
\begin{tikzcd}[arrows={-stealth}, cramped, sep=7]
& &  f_{ijk} \arrow[lld, no head] \arrow[rrd, no head] \arrow[d, no head]  & & \\[5pt]
f_{jk} \arrow[d, no head] \arrow[rrd, no head] && f_{ik} \arrow[rrd, no head] \arrow[lld, no head]   
     && f_{ij} \arrow[d, no head] \arrow[lld, no head] \\[5pt]
f_i \arrow[rrd, no head] && f_j \arrow[d, no head] && f_k\arrow[lld, no head] \\[5pt]
&& f
\end{tikzcd}
\end{equation}
coherent, cf.\ Proposition~\ref{pr:3D-consistency-pts+lines-3D}. 
Then there exists a unique point~$f_{1234}$ that makes all tiles in Figure~\ref{fig:4D-consistency-tikz} coherent. 
\end{theorem}

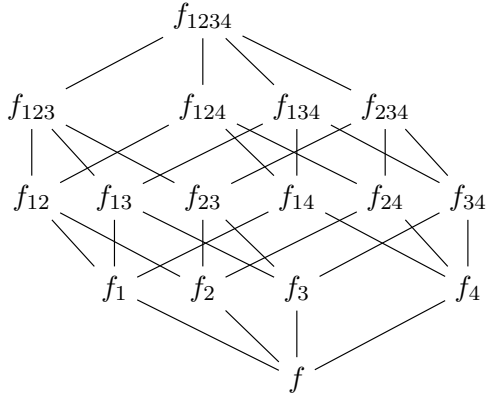
\begin{figure}[ht]
\vspace{-8pt}
\begin{equation*}
\begin{tikzcd}[arrows={-stealth}, cramped, sep=8]
&& f_{1234} \edge{lld} \edge{d} \edge{rd} \edge{rrd} 
\\[10pt]
f_{123} \edge{d} \edge{rd} \edge{rrd} & & f_{124} \edge{lld} \edge{rd} \edge{rrd} & f_{134} \edge{lld} \edge{d} \edge{rrd} & f_{234} \edge{lld} \edge{d} \edge{rd}  \\[10pt]
f_{12} \edge{rd} \edge{rrd} & f_{13} \edge{d} \edge{rrd} & f_{23} \edge{d} \edge{rd} & f_{14} \edge{lld} \edge{rrd} & f_{24} \edge{lld} \edge{rd} & f_{34} \edge{lld} \edge{d} \\[10pt]
& f_1 \edge{rrd} & f_2 \edge{rd} & f_3 \edge{d} && f_4 \edge{lld} \\[10pt]
&&& f
\end{tikzcd}
\end{equation*}
\vspace{-10pt}
\caption{4D consistency of coherent tilings. 
}
\vspace{-10pt}
\label{fig:4D-consistency-tikz}
\end{figure}

\vspace{-5pt}

\begin{proof}
We will use the notational conventions $f_{ij}=f_{ji}$ and $f_{ijk}=f_{jik}=f_{ikj}=\cdots$. 
The lines $(f f_{ij})$ and $f_i$ intersect (cf.~\eqref{eq:f-fi-fij-4}), so the point~$f_{ij}$ lies in~the plane~$(f f_i)$. 
Hence each~quadruple $\{f, f_{ij}, f_{ik}, f_{i\ell}\}$ is coplanar (all points lie in~$(f f_i)$),
so the four planes $(f_{ij}, f_{ik}, f_{i\ell})$ have a common point~$f$. 
By the octahedron theorem (or by the cube theorem), 
the four planes $(f_{ij}, f_{ik}, f_{jk})$ have a common point~$f_{1234}$. 

The coherence of~\eqref{eq:f-fi-fij-4}--\eqref{eq:fi-fij-fijk-4} and
Proposition~\ref{pr:3D-consistency-pts+lines-3D} imply that 
the four lines~$f_i$ lie in a common plane~$h$ and moreover 
$f_{ijk}=h\cap (f_{ij}f_{ik}f_{jk})$. 
Thus the four lines $f_{ijk}$ lie in~$h$ 
and we have $f_{ijk}\cap f_{ij\ell}=h\cap (f_{ij}f_{ik}f_{jk})\cap (f_{ij}f_{i\ell}f_{j\ell})=h\cap (f_{ij} f_{1234})$,
so the tile with labels $f_{ij}, f_{ijk}, f_{ij\ell}, f_{1234}$ is coherent. 

Uniqueness follows by arguing that $f_{1234}$ must lie on the planes~$(f_{ij}, f_{ik}, f_{jk})$. 
\end{proof}

\newpage

\section*{4D consistency for points and lines in the plane}

\begin{theorem}
\label{th:4d}
Theorem~\ref{th:4d-3D} holds verbatim for points and lines in the plane. 

More explicitly, let $f$ be a point on the real/complex plane and let $f_1, f_2, f_3, f_4$ be four generic lines. 
For each of the six instances of $1\le i<j\le 4$, pick a point $f_{ij}$
on the line passing through $f$ and $f_i\cap f_j$. 
For each triple $1\le i<j<k\le 4$, the triangles  $f_i f_j f_k$ and $f_{ij}, f_{ik}, f_{jk}$ are in perspective;
let~$f_{ijk}$ be their Desargues line. 
Then the six lines obtained by connecting $f_{ij}$ with $f_{ijk}\cap f_{ij\ell}$ intersect in a point. 
See Figure~\ref{fig:4D2D}. 
\end{theorem}

\begin{figure}[ht]
\begin{center}
\includegraphics[scale=0.54, trim=0cm 0cm 0cm 0.2cm, clip]{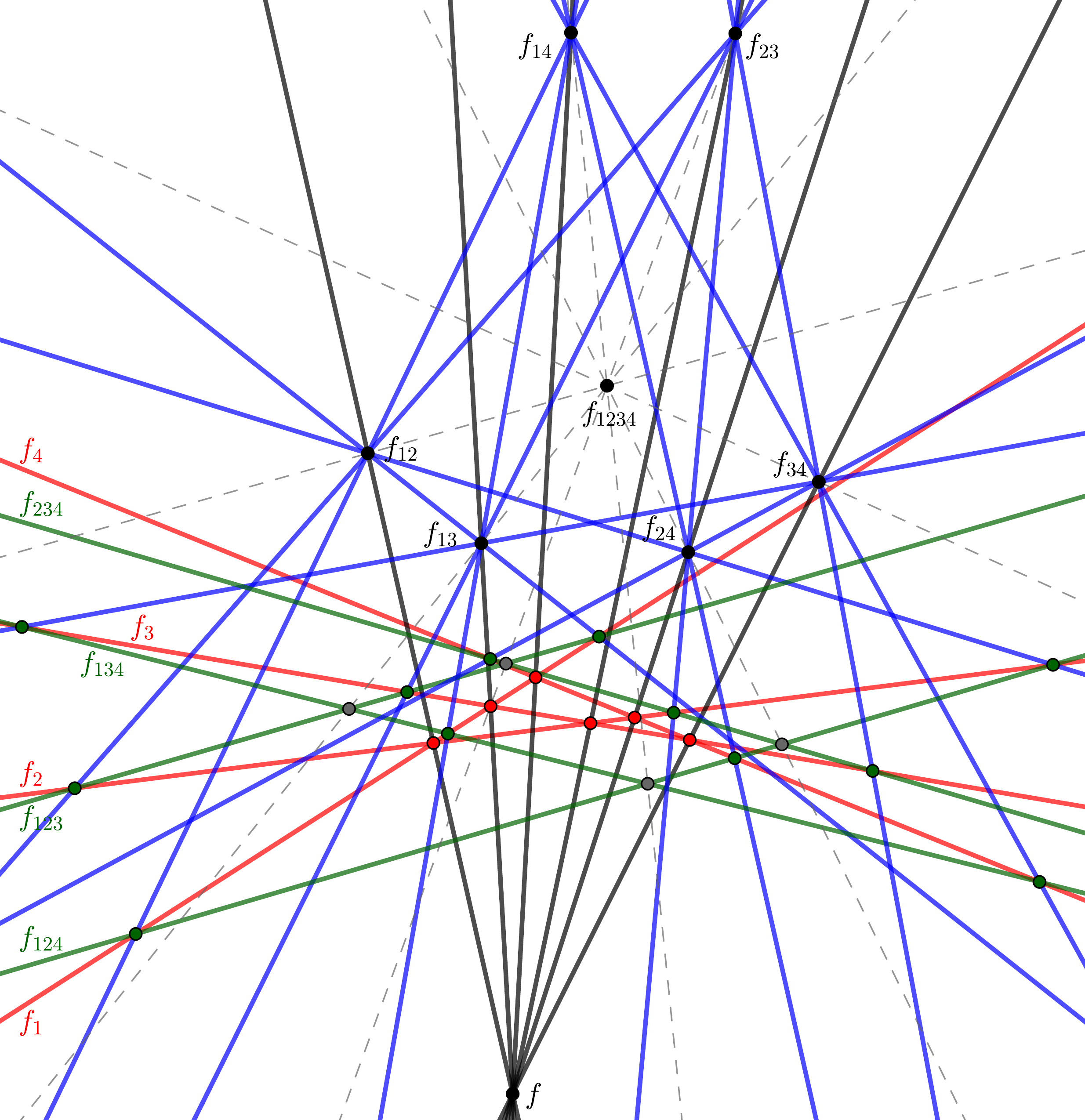}
\end{center}
\vspace{-.05in}
\caption{4D consistency of coherent tilings for points and lines in the plane. 
Two relevant points, namely $f_{123}\cap f_{124}$ and 
$f_{134}\cap f_{234}$, are outside the picture frame. 
}
\label{fig:4D2D}
\end{figure}

\vspace{-.2in}

\begin{proof}
Lift point $f$ out of the plane~$h$ containing the lines $f_1,f_2,f_3,f_4$ 
and reproduce the construction of Theorem~\ref{th:4d-3D}.
Continuously move~$f$ back into~$h$ and apply Proposition~\ref{pr:3D-consistency-Desargues}
and Theorem~\ref{th:4d-3D} to obtain the desired result.
\end{proof}

\newpage

\section*{4D consistency for points and planes in 3-space}

The phenomenon of 3D consistency, in a version analogous to Proposition~\ref{pr:3D-consistency-Desargues}
(cf.\ \cite[Definition~4.1]{bobenko-suris-book}), 
does not hold for points and planes in 3-space---not because we cannot complete the 3D cube,
but because there are infinitely many ways to do~it. 
Indeed, if a plane $f_{123}$ satisfies two of the three coherence conditions~\eqref{eq:fi-fij-fijk},
then the third condition is automatic by the bundle theorem (Theorem~\ref{th:bundle}). 
Since each condition requires the plane $f_{123}$ to pass through a certain point, 
the planes satisfying two such conditions form a one-dimensional pencil, so uniqueness of~$f_{123}$ fails. 

On the other hand, 4D consistency holds for points and planes in 3-space, 
even without relying on the coherence of the bottom tiles~\eqref{eq:f-fi-fij-4}: 

\begin{theorem}
\label{th:4d-3D-planes}
In the real/complex projective 3-space, 
let $f_{12}, f_{13}, f_{14}, f_{23}, f_{24}, f_{34}$ be 6~points  
and let $f_1, f_2, f_3, f_4, f_{123}, f_{124}, f_{134}, f_{234}$ be 8~planes such that the 12~tiles 
\begin{equation}
\label{eq:f-fi-fij-4-planes}
\begin{tikzpicture}[baseline= (a).base]
\node[scale=1] (a) at (0,0){
\begin{tikzcd}[arrows={-stealth}, sep=5, cramped]
& f_{ijk} \edge{ld} \edge{rd} \\[2pt]
f_{ij}  \edge{rd}  && f_{ik} \edge{ld} \\[2pt]
& f_i
\end{tikzcd}
};
\end{tikzpicture}
\end{equation}
are coherent and the 14 inputs $f_i, f_{ij}, f_{ijk}$ are generic subject to these conditions. 
(Here we use the conventions $f_{ij}=f_{ji}$ and $f_{ijk}=f_{jik}=f_{ikj}=\cdots$)
Then there exist unique points $f$ and~$f_{1234}$ that make all tiles in Figure~\ref{fig:4D-consistency-tikz} coherent. 
\end{theorem}

\begin{proof}
By symmetry, it suffices to show the existence and uniqueness of~$f_{1234}$. 
Let $h_{ij}$ be the plane through $f_{ij}$ and the line $f_{ijk}\cap f_{ij\ell}$. 
We need to show that the six planes $h_{ij}$ have a common point~$f_{1234}$. 
By the discussion preceding Theorem~\ref{th:4d-3D-planes}, each triple of planes $h_{ij}, h_{ik}, h_{i\ell}$ intersect along a line,
so $h_{ij}\cap h_{ik} = h_{ij}\cap h_{ik} \cap h_{i\ell}\,$. 
Therefore
\begin{equation*}
f_{1234}
\stackrel{\rm def}{=} h_{12}\cap h_{13}\cap h_{23}
=h_{12}\cap h_{13}\cap h_{14} \cap h_{23}
=\cdots = \bigcap h_{ij} \,. \qedhere
\end{equation*}
\end{proof}

\begin{figure}[ht]
\vspace{-5pt}
\begin{equation*}
\begin{tikzcd}[arrows={-stealth}, cramped, sep=5]
f_{34} \edge{dr} &&&&&& f_{34} \edge{dl} \\
& f_{134} \edge{rr} \edge{dd} && f_{13} \edge{dl} \edge{dr} \edge{rr} && f_3 \edge{dd} \\
&& f_1 \edge{dl} \edge{dr} && f_{123} \edge{dl} \edge{dr} \\
& f_{14} \edge{dd} \edge{dr} && f_{12} \edge{dl} \edge{dr} && f_{23} \edge{dl} \edge{dd} \\
&& f_{124}  \edge{dr} && f_2 \edge{dl} \\
& f_4 \edge{dl} \edge{rr} && f_{24} \edge{rr} && f_{234} \edge{dr} \\
f_{34} &&&&&& f_{34}
\end{tikzcd}
\end{equation*}
\vspace{-15pt}
\caption{The dependence among conditions~\eqref{eq:f-fi-fij-4-planes}. 
}
\vspace{-15pt}
\label{fig:4D-cube-trimmed}
\end{figure}
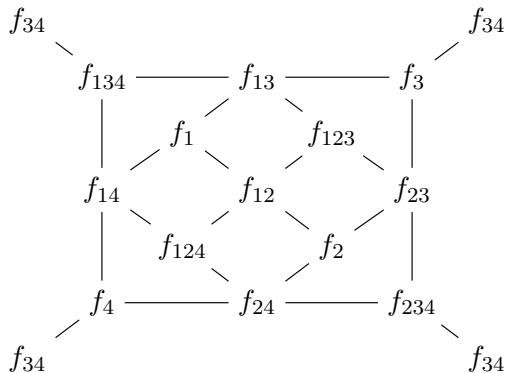

\begin{remark}
\label{rem:12-coherence-conditions}
The 12 coherence conditions~\eqref{eq:f-fi-fij-4-planes} in Theorem~\ref{th:4d-3D-planes} are not independent: 
any 11 among them imply the remaining one, cf.\ the tiling of the sphere in Figure~\ref{fig:4D-cube-trimmed}. 

Theorem~\ref{th:4d-3D-planes} can be viewed as a strengthening of a theorem of Cox 
\cite[Exercise~2.28]{bobenko-suris-book}. 

\end{remark}

\newpage

\section*{The octagon relation for Desargues flips}

\begin{definition}
A~\emph{flip} in a tiling 
is a local transformation that modifies three pairwise-adjacent tiles 
as shown 
in Figure~\ref{fig:flip-hexagon}. 
\end{definition}

\begin{figure}[ht]
\vspace{-10pt}
\begin{equation*}
\begin{array}{ccccc}
\begin{tikzpicture}[baseline= (a).base]
\node[scale=0.9] (a) at (0,0){
\begin{tikzcd}[arrows={-stealth}, cramped, outer sep=-2,sep=5]
&  \bullet \edge{ld} \edge{rd} \edge{dd}  & \\[-7pt]
\circ \edge{dd}  & & \circ \edge{dd}  \\[-7pt]
& \circ \edge{rd} \edge{ld}     \\[-7pt]
 \bullet \edge{rd}&&  \bullet\edge{ld} \\[-7pt]
& \circ
\end{tikzcd}
};
\end{tikzpicture}
& \longleftrightarrow&
\begin{tikzpicture}[baseline= (a).base]
\node[scale=0.9] (a) at (0,0){
\begin{tikzcd}[arrows={-stealth}, cramped, outer sep=-2, sep=5]
&   \bullet \edge{ld} \edge{rd}  & \\[-7pt]
\circ \edge{dd} \edge{rd} && \circ \edge{dd} \edge{ld} \\[-7pt]
&  \bullet \edge{dd} \\[-7pt]
 \bullet \edge{rd} &  &  \bullet\edge{ld} \\[-7pt]
& \circ
\end{tikzcd}
};
\end{tikzpicture}
\end{array}
\end{equation*}
\vspace{-15pt}
\caption{A flip in a tiling. (We can go either left-to-right or right-to-left.)
}
\vspace{-20pt}
\label{fig:flip-hexagon}
\end{figure}
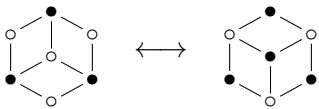

\begin{remark}
Flips of \emph{zonotopal tilings} play an important role in the combinatorics of symmetric groups, 
as they correspond to \emph{braid relations} between reduced words. 
A~classical result of J.~Tits~\cite[\S~4.3, Proposition~4]{tits-octagon} describes 
the syzygies among braid relations. 
To reformulate it in the language of tilings, fix a centrally symmetric convex polygon~$\mathbf{P}$ (a 2-dimensional zonotope) 
and consider the following graph~$G_\mathbf{P}$.
The~vertices of~$G_\mathbf{P}$ are the tilings of~$\mathbf{P}$ 
by 
parallelograms (zonotopal tiles). 
The edges of~$G_\mathbf{P}$ correspond to flips. 
Tits' theorem asserts that the cycles in~$G_\mathbf{P}$ are generated~by 
\begin{itemize}[leftmargin=.2in]
\item 
the 4-cycles that correspond to commuting braid moves, together with  
\item
the 8-cycles of the form shown in Figure~\ref{fig:octagon-of-flips}. 
\end{itemize}
These octagonal cycles arise in various fields of mathematics and theoretical physics 
in the context of the \emph{tetrahedron equation}
of A.~B.~Zamolodchikov~\cite{zamolodchikov-tetrahedra}, 
see, e.g., \cite{dimakis-muller-hoissen, kapranov-voevodsky, talalaev} and references therein. 
\end{remark}


\begin{figure}[ht]
\vspace{-10pt}
\begin{equation*}
\!\!
\begin{array}{ccccc}
\begin{tikzpicture}[baseline= (a).base]
\node[scale=0.8] (a) at (0,0){
\begin{tikzcd}[arrows={-stealth}, cramped, outer sep=-2, sep=5]
& \circ \edge{r} \edge{ld} \edge{rd} & \bullet \edge{rd} \\[0pt]
\bullet \edge{d} \edge{dr} & & \bullet \edge{r} \edge{d} \edge{ld}& \circ \edge{d} \\[0pt]
\circ \edge{rd} & \circ \edge{d} & \circ \edge{r} \edge{dl} & \bullet \edge{ld} \\[0pt]
& \bullet \edge{r} & \circ
\end{tikzcd}
};
\end{tikzpicture}
& \!\!\mutation{} \!\!& 
\begin{tikzpicture}[baseline= (a).base]
\node[scale=0.8] (a) at (0,0){
\begin{tikzcd}[arrows={-stealth}, cramped, outer sep=-2, sep=5]
& \circ \edge{r} \edge{ld} \edge{d} \edge{dr} & \bullet \edge{rd} \\[0pt]
\bullet \edge{d} & \bullet \edge{dl} \edge{dr} & \bullet \edge{r} \edge{d} & \circ \edge{d} \\[0pt]
\circ \edge{rd} & & \circ \edge{r} \edge{dl} & \bullet \edge{ld} \\[0pt]
& \bullet \edge{r} & \circ
\end{tikzcd}
};
\end{tikzpicture}
& \!\!\mutation{} \!\!& 
\begin{tikzpicture}[baseline= (a).base]
\node[scale=0.8] (a) at (0,0){
\begin{tikzcd}[arrows={-stealth}, cramped, outer sep=-2, sep=5]
& \circ \edge{r} \edge{ld} \edge{d} & \bullet \edge{rd} \edge{d} \\[0pt]
\bullet \edge{d} & \bullet \edge{r} \edge{dr} \edge{dl} & \circ \edge{dr} & \circ \edge{d} \\[0pt]
\circ \edge{rd} & & \circ \edge{r} \edge{dl} & \bullet \edge{ld} \\[0pt]
& \bullet \edge{r} & \circ
\end{tikzcd}
};
\end{tikzpicture}
\\[-8pt]
\\
\vmutation{}  &&&& \vmutation{} 
\\[-8pt]
\\
\begin{tikzpicture}[baseline= (a).base]
\node[scale=0.8] (a) at (0,0){
\begin{tikzcd}[arrows={-stealth}, cramped, outer sep=-2, sep=5]
& \circ \edge{r} \edge{ld} \edge{dr} & \bullet \edge{rd} \\[0pt]
\bullet \edge{d} \edge{rd} && \bullet \edge{r} \edge{dl} & \circ \edge{d} \edge{ld} \\[0pt]
\circ \edge{rd} & \circ \edge{r} \edge{d} & \bullet \edge{d} & \bullet \edge{ld} \\[0pt]
& \bullet \edge{r} & \circ
\end{tikzcd}
};
\end{tikzpicture}
&  & 
&  & 
\begin{tikzpicture}[baseline= (a).base]
\node[scale=0.8] (a) at (0,0){
\begin{tikzcd}[arrows={-stealth}, cramped, outer sep=-2, sep=5]
& \circ \edge{r} \edge{ld} \edge{d} & \bullet \edge{d} \edge{rd} \\[0pt]
\bullet \edge{d} & \bullet \edge{r} \edge{ld} & \circ \edge{dr} \edge{dl} & \circ \edge{d} \\[0pt]
\circ \edge{rd} \edge{r} & \bullet \edge{dr} & & \bullet \edge{ld} \\[0pt]
& \bullet \edge{r} & \circ
\end{tikzcd}
};
\end{tikzpicture}
\\[-8pt]
\\
\vmutation{}  &&&& \vmutation{}
\\[-8pt]
\\
\begin{tikzpicture}[baseline= (a).base]
\node[scale=0.8] (a) at (0,0){
\begin{tikzcd}[arrows={-stealth}, cramped, outer sep=-2, sep=5]
& \circ \edge{r} \edge{ld} & \bullet \edge{rd} \edge{ld} \\[0pt]
\bullet \edge{d} \edge{r} \edge{dr} & \circ \edge{rd} & & \circ \edge{d} \edge{dl} \\[0pt]
\circ \edge{rd} & \circ \edge{r} \edge{d} & \bullet \edge{d} & \bullet \edge{ld} \\[0pt]
& \bullet \edge{r} & \circ
\end{tikzcd}
};
\end{tikzpicture}
& \!\!\mutation{} \!\!& 
\begin{tikzpicture}[baseline= (a).base]
\node[scale=0.8] (a) at (0,0){
\begin{tikzcd}[arrows={-stealth}, cramped, outer sep=-2, sep=5]
& \circ \edge{r} \edge{ld} & \bullet \edge{rd} \edge{ld} \\[0pt]
\bullet \edge{d} \edge{r} & \circ \edge{d} \edge{dr} & & \circ \edge{d} \edge{dl} \\[0pt]
\circ \edge{rd} \edge{r} & \bullet \edge{dr} & \bullet \edge{d} & \bullet \edge{ld} \\[0pt]
& \bullet \edge{r} & \circ
\end{tikzcd}
};
\end{tikzpicture}
& \!\!\mutation{}\!\! & 
\begin{tikzpicture}[baseline= (a).base]
\node[scale=0.8] (a) at (0,0){
\begin{tikzcd}[arrows={-stealth}, cramped, outer sep=-2, sep=5]
& \circ \edge{r} \edge{ld} & \bullet \edge{rd} \edge{d} \edge{ld}\\[0pt]
\bullet \edge{d} \edge{r} & \circ \edge{d} & \circ \edge{dl} \edge{dr} & \circ \edge{d} \\[0pt]
\circ \edge{rd} \edge{r} & \bullet \edge{rd} & & \bullet \edge{ld} \\[0pt]
& \bullet \edge{r} & \circ
\end{tikzcd}
};
\end{tikzpicture}
\end{array}
\!\!
\end{equation*}
\vspace{-15pt}
\caption{The octagon of flips. 
}
\vspace{-20pt}
\label{fig:octagon-of-flips}
\end{figure}
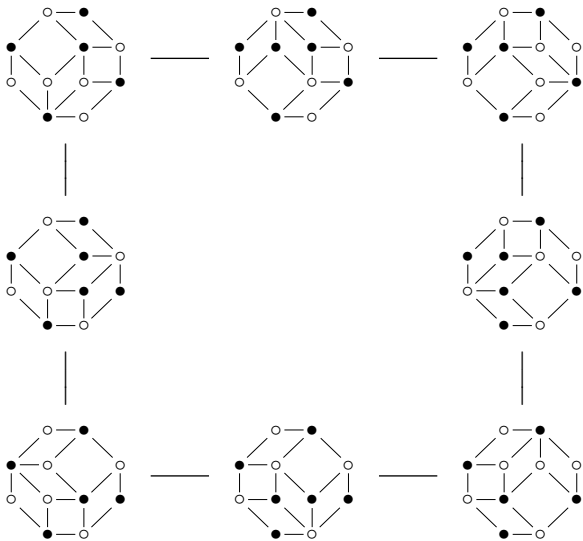


\begin{definition}
Let $\TT$ be a tiling of a polygon by coherent tiles.
Take three tiles in~$\TT$ that form a hexagon as in Figure~\ref{fig:flip-hexagon}.
Construct a new tiling~$\TT'$ by performing a flip and labeling the new vertex 
so that the tiling remains coherent, cf.\ \hbox{Figure~\ref{fig:6-cycle-tiling}(ii)--(iii)}. 
Under suitable genericity assumptions, Corollary~\ref{cor:tileable-hexagons} ensures that
this can be done in a unique way.
We then say that~$\TT'$ is obtained from~$\TT$ by a \emph{Desargues flip}. 
\end{definition}

We next show that Desargues flips satisfy the tetrahedron equation:  

\begin{theorem}
Take a zonotopal tiling of an octagon by coherent tiles, cf.\ Figure~\ref{fig:octagon-of-flips}. 
Apply a sequence of eight Desargues flips corresponding to going around the perimeter of the octagon.
Then the resulting  coherent labeling coincides with the original one. 
\end{theorem}

\begin{figure}[ht]
\vspace{-10pt}
\begin{equation*}
\!\!
\begin{array}{ccccc}
\begin{tikzpicture}[baseline= (a).base]
\node[scale=0.9] (a) at (0,0){
\begin{tikzcd}[arrows={-stealth}, cramped, sep=5]
& f_{123} \edge{r} \edge{ld} \edge{rd} & f_{1234} \edge{rd} \\[7pt]
f_{12} \edge{d} \edge{dr} & & f_{23} \edge{r} \edge{d} \edge{ld}& f_{234} \edge{d} \\[7pt]
f_1 \edge{rd} & f_2 \edge{d} & f_3 \edge{r} \edge{dl} & f_{34} \edge{ld} \\[7pt]
\TT & \varnothing \edge{r} & f_4
\end{tikzcd}
};
\end{tikzpicture}
& \mutation{123} & 
\begin{tikzpicture}[baseline= (a).base]
\node[scale=0.9] (a) at (0,0){
\begin{tikzcd}[arrows={-stealth}, cramped, sep=5]
& f_{123} \edge{r} \edge{ld} \edge{d} \edge{dr} & f_{1234} \edge{rd} \\[7pt]
f_{12} \edge{d} & f_{13} \edge{dl} \edge{dr} & f_{23} \edge{r} \edge{d} & f_{234} \edge{d} \\[7pt]
f_1 \edge{rd} & & f_3 \edge{r} \edge{dl} & f_{34} \edge{ld} \\[7pt]
\TT' & \varnothing \edge{r} & f_4
\end{tikzcd}
};
\end{tikzpicture}
& \mutation{124} & 
\begin{tikzpicture}[baseline= (a).base]
\node[scale=0.9] (a) at (0,0){
\begin{tikzcd}[arrows={-stealth}, cramped, sep=5]
& f_{123} \edge{r} \edge{ld} \edge{d} & f_{1234} \edge{rd} \edge{d} \\[7pt]
f_{12} \edge{d} & f_{13} \edge{r} \edge{dr} \edge{dl} & f_{134} \edge{dr} & f_{234} \edge{d} \\[7pt]
f_1 \edge{rd} & & f_3 \edge{r} \edge{dl} & f_{34} \edge{ld} \\[7pt]
& \varnothing \edge{r} & f_4
\end{tikzcd}
};
\end{tikzpicture}
\\[-5pt]
\\
\vmutation{234}\qquad  &&&& \vmutation{134}\qquad 
\\[-5pt]
\\
\begin{tikzpicture}[baseline= (a).base]
\node[scale=0.9] (a) at (0,0){
\begin{tikzcd}[arrows={-stealth}, cramped, sep=5]
& f_{123} \edge{r} \edge{ld} \edge{dr} & f_{1234} \edge{rd} \\[7pt]
f_{12} \edge{d} \edge{rd} && f_{23} \edge{r} \edge{dl} & f_{234} \edge{d} \edge{ld} \\[7pt]
f_1 \edge{rd} & f_2 \edge{r} \edge{d} & f_{24} \edge{d} & f_{34} \edge{ld} \\[7pt]
\TT'' & \varnothing \edge{r} & f_4
\end{tikzcd}
};
\end{tikzpicture}
&  & 
&  & 
\begin{tikzpicture}[baseline= (a).base]
\node[scale=0.9] (a) at (0,0){
\begin{tikzcd}[arrows={-stealth}, cramped, sep=5]
& f_{123} \edge{r} \edge{ld} \edge{d} & f_{1234} \edge{d} \edge{rd} \\[7pt]
f_{12} \edge{d} & f_{13} \edge{r} \edge{ld} & f_{134} \edge{dr} \edge{dl} & f_{234} \edge{d} \\[7pt]
f_1 \edge{rd} \edge{r} & f_{14} \edge{dr} & & f_{34} \edge{ld} \\[7pt]
& \varnothing \edge{r} & f_4
\end{tikzcd}
};
\end{tikzpicture}
\\[-5pt]
\\
\vmutation{134}\qquad  &&&& \vmutation{234}\qquad
\\[-5pt]
\\
\begin{tikzpicture}[baseline= (a).base]
\node[scale=0.9] (a) at (0,0){
\begin{tikzcd}[arrows={-stealth}, cramped, sep=5]
& f_{123} \edge{r} \edge{ld} & f_{1234} \edge{rd} \edge{ld} \\[7pt]
f_{12} \edge{d} \edge{r} \edge{dr} & f_{124} \edge{rd} & & f_{234} \edge{d} \edge{dl} \\[7pt]
f_1 \edge{rd} & f_2 \edge{r} \edge{d} & f_{24} \edge{d} & f_{34} \edge{ld} \\[7pt]
& \varnothing \edge{r} & f_4
\end{tikzcd}
};
\end{tikzpicture}
& \mutation{124} & 
\begin{tikzpicture}[baseline= (a).base]
\node[scale=0.9] (a) at (0,0){
\begin{tikzcd}[arrows={-stealth}, cramped, sep=5]
& f_{123} \edge{r} \edge{ld} & f_{1234} \edge{rd} \edge{ld} \\[7pt]
f_{12} \edge{d} \edge{r} & f_{124} \edge{d} \edge{dr} & & f_{234} \edge{d} \edge{dl} \\[7pt]
f_1 \edge{rd} \edge{r} & f_{14} \edge{dr} & f_{24} \edge{d} & f_{34} \edge{ld} \\[7pt]
& \varnothing \edge{r} & f_4
\end{tikzcd}
};
\end{tikzpicture}
& \mutation{123} & 
\begin{tikzpicture}[baseline= (a).base]
\node[scale=0.9] (a) at (0,0){
\begin{tikzcd}[arrows={-stealth}, cramped, sep=5]
& f_{123} \edge{r} \edge{ld} & f_{1234} \edge{rd} \edge{d} \edge{ld}\\[7pt]
f_{12} \edge{d} \edge{r} & f_{124} \edge{d} & f_{134} \edge{dl} \edge{dr} & f_{234} \edge{d} \\[7pt]
f_1 \edge{rd} \edge{r} & f_{14} \edge{rd} & & f_{34} \edge{ld} \\[7pt]
& \varnothing \edge{r} & f_4
\end{tikzcd}
};
\end{tikzpicture}
\end{array}
\!\!
\end{equation*}
\vspace{-10pt}
\caption{The octagon relation for Desargues flips. 
}
\vspace{-15pt}
\label{fig:frenkel-moore}
\end{figure}
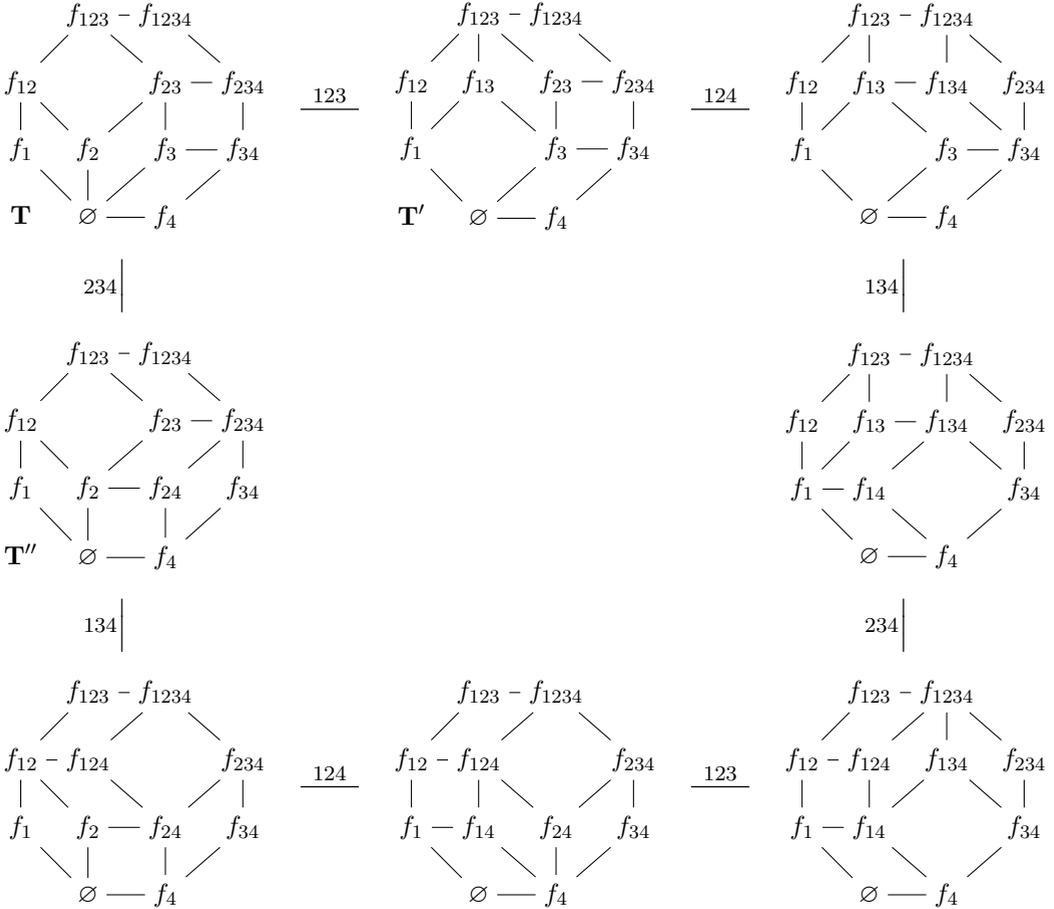


\begin{proof}
As we will explain, this result is a consequence of the 4D consistency property discussed above. 
Let us start with the coherent tiling~$\TT$ in 
Figure~\ref{fig:frenkel-moore}. 
Our goal is to show that applying Desargues flips along the top and right rims of Figure~\ref{fig:frenkel-moore}
produces the same outcome as the flips along the left and bottom rims. 
We note~that all the~tiles appearing in Figure~\ref{fig:frenkel-moore} correspond to 2D faces of the 4D cube, 
cf.\ Figure~\ref{fig:4D-consistency-tikz}. 

Apply Desargues flips $\TT\to\TT'$ and $\TT\to\TT''$ 
to determine $f_{13}$ and~$f_{24}$. 
We then~get 13 labels, namely $f_1$, $f_4$, and 11~labels in the Hamming ball of radius~2 centered at~$f_{23}$. \linebreak[3]
The latter 11~labels satisfy the conditions in Theorem~\ref{th:4d}
(with suitable re-indexing of the 4D cube). 
The remaining 5~labels 
are then uniquely determined by the coherence conditions in the 4D cube,
so they must match the original labels $f_1, f_4$ and the labels $f_{124}, f_{134}, f_{14}$
arising from Desargues flips in Figure~\ref{fig:frenkel-moore}. 
\end{proof}

\clearpage

\newpage

\section{Anticoherent polygons
}
\label{sec:anticoherent}

Many classical theorems of projective geometry that involve cross-ratios 
can be obtained using the tiling technique combined with the following notion, which may be viewed as a 
``negative counterpart'' of Definition~\ref{def:coherent-polygon}. 

\begin{definition}
\label{def:anticoherent-polygon}
Let $n\ge2$.
Let $A_1,\dots,A_n$ (resp., $\ell_1,\dots,\ell_n$) be an $n$-tuple of~points (resp., lines) on 
the real/complex projective plane. 
Assume that each point~$A_i$ does not lie on either of the lines~$\ell_i$ and~$\ell_{i-1}$
(with the indexing understood modulo~$n$). 
Let $\mathbf{P}$ be a $2n$-gon with vertices labeled $A_1, \ell_1, \dots, A_n,\ell_n$, 
in this order, see Figure~\ref{fig:coherent-polygon}.  
We call $\mathbf{P}$  \emph{anticoherent} if 
the associated generalized cross-ratio \eqref{eq:coherent-polygon} is equal to~$-1$: 
\begin{equation*}
(A_1,\dots,A_n;\ell_1,\dots,\ell_n)=-1. 
\end{equation*}
The anticoherence property is invariant under rotations and reflections of a polygon~$\mathbf{P}$. 
\end{definition}

We will repeatedly make use of the following simple observation. 

\begin{lemma}
\label{lem:doubling}
A $2n$-gon $\mathbf{P}$ with vertex labels $A_1, \ell_1, \dots, A_n,\ell_n$ 
(cf.\ Definition~\ref{def:anticoherent-polygon})
is anticoherent
if and only if $\mathbf{P}$ is not coherent
but the $4n$-gon $\mathbf{P}^{(2)}$ 
with vertices labeled by 
$A_1, \ell_1, \dots, A_n,\ell_n,A_1, \ell_1, \dots, A_n,\ell_n$, in this order, 
is coherent. 
\end{lemma}

We call the $4n$-gon $\mathbf{P}^{(2)}$ the \emph{double} of the $2n$-gon~$\mathbf{P}$. 

\begin{proof}
This statement is immediate from the identity
\begin{equation*}
(A_1,\dots,A_n,A_1,\dots,A_n;\ell_1,\dots,\ell_n,\ell_1,\dots,\ell_n)=(A_1,\dots,A_n;\ell_1,\dots,\ell_n)^2. \qedhere
\end{equation*}
\end{proof}

\section*{Harmonic property of a complete quadrangle}

In the case $n=2$, a quadrilateral tile 
\begin{equation}
\label{eq:anticoherent-tile}
\begin{tikzpicture}[baseline= (a).base]
\node[scale=1] (a) at (0,0){
\begin{tikzcd}[arrows={-stealth}, sep=20, cramped]
P_1  \edge{r}  \edge{d}& \ell_1 \edge{d} \\[0pt]
\ell_2 \edge{r} & P_2
\end{tikzcd}
};
\end{tikzpicture}
\end{equation}
is anticoherent if and only if the quadruple of points $P_1$, $P_2$, $(P_1P_2)\cap\ell_1$, ${(P_1P_2)\cap \ell_2}$ 
is \emph{harmonic}. \linebreak[3]
In the language of incidence geometry, this property can be expressed 
via a classical construction involving a complete quadrangle, 
which we~derive below (in two different versions, see~\eqref{eq:harmonic-quad}) using the tiling method. 
Cf.~also Theorem~\ref{th:harmonic-points-theorem}.

\begin{theorem}
\label{th:harmonic-quad}
Let $A_1, A_2, A_3,A_4$ be four generic points on the real/complex projective plane. 
Draw six lines $\ell_{ij}\!=\!(A_iA_j)$, ${1\le i<j\le 4}$. 
Let $P_{12,34}=\ell_{12}\cap\ell_{34}$~and $P_{14, 23}\!=\!\ell_{14}\cap\ell_{23}$. 
Draw the line ${\ell_\circ\!=\!(P_{12,34}P_{14, 23})}$. 
Set ${P_{13}\!=\!\ell_{13}\cap\ell_\circ}$ and ${P_{24}\!=\!\ell_{24}\cap\ell_\circ}$.
(See Figure~\ref{fig:harmonic-quad-config}.)
Then the tiles
\begin{equation}
\label{eq:harmonic-quad}
\begin{tikzpicture}[baseline= (a).base]
\node[scale=1] (a) at (0,0){
\begin{tikzcd}[arrows={-stealth}, sep=16, cramped]
P_{13}  \edge{r}  \edge{d}& \ell_{14} \edge{d} \\[5pt]
\ell_{12} \edge{r} & P_{24}
\end{tikzcd}
};
\end{tikzpicture}
\qquad\qquad
\begin{tikzpicture}[baseline= (a).base]
\node[scale=1] (a) at (0,0){
\begin{tikzcd}[arrows={-stealth}, sep=5, cramped]
P_{12,34}  \edge{r}  \edge{d}& \ell_{13} \edge{d} \\[15pt]
\ell_{24} \edge{r} & P_{14,23}
\end{tikzcd}
};
\end{tikzpicture}
\end{equation}
are anticoherent. 
\end{theorem}

\begin{proof}
Apply Lemma~\ref{lem:doubling} to the tilings shown in Figures~\ref{fig:anticoherent-quad}--\ref{fig:anticoherent-quad-2}. 
\end{proof}

\pagebreak[3]

\begin{figure}[ht]
\begin{center}
\includegraphics[scale=0.47, trim=0.1cm 0cm 0cm 0.5cm, clip]{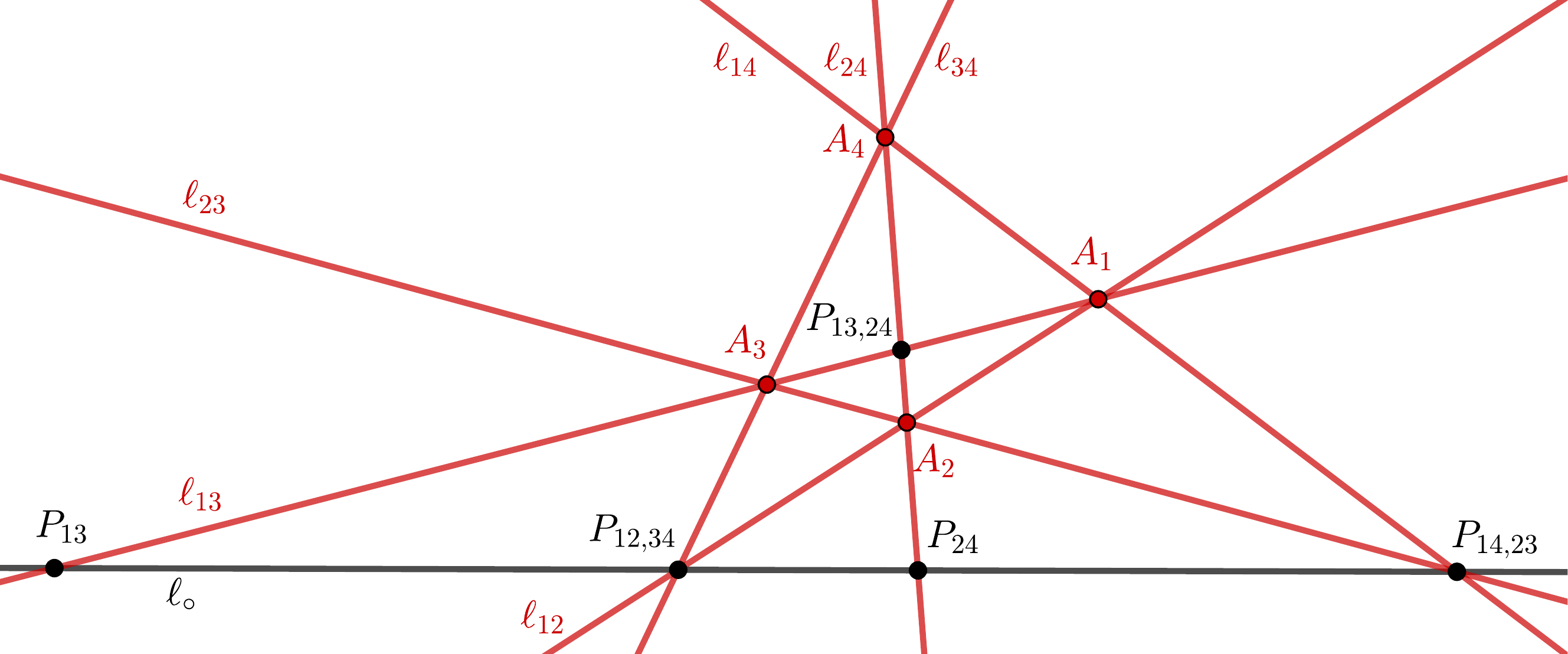}
\end{center}
\vspace{-8pt}
\caption{The harmonic property of a complete quadrangle. 
Here $P_{13,24}=\ell_{13}\cap \ell_{24}$. 
}
\label{fig:harmonic-quad-config}
\end{figure}

\begin{figure}[ht]
\vspace{-.3in}
\begin{equation*}
\begin{tikzpicture}[baseline= (a).base]
\node[scale=1] (a) at (0,0){
\begin{tikzcd}[arrows={-stealth}, cramped, sep=5]
P_{13} \edge{rr} \edge{dd} && \ell_{14} \edge{dd} & \\[3pt]
& \boxed{-1} && \\[3pt]
\ell_{12} \edge{rr} && P_{24}
\end{tikzcd}
};
\end{tikzpicture}
\qquad\qquad
\begin{tikzpicture}[baseline= (a).base]
\node[scale=1] (a) at (0,0){
\begin{tikzcd}[arrows={-stealth}, cramped, sep=15]
 & \ell_{12} \edge{r} \edge{rd} \edge{ld} & P_{13} \edge{rd} \\[8pt]
 P_{24} \edge{rd} \edge{d} && P_{13,24}\edge{ld} \edge{d} \edge{r} & \ell_{14} \edge{d} \\[6pt]
\ell_{14}  \edge{rd} & \ell_{23} \edge{d} & \ell_{34} \edge{ld} \edge{r} & P_{24} \edge{ld} \\[8pt]
 & P_{13} \edge{r} & \ell_{12} 
\end{tikzcd}
};
\end{tikzpicture}
\end{equation*}
\vspace{-.2in}
\caption{Anticoherent quadrilateral (marked $\boxed{-1}$) and a tiling of its double.
}
\vspace{-.2in}
\label{fig:anticoherent-quad}
\end{figure}
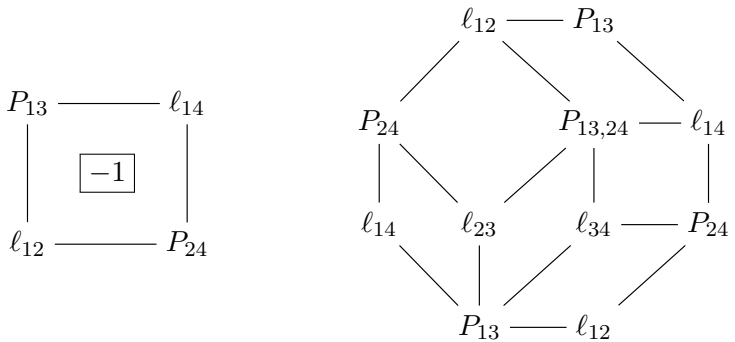

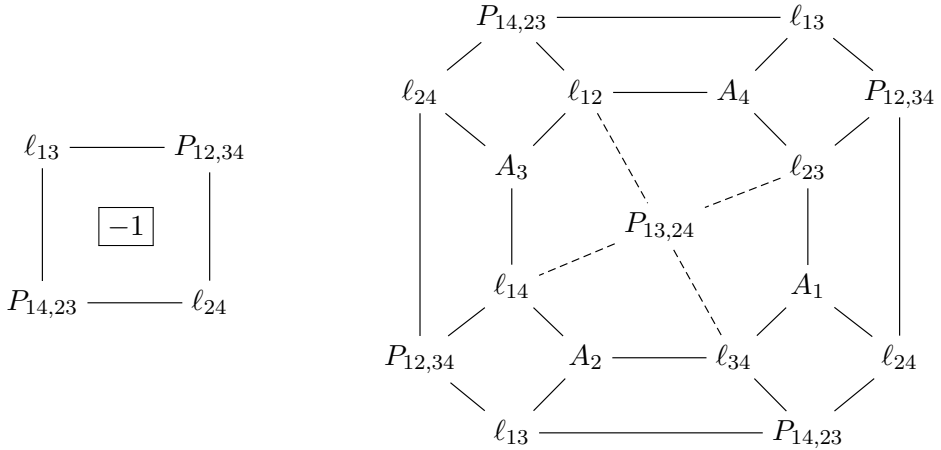
\begin{figure}[ht]
\vspace{-.1in}
\begin{equation*}
\begin{tikzpicture}[baseline= (a).base]
\node[scale=1] (a) at (0,0){
\begin{tikzcd}[arrows={-stealth}, cramped, sep=1]
\ell_{13} \edge{rr} \edge{dd} && P_{12,34} \edge{dd} & \\[9pt]
& \boxed{-1} && \\[9pt]
P_{14,23} \edge{rr} && \ell_{24}
\end{tikzcd}
};
\end{tikzpicture}
\qquad\qquad
\begin{tikzpicture}[baseline= (a).base]
\node[scale=1] (a) at (0,0){
\begin{tikzcd}[arrows={-stealth}, cramped, sep=1]
& P_{14,23} \edge{rrrr} \edge{dr} &&&& \ell_{13} \edge{dl} \edge{rd} & \\[10pt]
\ell_{24} \edge{dr} \edge{dddd} \edge{ru} & & \ell_{12} \edge{rr} \edge{dl} \edge{ddr, dashed}&& A_4 \edge{dr}  && P_{12,34} \edge{dl} \edge{dddd} \\[10pt]
& A_3 \edge{dd}  &&&& \ell_{23}\edge{dd} \edge{dll, dashed} \\[5pt]
&&& P_{13,24} \edge{ddr, dashed} \edge{dll, dashed} \\[5pt]
& \ell_{14} \edge{dl} \edge{dr} &&&& A_1 \edge{dl} \edge{dr} \\[10pt]
P_{12,34} \edge{rd} & & A_2 \edge{dl} \edge{rr} && \ell_{34} \edge{dr} & & \ell_{24} \\[10pt]
& \ell_{13} \edge{rrrr} &&&& P_{14,23} \edge{ru}
\end{tikzcd}
};
\end{tikzpicture}
\end{equation*}
\vspace{-.2in}
\caption{Another anticoherent quadrilateral and a tiling of its double.
}
\vspace{-.2in}
\label{fig:anticoherent-quad-2}
\end{figure}

\begin{remark}
Over a field of characteristic~2, the notions of coherence and anti\-coherence coincide---so 
the tile shown in Figure~\ref{fig:anticoherent-quad-2} on the left is coherent, 
meaning that the point $P_{13,24}$ where $\ell_{13}$ and~$\ell_{24}$ intersect 
lies on the line $(P_{12,34} P_{14,23})=\ell_\circ$.
Thus in characteristic 2 we have $P_{13}=P_{24}=P_{13,24}$, yielding the \emph{Fano configuration}. 
\end{remark}

\newpage

\section*{Ceva's theorem}

\begin{proposition}
\label{prop:hex-tile}
Let $\ell_1,\ell_2,\ell_3$ be three concurrent lines on the real/complex plane. 
Let $P_1,P_2,P_3$ be points such that $P_1\in\ell_1$, $P_2\in\ell_2$, $P_3\in\ell_3$
(but no other incidences). 
Then the hexagon $\mathbf{P}$ shown in Figure~\ref{fig:hex-tile} is anticoherent. 
\end{proposition}

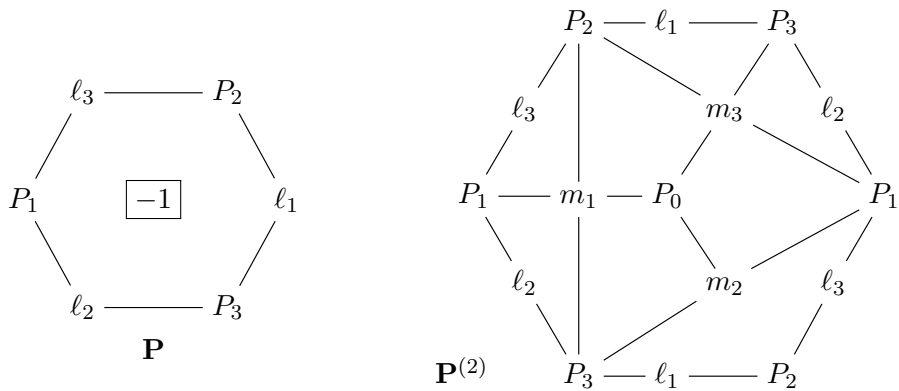
\begin{figure}[ht]
\vspace{-.15in}
\begin{equation*}
\begin{tikzpicture}[baseline= (a).base]
\node[scale=1] (a) at (0,0){
\begin{tikzcd}[arrows={-stealth}, cramped, sep=5]
{\ }&&&& \\[4pt]
& \ell_3 \edge{rr} && P_2 \edge{rd} & \\[17pt]
P_1 \edge{ru} \edge{rd} & &\boxed{-1}& & \ell_1 \\[17pt]
& \ell_2 \edge{rr}&& P_3 \edge{ru} \\[-4pt]
&& \mathbf{P}
\end{tikzcd}
};
\end{tikzpicture}
\qquad\qquad
\begin{tikzpicture}[baseline= (a).base]
\node[scale=1] (a) at (0,0){
\begin{tikzcd}[arrows={-stealth}, cramped, sep=2]
&&& P_2 \edge{rr} \edge{ld} \edge{dd} \edge{rrrd} && \ell_1 \edge{rr} && P_3 \edge{ld}\edge{rd} \\[15pt]
&& \ell_3 \edge{ld}  &&  && m_3 \edge{rrrd} \edge{ld} && \ell_2  \edge{rd}\\[15pt]
&P_1 \edge{rd} \edge{rr} && m_1 \edge{rr} \edge{dd}  & \hspace{3pt} & P_0 \edge{rd} &&  & & P_1\edge{ld} \\[15pt]
&& \ell_2 \edge{rd} &&  && m_2 \edge{rrru} \edge{llld} && \ell_3 \edge{ld} \\[15pt]
\mathbf{P}^{(2)} \hspace{-20pt}&&& P_3 \edge{rr} && \ell_1 \edge{rr} && P_2 
\end{tikzcd}
};
\end{tikzpicture}
\end{equation*}
\vspace{-.15in}
\caption{An anticoherent hexagon $\mathbf{P}$ and a tiling of its double~$\mathbf{P}^{(2)}$.
}
\vspace{-.2in}
\label{fig:hex-tile}
\end{figure}

\begin{proof}
Denote $Q_1=\ell_1\cap (P_2P_3)$,  $Q_2=\ell_2\cap (P_1P_3)$,  $Q_3=\ell_3\cap (P_1P_2)$,
$m_1=(Q_2Q_3)$, $m_2=(Q_1Q_3)$, $m_3=(Q_1Q_2)$. 
Also let $P_0$ be the point where $\ell_1,\ell_2,\ell_3$ meet, see Figure~\ref{fig:anticoherent-hexagon}. 
Then the claim follows from the tiling shown in Figure~\ref{fig:hex-tile}. 
\end{proof}

\begin{figure}[ht]
\begin{center}
\includegraphics[scale=0.5, trim=0.1cm 0cm 0cm 0.2cm, clip]{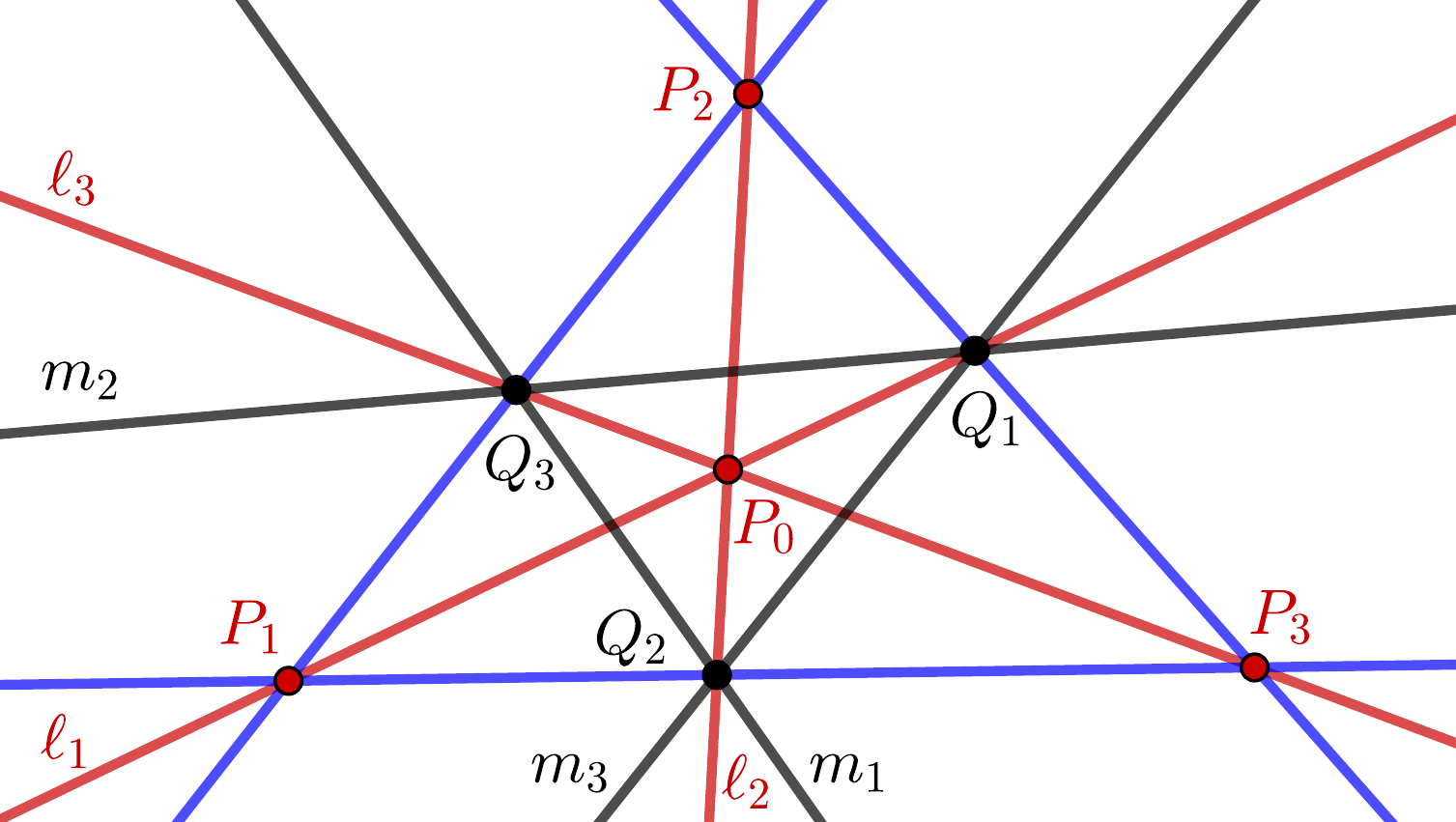}
\end{center}
\vspace{-8pt}
\caption{Three concurrent lines, with points on them. 
}
\label{fig:anticoherent-hexagon}
\end{figure}

Proposition~\ref{prop:hex-tile} also has a simple algebraic proof, which we omit. 

\begin{remark}
Proposition~\ref{prop:hex-tile} is essentially a reformulation of 
Ceva's theorem
(proved by G.~Ceva in 1678 and by al-Mu'taman ibn H\=ud in the 11th century~\cite{hogendijk}). 
Indeed, the mixed cross-ratio associated with the hexagon~$\mathbf{P}$ in Figure~\ref{fig:hex-tile} 
is equal, up to sign, to the ``Ceva ratio'' of six Euclidean lengths: 
\begin{equation*}
(P_1,P_2,P_3;\ell_3,\ell_1,\ell_2)
=
\dfrac{
\langle \mathbf{P_1}, \mathbf{\boldsymbol \ell_3} \rangle
\langle \mathbf{P_2}, \mathbf{\boldsymbol \ell_1} \rangle
\langle \mathbf{P_3}, \mathbf{\boldsymbol \ell_2} \rangle
}{
\langle \mathbf{P_2}, \mathbf{\boldsymbol \ell_3} \rangle
\langle \mathbf{P_3}, \mathbf{\boldsymbol \ell_1} \rangle
\langle \mathbf{P_1}, \mathbf{\boldsymbol \ell_2} \rangle
}
=
-\frac
{|P_1Q_3|\cdot |P_2Q_1| \cdot |P_3Q_2|}
{|P_2Q_3|\cdot |P_3Q_1|\cdot |P_1Q_2|}\,.
\end{equation*}
\end{remark}

\newpage

\section*{Another proof of Proposition~\ref{prop:10-gon}}

The use of anticoherent polygons allows us to simplify some of the tiling-based proofs of incidence theorems
that appeared earlier in this paper.
Instead of tiling an oriented surface with coherent tiles, we can tile it with both coherent and anticoherent polygons,
making sure that the number of anticoherent polygons is even.
Then any one of these (anti)coherence conditions is implied by the rest. 

We illustrate this approach by providing a simpler proof of Proposition~\ref{prop:10-gon}
that avoids using the rather complicated tilings in Figures~\ref{fig:10-gon-tiling-no-holes} and~\ref{fig:10-gon-tiling-holes}.
By extension, this also simplifies the tiling-based proofs of Theorems~\ref{th:goodman-pollack} and \ref{th:saam-5},
both of which relied on Proposition~\ref{prop:10-gon}. 

\begin{proof}[Proof~4 of Proposition~\ref{prop:10-gon}]
Consider the tiling of a torus by four polygons shown in Figure~\ref{fig:decagon-via-negative}. 
Two of these polygons (the hexagons marked with $\boxed{-1}$ in the picture) 
are anti\-coherent by Proposition~\ref{prop:hex-tile}. 
The octagon at the top of the picture is coherent by Proposition~\ref{pr:octagon}. 
It follows that the remaining polygon, which appears at the bottom of the picture, is coherent. 
This is precisely the decagon from Figure~\ref{fig:10-gon}. 
\end{proof}

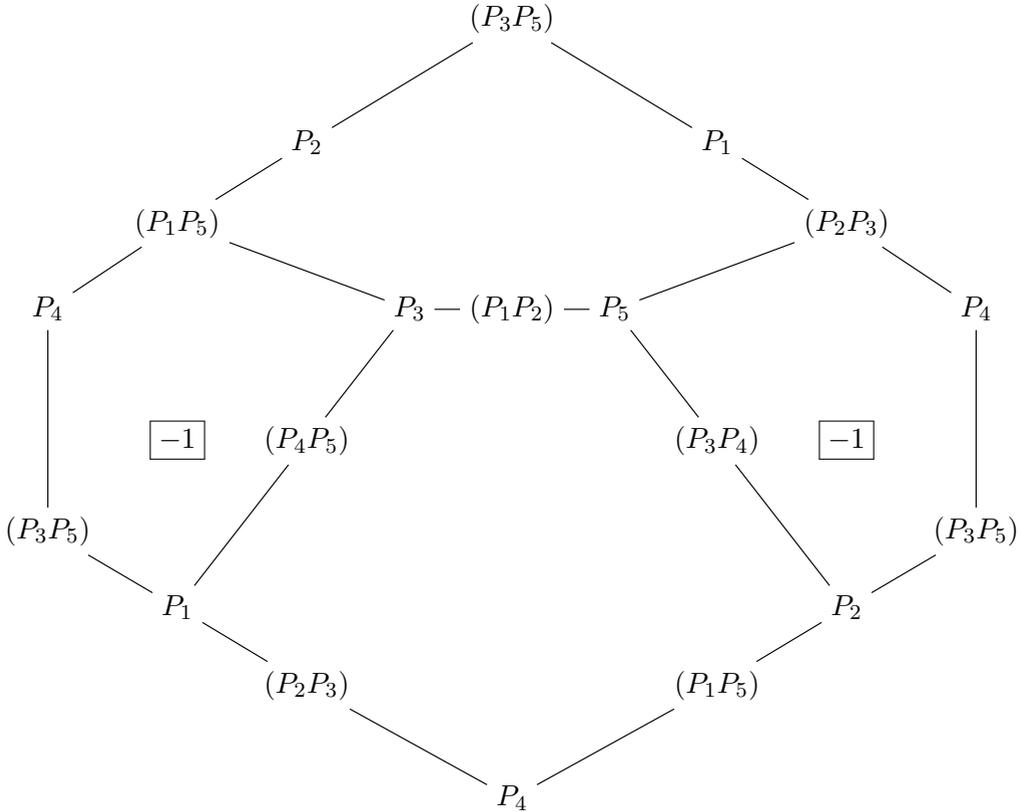
\begin{figure}[ht]
\begin{equation*}
\begin{tikzpicture}[baseline= (a).base]
\node[scale=1] (a) at (0,0){
\begin{tikzcd}[arrows={-stealth}, cramped, sep=10]
&& && (P_3P_5) \edge{lld}\edge{rrd} \\[20pt]
&&  P_2 \edge{ld} &&&& P_1 \edge{rd} \\[4pt]
& (P_1P_5) \edge{ld} \edge{rrd} &&& &&& (P_2P_3) \edge{lld} \edge{rd} \\[5pt]
P_4 \edge{dd} & && P_3 \edge{ld} \edge{r} &(P_1P_2)\edge{r} & P_5 \edge{rd} &&& P_4 \edge{dd} \\[20pt]
 &\boxed{-1}& (P_4P_5) \edge{ldd} &&&& (P_3P_4) \edge{rdd} &\boxed{-1}&  \\[5pt]
(P_3P_5) \edge{rd} &&  &&&&  && (P_3P_5) \edge{ld} \\[2pt]
& P_1 \edge{rd} &&&&&& P_2 \edge{ld} \\[3pt]
&& (P_2P_3) \edge{rrd} &&&& (P_1P_5) \edge{lld} \\[16pt]
&&&& P_4
\end{tikzcd}
};
\end{tikzpicture}
\end{equation*}
\vspace{-.1in}
\caption{Another proof of Proposition~\ref{prop:10-gon}. 
The opposite sides of the hexagonal fundamental domain should be glued to each other. 
}
\label{fig:decagon-via-negative}
\end{figure}

\newpage

\section*{Coherent annuli}

\begin{lemma}
\label{lem:coherent-annulus}
Let $P, A, B, C, D$ be five generic points in the real/complex plane.
Then the generalized mixed cross-ration associated to the annulus
\begin{equation}
\label{eq:coherent-annulus}
\begin{tikzpicture}[baseline= (a).base]
\node[scale=1] (a) at (0,0){
\begin{tikzcd}[arrows={-stealth}, cramped, sep=3]
(PA) \edge{rrr}  \edge{ddd}&&& B  \edge{ddd}\\[8pt]
& C  \edge{r}  \edge{d}& (PB) \edge{d}& \\[8pt]
& (PD) \edge{r}  & A & \\[8pt]
D \edge{rrr} & & & (PC)
\end{tikzcd}
};
\end{tikzpicture}
\end{equation}
is equal to 1:
\begin{equation*}
(B, D; (PC), (PA))\cdot (A, C; PB, PD) =1. 
\end{equation*}
\end{lemma}

\begin{proof}[Proof~1]
Applying the definition~\eqref{eq:coherent-polygon} of the generalized mixed cross-ratio,
we see that each of the four terms in the numerator cancels out a term in the denominator. 
\end{proof}

\begin{proof}[Proof~2]
Each of the two hexagons in the tiling
\begin{equation*}
\begin{tikzpicture}[baseline= (a).base]
\node[scale=1] (a) at (0,0){
\begin{tikzcd}[arrows={-stealth}, cramped, sep=6]
(PA) \edge{rrr} \edge{rd} \edge{ddd}&&& B  \edge{ddd}\\[8pt]
& C  \edge{r}  \edge{d}& (PB) \edge{d}& \\[8pt]
& (PD) \edge{r}  & A \edge{rd}& \\[8pt]
D \edge{rrr} & & & (PC)
\end{tikzcd}
};
\end{tikzpicture}
\end{equation*}
of the given annulus is anticoherent by Proposition~\ref{prop:hex-tile}. The claim follows. 
\end{proof}

\begin{proof}[Proof~3]
The statement follows from the tiling of the torus shown in Figure~\ref{fig:annulus-from-decagon}. 
\end{proof}

\enlargethispage{2cm}

\begin{figure}[ht]
\vspace{-5pt}
\begin{equation*}
\begin{tikzcd}[arrows={-stealth}, cramped, sep=9]
&& A \edge{r} \edge{ld} & (BD) \edge{r} \edge{r} & P \edge{r} \edge{ddd}& (AC) \edge{r} & B \edge{rd} \\[3pt]
& (PD) \edge{ld} &&&&&& (PC)\edge{rd}  \\[3pt]
B \edge{r} \edge{rd} &(PA)\edge{rd}  &&&&&&(PB) \edge{ld} \edge{r} & A \edge{ld} \\[3pt]
& (PC) \edge{r} \edge{rd} & D \edge{rr} && (AB) \edge{rr} && C \edge{r} & (PD)\edge{ld}  \\[3pt]
&& A \edge{r} & (BD) \edge{r} & P \edge{r} & (AC) \edge{r} & B 
\end{tikzcd}
\end{equation*}
\vspace{-5pt}
\caption{Third proof of Lemma~\ref{lem:coherent-annulus}. 
Opposite sides of the hexagonal fundamental domain should be glued to each other.
Two octagons at the top of the picture and the decagon at the bottom of it are coherent 
by Propositions~\ref{pr:octagon} and~\ref{prop:10-gon}, respectively. 
The remaining quadrilaterals form the boundary of the annulus in~\eqref{eq:coherent-annulus}. 
}
\vspace{-15pt}
\label{fig:annulus-from-decagon}
\end{figure}
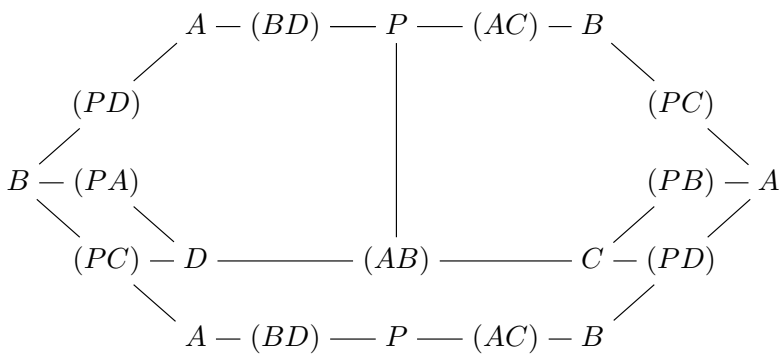

\newpage

\section*{Coherent 14-gons and 16-gons}

Lemma~\ref{lem:coherent-annulus} can be used to provide a tiling-based proof of 
the coherence of a $2n$-gon 
with boundary labels $P_1$\,-\,$(P_2P_3)$\,-\,$P_4$\,-\,$(P_6P_7)$\,-\,$\cdots$
for ${n\not\equiv 0\bmod 3}$, cf.\ Remark~\ref{rem:12-3-45-6-}. 
The cases $n=4$ and $n=5$ have been treated in 
Propositions~\ref{pr:octagon} and \ref{prop:10-gon}, respectively.  
We illustrate the general argument by showing the appropriate tilings for $n=7$ and $n=8$ in
Figures~\ref{fig:14-gon} and~\ref{fig:16-gon} below. 

\begin{figure}[ht]
\vspace{-10pt}
\begin{equation*}
\begin{tikzpicture}[baseline= (a).base]
\node[scale=.9] (a) at (0,0){
\begin{tikzcd}[arrows={-stealth}, cramped, sep=5.5]
&&&& (P_1P_7) \edge{r} \edge{dl} & P_6 \edge{r} & (P_4P_5) \edge{r} \edge{ddddd} & P_3 \edge{r} & (P_1P_2) \edge{dr} \\[3pt]
&&& P_4 \edge{dl} &&&&&& P_5 \edge{dr} \\[3pt]
&& (P_5P_6) \edge{dl} &&&&&&&& (P_3P_4) \edge{dr} \\[3pt]
& P_7 \edge{dr} \edge{dl} &&&&&&&&&& P_2 \edge{dl} \edge{dr} \\[3pt]
(P_1P_2) \edge{dr} && (P_1P_4) \edge{dl} &&&&&&&& (P_1P_5) \edge{dr} && (P_1P_7) \edge{dl} \\[3pt]
& P_5 \edge{dr} \edge{rrr} &&& (P_6P_7) \edge{rr} && P_1 \edge{rr} && (P_2P_3) \edge{rrr} &&& P_4 \edge{dl} \\[3pt]
&& (P_3P_4) \edge{dr} &&&&&&&& (P_5P_6) \edge{dl} \\[3pt]
&&& P_2 \edge{dr} &&&&&& P_7 \edge{dl} \\[3pt]
&&&& (P_1P_7) \edge{r} & P_6 \edge{r} & (P_4P_5) \edge{r} & P_3 \edge{r} & (P_1P_2) 
\end{tikzcd}
};
\end{tikzpicture}
\end{equation*}
\vspace{-10pt}
\caption{A coherent 14-gon appears of the bottom of the picture. 
Opposite sides of the hexagonal fundamental domain should be glued to each other.
The decagons at the top of the picture are coherent 
by Proposition~\ref{prop:10-gon}. 
The two quadrilaterals at the left and right ends form the boundary of the annulus in~\eqref{eq:coherent-annulus}. 
}
\vspace{-15pt}
\label{fig:14-gon}
\end{figure}
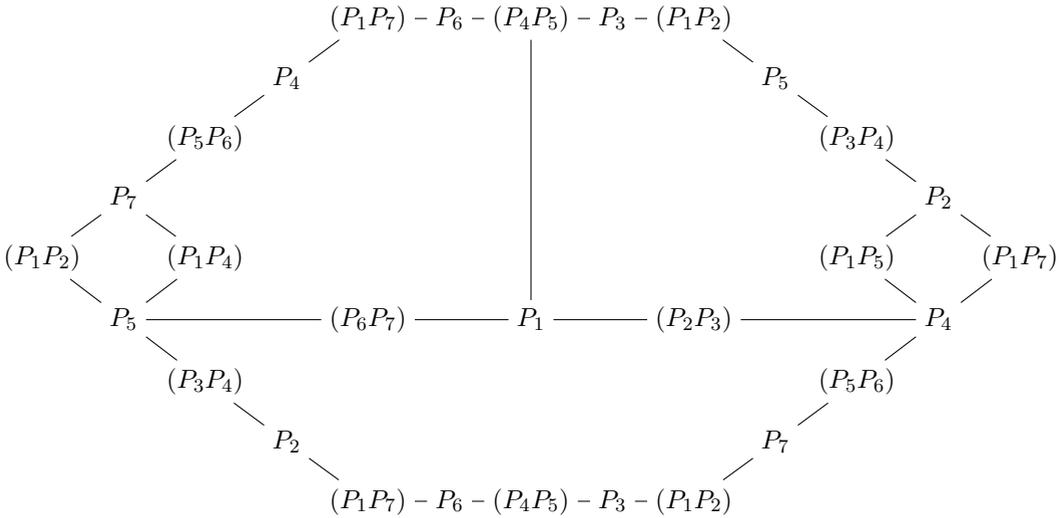

\begin{figure}[ht]
\vspace{-10pt}
\begin{equation*}
\begin{tikzpicture}[baseline= (a).base]
\node[scale=.9] (a) at (0,0){
\begin{tikzcd}[arrows={-stealth}, cramped, sep=5.5]
&& P_2 \edge{r} \edge{dl} & (P_1P_8) \edge{r} & P_7 \edge{r} \edge{rrrddddd}& (P_5P_6) \edge{r} & P_4 \edge{r} & (P_2P_3) \edge{r}& P_1  \edge{dr} \\[3pt]
& (P_7P_8) \edge{dl} &&&&&&&& (P_6P_7) \edge{dr} \\[3pt]
P_1 \edge{dr}  \edge{r} & (P_2P_7) \edge{dr} &&&&&&&&& P_5 \edge{dr} \\[3pt]
& (P_6P_7)\edge{dr}  \edge{r} &P_8 \edge{ddrrrrr} &&&&&&&& & (P_3P_4) \edge{dr} \\[3pt]
&& P_5 \edge{dr} &&&&&&&& & (P_1P_7) \edge{r} \edge{dl} & P_2 \edge{dl} \\[3pt]
&&& (P_3P_4)\edge{dr} &&&&(P_1P_2)\edge{r} &P_3\edge{r} &(P_4P_5)\edge{r} &P_6 \edge{r} & (P_7P_8) \edge{dl} \\[3pt]
&&&& P_2 \edge{r} & (P_1P_8) \edge{r} & P_7 \edge{r} & (P_5P_6) \edge{r} & P_4 \edge{r} & (P_2P_3) \edge{r} & P_1
\end{tikzcd}
};
\end{tikzpicture}
\end{equation*}
\vspace{-10pt}
\caption{A coherent 16-gon appears of the bottom of the picture. 
Opposite sides of the hexagonal fundamental domain should be glued to each other.
The octagon at the upper-left is coherent 
by Proposition~\ref{prop:10-gon}. 
The 14-gon at the upper-right is coherent 
by Figure~\ref{fig:14-gon}. 
The two quadrilaterals at the left and right ends form the boundary of the annulus in~\eqref{eq:coherent-annulus}. 
}
\vspace{-15pt}
\label{fig:16-gon}
\end{figure}
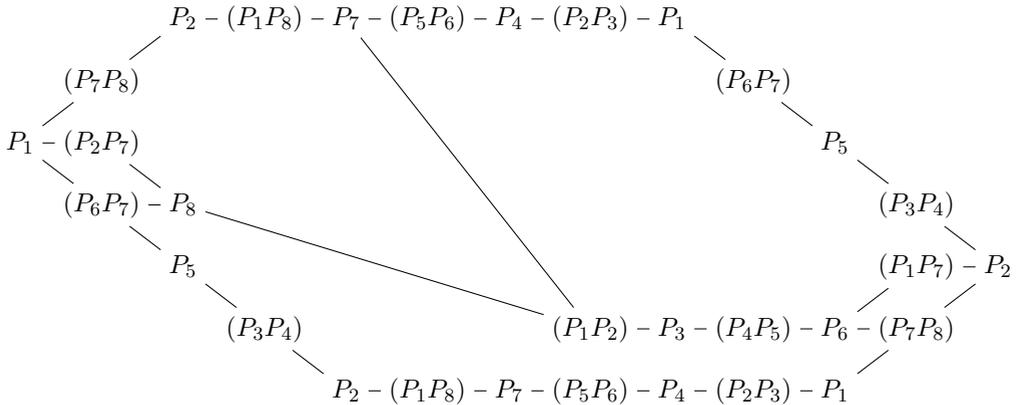

\newpage

\section{Variations of the master theorem
}
\label{sec:generalizations}

\section*{Master theorem for arbitrary dimensions}

The notions of mixed cross-ratio and coherent tile 
can be naturally extended to involve 
projective subspaces of arbitrary dimensions, as explained below. 
This leads to a generalization of our master theorem to tilings labeled by such subspaces. 

\medskip

We continue using the notation from Definition~\ref{def:mixed-cross-ratio}. 
In particular, $\PP\cong\CC\PP^d$ is the projectivization of a complex vector space~$\VV$ of dimension $d+1$. 
We fix a skew-symmetric multilinear \emph{volume form} $\operatorname{vol}: \bigwedge^{d+1}\VV\to\CC$. 

\begin{definition}
\label{def:extensors}
A (decomposable) \emph{extensor}  of step $k$ in~$\VV$  is a wedge product 
\begin{equation*}
\mathbf{e}=e_1\wedge  \cdots \wedge e_k \in \textstyle\bigwedge^k \VV 
\end{equation*}
of $k$ linearly independent vectors $e_1,\dots,e_k\in\VV$. 
The projectivization of~$\mathbf{e}$ is uniquely determined by 
the $k$-dimensional subspace $E\!=\!\operatorname{span}(e_1,\dots,e_k)$. 
That is, the extensor~$\mathbf{e}$ does not depend on the choice of a basis $e_1,\dots,e_k$ in~$E$,
up to a scalar~factor. 

Let $\mathbf{e}$ and $\mathbf{f}$ be extensors of steps $k$ and~$\ell$, respectively, with $k+\ell=\dim\VV=d+1$.
We then denote by $\langle\mathbf{e},\mathbf{f}\rangle = \operatorname{vol}(\mathbf{e}\wedge\mathbf{f})$ 
the natural pairing of $\mathbf{e}$ and $\mathbf{f}$,
\end{definition}

\begin{definition}
\label{def:mixed-cross-ratio-generalized}
Let $r$ and $s$ be positive integers such that $r+s=d-1$. 
Let $a_1, a_2, a_3, a_4$ be projective subspaces of~$\PP\cong\CC\PP^d$ of dimensions $r,s,r,s$, respectively. \linebreak[3]
Let $\mathbf{A_1}, \mathbf{A_2}, \mathbf{A_3}, \mathbf{A_4}$ be 
arbitrary choices of extensors associated with $a_1, a_2, a_3, a_4$, respectively, cf.\ Definition~\ref{def:extensors}.
Assume that all four intersections 
\begin{equation*}
\text{
$a_1\cap a_2$, 
$a_2\cap a_3$, 
$a_3\cap a_4$, 
$a_4\cap a_1$
}
\end{equation*}
are empty;
equivalently, all four pairings 
\begin{equation*}
\text{
$\langle \mathbf{A_1}, \mathbf{A_2} \rangle$, 
$\langle \mathbf{A_2}, \mathbf{A_3} \rangle$, 
$\langle \mathbf{A_3}, \mathbf{A_4} \rangle$, 
$\langle \mathbf{A_4}, \mathbf{A_1} \rangle$
}
\end{equation*}
are~nonzero.
The (generalized) \emph{mixed cross-ratio} $(a_1,a_3;a_2,a_4)$ is defined~by 
\begin{equation}
\label{eq:mixed-cross-ratio-generalized}
(a_1,a_3;a_2,a_4)=
\dfrac
{\langle \mathbf{A_1}, \mathbf{A_2} \rangle \langle \mathbf{A_3}, \mathbf{A_4} \rangle}
{\langle \mathbf{A_3}, \mathbf{A_2} \rangle \langle \mathbf{A_1}, \mathbf{A_4} \rangle}. 
\end{equation}
As in Definition~\ref{def:mixed-cross-ratio}, the mixed cross-ratio 
$(a_1,a_3;a_2,a_4)$ does not depend on the choice of extensors 
$\mathbf{A_{\boldsymbol i}}$ representing the subspaces~$a_i$. 

Generalizing Definition~\ref{def:coherent-tile}/Proposition~\ref{pr:coherence-algebraic}, we say that 
the tile
\begin{equation}
\label{eq:coherent-tile-generalized}
\begin{tikzpicture}[baseline= (a).base]
\node[scale=1] (a) at (0,0){
\begin{tikzcd}[arrows={-stealth}, sep=small, cramped]
a_1  \edge{r}  \edge{d}& a_2 \edge{d} \\[3pt]
a_4 \edge{r} & a_3
\end{tikzcd}
};
\end{tikzpicture}
\end{equation}
is \emph{coherent} if and only if 
\begin{equation}
\label{eq:cross-ratio=1-generalized}
(a_1,a_3;a_2,a_4)=1.
\end{equation}
\end{definition}

Our master theorem (Theorem~\ref{th:master}) generalizes---practically verbatim and with the same proof---to the
setting of Definition~\ref{def:mixed-cross-ratio-generalized}.
We label the black (resp., white) vertices of a tiling by projective subspaces of dimension~$r$ (resp.,~$s$)
such that adjacent vertices are labeled by disjoint subspaces. 
Then, if all tiles but one are coherent, then the remaining tile is coherent as well.  

\section*{Geometry of generalized mixed cross-ratios}

The geometric interpretation of the generalized mixed cross-ratio $(a_1,a_3;a_2,a_4)$ is more complicated
than the original case of vectors and covectors (cf.\ \eqref{eq:ABlm=ABLM}) might suggest. 
These complications ultimately stem from the following fact. 

\begin{lemma}
\label{lem:flats}
Let $\PP$ be a $d$-dimensional complex projective space. 
Let $a_1, a_2, a_3, a_4$ be generic subspaces of~$\PP$ of dimensions $r,s,r,s$, respectively,
where $r+s=d-1$. 
Then there are exactly $d-\max(r,s)$ lines that pierce all four subspaces $a_1, a_2, a_3, a_4$. 
\end{lemma}

\begin{proof}
This is an exercise in Schubert Calculus, see, e.g., \cite{fulton-young-tableaux, manivel}.
We need to count 2-dimensional subspaces of ~$\VV\cong\CC^{d+1}$ 
that nontrivially intersect given generic subspaces of dimensions $r+1, s+1, r+1, s+1$. 
This translates into computing the product of the corresponding special Schubert classes
in the cohomology ring $\operatorname{H}^*(\operatorname{Gr}_{2,d+1}(\CC))$. 
To be specific, the intersection number in question is equal to the coefficient of 
$s_{(d-1,d-1)}$ in the Schur function expansion of the product $(s_{(d-r-1)} s_{(d-s-1)})^2$. 
Assuming that $r\le s$ and using the Pieri rule, we see that the above number
is the number of semistandard Young tableaux of shape $(d\!-\!1,d\!-\!1)$ 
and content $({d\!-\!r\!-\!1},{d\!-\!s\!-\!1},{d\!-\!s\!-\!1},{d\!-\!r\!-\!1})$.
This number is easily seen to be equal to~${d\!-\!s}$. 
\end{proof}


The simplest instance of the above setting that provides new examples not covered before
is the case $d\!=\!3$, $r\!=\!s\!=\!1$. 
In that case, Lemma~\ref{lem:flats} asserts that four generic lines in~$\CC\PP^3$
can be pierced by exactly two lines, cf.\ Remark~\ref{rem:16-lines-Schubert}. 
For each of these lines, the corresponding quadruple of points of intersection 
gives rise to a cross-ratio.
The product of these cross-ratios yields the mixed cross-ratio of the original four lines: 

\begin{proposition}
\label{pr:tile-schubert}
Let $b$ and~$c$ be two distinct lines in~$\PP$ piercing four generic lines $a_1, a_2, a_3, a_4$. 
For $i=1,2,3,4$, let $B_i=b\cap a_i$ and $C_i=c\cap a_i$. 
We then have 
\begin{equation}
\label{eq:tile-schubert}
(a_1, a_3; a_2, a_4) 
=
(B_1, B_3; B_2, B_4)\cdot (C_1, C_3; C_2, C_4). 
\end{equation}
Consequently, the tile \eqref{eq:coherent-tile-generalized} is coherent if and only if the quadruples
$(B_1,B_2,B_3,B_4)$ and $(C_2,C_3,C_4,C_1)$ are projectively equivalent. 
\end{proposition}

\vspace{-8pt}


\begin{proof}
We will use the shorthand
$\langle \mathbf{a}, \mathbf{b}, \mathbf{c}, \mathbf{d}\rangle \stackrel{\rm def}{=} 
\operatorname{vol}(\mathbf{a}\wedge \mathbf{b}\wedge \mathbf{c}\wedge \mathbf{d})$.
Let $\bb_i$ (resp.,~$\cc_i$) be a vector in~$\VV$ representing the point~$B_i$ (resp.,~$C_i$). 
Since the points $B_1, \dots, B_4$ lie on the line~$b$, 
the vectors $\bb_1, \dots, \bb_4$ belong to a 2-dimensional subspace of~$V$.
We~can therefore rescale these four vectors so that their endpoints lie on a line.
In other words, we may assume that all the vectors $\bb_i-\bb_j$ are collinear. 
We may similarly assume that all the vectors $\cc_i-\cc_j$ are collinear. 
It follows that
\begin{align*}
(a_1, a_3; a_2, a_4) 
&=
\frac{\langle\bb_1\wedge \cc_1, \bb_2\wedge \cc_2\rangle \langle\bb_3\wedge \cc_3, \bb_4\wedge \cc_4\rangle}
       {\langle\bb_2\wedge \cc_2, \bb_3\wedge \cc_3\rangle \langle\bb_4\wedge \cc_4, \bb_1\wedge \cc_1\rangle} 
\\[-1pt]
&=
\frac{\langle\bb_1, \cc_1, \bb_2, \cc_2\rangle \langle\bb_3, \cc_3, \bb_4, \cc_4\rangle}
       {\langle\bb_2, \cc_2, \bb_3, \cc_3\rangle \langle\bb_4, \cc_4, \bb_1, \cc_1\rangle} 
\\[-1pt]
&=
\frac{\langle\bb_1, \cc_1, \bb_2-\bb_1, \cc_2-\cc_1\rangle \langle\bb_3, \cc_3, \bb_4-\bb_3, \cc_4-\cc_3\rangle}
       {\langle\bb_2, \cc_2, \bb_3-\bb_2, \cc_3-\cc_2\rangle \langle\bb_4, \cc_4, \bb_1-\bb_4, \cc_1-\cc_4\rangle}
\\[-1pt]
&=
\frac{\langle\bb_1, \cc_1, \bb_2-\bb_1, \cc_2-\cc_1\rangle \langle\bb_1, \cc_1, \bb_4-\bb_3, \cc_4-\cc_3\rangle}
       {\langle\bb_1, \cc_1, \bb_3-\bb_2, \cc_3-\cc_2\rangle \langle\bb_1, \cc_1, \bb_1-\bb_4, \cc_1-\cc_4\rangle}
       \\[-1pt]
&=
(B_1, B_3; B_2, B_4)\cdot (C_1, C_3; C_2, C_4). \qedhere
\end{align*}
\end{proof}

\newpage

\section*{Real cross-ratios}

Yet another variation of the master theorem arises from the following beautiful observation, 
see, e.g., \cite[\S5, Theorem~A]{schwerdtfeger} \cite[Exercise~3.2]{bobenko-suris-book}. 

\begin{proposition}
\label{pr:real-concyclicity}
Identify the real plane $\RR^2$ with the field of complex numbers~$\CC$ in the usual way. 
Then four distinct points in~$\RR^2$ lie on a circle if and only if their cross-ratio is real. 
\end{proposition}

We note that the order among the four points in Proposition~\ref{pr:real-concyclicity} does not matter
since all possible cross-ratios are related to each other by rational transformations that preserve reality.  


\begin{corollary}
\label{cor:concyclic-tilings}
Consider a tiling of a closed oriented surface by quadrilateral~tiles. 
Associate to the vertices in the tiling distinct points on the Euclidean plane. 
Suppose that for all tiles but one, 
the four points associated with the vertices of the tile lie on a circle.
Then the same property holds for the remaining~tile. 
\end{corollary}

\begin{proof}
This corollary is a direct consequence of  Proposition~\ref{pr:real-concyclicity}. 
The product of mixed cross-ratios associated with individual tiles is equal to~1.
If all of these cross-ratios, with the exception of one, are real, then the remaining cross-ratio must be real as well. 
\end{proof}

\begin{remark}
Corollary~\ref{cor:concyclic-tilings} immediately extends to the \emph{M\"obius plane}~$\RR^2\cup\{\infty\}$, a one-point
compactification of the Euclidean plane. (This can be seen by applying an inversion transformation.) 
On the M\"obius plane, 
straight lines are viewed as circles passing through~$\infty$. 
If we label one of the vertices of the tiling by~$\infty$, then the condition associated with a tile with labels $A, B, C, \infty$ 
will simply say that the points $A, B, C$ are collinear. 
\end{remark}

In the special case of the tiling of the sphere by six quadrilaterals 
(cf.\ Figure~\ref{fig:16-points-cube}), Corollary~\ref{cor:concyclic-tilings}  yields \emph{Miquel's theorem} reproduced below,
see, e.g., \cite[Theorem~9.21]{bobenko-suris-book} \cite[Theorem~18.5]{richter-gebert-book}. 


\begin{corollary}
\label{cor:miquel}
Let $P_1, \dots, P_6$ be distinct points on the Euclidean plane. 
If five of the six quadruples 
\begin{equation}
\label{eq:6-quadruples}
\begin{array}{l}
\{P_1, P_2, P_7, P_8\}, \{P_1, P_2, P_5, P_6\}, \{P_1, P_4, P_5, P_8\}, 
\\[.05in]
\{P_2, P_3, P_6, P_7\}, \{P_3, P_4, P_7, P_8\}, \{P_3, P_4, P_5, P_6\}
\end{array}
\end{equation}
are concyclic (i.e., lie on a circle), then the remaining quadruple is also concyclic. 
See Figure~\ref{fig:miquel}. 
\end{corollary}

\begin{figure}[ht]
\begin{center}
\includegraphics[scale=0.6, trim=0.1cm 0.1cm 0cm 0cm, clip]{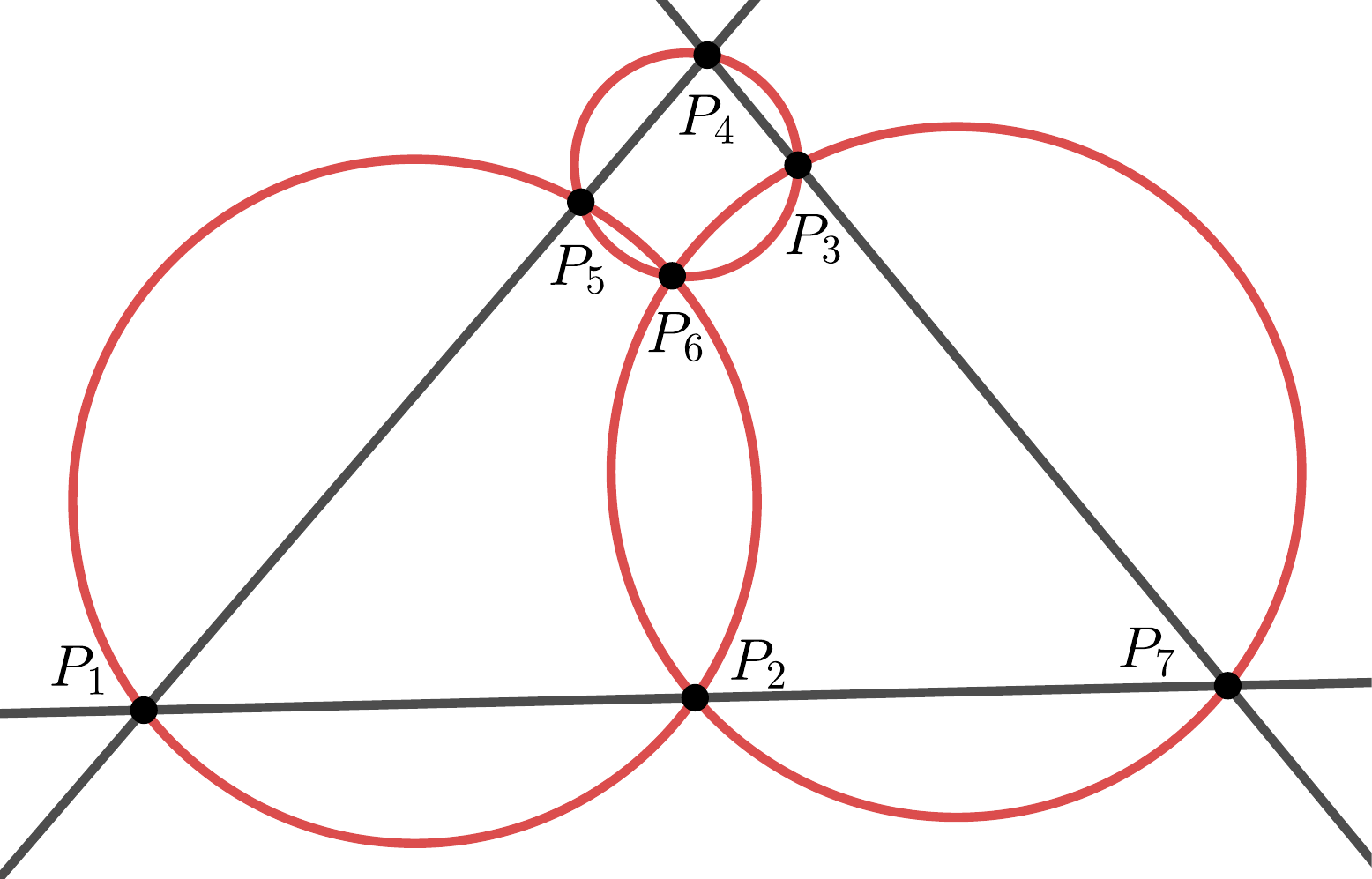}
\end{center}
\vspace{-10pt}
\caption{Miquel's theorem. Here $P_8=\infty$. 
}
\label{fig:miquel}
\end{figure}

While the proof of Miquel's theorem based on Proposition~\ref{pr:real-concyclicity}
is by no means new, 
applications of Corollary~\ref{cor:concyclic-tilings} that are based on other tilings can produce new results 
about configurations of circles and lines. 

\medskip

As we just saw, Miquel's theorem comes from the tiling of the sphere that was used to prove the Desargues theorem. 
Using instead a tiling of the torus that gave rise to the Pappus theorem, 
we obtain the following statement, cf.\ Figure~\ref{fig:nine-circles}. 

\begin{theorem}
\label{th:nine-circles}
Consider two circles on the plane passing through points $a$ and~$b$.
Draw a straight line through $a$ and denote by $P_1, P_2$ the points (other than~$a$) where this line intersects the two circles. 
Draw a straight line through $b$ and denote by $P_4, P_5$ the points (other than~$b$) where this line intersects the two circles. 
Set $P_3=(AP_4)\cap (BP_2)$ and $P_6=(AP_5)\cap (bP_1)$. 
Then the points $a, P_3, P_6, b$ are concyclic. 
\end{theorem}

\begin{proof}
This result is obtained by applying Corollary~\ref{cor:concyclic-tilings} to the tiling in Figure~\ref{fig:pappus-torus},
where we set $c=\infty$. 
\end{proof}

\begin{figure}[ht]
\begin{center}
\includegraphics[scale=0.5, trim=0.1cm 12cm 0cm 0.2cm, clip]{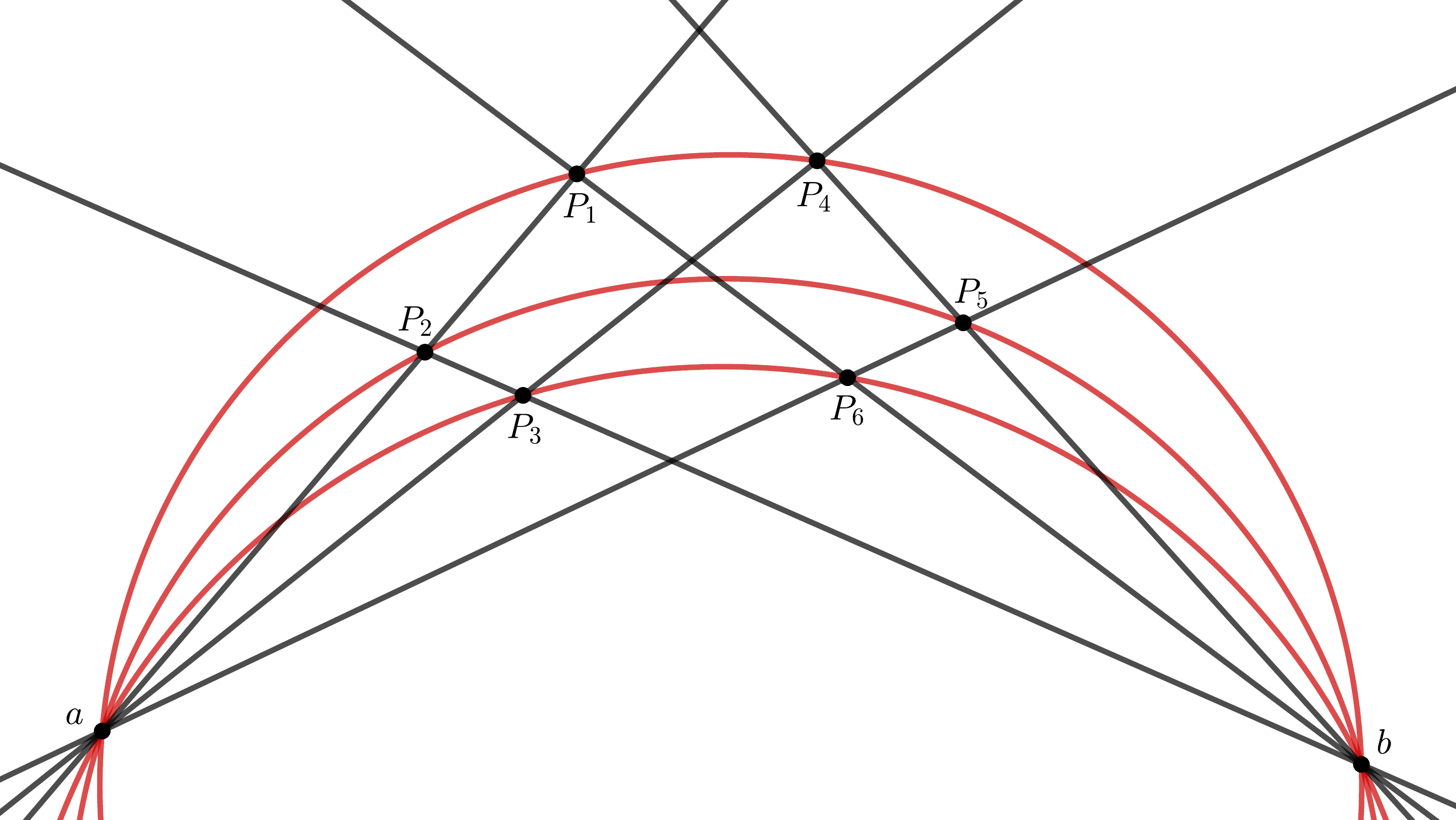}
\end{center}
\vspace{-8pt}
\caption{The configuration of circles and lines in Theorem~\ref{th:nine-circles}. 
}
\vspace{-10pt}
\label{fig:nine-circles}
\end{figure}


\clearpage

\newpage


\end{document}